\documentclass[reqno,a4paper]{siamart1116}
\usepackage[top=1.5cm, bottom=2.cm, left=1.25cm, right=1.25cm]{geometry}
\newtheorem{remark}{Remark}[section]
\usepackage{amsmath,amssymb,epsfig,float,color,mathtools,enumitem,comment}
\usepackage{arydshln}
\usepackage{booktabs}
\usepackage{graphicx,epstopdf} 
\usepackage[caption=false]{subfig} 

\newcommand{\dx}{\,\mbox{d}\bx}
\newcommand{\ds}{\,\mbox{d}s}
\newcommand{\dt}{\,\mbox{d}t}
\newcommand{\ddt}{\frac{\mbox{d}}{\dt}}
\numberwithin{equation}{section}
\numberwithin{figure}{section}
\numberwithin{theorem}{section}

\definecolor{lightblue}{rgb}{0.22,0.45,0.70}
\definecolor{mygray}{rgb}{0.7,0.7,0.7}

\newcommand\cero{\boldsymbol{0}}

\newcommand{\by}{\boldsymbol{y}}
\newcommand{\bx}{\boldsymbol{x}}
\newcommand{\bh}{\boldsymbol{h}}
\newcommand{\bn}{\boldsymbol{n}}
\newcommand{\bu}{\boldsymbol{u}}
\newcommand{\bv}{\boldsymbol{v}}
\newcommand{\bw}{\boldsymbol{w}}
\newcommand{\ff}{\boldsymbol{f}}
\newcommand\bchi{\boldsymbol{\chi}}
\newcommand\bphi{\boldsymbol{\phi}}
\newcommand\beps{\boldsymbol{\varepsilon}}
\newcommand\bsigma{\boldsymbol{\sigma}}
\newcommand\btau{\boldsymbol{\tau}}
\newcommand\bzeta{\boldsymbol{\zeta}}
\newcommand\bnabla{\boldsymbol{\nabla}}
\newcommand\bDelta{\boldsymbol{\Delta}}
\newcommand\bPi{\boldsymbol{\Pi}}

\newcommand\bL{\mathbf{L}}

\newcommand\bI{\mathbf{I}}
\newcommand\bQ{\mathbf{Q}}
\newcommand\bS{\mathbf{S}}
\newcommand\bV{\mathbf{V}}
\newcommand\bW{\mathbf{W}}
\newcommand\bZ{\mathbf{Z}}

\newcommand\bbD{\mathbb{D}}
\newcommand\bbS{\mathbb{S}}

\newcommand{\cmag}[1]{\textcolor{black}{#1}}

\allowdisplaybreaks

\title{A Lagrange multiplier-based method for Stokes-linearized poro-hyperelastic interface problems}
\author{Aparna Bansal\thanks{Department of Mathematics, Indian Institute of Technology Roorkee, Roorkee 247667, India. Email: \email{a\_bansal@ma.iitr.ac.in}.}
\and Nicol\'as A. Barnafi$^\P$\thanks{Instituto de Ingeniería Matemática y Computacional \& Facultad de Ciencias Biológicas, Pontificia Universidad Católica de Chile, Av Vicuña Mackenna 4860, Santiago, Chile. Email: \email{nicolas.barnafi@uc.cl}. Email: \email{nicolas.barnafi@uc.cl}.\\
\indent$^\P$XX and Centro de Modelamiento Matemático (CNRS IRL2807), Santiago, Chile.}
\and Dwijendra Narain Pandey\thanks{Department of Mathematics, Indian Institute of Technology Roorkee, Roorkee 247667, India. Email: \email{dwijpfma@iitr.ac.in}.}
\and Ricardo Ruiz-Baier\thanks{School of Mathematics,
Monash University, 9 Rainforest Walk, Melbourne, Victoria 3800, Australia. Email: \email{ricardo.ruizbaier@monash.edu}.}}
\date{\today}

\begin{document}

\maketitle
\begin{abstract}
We propose a model for the coupling between free fluid and a linearized poro-hyperelastic body. In this model, the Brinkman equation is employed for fluid flow in the porous medium, incorporating inertial effects into the fluid dynamics. A generalized poromechanical framework  is used,  incorporating fluid inertial effects in accordance with thermodynamic principles. We carry out the analysis of the unique solvability of the governing equations, and the existence proof relies on an auxiliary multi-valued parabolic problem. We propose a Lagrange multiplier-based mixed finite element method for its numerical approximation and show the well-posedness of both semi-and fully-discrete problems. Then, \textit{a priori} error estimates for both the semi- and fully-discrete schemes are derived. A series of numerical experiments is presented to confirm the theoretical convergence rates,  and we also employ the proposed monolithic scheme to simulate 2D physical phenomena in geophysical fluids and biomechanics of the brain function.
\end{abstract}

\section{Introduction} 
The interaction between a free-flowing fluid and a deformable porous medium presents a challenging multiphysics problem. 
The complexity arises from the disparate material properties across geometric interfaces, and examples of such challenges  widely exist in industrial problems, including groundwater flow in fractured aquifers, oil and gas extraction, and filter design.  These processes are also found in biomechanical applications such as perfused living tissues \cite{MR1909425}, 
transport of lipids and drugs in blood vessel walls \cite{MR2573342, MR4292314, MR2804646}, water transport and drug delivery in the brain \cite{MR2900644, MR2754490}, and addressing ocular diseases like glaucoma \cite{MR3272403, MR4353225} or diagnosing fibrosis in the lungs \cite{MR3454811}. 

There is an extensive literature on fluid poroelastic structure interaction (FPSI) problems, which exhibit features of both coupled  Stokes--Darcy interfacial flows \cite{MR1936102, MR2519594,  MR3653774, MR2139232} and fluid structure interaction  (FSI) \cite{MR2546594, MR2536639}. In FPSI scenarios, the behavior of the free fluid is described by the Stokes (Navier--Stokes) equations, while the flow within deformable porous media is governed by the Biot system of poroelasticity \cite{MR66874}. This system amalgamates an elasticity equation governing the deformation of the elastic porous matrix with a Darcy flow model to account for the fluid mass conservation within the pores. The coupling of the Stokes and Biot regions involves interface conditions ensuring the continuity of normal flux, the Beavers--Joseph--Saffman (BJS) slip condition for tangential velocity, the balance of forces, and continuity of normal stress. There exist  typical difficulties in solving FPSI equations, such as   nonlinearities of different types, the coupling of multiple physical fields and constitutive constraints, possibly large deformations, multiscale and inertial effects in both fluid and solid phases, among others. It is still challenging to construct and analyze stable and efficient numerical methods for this type of problems. Some of the computational and theoretical issues in FPSI are also present in simpler Darcy--Stokes equations or in FSI systems.
One of the first theoretical studies of the Stokes--Biot model is provided in \cite{MR2149168}, where well-posedness is demonstrated using semigroup methods. A numerical investigation is presented in \cite{MR2573342}, employing the variational multiscale finite element method (FEM) and proposing both monolithic and iterative partitioned methods. In \cite{MR3343599}, a non-iterative operator-splitting method is developed for the coupled Navier--Stokes--Biot model. Additionally, readers are also referred to \cite{MR3347244}, where a loosely coupled partitioned approach is utilized based on Nitsche's method. An analysis of a Lagrange multiplier formulation for imposing normal flux continuity is provided in \cite{MR3851065}, and for an extension to non-Newtonian fluids, see \cite{MR4022710}. A Stokes–Biot model with a total pressure formulation, which does not require Lagrange multipliers for the imposition of interface conditions, is studied in \cite{MR4353225}. Well-posedness for a Stokes–Biot system with a multilayered porous medium using Rothe's method is obtained in \cite{MR4332970}.

More recently, researchers have focused on utilizing the framework of Biot theory with finite strain to develop general poromechanics formulations \cite{MR3200403, MR3283808}. A thermodynamically consistent poromechanics formulation was introduced in \cite{MR3200403, MR3918635}. In particular, \cite{MR3200403} develops a model for the general case of large deformations, illustrating a mixture of fluid and solid phases coexisting at every point in the computational domain. The  nonlinearity of constitutive equations and geometric nonlinearity resulting from large deformations were avoided in \cite{MR3918635} by considering a linearized version of the aforementioned poromechanic model   under the assumption of small deformations. Such a model is known as  linearized poro-hyperelastic or generalized Biot model. The existence and uniqueness of strong and weak solutions for the generalized Biot  model are discussed in \cite{MR4547104}.

In this paper, our objective is to investigate the solvability of the Stokes-linearized poro-hyperelastic interface model. We utilize the Brinkman model for fluid flow to ensure mass conservation within the pores, integrating viscous effects into the fluid dynamics in the poroelastic region \cite{MR4253885}. Within this region, the fluid phase is strictly incompressible, while the solid phase is nearly incompressible. The weighting coefficients in the linear combination of velocities depend on the porosity of the material. This model presents a more precise alternative to the conventional Stokes--Biot model, particularly suited for scenarios involving thermodynamically consistent phenomena and varying porosity. Notably, the interface condition differs from that of the Stokes--Biot model. 

We enforce continuity of the normal velocity on the interface by utilizing a Lagrange multiplier. The resulting weak formulation fails to achieve a bound on the relative velocity in the energy norm using standard approaches in porous media. Furthermore, it incorporates the time derivative of displacement within certain non-coercive terms, thus yielding a time-dependent system that presents analytical challenges. To address these issues, we adopt an alternative formulation of mixed elasticity, with the primary variables being elastic stress and structural velocity \cite{MR2149168}. Consequently, the resulting system exhibits a structure similar to degenerate evolution saddle-point problems. Following the methodology outlined in \cite{MR2684313}, our analysis requires certain right-hand side terms to be zero, although, in typical applications, these terms may not vanish. Thus, we reframe the problem as a parabolic-type system to circumvent this constraint. Using results from classical semigroup theory for differential equations with monotone operators \cite{MR1422252}, we establish the existence of a unique solution for the parabolic system. Subsequently, we demonstrate that this solution satisfies the alternative formulation.
Furthermore, we prove the uniqueness of the solution for the original formulation and provide a stability bound. We then proceed to analyze the stability and error analysis of  semi- and fully-discrete FE approximations of the system. Specifically, we discretize the problem using finite differences in time and FEs in space. We address the numerical stability of a numerical scheme based on the Taylor--Hood FE family for the approximation of relative velocity, solid displacement, and pressure within the porous medium; both relative velocity and solid displacement are required to have a degree of approximation higher than that of the pressure \cite{MR4253885}. On the other hand, we employ Taylor--Hood, MINI, and (non-conforming) Crouzeix--Raviart FEs for fluid velocity and pressure approximation in the Stokes medium. We establish appropriate discrete inf-sup conditions  utilizing a conforming Lagrange multiplier discretization, that also contribute to ensure accuracy across non-matching grids at the Stokes-generalized Biot interface. Further, we obtain a sub-optimal convergence rate for the relative velocity and solid displacement, which is expected due to the lack of achieving an error bound in the energy norm, as mentioned earlier.  To the best of the authors' knowledge, this work represents a novel contribution to the field of theoretical and numerical partial differential equations in interface coupled problems.

\noindent\textbf{Outline of the paper.} The rest of the paper is organized as follows. Section \ref{sec:prelim} establishes preliminaries and  notations, while we derive the mathematical model along with its weak formulation in Section \ref{section2}. Section \ref{section3} is devoted to an alternative formulation, necessary for the purpose of the analysis. In Section \ref{section4}, we prove the well-posedness of both the alternative and original formulations, along with the stability bounds for the original formulation.  The semi-discrete approximation and its well-posedness analysis are developed in Section \ref{section6}. Section \ref{section7} presents the analysis for the fully discrete scheme. In Section \ref{section8},  some numerical experiments are provided to test the theoretical results regarding spatio-temporal convergence, and we also simulate (i) a typical reservoir model with real data, (ii) a scenario with large displacements in the interface and (iii) a simplified but physiologically accurate brain biomechanics problem. We conclude in Section~\ref{sec:concl} with a summary of our results and state further extensions.

\section{Notation and preliminaries}\label{sec:prelim}
Throughout this manuscript, we utilize the classical Sobolev spaces $L^2(\Omega)$ and $H^1(\Omega)$, equipped with their respective norms $\|\cdot\|_{L^2(\Omega)}$ and $\|\cdot\|_{H^1(\Omega)}$. The $L^2$-inner product is denoted as $(\cdot, \cdot)$, and for any arbitrary Hilbert space $H$, we represent the duality pairing with its dual space $H^{\prime}$ as $\langle\cdot, \cdot\rangle_{H^{\prime}, H}$. For a positive function $\psi$, we also consider the weighted Lebesgue spaces $L^2(\Omega, \psi)$, defined by the norm $\|f\|_\psi^2 = (f, f)_\psi = \int_{\Omega} f^2 \psi \dx$. Also, we use the convention of denoting scalars, vectors, and tensors as $a, \mathbf{a}$, and $\mathbb{A}$, respectively.  Finally, we define the Bochner spaces $L^p(0, T ; X)$ and $L^{\infty}(0, T ; X)$ for any Banach space 
 $X$, with norms given by $\bigl(\int_0^T\|x(s)\|_X^q \ds\bigr)^{1/q}$ and $\sup_{s \in (0, T)} \|x(s)\|_X$, respectively. We consider weak time derivatives in $W^{k, p}(0, T ; X)$, defined as $\left\{x \in L^p(0, T ; X): D^{\alpha}x \in L^p(0, T ; X) \text{ for all } n \in \mathbb{N}, \alpha \leq k\right\}$, where $1 \leq p \leq \infty$. 

For the sake of simplicity, throughout the analysis, $C$ will denote a generic positive constant independent of the mesh size $h$ but possible dependent on the model parameters. We will also abuse notation by denoting $\epsilon$ as an arbitrary constant with different values at different occurrences, arising from the use of  Young's inequality. Additionally, whenever an inequality holds for positive constants independent of the mesh size and dependent on the parameters, we will use the symbols $\lesssim$ or $\gtrsim$ and omit specific constants. The assumption of homogeneity in the boundary conditions is made to simplify the subsequent analysis, as lifting operators have already been established \cite{MR3974685}. Non-homogeneous boundary conditions are utilized in the numerical tests in section \ref{section8}. 

\section{Multiphysics formulation of the model problem} \label{section2}
Let us consider a bounded Lipschitz domain $\Omega \subset \mathbb{R}^d$, $d \in \{2, 3\}$, together with a partition into non-overlapping and connected subdomains $\Omega_{\mathrm{S}}, \Omega_{\mathrm{P}}$ representing zones occupied by a  free fluid region with flow governed by the Stokes equations and a  poroelastic material governed by the general thermodynamically consistent linearized poro-hyperelastic system, respectively. The interface between the two subdomains is denoted as $\Sigma = \partial \Omega_{\mathrm{S}} \cap \partial \Omega_{\mathrm{P}}$. The boundary of the domain $\Omega$ is separated in terms of the boundaries of two individual subdomains, that is $\partial \Omega = \Gamma_{\mathrm{S}} \cup \Gamma_{\mathrm{P}}$.

The free fluid region $\Omega_{\mathrm{S}}$ is governed by the Stokes equations, with the primary variables being the Stokes fluid velocity $\bu_f^S$ and the fluid pressure $p^{\mathrm{S}}$:
\begin{subequations}
\begin{align}
-\bnabla \cdot \bsigma_f^{\mathrm{S}}\left(\bu_f^{\mathrm{S}}, p^{\mathrm{S}}\right)=\ff_{\mathrm{S}} \quad \quad \quad \quad & \text { in } \Omega_{\mathrm{S}} \times(0, T], \label{stokes1}\\
\nabla \cdot \bu_f^{\mathrm{S}}=r_{\mathrm{S}} \quad \quad \quad \quad & \text { in } \Omega_{\mathrm{S}} \times(0, T],  \label{stokes2}
\end{align}
\end{subequations}
where $T>0$ is the final time. Here  $\beps(\bu_f^{\mathrm{S}}) = \frac12 (\bnabla \bu_f^{\mathrm{S}}+ (\bnabla \bu_f^{\mathrm{S}})^T)$ denotes the deformation strain tensor; $\bsigma_f^{\mathrm{S}}(\bu_f^{\mathrm{S}}, p^{\mathrm{S}}) = 2 \mu_f \beps(\bu_f^{\mathrm{S}}) - p^{\mathrm{S}} \mathbf{I}$, stress tensor; $\ff_{\mathrm{S}}\in \mathbf{H}^{-1}(\Omega_{\mathrm{S}})$, external load; $r_{\mathrm{S}} \in L^2(\Omega_{\mathrm{S}})$, fluid source/sink; $\mu_f$, fluid viscosity.   

The poroelastic region $\Omega_{\mathrm{P}}$ is governed by the linearized poro-hyperelastic model (which includes viscoelastic properties), with the primary variables being the fluid velocity $\bu_f^{\mathrm{P}}$, interstitial pressure $p^{\mathrm{P}}$, solid displacement $\by_s^{\mathrm{P}}$, and solid velocity $\bu_{s}^{\mathrm{P}}$:
\begin{subequations}
\begin{align}
\rho_{f} \phi \partial_t \bu_f^{\mathrm{P}}- \bnabla \cdot \bsigma_f^{\mathrm{P}} \left(\bu_f^{\mathrm{P}},p^{\mathrm{P}}\right) -p^{\mathrm{P}} \nabla \phi+\phi^2 \kappa_{f}^{-1}\left(\bu_f^{\mathrm{P}}-\bu_s^{\mathrm{P}}\right) &=  \rho_{f} \phi \ff_{\mathrm{P}}+\theta \bu_f^{\mathrm{P}} &  \text { in } \Omega_{\mathrm{P}} \times(0, T], \label{poro1}\\
(1-\phi)^2 K^{-1} \partial_t p^{\mathrm{P}}+\nabla \cdot\left(\phi \bu_f^{\mathrm{P}}+(1-\phi) \bu_s^{\mathrm{P}}\right) & =\rho_{f}^{-1} \theta & \text { in } \Omega_{\mathrm{P}} \times(0, T], \label{poro2} \\                               \rho_{s}(1-\phi) \partial_t \bu_{s}^{\mathrm{P}}-\bnabla \cdot \bsigma_s^{\mathrm{P}} \left(\by_s^{\mathrm{P}}, p^{\mathrm{P}}\right) -p^{\mathrm{P}} \nabla(1-\phi)-\phi^2 {\kappa}_{f}^{-1}\left(\bu_f^{\mathrm{P}}-\bu_s^{\mathrm{P}}\right) &=\rho_{s}(1-\phi) \ff_{\mathrm{P}} & \text { in } \Omega_{\mathrm{P}} \times(0, T],\label{poro3} \\
\bu_s^{\mathrm{P}} &=\partial_t \by_s^{\mathrm{P}} &  \text { in } \Omega_{\mathrm{P}} \times(0, T] \label{poro4}.
\end{align}
\end{subequations}
The first equation is the conservation of momentum for the fluid phase, which turns out to be a generalized Stokes law which incorporates the Brinkman effect; the second equation represents mass conservation; the third one is the conservation of momentum of the solid phase and the last one relates solid displacement and velocity. The relevant parameters are given by: $\phi=\phi(\bx)$, porosity; $\rho_{f}, \rho_{s}$, fluid/solid density; $\mu_f$, fluid viscosity; $\kappa$, permeability tensor; $\ff_{\mathrm{P}} \in \mathbf{L}^{2}(\Omega_{\mathrm{P}})$, external load; $\theta \in L^2(\Omega_{\mathrm{P}})$, fluid source/sink; $K$, bulk modulus and $\lambda_p, \mu_p$, Lam\'{e} parameters. 
\begin{remark}
    The parameters $\rho_s, \rho_f, \mu_f, \lambda_p, \nu_p$ are assumed to be positive constants. 
\end{remark}
Let us now define stress tensors in the poroelastic sub-domain as 
\begin{equation}
\bsigma_f^{\mathrm{P}}\left(\bu_f^{\mathrm{P}},p^{\mathrm{P}}\right) := 2 \mu_f \phi \beps\left(\bu_f^{\mathrm{P}}\right) - \phi p^{\mathrm{P}} \mathbf{I}, \quad 
 \bsigma^{\mathrm{P}} \left(\by_s^{\mathrm{P}}, p^{\mathrm{P}}\right) := 2 \mu_p  \beps\left(\by_s^{\mathrm{P}}\right) + \lambda_p \nabla \cdot \by_s^{\mathrm{P}} \mathbf{I}, \quad 
\bsigma_s^{\mathrm{P}} \left(\by_s^{\mathrm{P}}, p^{\mathrm{P}}\right) := \bsigma^{\mathrm{P}} - (1-\phi) p^{\mathrm{P}} \mathbf{I} \label{stress3} .
\end{equation}

We rewrite the aforementioned generalized Biot model using the relative velocity between the fluid and solid phases, expressed as $\bu_r=\bu_f^{\mathrm{P}} - \bu_{s}^{\mathrm{P}}$, and also incorporate equations \eqref{poro1} and \eqref{poro3} to transform \eqref{poro3} into the total momentum equation:
\begin{align*}
 \rho_f \phi (\partial_{t} \bu_r^{\mathrm{P}}+\partial_t \bu_{s}^{\mathrm{P}})-\bnabla \cdot \bsigma_f^{\mathrm{P}} \left(\bu_r^{\mathrm{P}} + \bu_s^{\mathrm{P}},p^{\mathrm{P}}\right) - p^{\mathrm{P}} \nabla \phi + \phi^2 \kappa_{f}^{-1} \bu_r-\theta( \bu_{s}^{\mathrm{P}}+ \bu_r^{\mathrm{P}})&= \rho_f \phi \ff_{\mathrm{P}}&  \text { in } \Omega_{\mathrm{P}} \times(0, T], \\
 (1-\phi)^2 {K}^{-1} \partial_t {p^{\mathrm{P}}}+\partial_t\left(\nabla \cdot \by_s^{\mathrm{P}}\right)+\nabla \cdot\left(\phi \bu_r^{\mathrm{P}}\right)&={\rho}_f^{-1} \theta&  \text { in } \Omega_{\mathrm{P}} \times(0, T], \\
 \rho_{s}(1-\phi) \partial_t \bu_s^{\mathrm{P}}+ \rho_f \phi \partial_{t} \bu_r^{\mathrm{P}}+ \rho_f \phi \partial_t \bu_s^{\mathrm{P}}-\bnabla \cdot \bsigma_s^{\mathrm{P}} \left(\by_s^{\mathrm{P}}, p^{\mathrm{P}}\right)  & \\ -\bnabla \cdot \bsigma_f^{\mathrm{P}} \left(\bu_r^{\mathrm{P}} + \bu_s^{\mathrm{P}},p^{\mathrm{P}}\right)  -\theta( \bu_{s}^{\mathrm{P}}+ \bu_r^{\mathrm{P}}) & =\rho_{s}(1-\phi) \ff_{\mathrm{P}} + \rho_f \phi \ff_{\mathrm{P}}
&  \text { in } \Omega_{\mathrm{P}} \times(0, T], \\
 \bu_{s}^{\mathrm{P}}&=\partial_t \by_s^{\mathrm{P}} &  \text { in } \Omega_{\mathrm{P}} \times(0, T].
\end{align*}
Henceforth, we adopt the notation $\bsigma_f^{\mathrm{S}}$, $\bsigma_f^{\mathrm{P}}$ and $\bsigma_s^{\mathrm{P}}$ to denote $\bsigma_f^{\mathrm{S}}\left(\bu_{f,} p^{\mathrm{S}}\right)$, $\bsigma_f^{\mathrm{P}} \left(\bu_r^{\mathrm{P}} + \bu_s^{\mathrm{P}},p^{\mathrm{P}}\right)$ and $\bsigma_s^{\mathrm{P}} \left(\by_s^{\mathrm{P}}, p^{\mathrm{P}}\right)$, respectively. The resulting model is then defined as
\begin{subequations}
\begin{align}
 \rho_f \phi (\partial_{t} \bu_r^{\mathrm{P}}+\partial_t \bu_{s}^{\mathrm{P}})-\bnabla \cdot \bsigma_f^{\mathrm{P}} - p^{\mathrm{P}} \nabla \phi + \phi^2 \kappa_{f}^{-1} \bu_r-\theta( \bu_{s}^{\mathrm{P}}+ \bu_r^{\mathrm{P}}) &= \rho_f \phi \ff_{\mathrm{P}}&  \text { in } \Omega_{\mathrm{P}} \times(0, T], \label{relativeporo2} \\
(1-\phi)^2 {K}^{-1} \partial_t {p^{\mathrm{P}}}+\partial_t\left(\nabla \cdot \by_s^{\mathrm{P}}\right)+\nabla \cdot\left(\phi \bu_r^{\mathrm{P}}\right)&={\rho}_f^{-1} \theta&  \text { in } \Omega_{\mathrm{P}} \times(0, T], \label{relativeporo3} \\
 \rho_f \phi \partial_t \bu_r^{\mathrm{P}}+\rho_p \partial_t \bu_{s}^{\mathrm{P}}-\nabla \cdot\bsigma_f^{\mathrm{P}} -\bnabla \cdot \bsigma_s^{\mathrm{P}} -\theta \bu_r-\theta \bu_{s}^{\mathrm{P}}&=\rho_p \ff_{\mathrm{P}}&  \text { in } \Omega_{\mathrm{P}} \times(0, T], \label{relativeporo1} \\
\bu_{s}^{\mathrm{P}}&=\partial_t \by_s^{\mathrm{P}} &  \text { in } \Omega_{\mathrm{P}} \times(0, T], \label{relativeporo4}
\end{align}
\end{subequations}
where $\rho_p = \rho_s(1-\phi)+ \rho_f \phi$ denotes the density of the saturated porous medium. 
This system is complemented by the following set of boundary conditions
\[
\bu_f^{\mathrm{S}}= {\cero} \quad  \text {on}\quad \Gamma_{\mathrm{S}} \times(0, T], \qquad  
\by_s^{\mathrm{P}} = \bu_r^{\mathrm{P}} = {\cero} \quad \text {on}\quad \Gamma_{\mathrm{P}} \times(0, T].
\]
Non-homogeneous displacement and velocity conditions can be handled in a standard way by adding suitable lifting operators of the boundary data. The interface conditions on the fluid-poroelastic interface $\Sigma$ consist of mass conservation \eqref{int1}, fluid conservation \eqref{int2}, balance of contact forces \eqref{int3}, and the Beavers--Joseph--Saffman (BJS) condition modeling slip with friction \eqref{int4}:
\begin{subequations}\begin{align}
\bu_f^{\mathrm{S}}\cdot \bn_{\mathrm{S}}+\left(\partial_t \by_s^{\mathrm{P}}+\bu_r^{\mathrm{P}}\right) \cdot \bn_{\mathrm{P}}&= 0 & \text { on } \Sigma \times(0, T], \label{int1} \\
-\left(\bsigma^{\mathrm{S}} \bn_{\mathrm{S}}\right) \cdot \bn_{\mathrm{S}} & = -\left(\bsigma_f^{\mathrm{P}} \bn_{\mathrm{P}}\right) \cdot \bn_{\mathrm{P}}  & \text { on } \Sigma \times(0, T], \label{int2}\\
\bsigma_f^{\mathrm{S}} \bn_{\mathrm{S}}+\bsigma_f^{\mathrm{P}} \bn_{\mathrm{P}} +\bsigma_s^{\mathrm{P}} \bn_{\mathrm{P}}&={\cero}  & \text { on } \Sigma \times(0, T], \label{int3}\\
-\left(\bsigma_f^{\mathrm{S}} \bn_{\mathrm{S}}\right) \cdot \btau_{f, j}&=\mu_{f}
  \alpha_{\mathrm{BJS}} \sqrt{Z_j^{-1}}\left(\bu_f^{\mathrm{S}}-{\partial_t \by_s^{\mathrm{P}}}\right) \cdot \btau_{f, j} & \text { on } \Sigma \times(0, T] \label{int4},
\end{align}\end{subequations}
where $\bn_{\mathrm{S}}$ and $\bn_{\mathrm{P}}$ are the outward unit normal vectors to $\Omega_{\mathrm{S}}$, and $\Omega_{\mathrm{P}}$, respectively, $\btau_{f, j}, 1 \leq j \leq d-1$, is an orthogonal system of unit tangent vectors on $\Sigma$, we denote $Z_j=\left(\kappa_{f} \btau_{f, j}\right) \cdot \btau_{f, j}$, 
and $\alpha_{\mathrm{BJS}} \geq 0$ is an experimentally determined friction coefficient.

As the Brinkman equation is used for porous media, the viscous term with respect to relative velocity in the porous medium becomes significant. When we balance the conservation laws at the interface, terms such as 
\[\int_{\Sigma} P_T \bsigma_f^{\mathrm{P}} \bn_{\mathrm{P}} \bv_r^{\mathrm{P}} \ds,\]
(where $P_T = I - \bn_{\mathrm{P}} \otimes \bn_{\mathrm{P}}$ represents the projection onto the tangent space), will occur on the interface. In order to simplify the forthcoming analysis, we opt for requiring the following hypotheses:  
\begin{enumerate}
\item  $(P_T \bv_r^{\mathrm{P}})|_\Sigma = \cero \qquad \qquad $   (Dirichlet boundary condition along tangential component),  
\item $(P_T \bsigma_f^{\mathrm{P}} \bn_{\mathrm{P}})|_\Sigma = \cero \qquad~~ $ (Neumann boundary condition along tangential component, using that $P_T^t = P_T$),
\item $\bsigma_f^{\mathrm{P}}  \left(\bu_r^{\mathrm{P}} + \bu_s^{\mathrm{P}},p^{\mathrm{P}}\right) \bn_{\mathrm{P}} = -\alpha_R P_T \bv_r^{\mathrm{P}}$ on $\Sigma$ \quad (Robin boundary condition along tangential component), 
\end{enumerate}
where $\alpha_R$ is a positive constant. In this work we employ option 2 on the interface.
 We further set the initial conditions
$$
p^{\mathrm{P}}(\bx,0)=p^{\mathrm{P},0}\left(\bx\right), \quad 
\by_s^{\mathrm{P}}(\bx,0)=\by_{s,0}\left(\bx\right), \quad 
\bu_r^{\mathrm{P}}(\bx,0)=\bu_{r,0}\left(\bx\right), \quad 
\bu_s^{\mathrm{P}}(\bx,0)=\bu_{s,0}\left(\bx\right).
$$

\section{Weak formulation}\label{section3}
We consider the following functional spaces
\begin{gather*}
\mathbf{V}_f=\left\{\bu_f^{\mathrm{S}}\in \mathbf{H}^1\left(\Omega_{\mathrm{S}}\right): \bu_f^{\mathrm{S}}= \cero  \text { on } \Gamma_{\mathrm{S}}\right\}, \quad W_f=L^2\left(\Omega_{\mathrm{S}}\right),  \quad 
\mathbf{V}_{r}=\left\{\bu_r^{\mathrm{P}}  \in \mathbf{H}^1\left(\Omega_{\mathrm{P}} \right): \bu_r= \cero  \text { on } \Gamma_{\mathrm{P}}\right\}, \\ W_p=L^2\left(\Omega_{\mathrm{P}}\right),  \quad 
\mathbf{V}_{s}=\left\{\by_s^{\mathrm{P}} \in \mathbf{H}^1\left(\Omega_{\mathrm{P}}\right): \by_s^{\mathrm{P}}= \cero  \text { on } \Gamma_{\mathrm{P}}\right\}, \quad 
\mathbf{W}_s=\mathbf{L}^2\left(\Omega_{\mathrm{P}}\right),
\end{gather*}
endowed with the standard norms. The weak formulation of the Stokes model reads: find $(\bu_f^{\mathrm{S}},p^{\mathrm{S}}  ) \in \mathbf{V}_f \times W_f $ such that 
\begin{align*}
 2 \mu_f \left( \beps(\bu_f^{\mathrm{S}}), \beps(\bv_f^{\mathrm{S}})\right) -\left(p^{\mathrm{S}}, \nabla \cdot \bv_f^{\mathrm{S}} \right)  - \langle \bsigma_f^{\mathrm{S}} \bn_{\mathrm{S}}, \bv_f^{\mathrm{S}} \rangle_\Sigma & = \langle \ff_{\mathrm{S}}, \bv_f^{\mathrm{S}}\rangle_{\Omega_{\mathrm{S}}} \qquad \forall \bv_f^\mathrm{S} \in \bV_f,\\
\left(q^{\mathrm{S}}, \nabla \cdot\bu_f^{\mathrm{S}}\right) &= \left(r_{\mathrm{S}}, q^{\mathrm{S}}\right) \qquad \forall q^{\mathrm{S}} \in W_f,
\end{align*}
  where $\langle\cdot,\cdot\rangle_D$ denotes the duality pairing with respect to the $L^2(D)$ inner product. 

On the other hand, the weak formulation for the poromechanics reads: find $(\by_s^{\mathrm{P}}, \bu_r^{\mathrm{P}}, p^{\mathrm{P}},  \bu_s^{\mathrm{P}})$ in $\mathbf{V}_s\times \mathbf{V}_r \times W_p \times \mathbf{W}_s $:
\begin{align*} 
\left(\rho_{f} \phi \partial_t \bu_r^{\mathrm{P}}, \bv_r^{\mathrm{P}}\right)+ \left(\rho_{f} \phi \partial_t \bu_s^{\mathrm{P}}, \bv_r^{\mathrm{P}}\right) - \langle \bsigma_f^{\mathrm{P}} \bn_{\mathrm{P}}, \bv_r^{\mathrm{P}} \rangle_\Sigma +2 \mu_f \left(\phi \beps(\bu_r^{\mathrm{P}}), \beps\left(\bv_r^{\mathrm{P}}\right)\right)+2 \mu_f \left(\phi \beps\left(\bu_s^{\mathrm{P}}\right), \beps\left(\bv_r^{\mathrm{P}}\right)\right) &\\ 
+(\phi^2 \kappa_{f}^{-1} \bu_r^{\mathrm{P}}, \bv_r^{\mathrm{P}}) -\left(p^{\mathrm{P}}, \nabla \cdot\left(\phi \bv_r^{\mathrm{P}}\right)\right) -\left(\theta \bu_r^{\mathrm{P}}, \bv_r^{\mathrm{P}}\right)  -\left(\theta \bu_s^{\mathrm{P}}, \bv_r^{\mathrm{P}}\right)  &= (\rho_{f} \phi \ff_{\mathrm{P}}, \bv_r^{\mathrm{P}}),\\
\left({(1-\phi)^2}{K^{-1}} \partial_t p^{\mathrm{P}}, q^{\mathrm{P}}\right)+\left(q^{\mathrm{P}}, \nabla \cdot \bu_s^{\mathrm{P}}\right)+\left(q^{\mathrm{P}}, \nabla \cdot\left(\phi \bu_r^{\mathrm{P}}\right)\right) & =\left(\rho_{f}^{-1} \theta, q^{\mathrm{P}}\right), \\
\left(\rho_{f} \phi \partial_t \bu_r^{\mathrm{P}}, \bw_s^{\mathrm{P}}\right)+ \left(\rho_{p}\partial_t \bu_s^{\mathrm{P}}, \bw_s^{\mathrm{P}}\right) - \langle \bsigma_f^{\mathrm{P}} \bn_{\mathrm{P}}, \bw_s^{\mathrm{P}} \rangle_\Sigma +2 \mu_f \left(\phi \beps(\bu_r^{\mathrm{P}}), \beps\left(\bw_s^{\mathrm{P}}\right)\right) +2 \mu_f \left(\phi \beps\left(\bu_s^{\mathrm{P}}\right), \beps\left(\bw_s^{\mathrm{P}}\right)\right)&\\ 
  - \langle \bsigma_{s}^{\mathrm{P}} {\bn_{\mathrm{P}}}, \bw_s^{\mathrm{P}} \rangle_\Sigma +2 \mu_p \left(\beps\left(\by_s^{\mathrm{P}}\right), \beps\left(\bw_s^{\mathrm{P}}\right)\right) + \lambda_p \left(\nabla \cdot \by_s^{\mathrm{P}}, \nabla \cdot \bw_s^{\mathrm{P}} \right) 
  -\left(p^{\mathrm{P}}, \nabla \cdot \bw_s^{\mathrm{P}} \right) -\left(\theta \bu_r^{\mathrm{P}}, \bw_s^{\mathrm{P}}\right)& \\
 - \left(\theta \bu_s^{\mathrm{P}}, \bw_s^{\mathrm{P}}\right)&= (\rho_p \ff_{\mathrm{P}}, \bw_s^{\mathrm{P}}), \\
-\rho_p \left(\partial_t \by_s^{\mathrm{P}}, \bv_s^{\mathrm{P}}\right)+ \rho_p \left(\bu_s^{\mathrm{P}}, \bv_s^{\mathrm{P}}\right)  &=0,
\end{align*}
for every test function $\left(\bw_s^{\mathrm{P}}, \bv_r^{\mathrm{P}}, q^{\mathrm{P}}, \bv_{s}^P \right)$ in $\mathbf{V}_s\times \mathbf{V}_r \times W_p  \times \mathbf{W}_s $,  with initial conditions 
\[p^{\mathrm{P}}(0)=p^{\mathrm{P},0}, \quad 
\by_s^{\mathrm{P}}(0)=\by_{s,0}, \quad 
\bu_r^{\mathrm{P}}(0)=\bv_{r,0}, \quad 
\bu_s^{\mathrm{P}}(0)=\bv_{s,0},\] 
and we note that the fourth equation is multiplied by $\rho_p$ to maintain the symmetry of the block system.
We now define, for all  $\bu_f^{\mathrm{S}}, \bv_f^{\mathrm{S}} \in \mathbf{V}_f, \bu, \bv \in \mathbf{H}^1(\Omega_{\mathrm{P}}), \by_s^{\mathrm{P}}, \bw_s^{\mathrm{P}} \in \mathbf{V}_s$, the operators and associated bilinear forms related to the Stokes, Brinkman,   and elasticity operators, respectively:
\begin{align*}
\mathcal{A}_f^{\mathrm{S}} : \mathbf{V}_f \rightarrow \mathbf{V}_f', \qquad \langle \mathcal{A}_f^{\mathrm{S}} \bu_f^{\mathrm{S}}, \bv_f^{\mathrm{S}} \rangle = a_f^{\mathrm{S}}\left(\bu_f^{\mathrm{S}}, \bv_f^{\mathrm{S}}\right) & \coloneqq \left(2 \mu_f \beps (\bu_f^{\mathrm{S}}), \beps\left(\bv_f^{\mathrm{S}}\right)\right)_{\Omega_{\mathrm{S}}},\\
\mathcal{A}_f^{\mathrm{P}} : \mathbf{H}^1(\Omega_{\mathrm{P}}) \rightarrow \mathbf{H}^{-1}(\Omega_{\mathrm{P}}), \qquad \langle \mathcal{A}_f^{\mathrm{P}} \bu, \bv \rangle = a_{f}^{\mathrm{P}}\left(\bu,\bv\right) & \coloneqq \left(2 \mu_f \phi \beps \left( \bu \right), \beps\left( \bv \right)\right)_{\Omega_{\mathrm{P}}}, \\
\mathcal{A}_s^{\mathrm{P}} : \mathbf{V}_s \rightarrow \mathbf{V}_s', \qquad \langle \mathcal{A}_s^{\mathrm{P}} \by_s^{\mathrm{P}}, \bw_s^{\mathrm{P}} \rangle = a_s^{\mathrm{P}}\left(\by_s^{\mathrm{P}}, \bw_s^{\mathrm{P}}\right) & \coloneqq \left(2 \mu_p \beps\left(\by_s^{\mathrm{P}}\right), \beps\left(\bw_s^{\mathrm{P}}\right)\right)_{\Omega_{\mathrm{P}}}+\left(\lambda_p \nabla \cdot \by_s^{\mathrm{P}}, \nabla \cdot \bw_s^{\mathrm{P}}\right)_{\Omega_{\mathrm{P}}}.
\end{align*}
In addition, for all ${q}_{\mathrm{S}} \in W_{f}, {q}_{\mathrm{P}} \in W_{p}, \bv_f^{\mathrm{S}} \in \mathbf{V}_f, \bw_s^{\mathrm{P}} \in \bv_s^{\mathrm{P}}, \bv_r^{\mathrm{P}} \in \bV_r, \bw, \bzeta \in \bW_s$,   
let us define the following bilinear forms  
\begin{align*}
 \mathcal{B}^{\mathrm{S}} : \bV_f \rightarrow W_{s}', \quad \langle \mathcal{B}^{\mathrm{S}} \bv_f^{\mathrm{S}}, q^{\mathrm{S}} \rangle = b^{\mathrm{S}}(\bv_f^{\mathrm{S}}, q^{\mathrm{S}}) & \coloneqq -(\nabla \cdot \bv_f^{\mathrm{S}}, q^{\mathrm{S}}),\\
  \mathcal{B}_s^{\mathrm{P}} : \bv_s^{\mathrm{P}} \rightarrow W_{p}', \quad \langle \mathcal{B}_s^{\mathrm{P}} \bw_s^{\mathrm{P}}, q^{\mathrm{P}} \rangle = b^{\mathrm{P}}_s(\bw_s^{\mathrm{P}}, q^{\mathrm{P}}) &\coloneqq -(\nabla \cdot \bw_s^{\mathrm{P}}, q^{\mathrm{P}}),\\
 \mathcal{B}_f^{\mathrm{P}} : \bV_r \rightarrow W_{p}', \quad \langle \mathcal{B}_f^{\mathrm{P}} \bv_r^{\mathrm{P}}, q^{\mathrm{P}} \rangle = b^{\mathrm{P}}_f(\bv_r^{\mathrm{P}}, q^{\mathrm{P}}) &\coloneqq -(\nabla \cdot (\phi\bv_r^{\mathrm{P}}), q^{\mathrm{P}}),\\
\mathcal{M}_{\xi} : \bW_s \to \bW_s', \quad \langle \mathcal{M}_{\xi} \bw,\bzeta \rangle = m_{\xi}(\bw,\bzeta) &\coloneqq (\xi \bw,\bzeta).
\end{align*}
The interface terms present in the weak formulation are collected into 
$$
I_{\Sigma} \coloneqq  - \langle \bsigma_f^{\mathrm{S}} \bn_{\mathrm{S}}, \bv_f^{\mathrm{S}} \rangle_\Sigma - \langle \bsigma_f^{\mathrm{P}} \bn_{\mathrm{P}}, \bw_s^{\mathrm{P}} \rangle_\Sigma  - \langle \bsigma_{s}^{\mathrm{P}} \bn_{\mathrm{P}}, \bw_s^{\mathrm{P}} \rangle_\Sigma - \langle \bsigma_f^{\mathrm{P}} \bn_{\mathrm{P}}, \bv_r^{\mathrm{P}} \rangle_\Sigma.
$$
Using the interface condition \eqref{int2} we set
$
\lambda \coloneqq -\left(\bsigma^{\mathrm{S}} \bn_{\mathrm{S}}\right) \cdot \bn_{\mathrm{S}}= -\left(\bsigma_f^{\mathrm{P}} \bn_{\mathrm{P}}\right) \cdot \bn_{\mathrm{P}} \text{ on } \Sigma$, 
which is used as a Lagrange multiplier to impose   mass conservation on the interface. Utilizing the BJS condition and the balance of stress, we obtain 
\begin{align*}
I_{\Sigma}& =  - \int_{\Sigma} (\bsigma_f^{\mathrm{P}} \bn_{\mathrm{P}})\bn_{\mathrm{P}} (\bn_{\mathrm{S}} \cdot \bv_f^{\mathrm{S}} +\bn_{\mathrm{P}} \cdot \bv_r^{\mathrm{P}} + \bn_{\mathrm{P}} \cdot \bw_s^{\mathrm{P}}) \ds  - \int_{\Sigma} (\bsigma_f^{\mathrm{P}} \bn_{\mathrm{P}})\btau_{f,j} (\btau_{f,j} \cdot \bv_r^{\mathrm{P}}) \ds \\
& \quad - \int_{\Sigma} (\bsigma_f^{\mathrm{S}}  \bn_{\mathrm{S}})\btau_{f,j} \left(\bv_f^{\mathrm{S}}-\bw_s^{\mathrm{P}}\right) \cdot \btau_{f, j} \ds \\
& =: a_{\mathrm{BJS}}\left(\bu_f^{\mathrm{S}}, \partial_t \by_s^{\mathrm{P}} ; \bv_f^{\mathrm{S}}, \bw_s^{\mathrm{P}}\right)+b_{\Gamma}\left(\bu_f^{\mathrm{S}}, \bu_r^{\mathrm{P}}, \bw_s^{\mathrm{P}} ; \lambda\right) \quad \forall \bu_f^{\mathrm{S}}, \bv_f^{\mathrm{S}} \in \mathbf{V}_f,  \by_s^{\mathrm{P}}, \bw_s^{\mathrm{P}} \in \mathbf{V}_s,
\end{align*}
where in the last step we have used   the hypothesis $P_T \bsigma_f^{\mathrm{P}} \bn_{\mathrm{P}} = \cero$ on $\Sigma$, together with the definitions 
\begin{align*}
a_{\mathrm{BJS}}\left( \bu_f^{\mathrm{S}}, \by_s^{\mathrm{P}} ; \bv_f^{\mathrm{S}}, \bw_s^{\mathrm{P}}\right) & \coloneqq \sum_{j=1}^{d-1}\int_{\Sigma}\alpha_{\mathrm{BJS}} \sqrt{Z_j^{-1}}\left(\bu_f^{\mathrm{S}}-\by_s^{\mathrm{P}}\right) \cdot \btau_{f, j}\left(\bv_f^{\mathrm{S}}-\bw_s^{\mathrm{P}}\right) \cdot \btau_{f, j}\ds, \\
b_{\Gamma}\left(\bu_f^{\mathrm{S}}, \bu_r^{\mathrm{P}}, \bw_s^{\mathrm{P}} ; \mu\right) & \coloneqq \left\langle\bu_f^{\mathrm{S}}\cdot \bn_{\mathrm{S}}+\left(\bu_r^{\mathrm{P}}+\bw_s^{\mathrm{P}}\right) \cdot \bn_{\mathrm{P}}, \mu\right\rangle_{\Sigma} .
\end{align*}
 We further define
$$
\left|\bu_f^{\mathrm{S}}-\by_s^{\mathrm{P}}\right|_{\mathrm{BJS}}^2  \coloneqq a_{\mathrm{BJS}}\left(\bu_f^{\mathrm{S}}, \by_s^{\mathrm{P}} ; \bu_f^{\mathrm{S}}, \by_s^{\mathrm{P}}\right) 
 =\sum_{j=1}^{d-1} \mu_f \alpha_{\mathrm{BJS}}\|Z_j^{-1 / 4}\left(\bu_f^{\mathrm{S}}-\by_s^{\mathrm{P}}\right) \cdot \btau_{f, j}\|_{0,\Sigma}^2 .
$$
Note that for $b_{\Gamma}$ to be well-defined, it is necessary that $\lambda \in \Lambda = \left(\left.\mathbf{V}_s \cdot \bn_{\mathrm{P}}\right|_{\Sigma}\right)^{\prime}$, which is the space denoted as $H_{00}^{-1 / 2}(\Sigma)$. This space contains all $H^{-1/2}$ distributions that vanish at $\Gamma_{\mathrm{P}} = \partial \Omega_{\mathrm{P}} \setminus \Sigma$, and $\lambda= -(\bsigma_f^{\mathrm{P}} \bn_{\mathrm{P}}) \cdot \bn_{\mathrm{P}}$ on $\Sigma$. However, $-(\bsigma_f^{\mathrm{P}} \bn_{\mathrm{P}}) \cdot \bn_{\mathrm{P}}$ is not well defined on $\Gamma_{\mathrm{P}}$. Hence we use the Lagrange multiplier space $\Lambda=H^{-1 / 2}\left(\Sigma\right)$.
With the bilinear forms above we define the following operators
\begin{align*}
 \mathcal{A}_{fs}^{\mathrm{BJS}} : \mathbf{H}^1(\Omega_{\mathrm{S}}) \rightarrow \mathbf{H}^{1/2}(\Sigma), \quad (\mathcal{A}_{fs}^{\mathrm{BJS}} \bu_f^{\mathrm{S}}, \bw_s^{\mathrm{P}} )_{\Sigma} &=  a_{\mathrm{BJS}}(\bu_f^{\mathrm{S}}, \cero; \cero, \bw_s^{\mathrm{P}}) , \\ 
  \mathcal{A}_{ss}^{\mathrm{BJS}}  : \mathbf{H}^1(\Omega_{\mathrm{P}}) \rightarrow \mathbf{H}^{1/2}(\Sigma), \quad  (\mathcal{A}_{ss}^{\mathrm{BJS}} \by_s^{\mathrm{P}},\bw_s^{\mathrm{P}} )_{\Sigma}&=  a_{\mathrm{BJS}}(\cero,\by_s^{\mathrm{P}};\cero, \bw_s^{\mathrm{P}}),\\ 
 \mathcal{A}_{ff}^{\mathrm{BJS}}  : \mathbf{H}^1(\Omega_{\mathrm{S}}) \rightarrow \mathbf{H}^{1/2}(\Sigma), \quad  ( \mathcal{A}_{ff}^{\mathrm{BJS}}\bu_f^{\mathrm{S}},\bv_f^{\mathrm{S}})_{\Sigma}&=  a_{\mathrm{BJS}}(\bu_f^{\mathrm{S}}, \cero; \bv_f^{\mathrm{S}}, \cero), \\   
 \mathcal{B}_{f,\Gamma} : \mathbf{V}_f \rightarrow {H}^{1/2}(\Sigma), \quad  \langle \mathcal{B}_{f,\Gamma} \bv_f^{\mathrm{S}}, \mu \rangle_{\Sigma} &= b_{\Gamma}\left( \bv_f^{\mathrm{S}}, \cero,\cero; \mu \right), \\
 \mathcal{B}_{p,\Gamma} : \mathbf{V}_r  \rightarrow {H}^{1/2}(\Sigma), \quad    \langle \mathcal{B}_{p,\Gamma} \bv_r^{\mathrm{P}}, \mu \rangle_{\Sigma} &= b_{\Gamma}\left( \cero, \bv_r^{\mathrm{P}},\cero; \mu \right),\\ 
 \mathcal{B}_{s,\Gamma} : \mathbf{V}_s \rightarrow {H}^{1/2}(\Sigma), \quad \langle \mathcal{B}_{e,\Gamma}\bw_s^{\mathrm{P}}, \mu \rangle_{\Sigma} &= b_{\Gamma}\left(\cero,\cero, \bw_s^{\mathrm{P}}; \mu \right),
\end{align*}
where $( \cdot,\cdot )_{\Sigma}$ denotes the  $L^2(\Sigma)$ inner product (in this case, of two  $H^{1/2}(\Sigma)$ functions). See, e.g., \cite{buffa01}. 

 For the sake of the forthcoming analysis, we use $\partial_t \by_s^{\mathrm{P}}$ instead of $\bu_{s}^{\mathrm{P}}$ in the poroelastic region. Therefore we can write the weak formulation of the model problem as: for $t \in(0, T]$, find $\bu_f^{\mathrm{S}}(t) \in \mathbf{V}_f$, $p^{\mathrm{S}}(t) \in W_f, \bu_r(t) \in \mathbf{V}_r, p^{\mathrm{P}}(t) \in W_p, \by_s^{\mathrm{P}}(t) \in \mathbf{V}_s$, $\bv_s^{\mathrm{P}} \in \mathbf{W}_s$ and $\lambda(t) \in \Lambda$, such that $p^{\mathrm{P}}(0)=p^{\mathrm{P},0}, 
\by_s^{\mathrm{P}}(0)=\by_{s,0}, 
\bu_r(0)=\bu_{r,0}, 
\bu_s(0)=\bu_{s,0}$, and for all $\bv_f^{\mathrm{S}} \in \mathbf{V}_f,q^{\mathrm{S}} \in W_f, \bv_r^{\mathrm{P}} \in \mathbf{V}_r, q^{\mathrm{P}} \in W_p, \bw_s^{\mathrm{P}} \in \mathbf{V}_s$, $\bv_s^{\mathrm{P}} \in \mathbf{W}_s$ and $\mu \in \Lambda$: 
\begin{equation}\label{mixed-primal}
\mathbf{E} \partial_t X(t)+\mathbf{H} X(t)=L(t),\end{equation}
where
\begin{subequations}
\begin{gather}\label{E_matrix}
X  =\begin{bmatrix}
\bu_f^{\mathrm{S}}\\
\bu_r^{\mathrm{P}} \\
\by_s^{\mathrm{P}}\\
\bu_s^{\mathrm{P}}\\
{p}_{\mathrm{S}}\\
p^{\mathrm{P}} \\
{\lambda }
\end{bmatrix}, \quad L =\left[\begin{array}{c}
\mathcal{F}_{\bu_f^{\mathrm{S}}} \\
\mathcal{F}_{\bu_r} \\
\mathcal{F}_{\by_s^{\mathrm{P}}} \\
0\\
\mathcal{F}_{p^{\mathrm{S}}} \\
\mathcal{F}_{p^{\mathrm{P}}} \\
0
\end{array}\right],\quad \mathbf{E}=\left[\begin{array}{ccccccc}
0 & 0 & \mathcal{A}_{fs}^{\mathrm{BJS}} & 0 & 0 & 0 & 0 \\
0 & \mathcal{M}_{\rho_{f}\phi} & -M_{\theta}+\mathcal{A}^{\mathrm{P}}_f & \mathcal{M}_{\rho_{f}\phi} & 0 & 0 & 0\\
0 & \mathcal{M}_{\rho_{f}\phi} & A_{ss}^{\mathrm{BJS}}+ \mathcal{A}^{\mathrm{P}}_f -M_{\theta} & \mathcal{M}_{\rho_{p}} & 0 & 0  & 0\\
0 & 0 & -\mathcal{M}_{\rho_{p}} & 0 & 0 & 0 & 0\\
0 & 0 & 0 & 0 & \!\!\!\!\mathcal{M}_{\frac{(1-\phi)^2}{K}} & 0 & 0 \\
0 & 0 & -\mathcal{B}_s^{\mathrm{P}} & 0 & 0 & 0 & 0 \\
0 & 0 & -\mathcal{B}_{e, \Gamma} & 0 & 0 & 0 & 0 
\end{array}\right],
\\
\label{H_matrix}
\mathbf{H}  =\left[\begin{array}{ccccccc}
\mathcal{A}_f^{\mathrm{S}}+\mathcal{A}_{f f}^{\mathrm{BJS}} & 0 & 0 & 0 &(\mathcal{B}^{\mathrm{S}})^*& 0 & (\mathcal{B}_{f, \Gamma})^* \\
0 & \mathcal{A}_f^{\mathrm{P}}-\mathcal{M}_{\theta}+\mathcal{M}_{\phi^2/\kappa} & 0 & 0 & 0& (\mathcal{B}_{f}^{\mathrm{P}})^* & (\mathcal{B}_{p, \Gamma})^*  \\
(\mathcal{A}_{fs}^{\mathrm{BJS}})^* & \mathcal{A}_f^{\mathrm{P}}-\mathcal{M}_{\theta} & \mathcal{A}_s^{\mathrm{P}} & 0 & 0&(\mathcal{B}_{s}^{\mathrm{P}})^* & (\mathcal{B}_{e, \Gamma})^* \\
0 & 0 & 0 & \mathcal{M}_{\rho_{p}} & 0 & 0 & 0 \\
-\mathcal{B}^{\mathrm{S}} & 0 & 0 & 0 & 0 & 0 & 0\\
0 & -\mathcal{B}_{f}^{\mathrm{P}} & 0 & 0 & 0 & 0 & 0\\
-\mathcal{B}_{f, \Gamma} & -\mathcal{B}_{p, \Gamma} & 0 & 0 & 0 & 0 & 0 
\end{array}\right],
\end{gather}
\end{subequations}
and we note that the  matrix $\mathbf{E}+\mathbf{H}$ 
yields a generalized saddle-point structure 
of the form 
$$
\mathbf{E}+\mathbf{H} =
\left[
\begin{array}{cc}
\mathbf{A} & \mathbf{B}^T\\
-\mathbf{B} & \mathbf{C}
\end{array}
\right],
$$
where
$$
\mathbf{A} =
\left[
\begin{array}{cccc}
\mathcal{A}_f^{\mathrm{S}}+\mathcal{A}_{ff}^{\mathrm{BJS}} & 0 & \mathcal{A}_{fs}^{\mathrm{BJS}} & 0  \\
0 & \mathcal{A}_f^{\mathrm{P}}-\mathcal{M}_{\theta}+\mathcal{K}+\mathcal{M}_{\rho_{f}\phi} & -M_{\theta}+\mathcal{A}^{\mathrm{P}}_f & \mathcal{M}_{\rho_{f}\phi} \\
\mathcal{A}_{fs}^{BJS, T} & \mathcal{A}_f^{\mathrm{P}}-\mathcal{M}_{\theta}+\mathcal{M}_{\rho_{f}\phi} & \mathcal{A}_s^{\mathrm{P}}+ A_{ss}^{\mathrm{BJS}}+ \mathcal{A}^{\mathrm{P}}_f -M_{\theta} & \mathcal{M}_{\rho_p}\\
0 & 0 & -\mathcal{M}_{\rho_{p}} & \mathcal{M}_{\rho_{p}} 
\end{array}
\right],
$$
$$
\mathbf{B}=
\left[
\begin{array}{ccc}
\mathcal{B}^{\mathrm{S}} & 0 & \mathcal{B}_{f,\Gamma} \\
0 & \mathcal{B}_{f}^{\mathrm{H}} & \mathcal{B}_{p,\Gamma} \\
0 & \mathcal{B}_{s}^{\mathrm{H}} & \mathcal{B}_{s,\Gamma} \\
 0 & 0 & 0
\end{array}
\right], \quad \mathbf{C}=
\left[
\begin{array}{ccc}
\mathcal{M}_{\frac{(1-\phi)^2}{K}} & 0 & 0\\
 0 & 0 & 0\\
 0 & 0 & 0 
 \end{array}
\right].
$$

\subsection{Assumptions}\label{assumptions}
We will make use of the following assumptions:
\begin{enumerate}[label=(\textbf{H.\arabic*}), ref=$\mathrm{(H.\arabic*)}$]
    \item \label{(H1)} The porosity $\phi$ is such that $\phi, 1 / \phi,(1-\phi)$ and $1/(1-\phi)$ belong to $W^{s, r}(\Omega)$ with $s>d/r$, see \cite[Lemma 13]{MR4253885} and there exist two positive constants $\underline{\phi}$ and $\overline{\phi}$ such that $0<\underline{\phi} \leq \phi \leq \overline{\phi}<1$ almost everywhere in $\Omega$.
    \item \label{(H2)} $\theta $ represents a fluid sink.
\item \label{(H3)} The permeability tensor is symmetric and positive-definite, i.e., 
$$
\exists C_k>0:\left(\phi^2 \kappa^{-1} \bv_r^{\mathrm{P}}, \bv_r^{\mathrm{P}}\right) \gtrsim \|\bv_r^{\mathrm{P}}\|_{0,\Omega_{\mathrm{P}}}^2 \quad \forall \bv_r^{\mathrm{P}} \in \mathbf{V}_r .
$$
\end{enumerate}
From these assumptions, we obtain ellipticity properties that we use in both the well-posedness analysis and the energy estimates. We point out that  \ref{(H2)} is used to simplify the proof of  existence and   stability of solutions. However, it can be relaxed by means of 
a more refined approach (exploiting an exponential scaling of the velocity and choose a scaling factor). However in such a case the analysis turns out to be much more involved.

\begin{lemma}\label{general_C}
  The following inequalities hold for a.e. $t$ in $[0, T]$:   
\begin{align*}
\sqrt{\rho_f \underline{\phi}}\|\bu_r(t)\|_{0,\Omega_{\mathrm{P}}} & \leq\|\bu_r(t)\|_{\rho_f \phi} \leq \sqrt{\rho \bar{\phi}}\|\bu_r(t)\|_{0,\Omega_{\mathrm{P}}}, \\
\sqrt{\rho_s(1-\bar{\phi})}\|\bu_s(t)\|_{0,\Omega_{\mathrm{P}}} & \leq\|\bu_s(t)\|_{\rho_s(1-\phi)} \leq \sqrt{\rho_s(1-\underline{\phi})}\|\bu_s(t)\|_{0,\Omega_{\mathrm{P}}}, \\
\sqrt{K^{-1}(1-\bar{\phi})^2}\|p^{\mathrm{P}}(t)\|_{0,\Omega_{\mathrm{P}}} & \leq\|p^{\mathrm{P}}(t)\|_{(1-\phi)^2 K^{-1}} \leq \sqrt{K^{-1}(1-\underline{\phi})^2}\|p^{\mathrm{P}}(t)\|_{0,\Omega_{\mathrm{P}}}.
\end{align*}
\end{lemma} 
\begin{proof}
See \cite{MR4253885} for a proof.
\end{proof}
\begin{remark}
There are mainly two difficulties in analyzing this formulation directly. When applying standard approaches such as the Faedo--Galerkin method, it is not possible to obtain a bound on $\bu_r$ in the energy norm, and therefore we cannot achieve weak convergence of the bilinear forms. Additionally, $\partial_t \by_s^{\mathrm{P}}$ appears in several non-coercive terms complicating further the analysis of the direct formulation. Consequently, we concentrate on an equivalent mixed formulation. 
\end{remark}
\begin{remark}
 The mixed formulation is solely used to establish existence of solution, whereas we will provide proofs for uniqueness, stability, and error bounds using the original formulation. The equivalence of the two formulations will be addressed in Section \ref{equivalent_original}.
\end{remark}

\subsection{Mixed formulation}
Following the approach from \cite{MR2149168}, we  consider a mixed elasticity formulation with the structure velocity and elastic stress as primary variables. Let us recall   the inverse stress-strain relation
\begin{align}\label{derivative}
A \bsigma=\beps(\by_s^{\mathrm{P}}),
\end{align}
where $A$ is a symmetric and positive definite compliance tensor. In the isotropic case $A$ has the form 
$$
A \bsigma=\frac{1}{2 \mu_p}\left(\bsigma-\frac{\lambda_p}{2 \mu_p+d \lambda_p} \operatorname{tr}\left(\bsigma\right) \mathbf{I}\right), \quad \text { with } A^{-1} \mathbf{\epsilon}=2 \mu_p \mathbf{\epsilon}+\lambda_p \operatorname{tr}\left(\mathbf{\epsilon}\right) \mathbf{I} .
$$
The regularity of the displacement implies that the functional space for the elastic stress is  $\bZ =\mathbb{L}_{\mathrm{sym}}^2\left(\Omega_{\mathrm{P}}\right)$ with the norm $\|\bsigma\|_{\bZ }^2:=\sum_{i, j=1}^d\|(\bsigma)_{i, j}\|_{0,\Omega_{\mathrm{P}}}^2$. 
The derivation of the mixed formulation handles \eqref{relativeporo1} in a different manner. We still test this equation, as before, against $\bv_s^{\mathrm{P}} \in \mathbf{V}_s$ and integrate by parts, but we use the constitutive relation \eqref{stress3}. This yields 
$$
-\int_{\Omega_{\mathrm{P}}} \bsigma_s^{\mathrm{P}} : \bnabla \bv_s^{\mathrm{P}} \dx=\int_{\Omega_{\mathrm{P}}} \left(\bsigma: \beps (\bv_s^{\mathrm{P}})-(1-\phi) p^{\mathrm{P}} \nabla \cdot \bv_s^{\mathrm{P}}\right) \dx-\int_{\Sigma} \bsigma_s^{\mathrm{P}} \bn_{\mathrm{P}} \cdot \bv_s^{\mathrm{P}} \ds.
$$  
We eliminate the displacement $\by_s^{\mathrm{P}}$ from the system by differentiating \eqref{derivative} in time and writing $\bu_{s}^{\mathrm{P}}$ instead of $\partial_t \by_s^{\mathrm{P}} \in \mathbf{V}_s$. Multiplication by a test function $\btau \in \bZ $ gives
$$
\int_{\Omega_{\mathrm{P}}}\left(A \partial_t \bsigma: \btau-\beps \left(\bu_s^{\mathrm{P}}\right): \btau\right) \dx=0.
$$
The rest of the equations are handled in the same way as in the original weak formulation, resulting in the same functionals and interfacial terms. Next, we  define the following bilinear forms $b_{\text{sig}}^{\mathrm{P}}(\cdot, \cdot): \mathbf{V}_s \times \bZ  \to \mathbb{R}$ and $a_p^{\mathrm{P}}(\cdot, \cdot): \bZ  \times \bZ  \to \mathbb{R}$ by 
$$
b_{\text{sig}}^{\mathrm{P}}\left(\bu_s^{\mathrm{P}}, \btau \right):=(\beps(\bu_s^{\mathrm{P}}), \btau)_{\Omega_{\mathrm{P}}}, \quad a_p^{\mathrm{P}}(\bsigma, \btau):=\left(A \bsigma, \btau\right)_{\Omega_{\mathrm{P}}},
$$
and then 
proceed to group the trial and test spaces and  functions in the following manner 
\begin{gather*}
\vec{\bu}:=\left(\bu_r^{\mathrm{P}},\bu_s^{\mathrm{P}}, \bu_f^{\mathrm{S}}\right), \quad 
\vec{\bv}:=\left(\bv_r^{\mathrm{P}},\bv_s^{\mathrm{P}}, \bv_f^{\mathrm{P}}\right) \quad \in 
\vec{\mathbf{V}}:=\mathbf{V}_r \times \mathbf{X}_d \times \mathbf{V}_f, \\ 
\vec{p}:=\left(p^{\mathrm{P}}, \bsigma, p^{\mathrm{S}}, \lambda\right), \quad \vec{q}:=\left(q^{\mathrm{P}}, \btau, q^{\mathrm{S}}, \mu \right) \quad \in 
\vec{Q}:=W_p \times \bZ  \times W_f \times \Lambda,\end{gather*}
and use the following norms 
$$
\|\vec{\bu}\|_{\vec{\mathbf{V}}} :=\|\bu_r^{\mathrm{P}}\|_{\mathbf{V}_r}+\|\bu_s^{\mathrm{P}}\|_{\mathbf{V}_s}+\|\bu_f^{\mathrm{S}}\|_{\mathbf{V}_f}, \qquad 
\|\vec{p}\|_{\vec{Q}} :=\|p^{\mathrm{P}}\|_{W_p}+\|\bsigma\|_{\bZ }+\|p^{\mathrm{S}}\|_{W_f}+\|\lambda\|_{\Lambda}.
$$
With this, the weak formulation 
is written as a degenerate mixed evolution problem 
\begin{subequations}\label{weak-mixed}
\begin{align}
\frac{\partial}{\partial t} \mathcal{E}_1 \vec{\bu}(t)+\mathcal{A} \vec{\bu}(t)+\mathcal{B}^{\prime} \vec{p}(t) &=\ff(t) \quad  \text { in } \vec{\mathbf{V}}^{\prime}, \label{evo1} \\
\frac{\partial}{\partial t} \mathcal{E}_2 \vec{p}(t)-\mathcal{B} \vec{\bu}(t)+\mathcal{C} \vec{p}(t)  &=g(t) \qquad \text { in } \vec{Q}^{\prime} \label{evo2},
\end{align}
\end{subequations}
where the operators $\mathcal{A}: \vec{\mathbf{V}} \rightarrow \vec{\mathbf{V}}^{\prime}, \mathcal{B}: \vec{\mathbf{V}} \rightarrow \vec{Q}^{\prime}, \mathcal{C}: \vec{Q} \rightarrow \vec{Q}^{\prime}, \mathcal{E}_1: \vec{\mathbf{V}} \rightarrow \vec{\mathbf{V}}^{\prime}, \mathcal{E}_2: \vec{Q} \rightarrow \vec{Q}^{\prime}$, and the functionals $\ff \in \vec{\mathbf{V}}^{\prime}, g \in \vec{Q}^{\prime}$ are defined as follows (where we are also using the notation $\langle \mathcal{B}_{\text{sig}}^{\mathrm{P}} \cdot_1,\cdot_2\rangle= b_{\text{sig}}^{\mathrm{P}}(\cdot_1,\cdot_2)$): 
$$
\mathcal{A} =
\left[
\begin{array}{cccc}
\mathcal{A}_f^{\mathrm{P}}-\mathcal{M}_{\theta}+\mathcal{K} & -M_{\theta}+\mathcal{A}^{\mathrm{P}}_f & 0 \\
-M_{\theta}+\mathcal{A}^{\mathrm{P}}_f & -M_{\theta}+\mathcal{A}^{\mathrm{P}}_f+  A_{ss}^{\mathrm{BJS}} & (\mathcal{A}_{fs}^{\mathrm{BJS}})^* \\
0 & \mathcal{A}_{fs}^{\mathrm{BJS}}  &  \mathcal{A}_f^{\mathrm{S}}+\mathcal{A}_{ff}^{\mathrm{BJS}}
\end{array}
\right],\quad
\mathcal{B}=
\left[
\begin{array}{cccc}
\mathcal{B}_{f}^{\mathrm{P}} & 0 & 0 & \mathcal{B}_{f,\Gamma} \\
\mathcal{B}_{s}^{\mathrm{P}} & \mathcal{B}_{\text{sig}}^{\mathrm{P}} & 0 & \mathcal{B}_{p,\Gamma} \\
0 & 0 &  \mathcal{B}^{\mathrm{S}} & \mathcal{B}_{e,\Gamma} 
\end{array}
\right],
$$ 
$$
 \mathcal{C}=\left[\mathbf{0}\right]_{4 \times 4}, \ \ff=\begin{pmatrix}
\rho_f \phi \ff_{\mathrm{P}} \\
\rho_p \ff_{\mathrm{P}} \\
\ff_{\mathrm{S}}
\end{pmatrix}, \ g=\begin{pmatrix}
\theta \\
0 \\
r_{\mathrm{S}} \\
0
\end{pmatrix}, \ 
\mathcal{E}_1=\left(\begin{array}{ccc}
\mathcal{M}_{\rho_f \phi}  & \mathcal{M}_{\rho_f \phi} & 0 \\
\mathcal{M}_{\rho_f \phi} & \mathcal{M}_{\rho_p} & 0 \\
0 & 0 & 0
\end{array}\right), \ \mathcal{E}_2=\left(\begin{array}{cccc}
\mathcal{M}_{(1-\phi)^2 K^{-1}} & 0 & 0 & 0 \\
0 & \mathcal{M}_{A} & 0 & 0 \\
0 & 0 & 0 & 0 \\
0 & 0 & 0 & 0
\end{array}\right).
$$

\section{Well-posedness of the model}\label{section4}
We focus first on the analysis of the mixed
formulation \eqref{weak-mixed}.
\subsection{Existence and uniqueness of a solution of the mixed formulation}
We start with exploring important properties of the operators introduced in the previous section.
\begin{lemma}\label{coercivity-continuity}
 Under Assumptions~\ref{(H1)}-\ref{(H3)}, the linear operators $\mathcal{A}$, $\mathcal{E}_1$ and $\mathcal{E}_2$ are continuous and monotone.
 \end{lemma}
\begin{proof}
From Cauchy--Schwarz and Young inequalities, there exist  constants $C_f, C_s, C_r, C_{\mathrm{BJS}}>0$  
 such that 
\begin{align*}
a_f^{\mathrm{S}}(\bu_f^{\mathrm{S}},\bv_f^{\mathrm{S}}) &\leq C_f \|\bu_f^{\mathrm{S}}\|_{1,\Omega_{\mathrm{S}}} \|\bv_f^{\mathrm{S}}\|_{1,\Omega_{\mathrm{S}}} 
, \\  a_f^{\mathrm{P}}(\bu_s^{\mathrm{P}},\bv_s^{\mathrm{P}}) -m_{\theta}( \bu_s^{\mathrm{P}},\bv_s^{\mathrm{P}}) &\leq C_s \|\bu_s^{\mathrm{P}}\|_{1,\Omega_{\mathrm{P}}} \|\bv_s^{\mathrm{P}}\|_{1,\Omega_{\mathrm{P}}}, \\ 
a_f^{\mathrm{P}}(\bu_r^{\mathrm{P}},\bv_r^{\mathrm{P}}) -m_{\theta}( \bu_r^{\mathrm{P}},\bv_r^{\mathrm{P}}) + m_{\phi^2/\kappa}(\bu_r^{\mathrm{P}},\bv_r^{\mathrm{P}}) &\leq C_r \|\bu_r^{\mathrm{P}}\|_{1,\Omega_{\mathrm{P}}} \|\bv_r^{\mathrm{P}}\|_{1,\Omega_{\mathrm{P}}}, \\
a_{\mathrm{BJS}}(\bu_f^{\mathrm{S}}, \bu_s^{\mathrm{P}}; \bv_f^{\mathrm{S}}, \bv_s^{\mathrm{P}}) 
&\leq C_{\mathrm{BJS}}\left(\|\bu_f^{\mathrm{S}}\|_{1,\Omega_{\mathrm{S}}}+\|\bu_s^{\mathrm{P}}\|_{1,\Omega_{\mathrm{P}}}\right)\left(\|\bv_f^{\mathrm{S}}\|_{1,\Omega_{\mathrm{S}}}+\|\bv_s^{\mathrm{P}}\|_{1,\Omega_{\mathrm{P}}}\right), 
\end{align*}
where we have also used the trace inequality.
Thus we have that $\mathcal{A}$, $\mathcal{E}_1$ and $\mathcal{E}_2$  are continuous:
\begin{align*}
\langle \mathcal{A} \vec{\bu}, \vec{\bv}\rangle & = a_f^{\mathrm{S}}\left(\bu_f^{\mathrm{S}}, \bv_f^{\mathrm{S}}\right)+a_{f}^{\mathrm{P}}(\bu_r^{\mathrm{P}}, \bv_s^{\mathrm{P}})+ a_{f}^{\mathrm{P}}\left(\bu_r^{\mathrm{P}}, \bv_r^{\mathrm{P}}\right)+ a_{f}^{\mathrm{P}}(\bu_s^{\mathrm{P}}, \bv_r^{\mathrm{P}}) + a_{f}^{\mathrm{P}}\left(\bu_s^{\mathrm{P}}, \bv_s^{\mathrm{P}}\right) +a_{\mathrm{BJS}}\left(\bu_f^{\mathrm{S}}, \bu_s^{\mathrm{P}}; \bv_f^{\mathrm{S}}, \bv_s^{\mathrm{P}}\right) \\ &
\quad   - m_{\theta}(\bu_r^{\mathrm{P}}, \bv_s^{\mathrm{P}}) - m_{\theta}(\bu_s^{\mathrm{P}}, \bv_s^{\mathrm{P}}) - m_{\theta}(\bu_s^{\mathrm{P}}, \bv_r^{\mathrm{P}})  - m_{\theta}(\bu_r^{\mathrm{P}}, \bv_r^{\mathrm{P}})+ m_{\phi^2/\kappa}(\bu_r^{\mathrm{P}}, \bv_r^{\mathrm{P}}) \\&
 \leq C\|\vec{\bu}\|_{\vec{\mathbf{V}}}\|\vec{\bv}\|_{\vec{\mathbf{V}}},\\
 \langle\mathcal{E}_1 \vec{\bu}, \vec{\bv}\rangle & = (\rho_f \phi \bu_r^{\mathrm{P}}, \bv_s^{\mathrm{P}}) + (\rho_p \bu_s^{\mathrm{P}}, \bv_s^{\mathrm{P}})  + (\rho_f \phi  \bu_r^{\mathrm{P}}, \bv_r^{\mathrm{P}}) + (\rho_f \phi \bu_s^{\mathrm{P}}, \bv_r^{\mathrm{P}}) \leq C\|\vec{\bu}\|_{\vec{\mathbf{V}}}\|\vec{\bv}\|_{\vec{\mathbf{V}}}, \\
\langle\mathcal{E}_2 \vec{p}, \vec{q}\rangle &=\left((1-\phi)^2 K^{-1} p^{\mathrm{P}}, q^{\mathrm{P}}\right)_{\Omega_{\mathrm{P}}}+\left(A\bsigma,\btau_p\right)_{\Omega_{\mathrm{P}}} \leq C\|\vec{p}\|_{\vec{Q}}\|\vec{q}\|_{\vec{Q}} .
\end{align*}

On the other hand, there exist positive constants $\alpha_f, \alpha_s, \alpha_r$ such that                                              
\begin{align*}
 a_f^{\mathrm{S}}\left(\bv_f^{\mathrm{S}},\bv_f^{\mathrm{S}}\right) &\geq \alpha_f\|\bv_f^{\mathrm{S}}\|_{1,\Omega_{\mathrm{S}}}^2 
,\\ a_f^{\mathrm{P}}(\bv_s^{\mathrm{P}},\bv_s^{\mathrm{P}}) -m_{\theta}( \bv_s^{\mathrm{P}},\bv_s^{\mathrm{P}}) &\geq \alpha_s\|\bv_s^{\mathrm{P}}\|_{1,\Omega_{\mathrm{P}}}^2,  \\ 
a_f^{\mathrm{P}}(\bv_r^{\mathrm{P}},\bv_r^{\mathrm{P}}) -m_{\theta}( \bv_r^{\mathrm{P}},\bv_r^{\mathrm{P}}) + m_{\phi^2/\kappa}(\bv_r^{\mathrm{P}},\bv_r^{\mathrm{P}}) &\geq \alpha_r \|\bv_r^{\mathrm{P}}\|_{1,\Omega_{\mathrm{P}}}^2,  \\ 
a_{\mathrm{BJS}}\left(\bv_f^{\mathrm{S}}, \bw_s^{\mathrm{P}} ; \bv_f^{\mathrm{S}}, \bw_s^{\mathrm{P}}\right) &\geq \mu_f \alpha_{\mathrm{BJS}} K_{{\mathrm{max}}}^{-1/2} |\left(\bv_f^{\mathrm{S}}-\bw_s^{\mathrm{P}}\right)|_{\mathrm{BJS}}^2 ,
\end{align*}
where we have used K\"orn's inequality \cite{MR2373954} and the assumptions in Section \ref{assumptions}. Therefore  $\mathcal{A}$ is monotone. 
In addition, the monotonicity of $\mathcal{E}_1,\mathcal{E}_2$ follows straightforwardly from
$$
\langle\mathcal{E}_1 \vec{\bv}, \vec{\bv}\rangle= \rho_f \phi \|\bv_r^{\mathrm{P}}+ \bv_s^{\mathrm{P}}\|^2_{0,\Omega_{\mathrm{P}}} + (1-\phi)\rho_s \| \bv_s^{\mathrm{P}}\|^2_{0,\Omega_{\mathrm{P}}},  \qquad  
\left(\mathcal{E}_2 \vec{q}, \vec{q}\right) = \|(1-\phi) K^{-1/2} q^{\mathrm{P}}\|_{0,\Omega_{\mathrm{P}}}^2 + \|A^{1/2} \btau\|^2_{0,\Omega_{\mathrm{P}}}.
$$
\end{proof}
\begin{lemma}\label{continuity-b}
The operator $\mathcal{B}$ and its adjoint $\mathcal{B}^{*}$ are bounded and continuous. 
\end{lemma}
\begin{proof}
For all $\vec{\bv}=(\bv_r^{\mathrm{P}}, \bv_s^{\mathrm{P}}, \bv_f^{\mathrm{S}}) \in \vec{\mathbf{V}}$ and $\vec{q}=\left(q^{\mathrm{P}}, \btau, q^{\mathrm{S}}, \mu\right) \in \vec{Q}$ we have 
\begin{align*}
\langle \mathcal{B}(\vec{\bv}),\vec{q}\rangle&=b^{\mathrm{S}}\left(\bv_f^{\mathrm{S}}, q^{\mathrm{S}}\right)+b_s^{\mathrm{P}}\left(\bv_s^{\mathrm{P}}, q^{\mathrm{P}}\right)+ b_f^{\mathrm{P}}\left(\bv_r^{\mathrm{P}}, q^{\mathrm{P}}\right)+b_{\text{sig}}^{\mathrm{P}}\left(\bv_s^{\mathrm{P}}, \btau\right)+b_{\Gamma}\left(\bv_f^{\mathrm{S}}, \bv_r^{\mathrm{P}}, \bv_s^{\mathrm{P}} ; \mu\right) \\
& \leq C\left(\|\bv_f^{\mathrm{S}}\|_{1,\Omega_{\mathrm{S}}}\|q^{\mathrm{S}}\|_{0,\Omega_{\mathrm{S}}}+\|\bv_s^{\mathrm{P}}\|_{1,\Omega_{\mathrm{P}}}\|q^{\mathrm{P}}\|_{0,\Omega_{\mathrm{P}}}+\|\bv_r^{\mathrm{P}}\|_{1,\Omega_{\mathrm{P}}}\|q^{\mathrm{P}}\|_{0,\Omega_{\mathrm{P}}} +\|\bv_s^{\mathrm{P}}\|_{1,\Omega_{\mathrm{P}}}\|\btau\|_{0,\Omega_{\mathrm{P}}}\right. \\
& \qquad \left.+\|\bv_f^{\mathrm{S}}\|_{1,\Omega_{\mathrm{S}}}\|\mu\|_{-1/2,\Sigma}+\|\bv_r^{\mathrm{P}}\|_{1,\Omega_{\mathrm{P}}}\|\mu\|_{-1/2,\Sigma}+\|\bv_s^{\mathrm{P}}\|_{1,\Omega_{\mathrm{P}}}\|\mu\|_{-1/2,\Sigma}\right)\\
&  \leq C\|\vec{\bv}\|_{\vec{\mathbf{V}}}\|\vec{q}\|_{\vec{Q}}, 
\end{align*}
which completes the proof. 
\end{proof}
The next result establishes the Ladyzhenskaya--Babu\v{s}ka--Brezzi (LBB) condition for the mixed form.
\begin{lemma}\label{usual_inf}
There exists a constant $\xi_1(\Omega)>0$ such that
$$
\inf _{\substack{\left(q^{\mathrm{S}}, \mathbf{0},q^{\mathrm{P}},0 \right) \in \vec{Q} }} \sup _{\substack{\left(\bv_r^{\mathrm{P}},\mathbf{0},\bv_f^{\mathrm{S}}\right) \in \vec{\mathbf{V}}}} \frac{b^{\mathrm{S}} (\bv_f^{\mathrm{S}}, q^{\mathrm{S}})+b_f^{\mathrm{P}}\left(\phi \bv_r^{\mathrm{P}},q^{\mathrm{P}}\right)}{\|(\bv_r^{\mathrm{P}}, \mathbf{0},\bv_f^{\mathrm{S}})\|_{\vec{\mathbf{V}}}\|(q^{\mathrm{P}}, \mathbf{0}, q^{\mathrm{S}}, 0)\|_{\vec{Q}}} \geq \xi_1 >0 .
$$
\end{lemma}
\begin{proof}
This is proven by using the usual inf-sup condition for the Stokes problem \cite{MR3097958} and the weighted inf-sup condition in \cite[Lemma 14]{MR4253885}.
\end{proof}
\begin{lemma}\label{sigma_inf}
 There is a constant $\xi_2(\Omega)>0$, such that
$$
\inf _{\substack{\left(0, \mathbf{0},0,\mu \right) \in \vec{Q} }} \sup _{\substack{\left(\bv_r^{\mathrm{P}},\mathbf{0},\bv_f^{\mathrm{S}}\right) \in \vec{\mathbf{V}} }} \frac{b_{\Gamma}(\bv_r^{\mathrm{P}},\mathbf{0},\bv_f^{\mathrm{S}};\mu)}{\|(\bv_r^{\mathrm{P}}, \mathbf{0},\bv_f^{\mathrm{S}})\|_{\vec{\mathbf{V}}}\|(0, \mathbf{0}, 0, \mu)\|_{\vec{Q}}} \geq \xi_2>0 .
$$
\end{lemma}
\begin{proof}
Thanks to the Riesz representation theorem, for a given $\mu \in H^{-1/2}(\Sigma)$ there exists $\tilde{\xi} \in {H}^{1 / 2}\left(\Sigma\right)$ and $\bn_{\mathrm{S}}$ sufficiently smooth such that $\|\tilde{\xi}\|_{1 / 2, \Sigma}=\|\mu\|_{-1 / 2, \Sigma}$. Let us consider the following auxiliary Stokes problem 
\begin{align}\label{S}
-\mathbf{\Delta} \hat{\bv}_f+\nabla \zeta &=\mathbf{0} \qquad \text { in } \Omega_{\mathrm{S}}, \nonumber \\ \nabla \cdot \hat{\bv}_f & = 0 \qquad \text { in } \Omega_{\mathrm{S}},  \\
\nonumber\hat{\bv}_f  =\mathbf{0} \quad \text { on } \Gamma_{\mathrm{S}}, \text{ and }
\hat{\bv}_f & =\tilde{\xi} \bn_{\mathrm{S}} \quad \text { on } \Sigma. 
\end{align}
Thanks to \cite{MR548867}, we can assert that there exists a unique velocity solution to \eqref{S}, for which there holds
\begin{equation}\label{tilde}
\|\hat{\bv}_f\|_{1, \Omega_{\mathrm{S}}} \lesssim \| \hat{\xi} \bn_{\mathrm{S}} \|_{1 / 2, \Sigma} \lesssim \| \hat{\xi}  \|_{1 / 2, \Sigma} \quad  
\text{ and } \quad \|\hat{\bv}_f\|_{1, \Omega_{\mathrm{S}}} \lesssim \|\mu\|_{-1 / 2, \Sigma},
\end{equation}
and so $\hat{\bv}_f \in \mathbf{H}^1_{\star}(\Omega_{\mathrm{S}})=\{ \bv_f^{\mathrm{S}}\in \mathbf{H}^1(\Omega_{\mathrm{S}}): \bv_f|_{\Sigma} = \tilde{\xi} \bn_{\mathrm{S}} \}$.
In this way, we can choose $\bv_r^{\mathrm{P}}=\cero$ and write it as
\begin{align*}
 \sup _{\substack{\left(\bv_r^{\mathrm{P}},\mathbf{0},\bv_f^{\mathrm{S}}\right) \in \vec{\mathbf{V}} }} \frac{b_{\Gamma}(\bv_r^{\mathrm{P}},\mathbf{0},\bv_f^{\mathrm{S}};\mu)}{\|(\bv_r^{\mathrm{P}}, \mathbf{0},\bv_f^{\mathrm{S}})\|_{\vec{\mathbf{V}}}} \geq & \frac{\langle \hat{\bv}_f \cdot \bn_{\mathrm{S}}, \mu \rangle_{\Sigma}}{\|(\mathbf{0}, \mathbf{0},\bv_f^{\mathrm{S}})\|_{\vec{\mathbf{V}}}} = \frac{\langle \tilde{\xi} \bn_{\mathrm{S}} \cdot \bn_{\mathrm{S}}, \mu \rangle_{\Sigma}}{\| \hat{\bv}_f\|_{1,\Omega_{\mathrm{S}}}}  = \frac{\langle \tilde{\xi} , \mu \rangle_{\Sigma}}{\| \hat{\bv}_f\|_{1,\Omega_{\mathrm{S}}}}  =  \frac{\| \mu\|^2_{-1 / 2, \Sigma}}{\| \hat{\bv}_f\|_{1,\Omega_{\mathrm{S}}}} \gtrsim \xi_2 \| \mu\|_{-1 / 2, \Sigma}, 
\end{align*}
where we have used \eqref{tilde} and hence this concludes the result.
\end{proof}
\begin{lemma}\label{LBB}
There exist constants $\xi_3(\Omega), \xi_4(\Omega)>0$, such that
\begin{subequations}
\begin{align}
 \inf _{\left(\mathbf{0}, \bv_s^{\mathrm{P}}, \mathbf{0}\right) \in \vec{\mathbf{V}}} \sup _{\left(0, \btau, 0,0\right) \in \vec{Q}} \frac{b_{\text{sig}}^{\mathrm{P}}\left(\bv_s^{\mathrm{P}}, \btau\right)}{\|(\mathbf{0}, \bv_s^{\mathrm{P}}, \mathbf{0})\|_{\vec{\mathbf{V}}}\|(0, \btau, 0,0)\|_{\vec{Q}}} &\geq \xi_3, \label{inf-sup1}\\
 \inf _{\left(q^{\mathrm{P}}, \mathbf{0}, q^{\mathrm{S}}, \mu \right) \in \vec{Q}} \sup _{\left(\bv_r^{\mathrm{P}}, \mathbf{0}, \bv_f^{\mathrm{S}}\right) \in \vec{\mathbf{V}}} \frac{b^{\mathrm{S}} (\bv_f^{\mathrm{S}}, q^{\mathrm{S}})+b_f^{\mathrm{P}}\left(\bv_r^{\mathrm{P}},q^{\mathrm{P}}\right)+b_{\Gamma}\left(\bv_r^{\mathrm{P}}, \bv_f^{\mathrm{S}}, \mathbf{0} ; \mu \right)}{\|(\bv_r^{\mathrm{P}}, \mathbf{0}, \bv_f^{\mathrm{S}})\|_{\vec{\mathbf{V}}}\|(q^{\mathrm{P}}, \mathbf{0}, q^{\mathrm{S}}, \lambda)\|_{\vec{Q}}} &\geq \xi_4 \label{inf-sup2}.
\end{align}
\end{subequations}
\end{lemma}
\begin{proof}
Let $\mathbf{0} \neq\left(\mathbf{0}, \bv_s^{\mathrm{P}}, \mathbf{0}\right) \in \vec{\mathbf{V}}$ be given. We choose $\btau=\beps (\bv_s^{\mathrm{P}})$, and using K\"orn's inequality \cite{MR2373954}, we obtain
$$
\frac{b_{\text{sig}}^{\mathrm{P}}\left(\bv_s^{\mathrm{P}}, \btau\right)}{\|\btau\|_{0,\Omega_{\mathrm{P}}}}=\frac{\|\beps (\bv_s^{\mathrm{P}})\|_{0,\Omega_{\mathrm{P}}}^2}{\|\beps (\bv_s^{\mathrm{P}})\|_{0,\Omega_{\mathrm{P}}}}=\|\beps (\bv_s^{\mathrm{P}})\|_{0,\Omega_{\mathrm{P}}} \geq \xi_3\|\bv_s^{\mathrm{P}}\|_{1,\Omega_{\mathrm{P}}}.
$$
Therefore, \eqref{inf-sup1} holds.
Combining Lemmas \ref{usual_inf} and \ref{sigma_inf} implies the result stated in \eqref{inf-sup2}.

\end{proof}
\begin{corollary}
There exists a constant $\xi_5 (\Omega) >0$, such that
$$
 \inf _{\vec{q} \in \vec{Q}} \sup_{\vec{\bv} \in \vec{\mathbf{V}}} \frac{
 \langle \mathcal{B}(\vec{\bv}),\vec{q}\rangle}
 {\|\vec{\bv}\|_{\vec{\mathbf{V}}}\|\vec{q}\|_{\vec{Q}}}   \geq \xi_5.
$$
\end{corollary}
\begin{proof}
 The statement follows from Lemma \ref{LBB} by simply taking $\bv_s^{\mathrm{P}}=\cero$.
 \end{proof}
 The following  result (cf. \cite[Theorem 6.1(b)]{MR1422252}) is used to establish the existence of a solution to \eqref{weak-mixed}. 
\begin{theorem}\label{N,M}
Let the linear, symmetric and monotone operator $\mathcal{N}$ be given for the real vector space $E$ to its algebraic dual $E^*$, and let $E_b^{\prime}$ be the Hilbert space which is the dual of $E$ with the seminorm 
$$
|x|_b=\langle \mathcal{N} x, x \rangle^{1 / 2}, \quad x \in E .
$$

Let $\mathcal{M}^{\star} \subset E \times E_b^{\prime}$ be a relation with domain $D=\{x \in E: \mathcal{M}^{\star}(x) \neq \emptyset\}$.
Assume that $\mathcal{M}^{\star}$ is monotone and $\operatorname{Rg}(\mathcal{N}+\mathcal{M}^{\star})=E_b^{\prime}$. Then, for each $u_0 \in \bbD$ and for each $f \in W^{1,1}\left(0, T ; E_b^{\prime}\right)$, there is a solution $u$ of
$$
\ddt (\mathcal{N} u(t))+\mathcal{M}^{\star}(u(t)) \ni f(t), \quad 0<t<T,
$$
with $\mathcal{N} u \in W^{1, \infty}\left(0, T ; E_b^{\prime}\right)$, $u(t) \in \bbD$ for all $0 \leq t \leq T$,  and  $\mathcal{N} u(0)=\mathcal{N} u_0$.
\end{theorem}
Note first that the seminorm induced by  $\mathcal{E}_1$ is $|\vec{\bv}|_{\mathcal{E}_1}^2 = \rho_f \phi \|\bv_r^{\mathrm{P}}+ \bv_s^{\mathrm{P}}\|^2_{0,\Omega_{\mathrm{P}}} + (1-\phi)\rho_s \| \bv_s^{\mathrm{P}}\|^2_{0,\Omega_{\mathrm{P}}}$ is equivalent to $\|\bv_r^{\mathrm{P}}\|^2_{0,\Omega_{\mathrm{P}}}+\|\bv_s^{\mathrm{P}}\|^2_{0,\Omega_{\mathrm{P}}}.$ We denote by $\mathbf{W}_{r,2}$, $\mathbf{W}_{s,2}$ $W_{p, 2}$ and $\bZ_2 $ the closure of   $\mathbf{V}_r$, $\mathbf{V}_s$, $W_p$ and $\bZ $ with respect to the norms
\begin{equation}\label{eq:norms}
\|\bu_r^{\mathrm{P}}\|_{\mathbf{W}_{r, 2}}^2\! := (\rho_f \phi \bu_r^{\mathrm{P}}, \bu_r^{\mathrm{P}})_{\Omega_{\mathrm{P}}}, \ 
\|\bu_s^{\mathrm{P}}\|_{\mathbf{W}_{s, 2}}^2\! :=(\rho_p \bu_s^{\mathrm{P}}, \bu_s^{\mathrm{P}})_{\Omega_{\mathrm{P}}}, \quad  
\|p^{\mathrm{P}}\|_{W_{p, 2}}^2\! :=\left(\frac{(1-\phi)^2}{K} p^{\mathrm{P}}, p^{\mathrm{P}}\right)_{\Omega_{\mathrm{P}}}, \quad  
\|\btau\|_{\bZ_2 }^2:=(A \btau, \btau)_{\Omega_{\mathrm{P}}} .
\end{equation}    
Let $\bbS_2:=\mathbf{W}_{r,2} \times \mathbf{W}_{s,2} \times W_{p, 2} \times \bZ_2 $. We introduce the inner product $(\cdot, \cdot)_{\bbS_2}$ by 
\begin{align*}
((\bu_r^{\mathrm{P}}, \bu_s^{\mathrm{P}}, p^{\mathrm{P}}, \bsigma),(\bv_r^{\mathrm{P}}, \bv_s^{\mathrm{P}}, q^{\mathrm{P}}, \btau))_{\bbS_2} 
:=& \left(\rho_f \phi \bu_r^{\mathrm{P}}, \bv_r^{\mathrm{P}}\right)_{\Omega_{\mathrm{P}}} +\left(\rho_f \phi \bu_r^{\mathrm{P}}, \bv_s^{\mathrm{P}}\right)_{\Omega_{\mathrm{P}}} +\left(\rho_f \phi \bu_s^{\mathrm{P}}, \bv_r^{\mathrm{P}}\right)_{\Omega_{\mathrm{P}}} \\& + \left(\rho_p \bu_s^{\mathrm{P}}, \bv_s^{\mathrm{P}} \right)_{\Omega_{\mathrm{P}}}+\left((1-\phi)^2 K^{-1 } p^{\mathrm{P}}, q^{\mathrm{P}}\right)_{\Omega_{\mathrm{P}}} + \left(A \bsigma, \btau\right)_{\Omega_{\mathrm{P}}}.
\end{align*}
Next we define the domain $\bbD \subseteq \bbS_2$ as 
\begin{align}\label{domaindefined}
\bbD:= & \left\{\left(\bu_r^{\mathrm{P}}, \bu_s^{\mathrm{P}}, p^{\mathrm{P}}, \bsigma\right) \in \mathbf{V}_r \times \mathbf{X}_d \times W_p \times \bZ : \text { for } \left(\ff_{\mathrm{S}}, r_{\mathrm{S}}\right) \in \left(\mathbf{V}_f^{\prime}, W_f^{ \prime}\right),\  \exists\left( \bu_f^{\mathrm{S}}, p^{\mathrm{S}}, \lambda \right)   \in \mathbf{V}_f \times W_f \times \Lambda \text { s.t. } \right. \nonumber \\  & \quad \forall (\vec{\bv},\vec{p}) \in \vec{\mathbf{V}} \times \vec{Q}:      \text{ \eqref{domain1}--\eqref{domain3} holds for some } \left(\bar{f}_r, \bar{f}_s, \bar{g}_p, \bar{g}_e\right) \in \bbS_2^\prime\},
\end{align}
associated with the following set of equations 
\begin{subequations}\label{eq:domaineq}
\begin{align}
& a_f^{\mathrm{S}}\left(\bu_f^{\mathrm{S}}, \bv_f^{\mathrm{S}}\right)+a_{f}^{\mathrm{P}}(\bu_r^{\mathrm{P}}, \bv_s^{\mathrm{P}})+ a_{f}^{\mathrm{P}}\left(\bu_r^{\mathrm{P}}, \bv_r^{\mathrm{P}}\right)+ a_{f}^{\mathrm{P}}(\bu_s^{\mathrm{P}}, \bv_r^{\mathrm{P}}) + a_{f}^{\mathrm{P}}\left(\bu_s^{\mathrm{P}}, \bv_s^{\mathrm{P}}\right) +a_{\mathrm{BJS}}\left(\bu_f^{\mathrm{S}}, \bu_s^{\mathrm{P}}; \bv_f^{\mathrm{S}}, \bv_s^{\mathrm{P}}\right)  +b^{\mathrm{S}}\left(\bv_f^{\mathrm{S}}, p^{\mathrm{S}}\right) \nonumber \\&\quad  
+b_s^{\mathrm{P}}\left(\bv_s^{\mathrm{P}}, p^{\mathrm{P}}\right) + b_f^{\mathrm{P}}\left( \phi \bv_r^{\mathrm{P}}, p^{\mathrm{P}}\right)  +b_{\Gamma}\left(\bv_f^{\mathrm{S}}, \bv_r^{\mathrm{P}}, \bv_s^{\mathrm{P}} ; \lambda\right) - m_{\theta}(\bu_r^{\mathrm{P}}, \bv_s^{\mathrm{P}}) - m_{\theta}(\bu_s^{\mathrm{P}}, \bv_s^{\mathrm{P}}) - m_{\theta}(\bu_s^{\mathrm{P}}, \bv_r^{\mathrm{P}})  - m_{\theta}(\bu_r^{\mathrm{P}}, \bv_r^{\mathrm{P}}) \nonumber \\ & \quad   + m_{\phi^2/\kappa}(\bu_r^{\mathrm{P}}, \bv_r^{\mathrm{P}}) + m_{\rho_f \phi}(\bu_r^{\mathrm{P}}, \bv_s^{\mathrm{P}}) + m_{\rho_p}(\bu_s^{\mathrm{P}}, \bv_s^{\mathrm{P}})  + m_{\rho_f \phi }(\bu_r^{\mathrm{P}}, \bv_r^{\mathrm{P}})  + m_{\rho_f \phi}(\bu_s^{\mathrm{P}}, \bv_r^{\mathrm{P}}) + b_\text{sig}^{\mathrm{P}}\left(\bv_s^{\mathrm{P}}, \bsigma\right) \nonumber  \\
& \qquad \qquad \qquad \quad \ = \left(\rho_p \bar{f}_s, \bv_s^{\mathrm{P}}\right)_{\Omega_{\mathrm{P}}}  + \left(\rho_f \phi \bar{f}_r, \bv_r^{\mathrm{P}}\right)_{\Omega_{\mathrm{P}}} + \left(\rho_f \phi \bar{f}_r, \bv_s^{\mathrm{P}}\right)_{\Omega_{\mathrm{P}}}+ \left(\rho_f \phi \bar{f}_s, \bv_r^{\mathrm{P}}\right)_{\Omega_{\mathrm{P}}}     + \langle \ff_{\mathrm{S}}, \bv_f^{\mathrm{S}}\rangle_{\Omega_{\mathrm{S}}},\label{domain1}\\  
& \left((1-\phi)^2 K^{-1} p^{\mathrm{P}}, q^{\mathrm{P}}\right)_{\Omega_{\mathrm{P}}}-b_s^{\mathrm{P}}\left(\bu_s^{\mathrm{P}}, q^{\mathrm{P}}\right)-b_f^{\mathrm{P}}\left(\phi \bu_r^{\mathrm{P}}, q^{\mathrm{P}}\right)-b^{\mathrm{S}}\left(\bu_f^{\mathrm{S}}, q^{\mathrm{S}}\right)+a_p^{\mathrm{P}}\left( \bsigma, \btau\right)-b_{\text{sig}}^{\mathrm{P}}\left(\bu_s^{\mathrm{P}}, \btau\right) \nonumber  \\ 
& \qquad \qquad \qquad \quad \ =\left(r_{\mathrm{S}}, q^{\mathrm{S}}\right)_{\Omega_{\mathrm{S}}} +\left( (1-\phi)^2 K^{-1} \bar{g}_p, q^{\mathrm{P}}\right)_{\Omega_{\mathrm{P}}}+\left(A \bar{g}_e, \btau\right)_{\Omega_{\mathrm{P}}}, \label{domain2}\\
& b_{\Gamma}\left(\bu_f^{\mathrm{S}}, \bu_r^{\mathrm{P}}, \bu_s^{\mathrm{P}}; \mu\right)=0.\label{domain3}
\end{align}
\end{subequations}
 It is important to remark that the definitions of $\bbS_2$
and   $\bbD$ imply that, in order to
apply Theorem \ref{N,M} to problem \eqref{evo1}-\eqref{evo2}, we require $\ff_{\mathrm{S}}=\cero$ and $r_{\mathrm{S}}=0$. To avoid this restriction we   employ a translation argument (cf. \cite{MR2684313}) that permits us   to rewrite  \eqref{domain1}-\eqref{domain3} in the following {operator} form 
\begin{subequations}
\begin{align}
\left(\mathcal{E}_1 +\mathcal{A } \right)\vec{\bv}+\mathcal{B}^{\prime} \vec{p}&=\vec{f} \quad \text { in } \vec{\mathbf{V}}^{\prime} \text {, } \label{operator_form1} \\
-\mathcal{B} \vec{\bv}+\mathcal{E}_2 \vec{p} &=\vec{g}  \quad \text { in } \vec{Q}^{\prime} \text {, } \label{operator_form2}
\end{align}
\end{subequations}
where $\left(\vec{f},\vec{g} \right) = (\left(\bar{f}_r, \bar{f}_s, \ff_{\mathrm{S}}\right), \left(\bar{g}_p, \bar{g}_e, r_{\mathrm{S}}, \cero \right)) \in \vec{\mathbf{V}}^{\prime} \times \vec{Q}^{\prime}$.

We also stress that there may be more than one $\left(\bar{f}_r, \bar{f}_s,\bar{g}_p, \bar{g}_e\right) \in  \bbS_2^{\prime}$ that generate the same $\left(\bu_r^{\mathrm{P}}, \bu_s^{\mathrm{P}}, p^{\mathrm{P}}, \bsigma\right) \in \bbD$.
 In view of this, we introduce the multivalued operator $\mathcal{M}(\cdot)$ with domain $\bbD$ defined by
\begin{align}\label{Mset}
 \mathcal{M}\left((\bu_r^{\mathrm{P}}, \bu_s^{\mathrm{P}}, p^{\mathrm{P}}, \bsigma)\right):=\{\left(\bar{f}_r-\bu_r^{\mathrm{P}}, \bar{f}_s-\bu_s^{\mathrm{P}}, \bar{g}_p-p^{\mathrm{P}}, \bar{g}_e-\bsigma\right)
 \in \bbS_2^\prime  :\left(\bu_r^{\mathrm{P}}, \bu_s^{\mathrm{P}},p^{\mathrm{P}}, \bsigma\right)   
 \text { solves \eqref{eq:domaineq} for }  \left(\bar{f}_r, \bar{f}_s, \bar{g}_p, \bar{g}_e\right)     \in  \bbS_2^\prime\}.   
\end{align}

Consider the following problem: given $\bh_r \in W^{1,1}\left(0, T ; \mathbf{W}_{r, 2}^{ \prime}\right)$, $\bh_s \in W^{1,1}\left(0, T ; \mathbf{W}_{s, 2}^{ \prime}\right)$, $h_p  \in W^{1,1}\left(0, T ; W_{p, 2}^{ \prime}\right)$ and $h_e \in W^{1,1}\left(0, T ; \bZ_2 ^{\prime}\right)$, find $\left(\bu_r^{\mathrm{P}}, \bu_s^{\mathrm{P}}, p^{\mathrm{P}}, \bsigma\right) \in \bbD$ satisfying
\begin{align}\label{parabolic}
   \ddt \left(\begin{array}{l}
\bu_r(t)\\
\bu_s(t)\\
p^{\mathrm{P}}(t) \\
\bsigma(t)
\end{array}\right)+\mathcal{M}\left(\begin{array}{l}
\bu_r(t)\\
\bu_s(t)\\
p^{\mathrm{P}}(t) \\
\bsigma(t)
\end{array}\right) \ni\left(\begin{array}{l}
\bh_r(t)\\
\bh_s(t) \\
h_p(t) \\
h_e(t)
\end{array}\right) . 
\end{align}
Using Theorem \ref{N,M}, we can show that the problem \eqref{weak-mixed} is well-posed.
\begin{theorem}\label{main}
Suppose \ref{(H1)}--\ref{(H3)}. For each $\ff_{\mathrm{S}}\in W^{1,1}(0, T ; \mathbf{V}_f^{\prime}), \ff_{\mathrm{P}} \in W^{1,1}(0, T ; \mathbf{L}^{2}(\Omega_{\mathrm{P}})), r_{\mathrm{S}} \in W^{1,1}(0, T ; W_f^{ \prime})$, $\theta \in W^{1,1}(0, T ; W_p^{ \prime})$, and $p^{\mathrm{P}}(0)=p^{\mathrm{P},0} \in W_p, \bsigma(0)=A^{-1} \beps \left(\by_{s, 0}\right) \in \bZ  , \bu_r^{\mathrm{P}}(0)= \bu_{r,0} \in \mathbf{V}_r, \bu^{\mathrm{P}}_s(0)= \bu_{s,0} \in \mathbf{X}_d $,  there exists a solution of \eqref{weak-mixed} with $(\bu_f^{\mathrm{S}}, p^{\mathrm{S}}, \bu_r^{\mathrm{P}}, p^{\mathrm{P}}, \bu_s^{\mathrm{P}}, \bsigma,  \lambda) \in L^{\infty}\left(0, T ; \bV_f \right) \times L^{\infty}\left(0, T ; W_f \right) \times W^{1,\infty}(0, T ; \mathbf{V}_r) \times$ $W^{1, \infty}\left(0, T ; W_p \right) \times W^{1,\infty}\left(0, T ; \mathbf{V}_s\right) \times W^{1, \infty}\left(0, T ; \bZ \right) \times L^{\infty}(0, T ; \Lambda)$.
\end{theorem}
To prove Theorem \ref{main} we proceed in the following manner. \textbf{Step 1}: Establish that the domain $\bbD$ defined above is nonempty; 
\textbf{Step 2}:  Show solvability of the parabolic problem \eqref{parabolic}; and 
\textbf{Step 3}: Show that the original problem \eqref{weak-mixed} is a special case of \eqref{parabolic}. We address each step in what follows. 

\medskip 
\noindent\textbf{\underline{Step 1:} The domain $\bbD$ is nonempty.}
We begin with a number of preliminary results used in the proof. We first introduce operators that will be used to regularize the problem. 
Let $R_s: \mathbf{X}_d \to \mathbf{X}_d^{ \prime}, R_r: \mathbf{V}_r \to \mathbf{V}_r^{ \prime}, L_f: W_f\to W_f^{\prime}$, $L_p: W_p \to W_p^{\prime}$ be defined by 
\begin{gather*}
\langle R_s\left(\bu_s^{\mathrm{P}}\right),\bv_s^{\mathrm{P}}\rangle \coloneqq r_s\left(\bu_s^{\mathrm{P}}, \bv_s^{\mathrm{P}}\right)=\left(\beps \left(\bu_s^{\mathrm{P}}\right), \beps (\bv_s^{\mathrm{P}})\right)_{\Omega_{\mathrm{P}}},\qquad  \langle R_r(\bu_r^{\mathrm{P}}),\bv_r^{\mathrm{P}}\rangle \coloneqq r_r\left(\bu_r^{\mathrm{P}}, \bv_r^{\mathrm{P}}\right)=\left(\beps \left(\bu_r^{\mathrm{P}} \right), \beps \left(\bv_r^{\mathrm{P}}\right)\right)_{\Omega_{\mathrm{P}}}, \\
\langle L_f\left(p^{\mathrm{S}}\right),q^{\mathrm{S}}\rangle:=l_f\left(p^{\mathrm{S}}, q^{\mathrm{S}}\right)=\left( p^{\mathrm{S}}, q^{\mathrm{S}}\right)_{\Omega_{\mathrm{S}}}, \qquad 
 \langle L_p\left(p^{\mathrm{P}}\right),q^{\mathrm{P}}\rangle:=l_p\left(p^{\mathrm{P}}, q^{\mathrm{P}}\right)=\left( p^{\mathrm{P}}, q^{\mathrm{P}}\right)_{\Omega_{\mathrm{P}}} .
\end{gather*}

\begin{lemma}\label{continuity_R}
    The operators $R_s, R_r, L_f$, and $L_p$ are bounded, continuous, and coercive.
\end{lemma}
\begin{proof}
The operators satisfy the following continuity and coercivity bounds: 
\begin{align*}
 \langle R_s \left(\bu_s^{\mathrm{P}}\right),\bv_s^{\mathrm{P}}\rangle  &\lesssim \|\bu_s^{\mathrm{P}}\|_{1,\Omega_{\mathrm{P}}}\|\bv_s^{\mathrm{P}}\|_{1,\Omega_{\mathrm{P}}}   \quad\text{and}\quad \langle R_s\left(\bu_s^{\mathrm{P}}\right),\bu_s^{\mathrm{P}}\rangle  \gtrsim \|\bu_s^{\mathrm{P}}\|_{1,\Omega_{\mathrm{P}}}^2 \qquad \forall \bu_s^{\mathrm{P}}, \bv_s^{\mathrm{P}} \in \mathbf{X}_d, \\
 \langle R_r (\bu_r^{\mathrm{P}}), \bv_r^{\mathrm{P}}\rangle & \lesssim \|\bu_r^{\mathrm{P}}\|_{1,\Omega_{\mathrm{P}}}\|\bv_r^{\mathrm{P}}\|_{1,\Omega_{\mathrm{P}}}  \quad\text{and}\quad    \langle R_r (\bu_r^{\mathrm{P}}),\bu_r^{\mathrm{P}}\rangle \gtrsim \|\bu_r^{\mathrm{P}}\|_{1,\Omega_{\mathrm{P}}}^2  \qquad \forall \bu_r^{\mathrm{P}}, \bv_r^{\mathrm{P}} \in \mathbf{V}_r,\\
 \langle L_f\left(p^{\mathrm{S}}\right),q^{\mathrm{S}}\rangle & \lesssim \|p^{\mathrm{S}}\|_{0,\Omega_{\mathrm{S}}}\|q^{\mathrm{S}}\|_{0,\Omega_{\mathrm{S}}}  \quad\text{and}\quad \langle L_f\left(p^{\mathrm{S}}\right),p^{\mathrm{S}}\rangle  \gtrsim  \|p^{\mathrm{S}}\|^2_{0,\Omega_{\mathrm{S}}}  \qquad  \forall p^{\mathrm{S}}, q^{\mathrm{S}} \in W_f,\\
 \langle L_p\left(p^{\mathrm{P}}\right),q^{\mathrm{P}}\rangle & \lesssim \|p^{\mathrm{P}}\|_{0,\Omega_{\mathrm{P}}}\|q^{\mathrm{P}}\|_{0,\Omega_{\mathrm{P}}}  \quad\text{and}\quad \langle L_p\left(p^{\mathrm{P}}\right),p^{\mathrm{P}}\rangle  \gtrsim 
 \|p^{\mathrm{P}}\|_{0,\Omega_{\mathrm{P}}}  \qquad \forall p^{\mathrm{P}}, q^{\mathrm{P}} \in W_p.  
\end{align*}
The coercivity bounds follow directly from the definitions, using K\"orn's inequality \cite{MR2373954} for $R_s$, whereas the continuity bounds follow from Cauchy--Schwarz and Young's inequalities. 
\end{proof}
For the regularization of the Lagrange multiplier, let $\psi(\lambda) \in {H}^1\left(\Omega_{\mathrm{P}}\right)$ be the solution of the auxiliary problem 
$$
\begin{gathered}
-\nabla \cdot \nabla \psi(\lambda)=0 \quad \text { in } \quad \Omega_{\mathrm{P}}, \\
\nabla \psi(\lambda) \cdot \bn =\lambda \quad \text { on } \Sigma, \quad  \psi(\lambda) =0 \quad \text { on } \Gamma_{\mathrm{P}} .
\end{gathered}
$$
Note that the solution depends continuously on data, and the trace continuity imply that there exist positive constants $c^{\star}$ and $C^{\star} $, such that
\begin{equation}\label{lambda}
c^{\star}\|\psi(\lambda)\|_{1,\Omega_{\mathrm{P}}}\leq\|\lambda\|_{-1/2,\Sigma}  \leq C^{\star}\|\psi(\lambda)\|_{1,\Omega_{\mathrm{P}}}.
\end{equation}
\begin{lemma}\label{lambda_coercive}
   The operator $L_\lambda: \Lambda \rightarrow \Lambda^{\prime}$, defined as
   $$
\langle L_\lambda \lambda, \xi \rangle =l_\lambda(\lambda, \xi):=(\nabla \psi(\lambda), \nabla \psi(\xi))_{\Omega_{\mathrm{P}}},
$$
   is continuous and coercive. 
\end{lemma}
\begin{proof}
    It follows from \eqref{lambda} that there exist positive constants $c^{\star}$ and $C^{\star}$ such that 
\[
\langle L_\lambda \lambda, \xi\rangle \leq C^{\star} \|\lambda\|_{-1/2,\Sigma}\|\xi\|_{-1/2,\Sigma}, \quad  \langle L_\lambda \lambda, \lambda \rangle \geq c^{\star} \|\lambda\|_{-1/2,\Sigma}^2, \qquad  \forall \lambda, \xi \in \Lambda.
\]     
\end{proof} 
In order to establish that  $\bbD$ is nonempty we first show that there exists a solution to a regularization of \eqref{domain1}-\eqref{domain3} and then send the regularization parameter to zero.
\begin{lemma}\label{domain_nonempty}
The domain $\bbD$ specified by \eqref{domaindefined} is nonempty.
\end{lemma}
\begin{proof}
We follow four steps. 

\noindent\textbf{1. Regularization of \eqref{operator_form1}-\eqref{operator_form2}:}
For $\vec{\bv}^{i}=(\bv_{r}^{\mathrm{P},i}, \bv_{s}^{\mathrm{P},i}, \bv_{f}^{\mathrm{S},i}) \in \vec{\mathbf{V}}, \vec{q}^{i}=\left(q^{\mathrm{P},i}, \btau^i, q^{\mathrm{S},i}, \mu^i\right) \in \vec{Q}, i=1,2$, define the operators $\mathcal{R}: \vec{\mathbf{V}} \rightarrow \vec{\mathbf{V}}^{\prime}$ and $\mathcal{L}: \vec{Q} \rightarrow \vec{Q}^{\prime}$ as
\begin{align*}
\langle \mathcal{R}\vec{\bv}^{1},\vec{\bv}^{2}\rangle & := \langle R_s(\bv_{s}^{\mathrm{P},1}),\bv_{s}^{\mathrm{P},2}\rangle+ \langle R_r\left(\bv_{r}^{\mathrm{P},1}\right),\bv_{r}^{\mathrm{P},2}\rangle=r_s\left(\bv_{s}^{\mathrm{P},1}, \bv_{s}^{\mathrm{P},2}\right)+r_r\left(\bv_{r}^{\mathrm{P},1}, \bv_{r}^{\mathrm{P},2}\right), \\
\langle \mathcal{L}\vec{q}^{1},\vec{q}^{2}\rangle & := \langle L_f (q^{\mathrm{S},1}), q^{\mathrm{S},2}\rangle + \langle L_p (q^{\mathrm{P},1}),q^{\mathrm{P},2}\rangle +\langle L_{\Gamma}\left(\mu^1\right),\mu^2\rangle 
 =l_f\left(q^{\mathrm{S},1}, q^{\mathrm{S},2}\right)+l_p\left(q^{\mathrm{P},1}, q^{\mathrm{P},2}\right)+l_{\Gamma}\left(\mu^1, \mu^2\right) .
\end{align*}
For $\epsilon>0$, consider the regularized problem: 
given $\bar{f} \in \vec{\mathbf{V}}^{\prime}, \bar{g} \in \vec{Q}^{\prime}$, determine $\vec{\bv}_\epsilon \in \vec{\mathbf{V}}, \vec{q}_\epsilon \in \vec{Q}$ satisfying
\begin{subequations}
\begin{align}
(\epsilon \mathcal{R}+\mathcal{E}_1+\mathcal{A}) \vec{\bv}_\epsilon+\mathcal{B}^{\prime} \vec{q}_\epsilon & =\bar{f} \quad \text { in}\quad \vec{\mathbf{V}}^{\prime},  \label{R} \\
-\mathcal{B} \vec{\bv}_\epsilon+\left(\epsilon \mathcal{L}+\mathcal{E}_2\right) \vec{q}_\epsilon & =\bar{g} \quad \text { in}\quad \vec{Q}^{\prime} \label{L} .
\end{align}
\end{subequations}

\noindent\textbf{2. Existence of unique solution of \eqref{R}-\eqref{L}:}
Introduce the operator $\mathcal{O}: \vec{\mathbf{V}} \times \vec{Q} \rightarrow(\vec{\mathbf{V}} \times \vec{Q})^{\prime}$ defined as
$$
\mathcal{O}\left(\begin{array}{c}
\vec{\bv} \\
\vec{q}
\end{array}\right)=\left(\begin{array}{cc}
\epsilon \mathcal{R}+\mathcal{E}_1+\mathcal{A} & \mathcal{B}^{\prime} \\
-\mathcal{B} & \epsilon \mathcal{L}+\mathcal{E}_2
\end{array}\right)\left[\begin{array}{l}
\vec{\bv} \\
\vec{q}
\end{array}\right].
$$
From Lemmas \ref{coercivity-continuity}-\ref{continuity-b}  it can be shown that $\mathcal{O}$ is  bounded and continuous. Additionally, leveraging Lemmas \ref{coercivity-continuity} and \ref{lambda_coercive}, Assumptions in Section~\ref{assumptions}, and the  inequality 
\begin{equation}\label{ineq:aux} 
-(\xi a,a) -(\xi b,b) \leq 2 (\xi a,b) \leq (\xi a,a) + (\xi b,b),\end{equation} 
we can conclude that
\begin{align}
\nonumber\biggl\langle\mathcal{O}\begin{pmatrix}
    \vec{\bv}\\ \vec{q}\end{pmatrix},\begin{pmatrix}
    \vec{\bv}\\\vec{q}\end{pmatrix}\biggr\rangle 
& =  \epsilon r_s\left(\bv_s^{\mathrm{P}}, \bv_s^{\mathrm{P}}\right)+\epsilon r_r\left(\bv_r^{\mathrm{P}}, \bv_r^{\mathrm{P}}\right)+a_f^{\mathrm{S}}\left(\bv_f^{\mathrm{S}}, \bv_f^{\mathrm{S}}\right)+a_f^{\mathrm{P}}\left(\bv_r^{\mathrm{P}}, \bv_s^{\mathrm{P}}\right)+ a_f^{\mathrm{P}}\left(\bv_r^{\mathrm{P}}, \bv_r^{\mathrm{P}}\right)  + a_f^{\mathrm{P}}\left(\bv_s^{\mathrm{P}}, \bv_r^{\mathrm{P}}\right) \nonumber \\
& \quad + a_f^{\mathrm{P}}\left(\bv_s^{\mathrm{P}}, \bv_s^{\mathrm{P}}\right)+ a_{\mathrm{BJS}}\left(\bv_f^{\mathrm{S}}, \bv_s^{\mathrm{P}} ; \bv_f^{\mathrm{S}}, \bv_s^{\mathrm{P}}\right) -m_{\theta}(\bv_r^{\mathrm{P}},\bv_s^{\mathrm{P}}) - m_{\theta}(\bv_s^{\mathrm{P}},\bv_s^{\mathrm{P}})- m_{\theta}(\bv_s^{\mathrm{P}},\bv_r^{\mathrm{P}}) \nonumber \\
& \quad - m_{\theta}(\bv_r^{\mathrm{P}},\bv_r^{\mathrm{P}}) + m_{\phi^2/\kappa}(\bv_r^{\mathrm{P}}, \bv_r^{\mathrm{P}}) + \left( (1-\phi)^2   K^{-1} q^{\mathrm{P}},q^{\mathrm{P}}\right)   +a_p^{\mathrm{P}}\left(\btau, \btau\right)+\epsilon l_f\left(q^{\mathrm{S}}, q^{\mathrm{S}}\right) \nonumber \\
& \quad +\epsilon l_p\left(q^{\mathrm{P}}, q^{\mathrm{P}}\right)+\epsilon l_{\Gamma}(\mu, \mu)  + \left(\rho_f \phi \bv_r^{\mathrm{P}}, \bv_r^{\mathrm{P}} \right)_{\Omega_{\mathrm{P}}} + \left(\rho_f \phi \bv_s^{\mathrm{P}}, \bv_r^{\mathrm{P}} \right)_{\Omega_{\mathrm{P}}} + \left(\rho_f \phi \bv_r^{\mathrm{P}}, \bv_s^{\mathrm{P}} \right)_{\Omega_{\mathrm{P}}} + \left(\rho_p  \bv_s^{\mathrm{P}}, \bv_s^{\mathrm{P}} \right)_{\Omega_{\mathrm{P}}} \nonumber \\
& \geq  C\left(\epsilon\|\beps(\bv_r^{\mathrm{P}})\|_{0,\Omega_{\mathrm{P}}}^2+\epsilon\|\beps(\bv_s^{\mathrm{P}})\|_{0,\Omega_{\mathrm{P}}}^2+  \| \bv_r^{\mathrm{P}} \|_{0,\Omega_{\mathrm{P}}}^2  + \left|\bv_f^{\mathrm{S}}-\bv_s^{\mathrm{P}}\right|_{\mathrm{BJS}}^2 +\|\beps(\bv_f^{\mathrm{S}})\|_{0,\Omega_{\mathrm{S}}}^2 \right. \nonumber \\& \quad  +\|\btau\|_{0,\Omega_{\mathrm{P}}}^2 +\epsilon\|q^{\mathrm{S}}\|_{0,\Omega_{\mathrm{S}}}^2+(1-\phi)^2 K^{-1} \|q^{\mathrm{P}}\|_{0,\Omega_{\mathrm{P}}}^2 
\left.+\epsilon\|\mu\|_{-1/2,\Sigma}^2 + \| \bu_s^{\mathrm{P}}\|_{0,\Omega_{\mathrm{P}}}^2 \right) .\label{epsilon_coercive}
\end{align}
It follows that $\mathcal{O}$ is coercive. Thus, an application of the Lax--Miligram Lemma establishes the existence of a solution $\left(\vec{\bu}_\epsilon, \vec{p}_\epsilon\right) \in \vec{\mathbf{V}} \times \vec{Q}$ of \eqref{R}-\eqref{L}, where $\vec{\bu}_\epsilon=\left(\bu_{r,\epsilon}^{\mathrm{P}}, \bu_{s,\epsilon}^{\mathrm{P}}, \bu_{f,\epsilon}^{\mathrm{S}}\right)$ and $\vec{p}_\epsilon=\left(p^{\mathrm{P}}_{\epsilon}, \bsigma_{ \epsilon}, p^{\mathrm{S}}_{\epsilon}, \lambda_\epsilon\right)$.

\smallskip 
\noindent\textbf{3. Uniform boundedness:}
From inequality \eqref{epsilon_coercive} and \eqref{R}-\eqref{L}, we have that 
\begin{align}
&\epsilon \| \bu_{r,\epsilon}^{\mathrm{P}} \|_{1,\Omega_{\mathrm{P}}}^2+\epsilon \| \bu_{s,\epsilon}^{\mathrm{P}}\|^2_{1,\Omega_{\mathrm{P}}}+\| \bu_{f,\epsilon}^{\mathrm{S}}\|_{1,\Omega_{\mathrm{S}}}^2+\left|\bu_{f,\epsilon}^{\mathrm{S}}-\bu_{s,\epsilon}^{\mathrm{P}}\right|_{\mathrm{BJS}}^2   + \| \bu_{r,\epsilon}^{\mathrm{P}} \|_{0,\Omega_{\mathrm{P}}}^2 \nonumber \\
&
\quad +\|\bsigma_{\epsilon}\|_{0,\Omega_{\mathrm{P}}}^2+\epsilon\|p^{\mathrm{S}}_{\epsilon}\|_{0,\Omega_{\mathrm{S}}}^2+ (1-\phi)^2 K^{-1} \|p^{\mathrm{P}}_{\epsilon}\|_{0,\Omega_{\mathrm{P}}}^2  + \| \bu_{s,\epsilon}^{\mathrm{P}} \|_{0,\Omega_{\mathrm{P}}}^2 +\epsilon\|\lambda_{\epsilon} \|_{H^{-1/2}(\Sigma)}^2 \nonumber
\\ & \leq C\left(\|r_{\mathrm{S}}\|_{0,\Omega_{\mathrm{S}}}\|p^{\mathrm{S}}_{\epsilon}\|_{0,\Omega_{\mathrm{S}}} + \|\bar{f}_s\|_{0,\Omega_{\mathrm{P}}}\|\bu_{s,\epsilon}^{\mathrm{P}}\|_{0,\Omega_{\mathrm{P}}}  + \|\bar{f}_r\|_{0,\Omega_{\mathrm{P}}}\|\bu_{r,\epsilon}^{\mathrm{P}}\|_{0,\Omega_{\mathrm{P}}}  + \|\bar{f}_r\|_{0,\Omega_{\mathrm{P}}}\|\bu_{s,\epsilon}^{\mathrm{P}}\|_{0,\Omega_{\mathrm{P}}}   \right.  \nonumber  \\& \qquad   \left. + \|\bar{f}_s\|_{0,\Omega_{\mathrm{P}}}\|\bu_{r,\epsilon}^{\mathrm{P}}\|_{0,\Omega_{\mathrm{P}}} +\|\bar{g}_p\|_{0,\Omega_{\mathrm{P}}}\|p^{\mathrm{P}}_{\epsilon}\|_{0,\Omega_{\mathrm{P}}} +\|\bar{g}_e\|_{0,\Omega_{\mathrm{P}}}\|\bsigma_{ \epsilon}\|_{0,\Omega_{\mathrm{P}}} +   \|\ff_{\mathrm{S}}\|_{-1,\Omega_{\mathrm{P}}} \|\bu_{f,\epsilon}^{\mathrm{S}} \|_{1,\Omega_{\mathrm{S}}} \right) \label{inequalitycoercivity}.
\end{align}
On the other hand, as a consequence of  \eqref{domain2}, it follows that $\bsigma_{ \epsilon}$ and $\bv_{s, \epsilon}$ satisfy
\[
a_p^{\mathrm{P}}\left(\bsigma_{ \epsilon}, \btau\right)-b_{\text{sig}}^{\mathrm{P}}\left(\bu_{s,\epsilon}^{\mathrm{P}}, \btau\right)=\left(A \bar{g}_e, \btau\right)_{\Omega_{\mathrm{P}}} \qquad \forall \btau \in \bZ .
\]
Therefore, applying the inf-sup condition \eqref{inf-sup1}, we obtain:
\begin{align}
 \|\bu_{s,\epsilon}^{\mathrm{P}}\|_{1,\Omega_{\mathrm{P}}} & \lesssim  \sup _{\left(0, \btau, 0,0\right) \in \vec{Q}} \frac{b_{\text{sig}}^{\mathrm{P}}\left(\bu_{s,\epsilon}^{\mathrm{P}}, \btau\right)}{\|(0, \btau, 0,0)\|_{\vec{Q}}}= \sup _{\left(0, \btau, 0,0) \in \vec{Q}\right.} \frac{a_p^{\mathrm{P}}\left(\bsigma_{ \epsilon}, \btau\right)-\left(A \bar{g}_e, \btau\right)_{\Omega_{\mathrm{P}}}}{\|(0, \btau, 0,0)\|_{\vec{Q}}}  \lesssim  \|\bsigma_{ \epsilon}\|_{0,\Omega_{\mathrm{P}}}+\|{\bar{g}}_e\|_{0,\Omega_{\mathrm{P}}}  \label{inf-sup3} . 
\end{align}
Combining \eqref{inequalitycoercivity} and \eqref{inf-sup3}, and using Young's inequality,
we obtain
\begin{align*}
& \| \bu_{s,\epsilon}^{\mathrm{P}} \|_{1,\Omega_{\mathrm{P}}}^2+\epsilon \| \bu_{r,\epsilon}^{\mathrm{P}}\|_{1,\Omega_{\mathrm{P}}}^2+ \| \bu_{s,\epsilon}^{\mathrm{P}}\|^2_{0,\Omega_{\mathrm{P}}}+\| \bu_{f,\epsilon}^{\mathrm{S}}\|_{1,\Omega_{\mathrm{S}}}^2+\left|\bu_{f,\epsilon}^{\mathrm{S}}-\bu_{s,\epsilon}^{\mathrm{P}}\right|_{\mathrm{BJS}}^2 \nonumber \\&  + \| \bu_{r,\epsilon}^{\mathrm{P}} \|_{0,\Omega}^2 
+\| \bsigma_{\epsilon}\|_{0,\Omega_{\mathrm{P}}}^2  +\epsilon\|p^{\mathrm{S}}_{\epsilon}\|_{0,\Omega_{\mathrm{S}}}^2 +  (1-\phi)^2 K^{-1} \|p^{\mathrm{P}}_{\epsilon} \|_{0,\Omega_{\mathrm{P}}}^2+\epsilon\|\lambda_\epsilon \|_{-1/2,\Sigma}^2 \nonumber
\\ &
\leq  C\left( \|r_{\mathrm{S}}\|_{0,\Omega_{\mathrm{S}}}\|p^{\mathrm{S}}_{\epsilon}\|_{0,\Omega_{\mathrm{S}}} + 
\|\bar{g}_p\|_{0,\Omega_{\mathrm{P}}}\|p^{\mathrm{P}}_{\epsilon}\|_{0,\Omega_{\mathrm{P}}} + \|\bar{f}_r\|_{0,\Omega_{\mathrm{P}}}^2 +\|\bar{f}_s\|_{0,\Omega_{\mathrm{P}}}^2 + \|\bar{g}_e\|_{0,\Omega_{\mathrm{P}}}^2  +\| \ff_{\mathrm{S}}\|_{-1,\Omega_{\mathrm{S}}}^2 \right)\nonumber  \\&  
\quad  +\frac{1}{2} \left( 
\|\bu_{f,\epsilon}^{\mathrm{S}}\|_{1,\Omega_{\mathrm{S}}}^2+ \|\bu_{s,\epsilon}^{\mathrm{P}}\|_{0,\Omega_{\mathrm{P}}}^2  + \|\bu_{r,\epsilon}^{\mathrm{P}}\|_{0,\Omega_{\mathrm{P}}}^2 +      \|\bsigma_{ \epsilon}\|_{0,\Omega_{\mathrm{P}}}^2         \right), 
\end{align*}
from which it follows that
\begin{align}
& \| \bu_{s,\epsilon}^{\mathrm{P}} \|_{1,\Omega_{\mathrm{P}}}^2+\epsilon \| \bu_{r,\epsilon}^{\mathrm{P}} \|_{1,\Omega_{\mathrm{P}}}^2+\| \bu_{f,\epsilon}^{\mathrm{S}}\|^2_{1,\Omega_{\mathrm{S}}}+\left|\bu_{f,\epsilon}^{\mathrm{S}}-\bu_{s,\epsilon}^{\mathrm{P}}\right|_{\mathrm{BJS}}^2 + \| \bu_{r,\epsilon}^{\mathrm{P}} \|_{0,\Omega_{\mathrm{P}}}^2 + \| \bu_{s,\epsilon}^{\mathrm{P}} \|_{0,\Omega_{\mathrm{P}}}^2  +\| \bsigma_{\epsilon}\|_{0,\Omega_{\mathrm{P}}}^2 
 \nonumber \\
 &\quad  \lesssim  \|r_{\mathrm{S}}\|_{0,\Omega_{\mathrm{S}}}\|p^{\mathrm{S}}_{\epsilon}\|_{0,\Omega_{\mathrm{S}}} + 
\|\bar{g}_p\|_{0,\Omega_{\mathrm{P}}}\|p^{\mathrm{P}}_{\epsilon}\|_{0,\Omega_{\mathrm{P}}}    + \|\bar{f}_r\|_{0,\Omega_{\mathrm{P}}}^2  + \|\bar{f}_s\|_{0,\Omega_{\mathrm{P}}}^2  + \|\bar{g}_e\|_{0,\Omega_{\mathrm{P}}}^2 + \| \ff_{\mathrm{S}}\|_{-1,\Omega_{\mathrm{P}}}^2. \label{inequalitycoercivity2}
\end{align}
To obtain bounds for $p^{\mathrm{P}}_{\epsilon}, p^{\mathrm{S}}_{\epsilon}$, and $\lambda_\epsilon$ we use \eqref{inf-sup2}. With $\vec{p}=\left(p^{\mathrm{P}}_{\epsilon}, 0, p^{\mathrm{S}}_{\epsilon}, \lambda_\epsilon\right) \in \vec{Q}$, we have
\begin{align}
& \|p^{\mathrm{S}}_{\epsilon}\|_{0,\Omega_{\mathrm{S}}}+\|p^{\mathrm{P}}_{\epsilon}\|_{0,\Omega_{\mathrm{P}}}+\| \lambda_\epsilon \|_{-1/2,\Sigma} \nonumber \\ 
& \leq C \sup _{\left(\bv_r^{\mathrm{P}}, \mathbf{0}, \bv_f^{\mathrm{S}}\right) \in \vec{\mathbf{V}}} \frac{b^{\mathrm{S}}(\bv_f^{\mathrm{S}}, p^{\mathrm{S}}_{\epsilon})+b_f^{\mathrm{P}}(\bv_r^{\mathrm{P}}, p_{ H,\epsilon})+b_{\Gamma}(\bv_r^{\mathrm{P}}, \bv_f^{\mathrm{S}}, \mathbf{0} ; \lambda_{\epsilon})}{\|(\bv_r^{\mathrm{P}}, \mathbf{0}, \bv_f^{\mathrm{S}})\|_{\vec{\mathbf{V}}}} \nonumber \\
&  \lesssim  \epsilon \| \bu_{r,\epsilon}^{\mathrm{P}} \|_{1,\Omega_{\mathrm{P}}} +\| \bu_{f,\epsilon}^{\mathrm{S}}\|_{1,\Omega_{\mathrm{S}}} + \| \bu_{s,\epsilon}^{\mathrm{P}}\|_{1,\Omega_{\mathrm{P}}} +\left|\bu_{f,\epsilon}^{\mathrm{S}}-\bu_{s,\epsilon}^{\mathrm{P}}\right|_{\mathrm{BJS}} +  \| \bu_{s,\epsilon}^{\mathrm{P}}  \|_{0,\Omega_{\mathrm{P}}} + \| \bu_{r,\epsilon}^{\mathrm{P}} \|_{0,\Omega_{\mathrm{P}}}  + \| \bar{f}_r  \|_{0,\Omega_{\mathrm{P}}} + \| \ff_{\mathrm{S}}\|_{-1,\Omega_{\mathrm{S}}} \label{inequalitycoercivity3}.
\end{align}
Employing again Young's inequality, \eqref{inequalitycoercivity2}, and \eqref{inequalitycoercivity3}, {we} arrive at  
\begin{align*}
& \| \bu_{s,\epsilon}^{\mathrm{P}} \|_{1,\Omega_{\mathrm{P}}}^2+\epsilon \| \bu_{r,\epsilon}^{\mathrm{P}} \|_{1,\Omega_{\mathrm{P}}}^2+\| \bu_{f,\epsilon}^{\mathrm{S}}\|_{1,\Omega_{\mathrm{S}}}^2+\left|\bu_{f,\epsilon}^{\mathrm{S}}-\bu_{s,\epsilon}^{\mathrm{P}}\right|_{\mathrm{BJS}}^2 
+ \| \bu_{r,\epsilon}^{\mathrm{P}}\|_{0,\Omega_{\mathrm{P}}}^2  + \| \bu_{s,\epsilon}^{\mathrm{P}}\|_{0,\Omega_{\mathrm{P}}}^2 \nonumber \\ 
& \quad +\| \bsigma_{\epsilon}\|_{0,\Omega_{\mathrm{P}}}^2  +  \|p^{\mathrm{S}}_{\epsilon}\|^2_{0,\Omega_{\mathrm{S}}} +\|p^{\mathrm{P}}_{\epsilon}\|^2_{0,\Omega_{\mathrm{P}}}+\| \lambda_\epsilon \|^2_{-1/2,\Sigma} \\
&
\leq  C\left( \|r_{\mathrm{S}}\|_{0,\Omega_{\mathrm{S}}}^2  + 
\|\bar{g}_p\|_{0,\Omega_{\mathrm{P}}} ^2  + \|\bar{f}_r\|_{0,\Omega_{\mathrm{P}}}^2 + \|\bar{f}_s\|_{0,\Omega_{\mathrm{P}}}^2  + \|\bar{g}_e\|_{0,\Omega_{\mathrm{P}}}^2 + \| \ff_{\mathrm{S}}\|_{-1,\Omega_{\mathrm{S}}}^2\right) , 
\end{align*}
which implies that all the quantities $\|\bu_{s,\epsilon}^{\mathrm{P}}\|_{1,\Omega_{\mathrm{P}}}$, $\|\bu_{f,\epsilon}^{\mathrm{S}} \|_{1,\Omega_{\mathrm{S}}}$, $\|\bu_{r,\epsilon}^{\mathrm{P}}\|_{0,\Omega_{\mathrm{P}}}$, $\|\bu_{s,\epsilon}^{\mathrm{P}}\|_{0,\Omega_{\mathrm{P}}}$, 
$\|\bsigma_{ \epsilon}\|_{0,\Omega_{\mathrm{P}}}$, $\|p^{\mathrm{S}}_{\epsilon}\|_{0,\Omega_{\mathrm{S}}}$,  $\|p^{\mathrm{P}}_{\epsilon}\|_{0,\Omega_{\mathrm{P}}}$, and $\|\lambda_\epsilon\|_{-1/2,\Sigma}$ are bounded independently of $\epsilon$. Using   \eqref{R} and the continuity of $R_r$ (cf.  Lemma \ref{continuity_R}), we can infer that 
\begin{align*}
\left(\epsilon +1 \right) \| \bu_{r,\epsilon}^{\mathrm{P}}\|_{1,\Omega_{\mathrm{P}}}  &\lesssim  \| \bu_{r,\epsilon}^{\mathrm{P}}\|_{0,\Omega_{\mathrm{P}}}+\| \bu_{s,\epsilon}^{\mathrm{P}}\|_{1,\Omega_{\mathrm{P}}} + \| p^{\mathrm{P}}_{\epsilon}\|_{0,\Omega_{\mathrm{P}}} + \|\lambda_\epsilon\|_{-1/2,\Sigma}  + \|\bar{f}_r\|_{0,\Omega_{\mathrm{P}}}  + \|\bar{f}_s\|_{0,\Omega_{\mathrm{P}}} + \| \ff_{\mathrm{S}}\|_{-1,\Omega_{\mathrm{S}}},\\
\| \bu_{r,\epsilon}^{\mathrm{P}}\|_{1,\Omega_{\mathrm{P}}}  & \lesssim  \| \bu_{r,\epsilon}^{\mathrm{P}}\|_{0,\Omega_{\mathrm{P}}}+\| \bu_{s,\epsilon}^{\mathrm{P}}\|_{1,\Omega_{\mathrm{P}}} + \| p^{\mathrm{P}}_{\epsilon}\|_{0,\Omega_{\mathrm{P}}} + \|\lambda_\epsilon\|_{-1/2,\Sigma}  + \|\bar{f}_r\|_{0,\Omega_{\mathrm{P}}} + \|\bar{f}_s\|_{0,\Omega_{\mathrm{P}}} + \| \ff_{\mathrm{S}}\|_{-1,\Omega_{\mathrm{S}}}.
 \end{align*}
Therefore $\|\bu_{r,\epsilon}^{\mathrm{P}}\|_{1,\Omega_{\mathrm{P}}}$ is also bounded independently of $\epsilon$.

\noindent\textbf{4. Passing to the limit:}
Since $\vec{\mathbf{V}}$ and $\vec{Q}$ are reflexive Banach spaces, as $\epsilon \rightarrow 0$ we can apply the Banach--Alaoglu--Bourbaki Theorem \cite{MR2759829} to extract weakly convergent subsequences $\left\{\vec{\bv}_{\epsilon, n}\right\}_{n=1}^{\infty}$, $\left\{\vec{q}_{\epsilon, n}\right\}_{n=1}^{\infty}$ and $\left\{\mathcal{A} \vec{\bv}_{\epsilon, n}\right\}_{n=1}^{\infty}$, such that $\vec{\bv}_{\epsilon, n} \rightarrow \vec{\bv}$ in $\vec{\mathbf{V}}, \vec{q}_{\epsilon, n} \rightarrow \vec{q}$ in $\vec{Q}, \mathcal{A} \vec{\bv}_{\epsilon, n} \rightarrow \zeta$ in $\vec{\mathbf{V}}^{\prime}$, and
$$
\zeta+\mathcal{E}_1 \vec{\bv} +\mathcal{B}^{\prime} \vec{q}=\bar{f}  \quad\text{in}\quad \vec{\mathbf{V}}^{\prime}, \qquad 
\mathcal{E}_2 \vec{q}-\mathcal{B} \vec{\bv}=\bar{g}  \quad\text{in}\quad \vec{Q}^{\prime}.
$$
Moreover, from \eqref{R}--\eqref{L} we can infer that 
\begin{align*}
\limsup _{\epsilon \rightarrow 0}\left(\langle \mathcal{A} \vec{\bv}_\epsilon,\vec{\bv}_\epsilon\rangle+ \langle\mathcal{E}_1\vec{\bv}_\epsilon,\vec{\bv}_\epsilon\rangle+\langle\mathcal{E}_2\vec{q}_\epsilon,\vec{q}_\epsilon\rangle\right) & =\limsup _{\epsilon \rightarrow 0}\left(\bar{f}\left(\vec{\bv}_\epsilon\right)+\bar{g}\left(\vec{q}_\epsilon\right)-\epsilon \langle \mathcal{R}\vec{\bv}_\epsilon,\vec{\bv}_\epsilon\rangle-\epsilon \langle\mathcal{L}\vec{q}_\epsilon,\vec{q}_\epsilon\rangle\right) \\
& \leq \bar{f}(\vec{\bv})+\bar{g}(\vec{q})
=\zeta(\vec{\bv})+\langle \mathcal{E}_1 \vec{\bv},\vec{\bv}\rangle+\langle \mathcal{E}_2 \vec{q},\vec{q}\rangle.
\end{align*}
Since $\mathcal{A}+\mathcal{E}_1+\mathcal{E}_2$ is monotone and continuous, it follows (see \cite[Def. on p. 38]{MR1422252}) that $\mathcal{A} \vec{\bv}=\zeta$. Hence, $\vec{\bv}$ and $\vec{q}$ solve \eqref{domain1}-\eqref{domain3}, which establishes that $\bbD$ is nonempty.
\end{proof} 
\begin{corollary}\label{resolvent}
Under the assumptions \ref{(H1)}-\ref{(H3)}, 
we have that $\operatorname{Rg}(I+\mathcal{M})= \bbS_2^{\prime}$.
\begin{proof}
We need to show that for $\ff \in  \bbS_2^{\prime}$ there exists $\bv \in \bbD$ such that $\ff \in(I+\mathcal{M})(\bv)$.
Let $\left(\bar{f}_r, \bar{f}_s,\bar{g}_p, \bar{g}_e\right) \in  \bbS_2^{\prime}$. From Lemma \ref{domain_nonempty}, there exists $\left(\tilde{\bu}_r, \tilde{\bu}_r, \tilde{p}_{\mathrm{P}}, \tilde{\bsigma}_e\right) \in \bbD$ solving \eqref{domain1}-\eqref{domain3}. Hence $\left(\bar{f}_r-\tilde{\bu}_r, \bar{f}_s-\tilde{\bu}_s, \bar{g}_p-\tilde{p}_{\mathrm{P}},  \bar{g}_e-\tilde{\bsigma}_e\right)$ $\in \mathcal{M}\left(\tilde{\bu}_r, \tilde{\bu}_s,\tilde{p}_{\mathrm{P}}, \tilde{\bsigma}_e\right)$ and therefore it follows that $\left(\bar{f}_r,\bar{f}_s,\bar{g}_p, \bar{g}_e\right) \in(I+\mathcal{M})\left(\tilde{\bu}_r, \tilde{\bu}_s, \tilde{p}_{\mathrm{P}}, \tilde{\bsigma}_e \right)$.    
\end{proof}
\end{corollary}

\medskip 
\noindent\textbf{\underline{Step 2:} Solvability of the parabolic problem.}
We begin by showing that $\mathcal{M}$ (cf. \eqref{Mset}) is a monotone operator.
\begin{lemma}\label{monotonicity}
Under the assumptions \ref{(H1)}-\ref{(H3)}, the operator $\mathcal{M}$ defined by \eqref{parabolic} is monotone.
\end{lemma} 
\begin{proof}
  We need to show for $\ff \in \mathcal{M}(\bv), \ff \in \mathcal{M}(\tilde{\bv})$ that $(\ff-\tilde{\ff}, \bv-\tilde{\bv})_{\bbS_2} \geq 0$.
For $\left(\bu_r^{\mathrm{P}}, \bu_s^{\mathrm{P}}, p^{\mathrm{P}}, \bsigma\right) \in \bbD$, it holds that  $(\bar{f}_r - \bu_r^{\mathrm{P}}, \bar{f}_s- \bu_s^{\mathrm{P}}, \bar{g}_p-p^{\mathrm{P}}, \bar{g}_e-\bsigma) \in \mathcal{M}\left(\bu_r^{\mathrm{P}}, \bu_s^{\mathrm{P}},  p^{\mathrm{P}}, \bsigma\right)$, $\left(\bv_r^{\mathrm{P}}, \bv_s^{\mathrm{P}}, q^{\mathrm{P}}, \btau\right) \in \bbS_2$. And using the definition of inner product of $((\cdot,\cdot))_{\bbS_2}$, \eqref{domain1} and \eqref{domain2} we have
\begin{align}
&\left(\left(\bar{f}_r-\bu_r^{\mathrm{P}}, \bar{f}_s-\bu_s^{\mathrm{P}}, \bar{g}_p-p^{\mathrm{P}}, \bar{g}_e-\bsigma\right),\left(\bv_r^{\mathrm{P}},\bv_s^{\mathrm{P}},q^{\mathrm{P}}, \btau\right)\right)_{\bbS_2} \nonumber \\
& = a_f^{\mathrm{S}}\left(\bu_f^{\mathrm{S}}, \bv_f^{\mathrm{S}}\right)+a_{f}^{\mathrm{P}}(\bu_r^{\mathrm{P}}, \bv_s^{\mathrm{P}})+ a_{f}^{\mathrm{P}}\left(\bu_r^{\mathrm{P}}, \bv_r^{\mathrm{P}}\right)+ a_{f}^{\mathrm{P}}\left(\bu_s^{\mathrm{P}}, \bv_r^{\mathrm{P}} \right) + a_{f}^{\mathrm{P}}\left(\bu_s^{\mathrm{P}}, \bv_s^{\mathrm{P}}\right)  +a_{\mathrm{BJS}}\left(\bu_f^{\mathrm{S}}, \bu_s^{\mathrm{P}}; \bv_f^{\mathrm{S}}, \bv_s^{\mathrm{P}}\right) +b^{\mathrm{S}}\left(\bv_f^{\mathrm{S}}, p^{\mathrm{S}}\right) \nonumber \\
& \quad +b_s^{\mathrm{P}}\left(\bv_s^{\mathrm{P}} , p^{\mathrm{P}}\right) + b_f^{\mathrm{P}}\left( \phi \bv_r^{\mathrm{P}} , p^{\mathrm{P}}\right)
+b_{\Gamma}\left(\bv_f^{\mathrm{S}}, \bv_r^{\mathrm{P}} , \bv_s^{\mathrm{P}} ; \lambda\right) - m_{\theta}(\bu_r^{\mathrm{P}}, \bv_s^{\mathrm{P}}) - m_{\theta}(\bu_s^{\mathrm{P}}, \bv_s^{\mathrm{P}}) - m_{\theta}(\bu_s^{\mathrm{P}}, \bv_r^{\mathrm{P}}) - m_{\theta}(\bu_r^{\mathrm{P}}, \bv_r^{\mathrm{P}})  \nonumber \\
& \quad + m_{\phi^2/\kappa}(\bu_r^{\mathrm{P}}, \bv_r^{\mathrm{P}}) + b_\text{sig}^{\mathrm{P}}\left(\bv_s^{\mathrm{P}}, \bsigma\right)  -b_s^{\mathrm{P}}\left(\bu_s^{\mathrm{P}}, q^{\mathrm{P}} \right)-b_f^{\mathrm{P}}\left(\phi \bu_r^{\mathrm{P}}, q^{\mathrm{P}}\right)  - b^{\mathrm{S}}\left(\bu_f^{\mathrm{S}}, q^{\mathrm{S}}\right) - b_{\text{sig}}^{\mathrm{P}}\left(\bu_s^{\mathrm{P}}, \btau\right)  - \left(r_{\mathrm{S}}, q^{\mathrm{S}}\right) - \left(\ff_{\mathrm{S}},\bv_f^{\mathrm{S}}\right).\label{bartypeeq} 
\end{align}
Recalling the notation $ \bv = (\bu_r^{\mathrm{P}},\bu_s^{\mathrm{P}},p^{\mathrm{P}},\bsigma), \tilde{\bv} = (\tilde{\bu}_r,\tilde{\bu}_s,\tilde{p}_{\mathrm{P}},\tilde{\bsigma}_e),  \ff= (\bar{f}_r - \bu_r^{\mathrm{P}}, \bar{f}_s- \bu_s^{\mathrm{P}}, \bar{g}_p-p^{\mathrm{P}}, \bar{g}_e-\bsigma), \tilde{\ff}= (\tilde{\bar{f}}_r - \tilde{\bu}_r, \tilde{\bar{f}}_s- \tilde{\bu}_s, \tilde{\bar{g}}_p-\tilde{p}_{\mathrm{P}}, \tilde{\bar{g}}_e-\tilde{\bsigma}_e)$, we stress that the main objective is to prove, using \eqref{bartypeeq}, that 
$$
(\ff-\tilde{\ff}, \bv-\tilde{\bv} )\geq 0.
$$
As the model problem is linear, demonstrating the coercivity of the bilinear forms provided in \eqref{bartypeeq} is sufficient. We have already proven the coercivity of these bilinear forms in Lemmas \ref{coercivity-continuity} and \ref{LBB}, and the continuity of $(r_{\mathrm{S}}, q^{\mathrm{S}})$ and $(\ff_{\mathrm{S}}, \bv_f^{\mathrm{S}})$ is straightforward. This implies that the operator $\mathcal{M}$ is monotone.
\end{proof}
\begin{lemma}\label{mainlemma}
  Assume that \ref{(H1)}-\ref{(H3)} hold. Then, 
  for each $\bh_r \in W^{1,1}(0, T ; \mathbf{W}_{r, 2}^{\prime})$, $\bh_s \in W^{1,1}(0, T ; \mathbf{W}_{s, 2}^{ \prime})$, $h_p \in W^{1,1}(0, T ; W_{p, 2}^{ \prime})$ and $h_e \in W^{1,1}(0, T ; \bZ_2 ^{\prime})$, and $ \bu^{\mathrm{P}}_r(0) \in \mathbf{V}_r, \bu^{\mathrm{P}}_s(0) \in \mathbf{X}_d, p^{\mathrm{P}}(0) \in W_p, \bsigma(0) \in \bZ $, there exists a solution to \eqref{parabolic} with $ \bu_r^{\mathrm{P}} \in W^{1, \infty}(0, T ; \mathbf{V}_r) , \bu_s^{\mathrm{P}}\in W^{1, \infty}\left(0, T ; \mathbf{V}_s\right), p^{\mathrm{P}} \in W^{1, \infty}\left(0, T ; W_p\right)$ and $\bsigma \in W^{1, \infty}\left(0, T ; \bZ \right)$.  
\end{lemma}
\begin{proof}
Applying Theorem \ref{N,M} with $\mathcal{N}=I, \mathcal{M}^{\star}=\mathcal{M}, E= \bW_{r,2} \times \mathbf{W}_{s, 2} \times W_{p,2} \times \bZ_{ 2}$, $  
 E_b^{\prime}= \bW_{r,2}^{ \prime} \times \mathbf{W}_{s, 2}^{ \prime} \times W_{p,2}^{\prime} \times \bZ_{ 2}^{\prime}$, and using Lemma \ref{monotonicity} and Corollary \ref{resolvent}, we obtain existence of a solution to \eqref{parabolic}.
\end{proof}

\medskip
\noindent\textbf{\underline{Step 3:} The mixed problem  \eqref{weak-mixed}  is a special case of  \eqref{parabolic}.}
Finally, we establish the existence of a solution to \eqref{weak-mixed} as a corollary of Lemma \ref{mainlemma}.
\begin{lemma}\label{parabolicequi}
    If $(\bu^{\mathrm{P}}_r(t), \bu^{\mathrm{P}}_s(t), p^{\mathrm{P}}(t), \bsigma_f^{\mathrm{S}}(t)) \in \bbD$ solves \eqref{parabolic} for $\bh_r = \mathbf{0}, \bh_s=\ff_{\mathrm{P}}, h_p= \rho_f^{-1} (1-\phi)^{-2} K \theta$ and $h_e=0$, then it also solves \eqref{weak-mixed}. 
\end{lemma}
\begin{proof}
Let $\left(\bu_r(t), \bu_s(t), p^{\mathrm{P}}(t), \bsigma(t)\right)$ $ \in \bbD$ solve \eqref{parabolic}. For $\bh_r=\mathbf{0}, \bh_s=\ff_{\mathrm{P}}, h_p=\rho_f^{-1} (1-\phi)^{-2} K \theta$ and $h_{e}=0$, there exist $\left(\bar{f}_r, \bar{f}_s, \bar{g}_p, \bar{g}_e\right) \in \bbS_2^{\prime}$ such that $\left(\bar{f}_r - \bu_r^{\mathrm{P}}, \bar{f}_s- \bu_s^{\mathrm{P}}, \bar{g}_p-p^{\mathrm{P}}, \bar{g}_e-\bsigma\right) \in \mathcal{M}\left(\bu_r^{\mathrm{P}}, \bu_s^{\mathrm{P}}, p^{\mathrm{P}}, \bsigma\right)$ 
which satisfies 
$$
\ddt\left(\begin{array}{l}
\bu_r^{\mathrm{P}}\\
\bu_s^{\mathrm{P}}\\
p^{\mathrm{P}} \\
\bsigma
\end{array}\right)+\left(\begin{array}{c}
\bar{f}_r-\bu_r^{\mathrm{P}} \\
\bar{f}_s -\bu_s^{\mathrm{P}}\\
\bar{g}_p-p^{\mathrm{P}} \\
\bar{g}_e-\bsigma
\end{array}\right)=\left(\begin{array}{c}
\mathbf{0} \\
\ff_{\mathrm{P}} \\
\rho_f^{-1} (1-\phi)^{-2} \kappa \theta \\
0
\end{array}\right) .
$$
Then,
$$
\left(\ddt \left(\begin{array}{c}
\bu_r^{\mathrm{P}}\\
\bu_s^{\mathrm{P}}\\
p^{\mathrm{P}} \\
\bsigma
\end{array}\right),\left(\begin{array}{c}
\bv_r^{\mathrm{P}} \\
\bv_s^{\mathrm{P}} \\
q^{\mathrm{P}} \\
\btau
\end{array}\right)\right)_{\bbS_2}+\left(\left(\begin{array}{c}
\bar{f}_r-\bu_r^{\mathrm{P}} \\
\bar{f}_s-\bu_s^{\mathrm{P}}\\
\bar{g}_p-p^{\mathrm{P}} \\
\bar{g}_e-\bsigma
\end{array}\right),\left(\begin{array}{c}
\bv_r^{\mathrm{P}} \\
\bv_s^{\mathrm{P}} \\
q^{\mathrm{P}} \\
\btau
\end{array}\right)\right)_{\bbS_2}=\left(\left(\begin{array}{c}
 \mathbf{0} \\
\ff_{\mathrm{P}}\\
\rho_f^{-1} (1-\phi)^{-2} K \theta \\
0
\end{array}\right),\left(\begin{array}{c}
\bv_r^{\mathrm{P}} \\
\bv_s^{\mathrm{P}} \\
q^{\mathrm{P}} \\
\btau
\end{array}\right)\right)_{\bbS_2}.
$$
From the definition of the inner product $(\cdot, \cdot)_{\bbS_2}$, along with \eqref{domain1} and \eqref{domain2}, we can deduce part of the first two equations in \eqref{weak-mixed}. Equation \eqref{domain3}, stemming directly from the definition of the domain $\bbD$, implies the remainder of \eqref{weak-mixed}.
\end{proof}

\medskip 
\begin{proof}[Proof of Theorem \ref{main}] 
The existence of a solution of \eqref{weak-mixed} follows from Lemma \ref{mainlemma} and Lemma \ref{parabolicequi}. From Lemma \ref{mainlemma} we have that $ \bu_r^{\mathrm{P}} \in W^{1, \infty}(0, T ; \mathbf{V}_r) , \bu_s^{\mathrm{P}}\in W^{1, \infty}\left(0, T ; \mathbf{V}_s\right), p^{\mathrm{P}}  \in W^{1, \infty}\left(0, T ; W_p\right)$ and $\bsigma \in W^{1, \infty}\left(0, T ; \bZ \right)$. Next, by taking 
$(\bv_f^{\mathrm{S}}, q^{\mathrm{S}}, \bv_r^{\mathrm{P}}, q^{\mathrm{P}}, \bv_s^{\mathrm{P}}, \btau, \mu)=(\bu_f^{\mathrm{S}}, p^{\mathrm{S}}, \bu_r^{\mathrm{P}}, p^{\mathrm{P}}, \bu_s^{\mathrm{P}}, \bsigma, \lambda)$ in \eqref{weak-mixed}, we obtain that $\bv_f^{\mathrm{S}}\in L^{\infty}\left(0, T ; \mathbf{V}_f\right)$. Finally, the inf-sup condition \eqref{inf-sup2} and the first equation in \eqref{weak-mixed} imply that $p^{\mathrm{S}} \in L^{\infty}(0, T ; W_f)$ and $\lambda \in L^{\infty}(0, T ; \Lambda)$.
\end{proof}

\subsection{Existence and uniqueness of solution of the original formulation}\label{equivalent_original}
In this section we discuss how the well-posedness of  formulation \eqref{mixed-primal} follows from the existence of a solution of the mixed formulation \eqref{weak-mixed}. First we recall that $\bu_{s}^{\mathrm{P}}$ is the structure velocity, so the displacement solution can be recovered from
\begin{align}\label{y_s}
\by_s^{\mathrm{P}}(t)=\by_{s, 0}+\int_0^t\bu_s(s) \ds, \quad \forall t \in(0, T] .
\end{align}
Since $\bu_s^{\mathrm{P}}\in L^{\infty}(0, T ; \mathbf{X}_d)$, then $\by_s^{\mathrm{P}} \in W^{1, \infty}(0, T ; \mathbf{X}_d)$ for any $\by_{s, 0} \in \mathbf{X}_d$. By construction, $\bu_{s}^{\mathrm{P}}=\partial_t \by_s^{\mathrm{P}}$ and $\by_s^{\mathrm{P}}(0)=\by_{s, 0}$.
\begin{theorem}\label{original_c}
Assume \ref{(H1)}--\ref{(H3)}. Then, for data $\ff_{\mathrm{S}}\in W^{1,1}(0, T ; \mathbf{V}_f^{\prime})$, $\ff_{\mathrm{P}} \in W^{1,1}(0, T ; \mathbf{L}^{2}(\Omega_{\mathrm{P}}))$,  $r_{\mathrm{S}},\theta \in  W^{1,1}(0, \\ T ;W_f^{ \prime})$, and initial conditions $ \bu^{\mathrm{P}}_r(0)= \bu_{r,0} \in \mathbf{V}_r$, $\bu^{\mathrm{P}}_s(0)= \bu_{s,0} \in \mathbf{W}_s$, $p^{\mathrm{P}}(0)=p^{\mathrm{P},0} \in W_p$, and $\by_s^{\mathrm{P}}(0)= \by_{s, 0} \in \mathbf{V}_s$, there exists a unique solution $(\bu_f^{\mathrm{S}}, p^{\mathrm{S}},  \bu_r^{\mathrm{P}}, p^{\mathrm{P}}, \by_s^{\mathrm{P}}$  $, \bu_s^{\mathrm{P}}, \lambda)  \in L^{\infty}\left(0, T ; \mathbf{V}_f\right) \times L^{\infty}(0, T ; W_f) \times W^{1, \infty}(0, T ; \mathbf{V}_r) \times  W^{1, \infty}\left(0, T ; W_p\right) \times W^{1, \infty}(0, T ; \mathbf{V}_s) \times W^{1, \infty}(0, T ; \mathbf{W}_s) \times L^{\infty}(0, T ; \Lambda)$ of \eqref{mixed-primal}.    
\end{theorem}
\begin{proof}
We use the solvability of the mixed formulation \eqref{weak-mixed} to establish the solvability of the original formulation \eqref{mixed-primal}. Let $(\bu_f^{\mathrm{S}}, p^{\mathrm{S}}, \bu_r^{\mathrm{P}}, p^{\mathrm{P}}, \bu_s^{\mathrm{P}}, \bsigma, \lambda)$ be a solution to \eqref{weak-mixed}. Define $\by_s^{\mathrm{P}}$ as in \eqref{y_s}, so that $\bu_{s}^{\mathrm{P}}=\partial_t \by_s^{\mathrm{P}}$. Then the second equation in \eqref{weak-mixed}, with $\btau=\mathbf{0}$, implies the second equation in \eqref{mixed-primal}, and the third equation in \eqref{weak-mixed} implies that in \eqref{mixed-primal}. Additionally, we note that the first equations in \eqref{weak-mixed} and \eqref{mixed-primal} differ only in their respective terms $a_s^{\mathrm{P}}\left(\by_s^{\mathrm{P}}, \bw_s^{\mathrm{P}}\right)$ and $b_{\text{sig}}\left(\bv_s^{\mathrm{P}}, \bsigma\right)$.
Testing the second equation in \eqref{weak-mixed} with $\btau \in \bZ $ yields $\left(\partial_t\left(A \bsigma-\beps\left(\by_s^{\mathrm{P}}\right)\right), \btau\right)_{\Omega_{\mathrm{P}}}=0$. Using the fact that $\beps\left(\mathbf{V}_s\right) \subseteq \bZ $, this implies $\partial_t\left(A \bsigma-\beps\left(\by_s^{\mathrm{P}}\right)\right) =\mathbf{0}$. Integrating from 0 to $t \in(0, T]$, and using $\bsigma(0)=A^{-1} \beps\left(\by_s^{\mathrm{P}}(0)\right)$, we conclude that $\bsigma(t)=A^{-1} \beps\left(\by_s^{\mathrm{P}}(t)\right)$.
Hence, we have
$$
b_{\text{sig}}\left(\bv_s^{\mathrm{P}}, \bsigma\right)=\left(\bsigma, \beps(\bv_s^{\mathrm{P}})\right)_{\Omega_{\mathrm{P}}}=\left(A^{-1} \beps\left(\by_s^{\mathrm{P}}\right), \beps(\bv_s^{\mathrm{P}})\right)_{\Omega_{\mathrm{P}}}=a_s^{\mathrm{P}}\left(\by_s^{\mathrm{P}}, \bv_s^{\mathrm{P}}\right) .
$$
Therefore, the first equation in \eqref{mixed-primal} implies that of \eqref{weak-mixed}, establishing that $\bigl(\bu_f^{\mathrm{S}}, p^{\mathrm{S}}, \bu_r^{\mathrm{P}}, p^{\mathrm{P}}, \by_{s,0}+\int_0^t\bu_s(s) \ds, \lambda\bigr)$ is a solution of \eqref{mixed-primal}. From the weak formulation \eqref{mixed-primal}, it can ben seen that the structural velocity $\bu_s^{\mathrm{P}}$ is in  $\mathbf{W}_s$.
\end{proof}

Next, we provide a stability bound for the solution of \eqref{mixed-primal}.
\begin{lemma}\label{stability}
Assuming sufficient regularity of the data as well as  \ref{(H1)}--\ref{(H3)}, there exists a positive constant 
$\hat{C}$ (possibly depending on $K,\kappa,\rho_f,\rho_s,\lambda_p, \mu_f,\mu_p,\phi, \alpha_{\mathrm{BJS}}, C_K$) such that
\begin{align*}
& \|\bu_s^{\mathrm{P}}\|_{L^{\infty}\left(0, T ; \mathbf{L}^2(\Omega_{\mathrm{P}})\right)}+\|p^{\mathrm{P}}\|_{L^{\infty}(0, T ; L^2(\Omega_{\mathrm{P}}))} + 
\left| \bu_f^{\mathrm{S}}- \partial_t \by_s^{\mathrm{P}}\right|_{L^2(0,T;\mathrm{BJS})} + \|\bu_r^{\mathrm{P}}\|_{L^{\infty}\left(0, T ; \mathbf{L}^2(\Omega_{\mathrm{P}})\right)}+\|\by_s^{\mathrm{P}}\|_{L^{\infty}\left(0, T ; \mathbf{H}^1(\Omega_{\mathrm{P}})\right)} \\
&  \quad +\|\bu_f^{\mathrm{S}}\|_{L^2\left(0, T ; \mathbf{H}^1(\Omega_{\mathrm{S}})\right)} + \|\bu_r^{\mathrm{P}}\|_{L^2\left(0, T ; {L}^2(\Omega_{\mathrm{P}})\right)}   + \|p^{\mathrm{S}}\|_{L^2(0,T; L^2(\Omega_{\mathrm{S}}))}   +\|p^{\mathrm{P}}\|_{L^2\left(0,T; L^2(\Omega_{\mathrm{P}})\right)}    + \|\lambda\|_{L^2\left(0,T; \Lambda\right)} \\
& \leq \hat{C} \left(  \|\ff_{\mathrm{P}}\|_{L^2(0,T;\mathbf{L}^{2}(\Omega_{\mathrm{P}}))} +\|\ff_{\mathrm{S}}\|_{L^2(0,T;\mathbf{H}^{-1}(\Omega))} + \|\theta\|_{L^2(0,T;L^2(\Omega_{\mathrm{S}}))}^2 + \|r_{\mathrm{S}}\|_{L^2(0,T;L^2(\Omega_{\mathrm{S}}))} +\|\bu_s(0)\|_{0,\Omega_{\mathrm{P}}}  \right. \\& \left.
\qquad +  \|p^{\mathrm{P}}(0)\|_{0,\Omega_{\mathrm{P}}}
 +\|\by_s^{\mathrm{P}}(0)\|_{1,\Omega_{\mathrm{P}}}  + \|\bu_r(0)\|_{0,\Omega_{\mathrm{P}}} \right),
\end{align*}
where we introduced the notation $| \bv|_{L^2(0,T;\mathrm{BJS})}: = \int_0^t | \bv|_{\mathrm{BJS}} \ds$.
\end{lemma}
\begin{proof}
We test  the system against $(\bu_f^{\mathrm{S}}, p^{\mathrm{S}}, \bu_r^{\mathrm{P}}, p^{\mathrm{P}}, \partial_t \by_s^{\mathrm{P}}, \bu_s^{\mathrm{P}}, \lambda)$ and add the first three equations. 
Using the inequality \eqref{ineq:aux} and we can derive the following estimate 
\begin{align}
\label{rhs}
&  \frac{1}{2} \partial_t \left( \left(\rho_s (1-\phi) \bu_s^{\mathrm{P}}, \bu_s^{\mathrm{P}}\right)_{\Omega_{\mathrm{P}}} + \left( \sqrt{\rho_f \phi} (\bu_r^{\mathrm{P}}+\bu_{s}^{\mathrm{P}}),  \sqrt{\rho_f \phi} (\bu_r^{\mathrm{P}}+\bu_{s}^{\mathrm{P}})\right) +\left({(1-\phi)^2}{K}^{-1} p^{\mathrm{P}},p^{\mathrm{P}} \right)_{\Omega_{\mathrm{P}}} +  \left( 2 \mu_p \beps(\by
_s^{\mathrm{P}}), \beps(\by_s^{\mathrm{P}})\right)_{\Omega_{\mathrm{P}}} \right.\nonumber \\& 
\quad \left.  - \left( 2 \mu_f \phi \beps(\bu_r^{\mathrm{P}}), \beps(\bu_r^{\mathrm{P}})\right)_{\Omega_{\mathrm{P}}} + (\lambda_p \nabla \cdot \by_s^{\mathrm{P}}, \nabla \cdot \by_s^{\mathrm{P}})_{\Omega_{\mathrm{P}}} \right)  + \left| \bu_f^{\mathrm{S}}- \partial_t \by_s^{\mathrm{P}}\right|_{\mathrm{BJS}} + \left({\phi^2}{\kappa}^{-1} \bu_r^{\mathrm{P}}, \bu_r^{\mathrm{P}} \right)_{\Omega_{\mathrm{P}}} \nonumber \\& \quad  +  (2 \mu_f  \beps(\bu_f^{\mathrm{S}}), \beps(\bu_f^{\mathrm{S}})_{\Omega_{\mathrm{S}}}  + \left(2 \mu_f \phi \beps(\bu_{s}^{\mathrm{P}}), \beps(\bu_{s}^{\mathrm{P}})\right)_{\Omega_{\mathrm{P}}}   + \left( 2 \mu_f \phi \beps(\bu_r^{\mathrm{P}}), \beps(\bu_r^{\mathrm{P}}) \right)_{\Omega_{\mathrm{P}}} - \left( 2 \mu_f \phi \beps(\partial_t \by_s^{\mathrm{P}}), \beps(\partial_t \by_s^{\mathrm{P}})\right)_{\Omega_{\mathrm{P}}}\nonumber \\&  \leq \langle \ff_{\mathrm{S}}, \bu_f^{\mathrm{S}}\rangle_{\Omega_{\mathrm{S}}}  + \left( \rho_p \ff_{\mathrm{P}}, \bu_s^{\mathrm{P}}\right)_{\Omega_{\mathrm{P}}} + (\rho_f \phi  \ff_{\mathrm{P}}, \bu_r^{\mathrm{P}} )_{\Omega_{\mathrm{P}}}  + \left( r_{\mathrm{S}}, p^{\mathrm{S}}\right)_{\Omega_{\mathrm{S}}} + (\rho_f^{-1} \theta, p^{\mathrm{P}} )_{\Omega_{\mathrm{P}}}. 
\end{align}
From now on we denote by $c_1=c_1\left(\rho_f, \rho_s, \phi, K, \alpha_s, \alpha_f, \alpha_{\mathrm{BJS}}\right)$ a data-dependent constant used for lower bounds, and with $c_2=c_2\left(\rho_f, \rho_s, \phi, K, \lambda_p, \mu_f, \mu_p, \alpha_{\mathrm{BJS}}\right)$ one for upper bounds. From  \eqref{rhs} and using Young's inequality with $\epsilon>0$, we proceed to expand the right-hand side of that relation as 
\begin{align*} 
\texttt{RHS}_{\eqref{rhs}} 
& \leq \frac{c_2}{\epsilon_1}\left(\|\ff_{\mathrm{P}}(t)\|_{0,\Omega_{\mathrm{P}}}^2+ \|\ff_{\mathrm{S}}(t)\|_{-1,\Omega_{\mathrm{S}}}^2+\|\theta(t)\|_{L^2(\Omega)}^2+ \|r_{\mathrm{S}}(t)\|_{0,\Omega_{\mathrm{S}}}^2\right)    \\& \quad  +c_2\epsilon_1 \left(\|\bu_f^{\mathrm{S}}(t)\|_{1,\Omega_{\mathrm{S}}}^2 +\|p^{\mathrm{P}}(t)\|_{0,\Omega_{\mathrm{P}}}^2+\|p^{\mathrm{S}}(t)\|_{0,\Omega_{\mathrm{S}}}^2+\|\bu_s(t)\|_{0,\Omega_{\mathrm{P}}}^2+\|\bu_r(t)\|_{0,\Omega_{\mathrm{P}}}^2\right) .
\end{align*}
Let us denote now the  left-hand side of \eqref{rhs} as $\texttt{LHS}_{\eqref{rhs}}$.  Using  \ref{(H1)}--\ref{(H3)}, K\"orn's inequality,  
and the bound $\|\bu_r^{\mathrm{P}}\|_{\rho_f \phi }^2 \leq \|\bu_r^{\mathrm{P}} + \bu_s^{\mathrm{P}}\|_{\rho_f \phi }^2$, we have 
\[\texttt{LHS}_{\eqref{rhs}} \geq  \frac{1}{2} \partial_t \bigl( \|\bu_s(t)\|_{\rho_s (1-\phi)}^2 + \|\bu_r^{\mathrm{P}}\|_{\rho_f \phi }^2+ \|p^{\mathrm{P}}\|_{\frac{(1-\phi)^2}{K}}^2 + \|\by_s^{\mathrm{P}} \|_{1,\Omega_{\mathrm{P}}} ^2 \bigr)+ \left| \bu_f^{\mathrm{S}}- \partial_t \by_s^{\mathrm{P}}\right|_{\mathrm{BJS}}  + (\frac{\phi^2}{\kappa} \bu_r^{\mathrm{P}},\bu_r^{\mathrm{P}}) + \|  \bu_f^{\mathrm{S}}\|_{1,\Omega_{\mathrm{S}}}^2.
\]
Integrating $\texttt{LHS}_{\eqref{rhs}}$ in time over $(0,t]$ for arbitrary $t \in (0, T]$, and invoking Lemma \ref{general_C}, yields 
\begin{align*}
\int_0^t\texttt{LHS}_{\eqref{rhs}} \geq & \left.\left(\|\bu_s(s)\|_{\rho_s(1- \phi)}^2+ \|\bu_r(s)\|_{\rho_f \phi }^2 +\|p^{\mathrm{P}}(s)\|^2_{\frac{(1-\phi)^2} {K}}+\|\by_s^{\mathrm{P}}(s)\|_{1,\Omega_{\mathrm{P}}}^2 \right)\right|_{s=0} ^{s=t} \\& \quad +\int_0^t\left| \bu_f^{\mathrm{S}}- \partial_t \by_s^{\mathrm{P}}\right|_{\mathrm{BJS}}^2 \ds
+\alpha_f \int_0^t \|\bu_r(s)\|_{0,\Omega_{\mathrm{P}}}^2\ds  + \int_0^t\|  \bu_f^{\mathrm{S}}\|_{1,\Omega_{\mathrm{S}}}^2 \ds \\
\geq&c_1\left(\|\bu_s(t)\|_{0,\Omega_{\mathrm{P}}}^2+ \|\bu_r(t)\|_{0,\Omega_{\mathrm{P}}}^2+\|p^{\mathrm{P}}(t)\|_{L^2(\Omega)}^2+\|\by_s^{\mathrm{P}}(t)\|_{1,\Omega_{\mathrm{P}}}^2+\int_0^t\left| \bu_f^{\mathrm{S}}- \partial_t \by_s^{\mathrm{P}}\right|_{\mathrm{BJS}}^2 +\int_0^T\|\bu_r(s)\|_{0,\Omega_{\mathrm{P}}}^2\ds \right. \\& \quad \left.  + \int_0^t\|  \bu_f^{\mathrm{S}}(s) \|_{1,\Omega_{\mathrm{S}}}^2 \ds\right)
-c_2\left( \|\bu_s(0)\|_{0,\Omega_{\mathrm{S}}}^2+\|p^{\mathrm{P}}(0)\|_{0,\Omega_{\mathrm{P}}}^2  +\|\by_s^{\mathrm{P}}(0)\|_{1,\Omega_{\mathrm{P}}}^2  +\|\bu_r(0)\|_{0,\Omega_{\mathrm{P}}}^2 \right). 
\end{align*}
Then we combine the lower and upper bounds together to obtain the following estimate 
\begin{align*}
&\|\bu_s(t)\|_{0,\Omega_{\mathrm{P}}}^2+\|\bu_r(t)\|_{0,\Omega_{\mathrm{P}}}^2+\|p^{\mathrm{P}}(t)\|_{L^2(\Omega)}^2+\|\by_s^{\mathrm{P}}(t)\|_{1,\Omega_{\mathrm{P}}}^2+\int_0^t\bigl(\left| \bu_f^{\mathrm{S}}- \partial_t \by_s^{\mathrm{P}}\right|_{\mathrm{BJS}}^2  +  \|\bu_r(s)\|_{0,\Omega_{\mathrm{P}}}^2 +   \|  \bu_f^{\mathrm{S}}(s) \|_{1,\Omega_{\mathrm{S}}}^2\bigr)\ds \\
&\quad \leq \frac{c_2}{c_1\epsilon_1}  \int_0^t\left( \|\ff_{\mathrm{P}}(t)\|_{0,\Omega_{\mathrm{P}}}^2+\|\ff_{\mathrm{S}}(t)\|_{-1,\Omega_{\mathrm{S}}}^2  +\|\theta(t)\|_{0,\Omega_{\mathrm{P}}}^2  + \|r_{\mathrm{S}}(t)\|_{L^2(\Omega_{\mathrm{S}})}^2\right)\\
&\qquad +\frac{c_2\epsilon_1}{c_1} \int_0^t\left(\|\bu_f^{\mathrm{S}}(t)\|_{1,\Omega_{\mathrm{S}}}^2+\|p^{\mathrm{P}}(t)\|_{0,\Omega_{\mathrm{P}}}^2+\|p^{\mathrm{S}}(t)\|_{0,\Omega_{\mathrm{P}}}^2 +\|\bu_s(t)\|_{0,\Omega_{\mathrm{P}}}^2 +\|\bu_r(t)\|_{0,\Omega_{\mathrm{P}}}^2\right)  \\& \qquad   +\frac{c_2}{c_1}\left( \|\bu_s(0)\|_{0,\Omega_{\mathrm{P}}}^2 
 +\|p^{\mathrm{P}}(0)\|_{0,\Omega_{\mathrm{P}}}^2 +\|\by_s^{\mathrm{P}}(0)\|_{1,\Omega_{\mathrm{P}}}^2  +\|\bu_r(0)\|_{0,\Omega_{\mathrm{P}}}^2 \right) .
\end{align*}
Taking the supremum of $t$ in $(0, T]$ and using the upper bound $\int_0^t\varphi(s)\ds \leq T\|\varphi\|_{\infty}$ we can readily get 
\begin{align}
& (c_1-c_2 T \epsilon_1)(\|\bu_s^{\mathrm{P}}\|_{L^{\infty}\left(0, T ; \mathbf{L}^2(\Omega_{\mathrm{P}})\right)}^2+\|p^{\mathrm{P}}\|_{L^{\infty}(0, T ; L^2(\Omega_{\mathrm{P}}))}^2 +\|\by_s^{\mathrm{P}}\|_{L^{\infty}\left(0, T ;\mathbf{H}^1(\Omega_{\mathrm{P}})\right)}^2 + \|\bu_r^{\mathrm{P}}\|_{L^{\infty}\left(0, T ;\mathbf{L}^2(\Omega_{\mathrm{P}})\right)}^2) \nonumber \\& \quad + c_1\int_0^t\left| \bu_f^{\mathrm{S}}- \partial_t \by_s^{\mathrm{P}}\right|_{\mathrm{BJS}}^2 +(c_1-c_2\epsilon_1)\|\bu_f^{\mathrm{S}}\|_{L^2\left(0, T ; \mathbf{H}^1(\Omega_{\mathrm{S}})\right)}^2 +c_1\|\bu_r^{\mathrm{P}}\|_{L^2\left(0, T ; \mathbf{L}^2(\Omega_{\mathrm{P}})\right)}^2   \nonumber  \\&   
\leq \frac{c_2}{\epsilon_1} \left(\|\ff_{\mathrm{P}}\|_{L^2(0,T;\mathbf{L}^2(\Omega_{\mathrm{P}}))}^2+ \|\ff_{\mathrm{S}}\|_{L^2(0,T;\mathbf{H}^{-1}(\Omega_{\mathrm{S}}))}^2\right) 
+\frac{c_2}{\epsilon_1} \left(\|\theta\|_{L^2(0,T;L^2(\Omega_{\mathrm{P}}))}^2 + \|r_{\mathrm{S}}\|_{L^2(0,T;L^2(\Omega_{\mathrm{S}}))}^2\right) \nonumber \\& \quad +c_2 \epsilon_1 (\|\bu_r^{\mathrm{P}}\|_{L^2(0,T;\mathbf{L}^2(\Omega_{\mathrm{P}}))} + \|p^{\mathrm{S}}\|_{L^2(0,T;L^2(\Omega_{\mathrm{S}}))})+c_2( \|\bu_s(0)\|_{0,\Omega_{\mathrm{P}}}^2  +\|p^{\mathrm{P}}(0)\|_{0,\Omega_{\mathrm{P}}}^2 
  +\|\by_s^{\mathrm{P}}(0)\|_{1,\Omega_{\mathrm{P}}}^2 + \|\bu_r(0)\|_{0,\Omega_{\mathrm{P}}}^2 ). \label{stability1}
\end{align}
Finally, we use the inf-sup condition \eqref{inf-sup2} for $\tilde{p}_{\mathrm{S}}, \tilde{p}_{\mathrm{P}}, \tilde{\lambda}$ together with \eqref{mixed-primal} to obtain
\begin{align*}
&\|({p}^{\mathrm{S}}, {p}^{\mathrm{P}}, {\lambda})\|_{W_f \times W_p \times \Lambda} \lesssim  \sup _{(\bv_f^{\mathrm{S}}, \bv_r^{\mathrm{P}}) \in \mathbf{V}_f \times \mathbf{V}_r} \frac{b^{\mathrm{S}}(\bv_f^{\mathrm{S}}, {p}^{\mathrm{S}})+b_f^{\mathrm{P}}(\bv_r^{\mathrm{P}}, {p}_{\mathrm{P}})+b_{\Gamma}(\bv_f^{\mathrm{S}}, \bv_r^{\mathrm{P}}, \mathbf{0} ; {\lambda})}{\|(\bv_f^{\mathrm{S}}, \bv_r^{\mathrm{P}})\|_{\mathbf{V}_f \times \mathbf{V}_r}} \\
&\quad = \sup _{(\bv_f^{\mathrm{S}}, \bv_r^{\mathrm{P}}) \in \mathbf{V}_f \times \mathbf{V}_r}\left(\frac{\langle \ff_{\mathrm{S}},\bv_f^{\mathrm{S}}\rangle_{\Omega_{\mathrm{S}}}-a_f^{\mathrm{S}}(\bu_f^{\mathrm{S}}, \bv_f^{\mathrm{S}})-a_f^{\mathrm{P}}\left({\bu}_r^{\mathrm{P}}, \bv_r^{\mathrm{P}}\right)-a_f^{\mathrm{P}}\left(\partial_t {\by}_s, \bv_r^{\mathrm{P}}\right)-a_{\mathrm{BJS}}(\bu_f^{\mathrm{S}}, \partial_t {\by}_s ; \bv_f^{\mathrm{S}}, \mathbf{0})+m_\theta(\partial_t {\by}_s,\bv_r^{\mathrm{P}})
}{\|(\bv_f^{\mathrm{S}}, \bv_r^{\mathrm{P}})\|_{\mathbf{V}_f \times \mathbf{V}_r}}\right. \\
&\qquad\qquad \qquad \qquad \quad \left. + \frac{m_{\theta}\left(\bu_r^{\mathrm{P}},\bv_r^{\mathrm{P}}\right)-m_{\phi^2/\kappa}(\bu_r^{\mathrm{P}},\bv_r^{\mathrm{P}})-m_{\rho_f \phi}\left(\partial_t\bu_r^{\mathrm{P}},\bv_r^{\mathrm{P}}  \right) -m_{\rho_f \phi}\left(\partial_t\bu_s^{\mathrm{P}},\bv_r^{\mathrm{P}}  \right) + \left( \rho_f \phi \ff_{\mathrm{P}}, \bv_r^{\mathrm{P}}\right)_{\Omega_{\mathrm{P}}}} {\|(\bv_f^{\mathrm{S}}, \bv_r^{\mathrm{P}})\|_ {\mathbf{V}_f \times \mathbf{V}_r}} \right). 
\end{align*}
Using the continuity bounds available from Lemmas \ref{coercivity-continuity} and \ref{continuity-b}, we have 
\[
\|p^{\mathrm{S}}\|_{W_f} + \|p^{\mathrm{P}}\|_{W_p} + \|\lambda\|_{\Lambda}
    \lesssim  \|\bu_f^{\mathrm{S}}\|_{1,\Omega_{\mathrm{S}}} + \| \bu_r^{\mathrm{P}}\|_{0,\Omega_{\mathrm{P}}}   + \left|\bu_f^{\mathrm{S}}- \partial_t \by_s^{\mathrm{P}}\right|_{\mathrm{BJS}} + \|\ff_{\mathrm{S}}\|_{-1,\Omega_{\mathrm{S}}} +\|\ff_{\mathrm{P}}\|_{0,\Omega_{\mathrm{P}}},
\]
which in turn implies that  
\begin{equation}
\epsilon_2 \int_0^t\left(\|p^{\mathrm{S}}\|_{ 0,\Omega_{\mathrm{S}}}^2 + \|p^{\mathrm{P}}\|_{ 0,\Omega_{\mathrm{P}}}^2 + \|\lambda\|_{\Lambda}^2 \right)
    \leq  \tilde{C} \epsilon_2 \int_0^t\left(\|\bu_f^{\mathrm{S}}\|_{1,\Omega_{\mathrm{S}}}^2 + \| \bu_r^{\mathrm{P}}\|_{0,\Omega_{\mathrm{P}}}^2      + \left|\bu_f^{\mathrm{S}}- \partial_t \by_s^{\mathrm{P}}\right|_{\mathrm{BJS}}^2 + \|\ff_{\mathrm{S}}\|_{-1,\Omega_{\mathrm{S}}}^2 +\|\ff_{\mathrm{P}}\|_{0,\Omega_{\mathrm{P}}}^2\right).
    \label{stability2}
\end{equation}
Adding \eqref{stability1} and \eqref{stability2} and choosing $\epsilon_2$ and $\epsilon_1$ small enough, readily yields the following bound 
\begin{align*}
& \|\bu_s^{\mathrm{P}}\|_{L^{\infty}\left(0, T ; \mathbf{L}^2(\Omega_{\mathrm{P}})\right)}^2+\|p^{\mathrm{P}}\|_{L^{\infty}(0, T ; L^2(\Omega_{\mathrm{P}}))}^2 + \int_0^t\left| \bu_f^{\mathrm{S}}- \partial_t \by_s^{\mathrm{P}}\right|_{\mathrm{BJS}}^2 + \|\bu_r^{\mathrm{P}}\|_{L^{\infty}\left(0, T ; \mathbf{L}^2(\Omega_{\mathrm{P}})\right)}^2 +\|\by_s^{\mathrm{P}}\|_{L^{\infty}\left(0, T ; \mathbf{H}^1(\Omega_{\mathrm{P}})\right)}^2 \\
& +\|\bu_f^{\mathrm{S}}\|_{L^2\left(0, T ; \mathbf{H}^1(\Omega_{\mathrm{S}})\right)}^2 + \|\bu_r^{\mathrm{P}}\|_{L^2\left(0, T ; {\bL}^2(\Omega_{\mathrm{P}})\right)}^2   + \|p^{\mathrm{S}}\|_{L^2(0,T; L^2(\Omega_{\mathrm{S}}))}^2  \ +\|p^{\mathrm{P}}\|_{L^2\left(0,T; L^2(\Omega_{\mathrm{P}})\right)}^2     + \|\lambda\|_{L^2\left(0,T; {H}^{-1/2}(\Sigma)\right)} ^2\\
&\quad \lesssim   \|\ff_{\mathrm{P}}\|_{L^2(0,T;\mathbf{L}^{2}(\Omega_{\mathrm{P}}))}^2   +\|\ff_{\mathrm{S}}\|_{L^2(0,T;\mathbf{H}^{-1}(\Omega))}^2 
+ \|\theta\|_{L^2(0,T;L^2(\Omega_{\mathrm{S}}))}^2  + \|r_{\mathrm{S}}\|_{L^2(0,T;L^2(\Omega_{\mathrm{S}}))}^2  \\
& \qquad +\|\bu_s(0)\|_{0,\Omega_{\mathrm{P}}}^2 +
 \|p^{\mathrm{P}}(0)\|_{0,\Omega_{\mathrm{P}}}^2  +\|\by_s^{\mathrm{P}}(0)\|_{1,\Omega_{\mathrm{P}}}^2  + \|\bu_r(0)\|_{0,\Omega_{\mathrm{P}}}^2,
\end{align*}
which completes the proof. 
\end{proof}

\begin{corollary}
   Assume \ref{(H1)}--\ref{(H3)}. Then, there exists a unique   $(\bu_f^{\mathrm{S}},p^{\mathrm{S}},  \bu_r^{\mathrm{P}}, p^{\mathrm{P}}, \by_s^{\mathrm{P}},\bu_s^{\mathrm{P}}, \lambda)$ in  the product space  $L^{\infty}\left(0, T ; \mathbf{V}_f\right) \times L^{\infty}\left(0, T ; W_f\right) \times W^{1, \infty}(0, T ; \mathbf{V}_r) \times  W^{1, \infty}\left(0, T ; W_p\right) \times W^{1, \infty}(0, T ; \mathbf{X}_d) \times W^{1, \infty}\left(0, T ; \mathbf{V}_s\right) \times L^{\infty}(0, T ; \Lambda)$, solution of problem \eqref{mixed-primal}.
\end{corollary}
\begin{proof}
     Let $( \bu_f^{\mathrm{S},i}, p^{\mathrm{S},i}, \bu_r^{\mathrm{P},i}, p^{\mathrm{P},i}, \by_{s}^i, \bv_{s}^{\mathrm{P},i}, \lambda^i), i=1,2$, be two solutions of \eqref{mixed-primal} with the same data. Subtracting both weak formulations yields  homogeneous initial data and zero right-hand side. Consider $\tilde{\bu}_f= \bu_f^{\mathrm{S},1}-\bu_f^{\mathrm{S},2}, \tilde{p}_{\mathrm{S}}= p^{\mathrm{S},1}-p^{\mathrm{S},2}, \tilde{\bu}_r =  \bu_r^{\mathrm{P},1}-\bu_r^{\mathrm{P},2}, \tilde{p}_{\mathrm{P}}=p^{\mathrm{P},1}-p^{\mathrm{P},2}, \tilde{\by}_s=\by_{s}^1-\by_{s}^2,
\tilde{\bu}_s=\bu_{s}^{\mathrm{P},1}-\bu_{s}^{\mathrm{P},2}, \tilde{\lambda}= \lambda^1 - \lambda^2$, and substitute $\left(\tilde{\bu}_s , \tilde{p}_{\mathrm{S}} ,\tilde{\bu}_r, \tilde{p}_{\mathrm{P}}, \tilde{\by}_s, \tilde{\bu}_s, \tilde{\lambda}\right)$ instead of $\left({\bu}_s ,{p}_{\mathrm{S}} ,\bu_r^{\mathrm{P}}, {p}_{\mathrm{P}}, {\by}_s,\bu_s^{\mathrm{P}}, {\lambda}\right)$ with homogeneous initial data and forcing terms in Lemma \ref{stability}. This readily gives 
\begin{gather*}
\|\tilde{\bu}_s\|_{L^{\infty}\left(0, T ; \mathbf{L}^2(\Omega_{\mathrm{P}})\right)}=\|\tilde{\bu}_f\|_{L^2\left(0, T ; \mathbf{H}^1(\Omega_{\mathrm{S}})\right)}=\|\tilde{p}_{\mathrm{P}}\|_{L^{\infty}(0, T ; L^2(\Omega_{\mathrm{P}}))}=\|\tilde{\by}_s\|_{L^{\infty}\left(0, T ; \mathbf{H}^1(\Omega_{\mathrm{P}})\right)}=0, \\ \|\tilde{\bu}_r\|_{L^2\left(0, T ; \mathbf{L}^2(\Omega_{\mathrm{P}})\right)} = \|\tilde{p}_{\mathrm{S}}\|_{L^2(0, T ; {L}^2(\Omega_{\mathrm{S}}))}= \|\tilde{\lambda} \|_{L^2(0, T ; {H}^{-1/2}(\Sigma))}=0.
\end{gather*}
Hence, the solution to problem \eqref{mixed-primal} must be unique.
\end{proof}

\section{Semi-discrete formulation}\label{section6}
Suppose that $\mathcal{T}_h^{\mathrm{S}}$ and $\mathcal{T}_h^{\mathrm{P}}$ are shape-regular  quasi-uniform partitions of  $\Omega_{\mathrm{S}}$ and $\Omega_{\mathrm{P}}$, respectively, both consisting of affine elements with maximal element diameter $h$. The two partitions may be non-matching at the interface  $\Sigma$. For the discretization of the fluid velocity and pressure we choose {conforming} FE spaces $\mathbf{V}_{f,h} \subset \mathbf{V}_f$ and
$W_{f,h} \subset W_f$, which are assumed to be inf–sup stable (such as MINI elements, Taylor--Hood, and conforming Crouzeix--Raviart elements). For the discretization of the generalized Biot unknowns, we define $X_h^k=\left\{q \in C({\Omega}):\left.q\right|_K \in  \mathbb{P}_k(K) \quad \forall K \in \mathcal{T}_h\right\}$ where
 $\mathbb{P}_k(K)$ denotes the space of polynomials of degree $k \geq 1$ defined over $K$. With them, we define the following conforming discrete spaces:
$$
 \mathbf{V}_{r, h}=\mathbf{V}_r \cap\left[X_h^{k+1}\right]^d, \quad \mathbf{V}_{s,h}=\mathbf{V}_s \cap\left[X_h^{k+1}\right]^d, \quad W_{p,h}= W_p \cap X_h^k, \quad \mathbf{W}_{s, h}=\mathbf{W}_s \cap\left[X_h^k\right]^d. 
$$
On $\mathbf{V}_{f,h}, \mathbf{V}_{r,h}$ and $\mathbf{V}_{s,h}$ we prescribe 
homogeneous boundary conditions on the external boundaries $\Gamma_{\mathrm{S}}$ and $\Gamma_{\mathrm{P}}$, and the discrete space for the Lagrange multiplier is 
$$
\Lambda_h=\mathbf{V}_{r, h} \cdot \bn_{\mathrm{P}}|_\Sigma,
$$
and it is equipped with the norm
$\| \cdot \|_{\Lambda_h} = \| \cdot \|_{-1/2,\Sigma}$.

The first semi-discrete problem  reads: find $(\bu_{f,h}^{\mathrm{S}}, p^{\mathrm{S}}_h,  \bu_{r,h}^{\mathrm{P}}, p^{\mathrm{P}}_h, \by_{s, h}^{\mathrm{P}} ,\bu_{s,h}^{\mathrm{P}}, \lambda_h) \in L^{\infty}\left(0, T ; \mathbf{V}_{f, h}\right) \times L^{\infty}\left(0, T ; W_{f,h}\right) \times W^{1,\infty}\left(0, T ; \mathbf{V}_{r, h}\right)$ $ \times W^{1, \infty}\left(0, T ; W_{p,h}\right) \times   W^{1, \infty}\left(0, T ; \mathbf{V}_{s, h}\right) \times W^{1, \infty}\left(0, T ; \mathbf{W}_{s, h}\right) \times L^{\infty}\left(0, T ; \Lambda_h\right)$, such that 
\begin{subequations}
 \begin{align}
&a_f^{\mathrm{S}}(\bu^{\mathrm{S}}_{f,h}, \bv_{f,h}^{\mathrm{S}})+a_{f}^{\mathrm{P}}\left(\bu_{r,h}^{\mathrm{P}}, \bw_{s,h}^{\mathrm{P}}\right)+a_s^{\mathrm{P}}\left(\by_{s,h}, \bw_{s,h}^{\mathrm{P}}\right)+ a_{f}^{\mathrm{P}}\left(\bu_{r,h}^{\mathrm{P}}, \bv_{r,h}^{\mathrm{P}}\right)+ a_{f}^{\mathrm{P}}\left(\partial_t \by_{s,h}, \bv_{r,h}^{\mathrm{P}}\right) + a_{f}^{\mathrm{P}}\left(\partial_t \by_{s,h}, \bw_{s,h}^{\mathrm{P}}\right)\nonumber \\
&\quad +a_{\mathrm{BJS}}\left(\bu^{\mathrm{S}}_{f,h}, \partial_t \by_{s,h} ; \bv_{f,h}^{\mathrm{S}}, \bw_{s,h}^{\mathrm{P}}\right) 
 +b^{\mathrm{S}}\left(\bv_{f,h}^{\mathrm{S}}, p^{\mathrm{S}}_h\right)+b_s^{\mathrm{P}}\left(\bw_{s,h}^{\mathrm{P}}, p^{\mathrm{P}}_h\right) + b_f^{\mathrm{P}}\left(\bv_{r,h}^{\mathrm{P}}, p^{\mathrm{P}}_h\right)  +b_{\Gamma}\left(\bv_{f,h}^{\mathrm{S}}, \bv_{r,h}^{\mathrm{P}}, \bw_{s,h}^{\mathrm{P}} ; \lambda_h \right) \nonumber \\
& \quad - m_{\theta}(\bu_{r,h}^{\mathrm{P}}, \bw_{s,h}^{\mathrm{P}}) - m_{\theta}(\partial_t \by_{s,h}, \bw_{s,h}^{\mathrm{P}})  - m_{\theta}(\bu_{r,h}^{\mathrm{P}}, \bv_{r,h}^{\mathrm{P}})   - m_{\theta}(\partial_t \by_{s,h}, \bv_{r,h}^{\mathrm{P}})  + m_{\phi^2/\kappa}(\bu_{r,h}^{\mathrm{P}}, \bv_{r,h}^{\mathrm{P}}) \nonumber \\ 
& \quad + m_{\rho_f \phi}(\partial_t\bu_{r,h}^{\mathrm{P}}, \bw_{s,h}^{\mathrm{P}})  + m_{\rho_p}(\partial_t\bu_{s,h}^{\mathrm{P}}, \bw_{s,h}^{\mathrm{P}}) + m_{\rho_f \phi }(\partial_t\bu_{r,h}^{\mathrm{P}}, \bv_{r,h}^{\mathrm{P}})  + m_{\rho_f \phi}(\partial_t\bu_{s,h}^{\mathrm{P}}, \bv_{r,h}^{\mathrm{P}}) \nonumber \\
& \qquad \qquad \qquad \qquad \qquad \qquad \qquad \qquad \qquad \qquad \qquad \qquad= \langle\ff_{\mathrm{S}},\bv_{f,h}^{\mathrm{S}}\rangle_{\Omega_{\mathrm{P}}} +(\rho_p \ff_{\mathrm{P}}, \bw_{s,h}^{\mathrm{P}})_{\Omega_{\mathrm{P}}} +( \rho_f \phi  \ff_{\mathrm{P}}, \bv_{r,h}^{\mathrm{P}})_{\Omega_{\mathrm{P}}}, \label{semi1}
\\ 
&\left(\frac{(1-\phi)^2}{K}\partial_t p^{\mathrm{P}}_h, q^{\mathrm{P}}_h\right)_{\Omega_{\mathrm{P}}}-b_s^{\mathrm{P}}(\partial_t \by_{s,h}, q^{\mathrm{P}}_h)-b_f^{\mathrm{P}}( \bu_{r,h}^{\mathrm{P}}, q^{\mathrm{P}}_h)-b^{\mathrm{S}}(\bu^{\mathrm{S}}_{f,h}, q^{\mathrm{S}}_h) 
=\left(r_{\mathrm{S}}, q^{\mathrm{S}}_h\right)_{\Omega_{\mathrm{S}}}+\left(\rho_f^{-1} \theta, q^{\mathrm{P}}_h\right)_{\Omega_{\mathrm{P}}}, \label{semi2}\\
& \qquad \qquad \qquad \qquad\qquad \qquad \quad b_{\Gamma}\left(\bu^{\mathrm{S}}_{f,h}, \bu_{r,h}^{\mathrm{P}}, \partial_t \by_{s,h} ; \mu_h \right)=0, \label{semi3} \\
& \qquad \qquad \qquad \quad \ - m_{\rho_p}(\partial_t\by_{s,h}, \bv_{s,h}^{\mathrm{P}}) + m_{\rho_p}(\bu_{s,h}^{\mathrm{P}}, \bv_{s,h}^{\mathrm{P}}) =0 \label{semi4},
\end{align}   
\end{subequations}
for all $\bv_{f,h}^{\mathrm{S}} \in \mathbf{V}_{f, h}, q^{\mathrm{S}}_h \in q^{\mathrm{S}}_h, \bv_{r,h}^{\mathrm{P}} \in \mathbf{V}_{r, h}, q^{\mathrm{P}}_h \in q^{\mathrm{P}}_h, \bw_{s,h}^{\mathrm{P}} \in \mathbf{X}_{d, h}, \bv_{s,h}^{\mathrm{P}} \in \mathbf{V}_{s, h} $ and $\mu_h \in \Lambda_h$. 
The initial conditions $\bv_{r,h}(0), \bv_{s,h}(0), p^{\mathrm{P}}_h(0)$ and $\by_{s, h}^{\mathrm{P}}(0)$ are  suitable approximations of $\bv^{\mathrm{P}}_{r,0}, \bv^{\mathrm{P}}_{s,0},$  $ p^{\mathrm{P}, 0}$ and $\by^{\mathrm{P}}_{s, 0}$, respectively.
To prove that \eqref{semi1}-\eqref{semi4} is well-posed we  follow the same strategy as in the  continuous case. 
Let $\mathbf{V}_{s, h}$ consist of polynomials of degree at most $k_s$. The stress FE space $\bZ_h \subset \bZ $ will have symmetric tensors with  discontinuous polynomial entries of degree up to $k_{s-1}$:
$$
\bZ_h:=\left\{\bsigma \in \bZ :\left.\bsigma\right|_{K} \in \mathbb{P}_{k_s-1}^{\mathrm{sym}}(K)^{d \times d},\ \forall K \in \mathcal{T}_h^{\mathrm{P}} \right\} .
$$
Then the second semi-discrete formulation is as before but we replace looking for $\by_{s,h}$ by seeking $\bsigma_h \in W^{1, \infty}\left(0, T ; \bZ_{h}\right)$ from the set of equations 
\begin{subequations}
\begin{align}
& a_f^{\mathrm{S}}(\bu^{\mathrm{S}}_{f,h}, \bv_{f,h}^{\mathrm{S}})+a_{f}^{\mathrm{P}}(\bu_{r,h}^{\mathrm{P}}, \bv_{s,h}^{\mathrm{P}})+ a_{f}^{\mathrm{P}}\left(\bu_{r,h}^{\mathrm{P}}, \bv_{r,h}^{\mathrm{P}}\right)+ 
a_{f}^{\mathrm{P}}\left(\bu_{s,h}^{\mathrm{P}}, \bv_{s,h}^{\mathrm{P}}\right) +a_{\mathrm{BJS}}\left(\bu^{\mathrm{S}}_{f,h}, \bu_{s,h}^{\mathrm{P}} ; \bv_{f,h}^{\mathrm{S}}, \bv_{s,h}^{\mathrm{P}}\right)  +b^{\mathrm{S}}\left(\bv_{f,h}^{\mathrm{S}}, p^{\mathrm{S}}_h\right) \nonumber \\&\quad  
+b_s^{\mathrm{P}}\left(\bv_{s,h}^{\mathrm{P}}, p^{\mathrm{P}}_h\right) + b_f^{\mathrm{P}}\left(\bv_{r,h}^{\mathrm{P}}, p^{\mathrm{P}}_h\right) +b_{\Gamma}\left(\bv_{f,h}^{\mathrm{S}}, \bv_{r,h}^{\mathrm{P}}, \bv_{s,h}^{\mathrm{P}} ; \lambda_h \right) - m_{\theta}(\bu_{r,h}^{\mathrm{P}}, \bv_{s,h}^{\mathrm{P}}) - m_{\theta}(\bu_{s,h}^{\mathrm{P}}, \bv_{s,h}^{\mathrm{P}}) - m_{\theta}(\bu_{s,h}^{\mathrm{P}}, \bv_{r,h}^{\mathrm{P}}) \nonumber  \\
& \quad    - m_{\theta}(\bu_{r,h}^{\mathrm{P}}, \bv_{r,h}^{\mathrm{P}})    + m_{\phi^2/\kappa}(\bu_{r,h}^{\mathrm{P}}, \bv_{r,h}^{\mathrm{P}}) + m_{\rho_f \phi}(\partial_t\bu_{r,h}^{\mathrm{P}}, \bv_{s,h}^{\mathrm{P}}) + m_{\rho_p}(\partial_t\bu_{s,h}^{\mathrm{P}}, \bv_{s,h}^{\mathrm{P}})  + m_{\rho_f \phi }(\partial_t\bu_{r,h}^{\mathrm{P}}, \bv_{r,h}^{\mathrm{P}}) \nonumber  \\ & \quad + m_{\rho_f \phi}(\partial_t\bu_{s,h}^{\mathrm{P}}, \bv_{r,h}^{\mathrm{P}})+ b_\text{sig}^{\mathrm{P}}\left(\bv_{s,h}^{\mathrm{P}}, \bsigma_{h}\right)=\langle\ff_{\mathrm{S}},\bv_{f,h}^{\mathrm{S}}\rangle_{\Omega_{\mathrm{S}}}+(\rho_p \ff_{\mathrm{P}}, \bv_{s,h}^{\mathrm{P}})_{\Omega_{\mathrm{P}}} +( \rho_f \phi  \ff_{\mathrm{P}}, \bv_{r,h}^{\mathrm{P}})_{\Omega_{\mathrm{P}}}, \label{semiweak1} \\
& \left((1-\phi)^2 K^{-1} \partial_t p^{\mathrm{P}}_h, q^{\mathrm{P}}_h\right)_{\Omega_{\mathrm{P}}}-b_s^{\mathrm{P}}\left(\bu_{s,h}^{\mathrm{P}}, q^{\mathrm{P}}_h\right)-b_f^{\mathrm{P}}\left( \bu_{r,h}^{\mathrm{P}}, q^{\mathrm{P}}_h\right)-b^{\mathrm{S}}\left(\bu^{\mathrm{S}}_{f,h}, q^{\mathrm{S}}_h\right)+a_p^{\mathrm{P}}\left(\partial_t \bsigma_{h}, \btau_{h}\right) -b_{\text{sig}}^{\mathrm{P}}\left(\bu_{s,h}^{\mathrm{P}}, \btau_{h}\right) \nonumber  \\
& \qquad \qquad \qquad \qquad \qquad \qquad \qquad \qquad \quad\ =\left(q^{\mathrm{S}}, q^{\mathrm{S}}_h\right)_{\Omega_{\mathrm{S}}}+\left(\rho_f^{-1} \theta, q^{\mathrm{P}}_h\right)_{\Omega_{\mathrm{P}}} ,\label{semiweak2}\\
& \qquad \qquad \qquad \quad\   b_{\Gamma}\left(\bu^{\mathrm{S}}_{f,h}, \bu_{r,h}^{\mathrm{P}}, \bu_{s,h}^{\mathrm{P}} ; \mu_h\right)=0  \label{semiweak3},
\end{align}
\end{subequations}
for all $\bv_{f,h}^{\mathrm{S}} \in \mathbf{V}_{f, h}, q^{\mathrm{S}}_h \in W_{f,h}, \bv_{r,h}^{\mathrm{P}} \in \mathbf{V}_{r, h}, q^{\mathrm{P}}_h \in W_{p,h}$, $\bv_{s,h}^{\mathrm{P}} \in \mathbf{V}_{s, h}, \btau_{h} \in \bZ_h$, and $\mu_h \in \Lambda_h$. 
We demonstrate the well-posedness of the mixed formulation and subsequently deduce the well-posedness of the primal formulation, akin to the continuous case.
The initial conditions $\bv_{r,h}(0), \bv_{s,h}(0), p^{\mathrm{P}}_h(0)$ and $\bsigma_{ h}(0)$ are suitable approximations of $\bv^{\mathrm{P}}_{r,0}, \bv^{\mathrm{P}}_{s,0},  p^{\mathrm{P}, 0}$ and $\bsigma_{ 0}=A^{-1} \beps \left(\by^{\mathrm{P}}_{s,0}\right)$.
We group the spaces and test functions as:
$$
\begin{array}{ll}
\vec{\mathbf{V}}_h \coloneqq \mathbf{V}_{r,h} \times \mathbf{V}_{s,h} \times \mathbf{V}_{f,h}, & \vec{Q}_h \coloneqq W_{p,h} \times \bZ_h \times W_{f,h} \times \Lambda_h, \\
\vec{\bu}_h \coloneqq (\bu_{r,h}^{\mathrm{P}},\bu_{s,h}^{\mathrm{P}}, \bu^{\mathrm{S}}_{f,h}) \in \vec{\mathbf{V}}_h, & \vec{p}_h \coloneqq (p_{p,h}, \bsigma_{h}, p^{\mathrm{S}}_h, \lambda_h ) \in \vec{Q}_h, \\
\vec{\bv}_h \coloneqq (\bv_{r,h}^{\mathrm{P}},\bv_{s,h}^{\mathrm{P}}, \bv_{f,h}^{\mathrm{S}}) \in \vec{\mathbf{V}}_h, & \vec{q}_h:=(q_{p,h}, \btau_{h}, q^{\mathrm{S}}_h, \mu_h) \in \vec{Q}_h,
\end{array}
$$
where the spaces $\vec{\mathbf{V}}_h$ and $\vec{Q}_h$ are endowed with the same norms as in the continuous spaces. 

\subsection{Discrete inf-sup conditions}
\begin{lemma}\label{DLBB}
There exist constants $\beta_1, \beta_2>0$, independent of $h$ such that
\begin{subequations}
\begin{align}
\inf _{\left(\mathbf{0}, \bv_{s,h}^{\mathrm{P}}, \mathbf{0}\right) \in \vec{\mathbf{V}}_h} \sup _{\left(0, \btau_{h}, 0,0\right) \in \vec{Q}_h} \frac{b_{\text{sig}}^{\mathrm{P}}\left(\bv_{s,h}^{\mathrm{P}}, \btau_{h}\right)}{\|(\mathbf{0}, \bv_{s,h}^{\mathrm{P}}, \mathbf{0})\|_{\vec{\mathbf{V}}}\|(0, \btau_{h}, 0,0)\|_{\vec{Q}}} &\geq \beta_1, \label{Dinf-sup1}\\
 \inf _{\left(q^{\mathrm{P}}_h, \mathbf{0}, q^{\mathrm{S}}_h, \mu_h \right) \in \vec{Q}_h} \sup _{\left(\bv_{r,h}^{\mathrm{P}}, \mathbf{0}, \bv_{f,h}^{\mathrm{S}} \right) \in \vec{\mathbf{V}}_h} \frac{b^{\mathrm{S}} \left(\bv_{f,h}^{\mathrm{S}}, q^{\mathrm{S}}_h \right)+b_f^{\mathrm{P}}\left(\bv_{r,h}^{\mathrm{P}},q^{\mathrm{P}}_h\right)+b_{\Gamma}\left(\bv_{r,h}^{\mathrm{P}}, \bv_{f,h}^{\mathrm{S}}, \mathbf{0} ; \mu_h \right)}{\|(\bv_{r,h}^{\mathrm{P}}, \mathbf{0}, \bv_{f,h}^{\mathrm{S}})\|_{\vec{\mathbf{V}}}\|(q^{\mathrm{P}}_h, \mathbf{0}, q^{\mathrm{S}}_h, \lambda_h )\|_{\vec{Q}}} &\geq \beta_2 \label{Dinf-sup2}.
\end{align}\end{subequations}
\end{lemma}
\begin{proof}
Let $\mathbf{0} \neq(\mathbf{0}, \bv_{s,h}^n, \mathbf{0}) \in \vec{\mathbf{V}}_h$ be given. We choose $\btau_{h}^n=\beps (\bv_{s,h}^n)$ and, using K\"orn's inequality \cite{MR2373954}, we obtain
$$
\frac{b_{\text{sig}}^{\mathrm{P}}(\bv_{s,h}^n, \btau_{h}^n)}{\|\btau_{h}^n\|_{0,\Omega_{\mathrm{P}}}}=\frac{\|\beps (\bv_{s,h}^n)\|_{0,\Omega_{\mathrm{P}}}^2}{\|\beps (\bv_{s,h}^n)\|_{0,\Omega_{\mathrm{P}}}}=\|\beps (\bv_{s,h}^n)\|_{0,\Omega_{\mathrm{P}}} \geq \beta_1 \|\bv_{s,h}^n\|_{1,\Omega_{\mathrm{P}}}.
$$
Therefore, \eqref{Dinf-sup1} holds.

Equation \eqref{Dinf-sup2} follows from the  inf-sup condition for the Stokes problem \cite{MR2934766, MR3097958} and the weighted inf-sup condition in \cite[Theorem 7]{MR4253885}. The proof of the inf-sup condition for  $b_{\Gamma}(\cdot,\cdot;\cdot)$
can be done similarly to
\cite[Corollary 3.5]{MR3974685}.
\end{proof}
\begin{corollary}
There exists a constant $\beta_3>0$, independent of $h$ such that
\[ \inf _{\vec{q_h} \in \vec{Q}} \sup_{\vec{\bv}_h \in \vec{\mathbf{V}}} \frac{
 \langle \mathcal{B}(\vec{\bv}_h),\vec{q}_h\rangle}{\|\vec{\bv}_h \|_{\vec{\mathbf{V}}}\|\vec{q}_h\|_{\vec{Q}}}
 \geq \beta_3.
\]
\end{corollary}
\begin{proof}
 The statement follows from Lemma \ref{LBB} by simply taking $\bv_{s,h}^{\mathrm{P}}=\cero$.
 \end{proof}

\subsection{Existence and uniqueness of solution}
In order to show well-posedness of \eqref{semiweak1}-\eqref{semiweak3}, we proceed as in the case of the continuous problem. We introduce $\mathbf{W}_{r,h}^2$, $\mathbf{W}_{s,h}^2$ $W_{p, h}^2$ and $\bZ_h^2$ as the closure of the spaces $\mathbf{V}_{r,h}$, $\mathbf{V}_{s,h}$, $W_{p,h}$ and $\bZ_h$ with respect to the norms \eqref{eq:norms}, and denote 
 $\bbS_{2,h}=\mathbf{W}_{r,h}^2 \times \mathbf{W}_{s,h}^2 \times W_{p, h}^2 \times \bZ_h^2$, and define the discrete  domain 
\begin{align}
\bbD_h:= & \left\{\left(\bu_{r,h}^{\mathrm{P}}, \bu_{s,h}^{\mathrm{P}}, p^{\mathrm{P}}_h, \bsigma_{h}\right) \in \mathbf{V}_{r,h} \times \mathbf{X}_{d,h} \times W_{p,h} \times \bZ_h: \text { for given } \left(\ff_{\mathrm{S}}, r_{\mathrm{S}}\right) \in \left(\mathbf{V}_f^{\prime}, W_f^{ \prime}\right) \right. \nonumber \\&  \exists\left( \bu^{\mathrm{S}}_{f,h}, p^{\mathrm{S}}_h, \lambda_h \right)   \in \mathbf{V}_{f,h} \times W_{f,h} \times \Lambda_h \text { such that } \forall\left(\left(\bv_{r,h}^{\mathrm{P}}, \bv_{s,h}^{\mathrm{P}}, \bv_{f,h}^{\mathrm{S}}\right),\left(p^{\mathrm{P}}_h, \btau_{h}, p^{\mathrm{S}}_h, \mu_h \right)\right) \nonumber \\
& \in \vec{\mathbf{V}}_h \times \vec{Q}_h \text{ satisfy } \eqref{DD1}-\eqref{DD3} \left. \text { for some }\left(\bar{f}_r, \bar{f}_s, \bar{g}_p, \bar{g}_e\right) \in \bbS_{2,h}^{\prime}\right\}. \label{Ddomaindefined}
\end{align}
The equations defining the domain are 
\begin{subequations}
\begin{align}
& a_f^{\mathrm{S}}(\bu^{\mathrm{S}}_{f,h}, \bv_{f,h}^{\mathrm{S}})+a_{f}^{\mathrm{P}}(\bu_{r,h}^{\mathrm{P}}, \bv_{s,h}^{\mathrm{P}})+ a_{f}^{\mathrm{P}}\left(\bu_{r,h}^{\mathrm{P}}, \bv_{r,h}^{\mathrm{P}}\right)+ a_{f}^{\mathrm{P}}\left(\bu_{s,h}^{\mathrm{P}}, \bv_{r,h}^{\mathrm{P}}\right)  + a_{f}^{\mathrm{P}}\left(\bu_{s,h}^{\mathrm{P}}, \bv_{s,h}^{\mathrm{P}}\right)+a_{\mathrm{BJS}}\left(\bu^{\mathrm{S}}_{f,h}, \bu_{s,h}^{\mathrm{P}} ; \bv_{f,h}^{\mathrm{S}}, \bv_{s,h}^{\mathrm{P}}\right)  \nonumber \\&\quad   
 +b^{\mathrm{S}}\left(\bv_{f,h}^{\mathrm{S}}, p^{\mathrm{S}}_h\right)+b_s^{\mathrm{P}}\left(\bu_{s,h}^{\mathrm{P}}, p^{\mathrm{P}}_h\right)  + b_f^{\mathrm{P}}\left( \bv_{r,h}^{\mathrm{P}}, p^{\mathrm{P}}_h\right) +b_{\Gamma}\left(\bv_{f,h}^{\mathrm{S}}, \bv_{r,h}^{\mathrm{P}}, \bv_{s,h}^{\mathrm{P}} ; \lambda_h\right) - m_{\theta}(\bu_{r,h}^{\mathrm{P}}, \bv_{s,h}^{\mathrm{P}}) - m_{\theta}(\bu_{s,h}^{\mathrm{P}}, \bv_{s,h}^{\mathrm{P}}) \nonumber \\ & \quad  - m_{\theta}(\bu_{s,h}^{\mathrm{P}}, \bv_{r,h}^{\mathrm{P}})  - m_{\theta}(\bu_{r,h}^{\mathrm{P}}, \bv_{r,h}^{\mathrm{P}})   + m_{\phi^2/\kappa}(\bu_{r,h}^{\mathrm{P}}, \bv_{r,h}^{\mathrm{P}}) + m_{\rho_f \phi}(\bu_{r,h}^{\mathrm{P}}, \bv_{s,h}^{\mathrm{P}}) + m_{\rho_p}(\bu_{s,h}^{\mathrm{P}}, \bv_{s,h}^{\mathrm{P}}) \nonumber  \\ & \quad + m_{\rho_f \phi }(\bu_{r,h}^{\mathrm{P}}, \bv_{r,h}^{\mathrm{P}})    + m_{\rho_f \phi}(\bu_{s,h}^{\mathrm{P}}, \bv_{r,h}^{\mathrm{P}}) + b_\text{sig}^{\mathrm{P}}\left(\bu_{s,h}^{\mathrm{P}}, \bsigma_{h}\right)= \left(\rho_p \bar{f}_s, \bv_{s,h}^{\mathrm{P}}\right)_{\Omega_{\mathrm{P}}} + \left(\rho_f \phi \bar{f}_r, \bv_{r,h}^{\mathrm{P}}\right)_{\Omega_{\mathrm{P}}} + \langle \ff_{\mathrm{S}}, \bv_{f,h}^{\mathrm{S}}\rangle_{\Omega_{\mathrm{S}}},\label{DD1}\\  
& \left(\frac{(1-\phi)^2}{K} p^{\mathrm{P}}_h, q^{\mathrm{P}}_h\right)_{\Omega_{\mathrm{P}}}-b_s^{\mathrm{P}}\left(\bu_{s,h}^{\mathrm{P}}, q^{\mathrm{P}}_h\right)-b_f^{\mathrm{P}}\left(\phi \bu_{r,h}^{\mathrm{P}}, q^{\mathrm{P}}_h\right)-b^{\mathrm{S}}\left(\bu^{\mathrm{S}}_{f,h}, q^{\mathrm{S}}_h\right) +a_p^{\mathrm{P}}\left( \bsigma_{h}, \btau_{h}\right)  -b_{\text{sig}}^{\mathrm{P}}\left(\bu_{s,h}^{\mathrm{P}}, \btau_{h}\right) \nonumber  \\ 
& \quad =\left(q^{\mathrm{S}}, q^{\mathrm{S}}_h\right)_{\Omega_{\mathrm{S}}} +\left( (1-\phi)^2 K^{-1} \bar{g}_p, q^{\mathrm{P}}_h\right)_{\Omega_{\mathrm{P}}}  +\left(A \bar{g}_e, \btau_{h}\right)_{\Omega_{\mathrm{P}}}, \label{DD2} \\
& b_{\Gamma}\left(\bu^{\mathrm{S}}_{f,h}, \bu_{r,h}^{\mathrm{P}}, \bu_{s,h}^{\mathrm{P}} ; \mu_h\right)=0. \label{DD3}
\end{align}
\end{subequations}
Analogously to the continuous formulation, we define $\mathcal{M}_h$ with domain $\bbD_h$ as
\begin{align*}
 \mathcal{M}_h\left(\left(\bu_{r,h}^{\mathrm{P}}, \bu_{s,h}^{\mathrm{P}}, p^{\mathrm{P}}_h, \bsigma_{h}\right)\right):=&\left\{\left(\bar{f}_r-\bu_{r,h}^{\mathrm{P}}, \bar{f}_s-\bu_{s,h}^{\mathrm{P}}, \bar{g}_p-p^{\mathrm{P}}_h, \bar{g}_e-\bsigma_{h}\right) \in \bbS_{2,h}^{\prime}  
 :\left(\bu_{r,h}^{\mathrm{P}}, \bu_{s,h}^{\mathrm{P}},p^{\mathrm{P}}_h, \bsigma_{h}\right) 
  \text { solves }  \eqref{DD1}-\eqref{DD3}  \right. \\ \left. \text {  for }   \left(\bar{f}_r, \bar{f}_s, \bar{g}_p, \bar{g}_e\right) \in \bbS_{2,h}^{\prime}\right\},  
\end{align*}
and consider the problem
\begin{align}\label{Dparabolic}
   \ddt \left(\begin{array}{l}
\bu_{r,h}^{\mathrm{P}}(t)\\
\bu_{s,h}^{\mathrm{P}}(t)\\
p^{\mathrm{P}}_h(t) \\
\bsigma_{h}(t)
\end{array}\right)+\mathcal{M}_h\left(\begin{array}{l}
\bu_{r,h}^{\mathrm{P}}(t)\\
\bu_{s,h}^{\mathrm{P}}(t)\\
p^{\mathrm{P}}_h(t) \\
\bsigma_{h}(t)
\end{array}\right) \ni\left(\begin{array}{l}
\mathbf{0}\\
\ff_{\mathrm{P}} \\
\rho_f^{-1} (1-\phi)^{-2} K \theta \\
\mathbf{0}
\end{array}\right) . 
\end{align}

\begin{theorem}\label{Dmain}
     For each $\ff_{\mathrm{S}}\in W^{1,1}(0, T ; \mathbf{V}_f^{\prime}), \ff_{\mathrm{P}} \in W^{1,1}(0, T ; \mathbf{V}_r^{ \prime}), r_{\mathrm{S}} \in W^{1,1}(0, T ;  W_f^{ \prime})$, $\theta \in W^{1,1}\left(0, T ; W_p^{ \prime}\right)$, and $p^{\mathrm{P}}_h(0) \in W_{p,h}, \bsigma_{h}(0)=A^{-1} \beps (\by_{s, h}^{\mathrm{P}}(0)) \in \bZ_h , \bu_{r,h}^{\mathrm{P}}(0) \in \mathbf{V}_{r,h}, \bu_{s,h}^{\mathrm{P}}(0) \in \mathbf{V}_{s,h} $ under   assumptions \ref{(H1)}-\ref{(H3)}, there exists a solution of \eqref{semiweak1}-\eqref{semiweak2} with $(\bu^{\mathrm{S}}_{f,h}, p^{\mathrm{S}}_h, \bu_{r,h}^{\mathrm{P}}, p^{\mathrm{P}}_h, \bu_{s,h}^{\mathrm{P}}$, $\bsigma_{h}, \lambda_h) \in L^{\infty}\left(0, T ; \mathbf{V}_{f,h}\right) \times L^{\infty}\left(0, T ; W_{f,h} \right) \times W^{1,\infty}\left(0, T ; \mathbf{V}_{r,h}\right) \times$ $W^{1, \infty}\left(0, T ; W_{p,h} \right) \times W^{1,\infty}\left(0, T ; \mathbf{V}_{s,h}\right) \times W^{1, \infty}\left(0, T ; \bZ_h\right) \times L^{\infty}(0, T ; \Lambda_h)$.
\end{theorem}
To prove Theorem \ref{Dmain} we proceed as in the continuous problem: 
first we establish that  $\bbD_h$  is nonempty, secondly we 
show solvability of the parabolic problem \eqref{Dparabolic}, and then show that the solution to \eqref{Ddomaindefined} satisfies \eqref{semiweak1}-\eqref{semiweak3}.
With the established discrete inf-sup conditions \eqref{Dinf-sup1} and \eqref{Dinf-sup2}, the proof follows closely the proof of Theorem \ref{main}.  The only difference is that the operator $L_{\Gamma}$ from Lemma \ref{lambda_coercive} is now defined as $L_{\Gamma}: \Lambda_h \rightarrow \Lambda_h^{\prime}, \langle L_{\Gamma} \mu_{h, 1},\mu_{h, 2}\rangle:=\langle\mu_{h, 1}, \mu_{h, 2}\rangle_{-\frac{1}{2},h,\Gamma_h}$. One needs to establish that $L_{\Gamma}$ is a bounded, continuous, coercive and monotone operator, which follows immediately from its definition, since $\langle L_{\Gamma} \mu_h , \mu_h  \rangle ^{1 / 2}=\|\mu_h\|_{-\frac{1}{2},\Lambda_h}$.

As a corollary of Theorem \ref{Dmain}, we obtain the following well-posedness result for the original semidiscrete problem \eqref{semi1}-\eqref{semi4}. The proof is identical to the proof of Theorem \ref{original_c}.
\begin{theorem}
Under the same assumptions as Theorem~\ref{Dmain}, there exists a unique solution $(\bu^{\mathrm{S}}_{f,h}, p^{\mathrm{S}}_h,  \bu_{r,h}^{\mathrm{P}}, p^{\mathrm{P}}_h,  \by_{s,h}, \bu_{s,h}^{\mathrm{P}},$ $ \lambda_h) \in L^{\infty}\left(0, T ; \mathbf{V}_{f,h}\right) \times L^{\infty}\left(0, T ; W_{f,h}\right) \times W^{1, \infty}\left(0, T ; \mathbf{V}_{r,h}\right) \times  W^{1, \infty}\left(0, T ; W_{p,h}\right) \times W^{1, \infty}\left(0, T ; \mathbf{V}_{s,h}\right) \times $ $W^{1, \infty}\left(0, T ; {W}_{s,h}\right) \times L^{\infty}(0, T ; \Lambda_h)$ of \eqref{semi1}-\eqref{semi4}.    
\end{theorem}
The proof of the following stability result is identical to the proof of Theorem \ref{stability}.
\begin{lemma}
    For the solution of \eqref{semi1}-\eqref{semi4}, assuming \ref{(H1)}--\ref{(H3)} as well as sufficient regularity of the data, there exists
$\hat{C}(K,\kappa,\rho_f,\rho_s,\lambda_p, \mu_f,\mu_p,\phi, \alpha_{\mathrm{BJS}}, C_K) > 0$ such that
\begin{align*}
& \|\bu_{s,h}^{\mathrm{P}}\|_{L^{\infty}(0, T ; \mathbf{L}^2(\Omega_{\mathrm{P}}))}+\|p^{\mathrm{P}}_h\|_{L^{\infty}(0, T ; L^2(\Omega_{\mathrm{P}}))} + | \bu^{\mathrm{S}}_{f,h} - \partial_t \by_{s,h}|_{L^2(0,T;\mathrm{BJS})}  +\|\by_{s,h}\|_{L^{\infty}\left(0, T ; \mathbf{H}^1(\Omega_{\mathrm{P}})\right)} +\|\bu_{r,h}^{\mathrm{P}}\|_{L^{\infty}(0, T ; \mathbf{L}^2(\Omega_{\mathrm{P}}))}\\&+\|\bu^{\mathrm{S}}_{f,h}\|_{L^2(0, T ; \mathbf{H}^1(\Omega_{\mathrm{S}}))} + \|\bu_{r,h}^{\mathrm{P}}\|_{L^2\left(0, T ; {L}^2(\Omega_{\mathrm{P}})\right)} + \|p^{\mathrm{S}}_h\|_{L^2(0,T; L^2(\Omega_{\mathrm{S}}))}   +\|p^{\mathrm{P}}_h\|_{L^2\left(0,T; L^2(\Omega_{\mathrm{P}})\right)}    + \|\lambda_h\|_{L^2\left(0,T; \Lambda_h\right)}  \\
& \leq \hat{C} \left( \|\ff_{\mathrm{P}}\|_{L^2(0,T;\mathbf{L}^2(\Omega_{\mathrm{P}}))} + \|\ff_{\mathrm{P}}\|_{L^2(0,T;\mathbf{H}^{-1}(\Omega_{\mathrm{P}}))} 
+ \|\theta\|_{L^2(0,T;L^2(\Omega_{\mathrm{P}}))} + \|r_{\mathrm{S}}\|_{L^2(0,T;L^2(\Omega_{\mathrm{S}}))} +\|\bu_{s,h}^{\mathrm{P}}(0)\|_{0,\Omega_{\mathrm{P}}}  
\right. \\& \left.
 \qquad +\|p^{\mathrm{P}}_h(0)\|_{0,\Omega_{\mathrm{P}}}
   +\|\by_{s,h}(0)\|_{1,\Omega_{\mathrm{P}}} +\|\bu_{r,h}^{\mathrm{P}}(0)\|_{0,\Omega_{\mathrm{P}}} \right).
\end{align*}
\end{lemma}

\subsection{Error analysis for the semi discrete scheme}
In this section we analyze the spatial discretization error. Let $k_f$ and $s_f$ be the degrees of polynomials in $\mathbf{V}_{f,h}$ and $W_{f,h}$, let $k_p$ and $s_p$ be the degrees of polynomials in $\mathbf{V}_{r,h}$ and $W_{p,h}$ respectively, and let $k_s$
and $s_s$ be the polynomial degree in $\mathbf{V}_{s,h}$ and $\mathbf{W}_{s,h}$.

\subsubsection{Approximation error}
Let $Q_{f, h}, Q_{p, h}$, and $\bQ_{s, h}$ be the $L^2$-projection operators onto $W_{f,h}, W_{p,h}$, and $\bW_{s,h}$ respectively, satisfying:
\begin{subequations}
\begin{align}
\left(p^{\mathrm{S}}-Q_{f, h} p^{\mathrm{S}}, q^{\mathrm{S}}_h\right)_{\Omega_{\mathrm{S}}}&=0 & \forall q^{\mathrm{S}}_h \in W_{f,h}, \label{ees1}\\
\left(p^{\mathrm{P}}-Q_{p, h} p^{\mathrm{P}}, q^{\mathrm{P}}_h\right)_{\Omega_{\mathrm{P}}}& =0 & \forall q^{\mathrm{P}}_h \in W_{p,h}, \label{ees2} \\
\left\langle\bu_{s}^{\mathrm{P}}-\bQ_{s, h} \bu_s^{\mathrm{P}}, \bv_{s,h}^{\mathrm{P}}\right\rangle_{\Omega_{\mathrm{P}}}&=0 & \forall \bv_{s,h}^{\mathrm{P}} \in \mathbf{W}_{s,h}. \label{ees3}
\end{align}
\end{subequations}
These operators satisfy the approximation properties \cite{MR1299729}:
\begin{subequations}
\begin{align}
\|p^{\mathrm{S}}-Q_{f, h} p^{\mathrm{S}}\|_{0,\Omega_{\mathrm{S}}} & \leq C_1^{\star} h^{r_{s f}+1}\|p^{\mathrm{S}}\|_{r_{s f}+1,\Omega_{\mathrm{S}}} & 0 \leq r_{s_f} \leq s_f, \label{ps}\\
\|p^{\mathrm{P}}-Q_{p, h} p^{\mathrm{P}}\|_{0,\Omega_{\mathrm{P}}} &\leq  C_1^{\star} h^{r_{s_p}+1}\|p^{\mathrm{P}}\|_{{r_{s_p}+1},\Omega_{\mathrm{P}}} & 0 \leq r_{s_p} \leq s_p, \label{ph}\\
\|\bu_{s}^{\mathrm{P}}-\bQ_{s, h} \bu_s^{\mathrm{P}}\|_{0,\Omega_{\mathrm{P}}} &\leq C_1^{\star} h^{{r}_{s_s}+1}\|\bu_s^{\mathrm{P}}\|_{{r_{s_s}+1},\Omega_{\mathrm{P}}} & 0 \leq {r}_{s_s} \leq s_s \label{vh}.
\end{align}
\end{subequations}
The definition    $\Lambda_h=\left.\mathbf{V}_{r, h} \cdot \bn_{\mathrm{P}}\right|_{\Sigma}$ implies the following approximation property \cite[Appendix A]{MR3117433} 
\begin{align}\label{lambdah}
\|\lambda- Q_{\lambda,h}\lambda\|_{{-1 / 2},\Sigma} \leq C_1^{\star} h^{r_{\tilde{k}_p}+\frac{1}{2}}\| \lambda\|_{{r_{\tilde{k}_p}},\Sigma} \qquad  {-1/2} \leq {r}_{\tilde{k}_p} \leq \tilde{k}_p -1/2  .
\end{align}
Next, we consider a Stokes-like projection operator $\left(\bS_{f, h}, R_{f, h}\right): \mathbf{V}_f \rightarrow \mathbf{V}_{f, h} \times W_{f,h}$, defined for all $\bv_f^{\mathrm{S}}\in \mathbf{V}_f$ as
\begin{subequations}
\begin{align}
a_f^{\mathrm{S}}\left(\bS_{f, h} \bv_f^{\mathrm{S}}, \bv_{f,h}^{\mathrm{S}}\right)-b_f^{\mathrm{S}}\left(\bv_{f,h}^{\mathrm{S}}, R_{f, h} \bv_f^{\mathrm{S}}\right)&=a_f^{\mathrm{S}}\left(\bv_f^{\mathrm{S}}, \bv_{f,h}^{\mathrm{S}}\right) & \forall \bv_{f,h}^{\mathrm{S}} \in \mathbf{V}_{f, h}, \label{sp1} \\
b_f^{\mathrm{S}}\left(\bS_{f, h} \bv_f^{\mathrm{S}}, q^{\mathrm{S}}_h\right)&=b_f\left(\bv_f^{\mathrm{S}}, q^{\mathrm{S}}_h\right) & \forall q^{\mathrm{S}}_h \in W_{f,h} \label{sp2}.
\end{align}
\end{subequations}
The operator $\bS_{f, h}$ satisfies the approximation property \cite{MR2788393, MR3851065}:
\begin{align}\label{sp3}
\|\bv_f^{\mathrm{S}}-\bS_{f, h} \bv_f^{\mathrm{S}}\|_{1,\Omega_{\mathrm{S}}} \leq C_1^{\star} h^{r_{k_f}}\|\bv_f^{\mathrm{S}}\|_{{r_{k_f}+1},\Omega_{\mathrm{S}}}, \quad 0 \leq r_{k_f} \leq k_f .
\end{align}
Let $\bPi_{r, h}$ be the Stokes projection onto $\mathbf{V}_{r, h}$ satisfying  for all $\bv_r^{\mathrm{P}} \in \mathbf{V}_r$,
\begin{subequations}
\begin{align}
&\left(\nabla \cdot \bPi_{r, h} \bv_r^{\mathrm{P}}, q^{\mathrm{P}}_h\right)=\left(\nabla \cdot \bv_r^{\mathrm{P}}, q^{\mathrm{P}}_h\right)  &\forall q^{\mathrm{P}}_h \in W_{p,h},& \label{mfe1}\\
&\left\langle\bPi_{r, h} \bv_r^{\mathrm{P}} \cdot \bn_{\mathrm{P}}, \bv_{r,h}^{\mathrm{P}} \cdot \bn_{\mathrm{P}}\right\rangle_{\Sigma}=\left\langle\bv_r^{\mathrm{P}} \cdot \bn_{\mathrm{P}}, \bv_{r,h}^{\mathrm{P}} \cdot \bn_{\mathrm{P}} \right\rangle_{\Sigma} & \forall \bv_{r,h}^{\mathrm{P}} \in \mathbf{V}_{r, h}.& \label{mfe2}
\end{align}
\end{subequations}
We will make use of the following estimates regarding $\bPi_{r, h}$:
\begin{subequations}
\begin{align}
 \|\bv_r^{\mathrm{P}}-\bPi_{r, h} \bv_r^{\mathrm{P}}\|_{0,\Omega_{\mathrm{P}}} &\leq C_1^{\star} h^{r_{k_p}+1}\|\bv_r^{\mathrm{P}}\|_{H^{r_{k_p}+1}\left(\Omega_{\mathrm{P}}\right)}, \quad 0 \leq r_{k_p} \leq k_p, \label{mfe3}\\
{\left\|\bPi_{r, h} \bv_r^{\mathrm{P}}\right\|_{1, \Omega_p}} & \leq C_1^{\star} \left\| \bv_r^{\mathrm{P}}\right\|_{1, \Omega_p}. \label{mfe333}
\end{align}\end{subequations}
Finally, let $\bS_{s, h}$ be the Scott--Zhang interpolant from $\mathbf{V}_s$ onto $\mathbf{V}_{s, h}$, satisfying \cite{MR1011446}:
\begin{align}\label{sz1}
& \|\by_s^{\mathrm{P}}-\bS_{s, h} \by_s^{\mathrm{P}}\|_{0,\Omega_{\mathrm{P}}}+h\left|\by_s^{\mathrm{P}}-\bS_{s, h} \by_s^{\mathrm{P}}\right|_{1,\Omega_{\mathrm{P}}} \leq C_1^{\star} h^{r_{k_s}+1}\|\by_s^{\mathrm{P}}\|_{{r_{k s}+1},\Omega_{\mathrm{P}}}, 
 \quad 0 \leq r_{k_s} \leq k_s .
\end{align}
\subsubsection{Construction of a weakly-continuous interpolant}
In this section we use the operators defined above to build an operator onto a space with  weak continuity of normal velocities. Let us consider 
\begin{align*}
\mathbf{U}&=\left\{\left(\bv_f^{\mathrm{S}}, \bv_r^{\mathrm{P}}, \bw_s^{\mathrm{P}}\right) \in \mathbf{V}_f \times \mathbf{V}_r  \times \mathbf{V}_s: \bv_f^{\mathrm{S}}\cdot \bn_{\mathrm{S}}+\bv_r^{\mathrm{P}} \cdot \bn_{\mathrm{P}}+\bw_s^{\mathrm{P}} \cdot \bn_{\mathrm{P}}=0 \quad\text{on} \ \Sigma\right\},\\
\mathbf{U}_h= & \left\{\left(\bv_{f,h}^{\mathrm{S}}, \bv_{r,h}^{\mathrm{P}}, \bw_{s,h}^{\mathrm{P}}\right) \in \mathbf{V}_{f, h} \times \mathbf{V}_{r, h} \times \mathbf{V}_{s, h}:b_{\Gamma}\left(\bv_{f,h}^{\mathrm{S}}, \bv_{r,h}^{\mathrm{P}}, \bw_{s,h}^{\mathrm{P}} ; \mu_h\right)=0, \forall \mu_h \in \Lambda_h\right\} .
\end{align*}
We will construct an interpolation operator $I_h: \mathbf{U} \rightarrow \mathbf{U}_h$ as a triple
$$
I_h\left(\bv_f^{\mathrm{S}}, \bv_r^{\mathrm{P}}, \bw_s^{\mathrm{P}}\right)=\left(\bI_{f, h} \bv_f^{\mathrm{S}}, \bI_{r, h} \bv_r^{\mathrm{P}}, \bI_{s, h} \bw_s^{\mathrm{P}}\right),
$$
with the following properties:
\begin{subequations}
\begin{align}
b_{\Gamma}\left(\bI_{f, h} \bv_f^{\mathrm{S}}, \bI_{r, h} \bv_r^{\mathrm{P}}, \bI_{s, h} \bw_s^{\mathrm{P}} ; \mu_h\right)&=0 & \forall \mu_h \in \Lambda_h, \label{fp1}\\
b_f^{\mathrm{S}}\left(\bI_{f, h} \bv_f^{\mathrm{S}}-\bv_f^{\mathrm{S}}, q^{\mathrm{S}}_h\right)&=0 & \forall q^{\mathrm{S}}_h \in W_{f,h}, \label{fp2}\\
b_f^{\mathrm{P}}\left(\bI_{r, h} \bv_r^{\mathrm{P}}-\bv_r^{\mathrm{P}}, q^{\mathrm{P}}_h\right)&=0 & \forall q^{\mathrm{P}}_h \in W_{p,h} \label{fp3} .
\end{align}
\end{subequations}
We let $\bI_{f, h}:=\bS_{f, h}$ and $\bI_{s, h}:=\bS_{s, h}$. To construct $\bI_{r, h}$, we first consider the following Stokes auxiliary problem 
\begin{align}\label{auxi1}
\nonumber - \bDelta \bzeta + \nabla p =\cero \quad \text{and} \quad 
\nabla \cdot \bzeta & =0 & \text { in } \Omega_{\mathrm{P}}, \\ 
 \bzeta&=\cero &  \text { on } \Gamma_{\mathrm{P}}, \\
\nonumber \bzeta \cdot \bn_{\mathrm{P}}&=\left(\bv_f^{\mathrm{S}}-\bI_{f, h} \bv_f^{\mathrm{S}}\right) \cdot \bn_{\mathrm{S}}+\left(\bw_s^{\mathrm{P}}-\bI_{s, h} \bw_s^{\mathrm{P}}\right) \cdot \bn_{\mathrm{P}} & \text { on } \Sigma .\end{align}
Define $\bw=\bzeta+\bv_r^{\mathrm{P}}$. From \eqref{auxi1} we have
\begin{subequations}
\begin{align}\label{auxi2}
\nabla \cdot \bw&=\nabla \cdot \bzeta+\nabla \cdot \bv_r^{\mathrm{P}}=\nabla \cdot \bv_r^{\mathrm{P}} \quad \text { in } \Omega_{\mathrm{P}}, \\
\bw \cdot \bn_{\mathrm{P}} & =\bzeta \cdot \bn_{\mathrm{P}}+\bv_r^{\mathrm{P}} \cdot \bn_{\mathrm{P}}
=-\bI_{f, h} \bv_f^{\mathrm{S}}\cdot \bn_{\mathrm{S}}-\bI_{s, h} \bw_s^{\mathrm{P}} \cdot \bn_{\mathrm{P}} \quad \text { on } \Sigma. \label{cond}
\end{align}\end{subequations}
We now let
\begin{align}\label{eqop}
\bI_{r, h} \bv_r^{\mathrm{P}}=\bPi_{r, h} \bw .
\end{align}
Next, we verify that the operator $I_h=\left(\bI_{f, h}, \bI_{r, h}, \bI_{s, h}\right)$ satisfies \eqref{fp1}-\eqref{fp3}. Property \eqref{fp2} follows immediately from \eqref{sp2}, while, using \eqref{auxi2} and \eqref{mfe1}, property \eqref{fp3} follows from
$$
\left(\nabla \cdot \bI_{r, h} \bv_r^{\mathrm{P}}, q^{\mathrm{P}}_h\right)_{\Omega_{\mathrm{P}}}  =\left(\nabla \cdot \bPi_{r, h} \bw, q^{\mathrm{P}}_h\right)_{\Omega_{\mathrm{P}}}=\left(\nabla \cdot \bw, q^{\mathrm{P}}_h\right)_{\Omega_{\mathrm{P}}}  =\left(\nabla \cdot \bv_r^{\mathrm{P}}, q^{\mathrm{P}}_h\right)_{\Omega_{\mathrm{P}}} \qquad \forall q^{\mathrm{P}}_h \in W_{p,h} .
$$
Using \eqref{cond} and \eqref{mfe2}, we have for all $\mu_h \in \Lambda_h$,
$$
\left\langle \bI_{r, h} \bv_r^{\mathrm{P}} \cdot \bn_{\mathrm{P}}, \mu_h\right\rangle_{\Sigma}  =\left\langle\bPi_{r, h} \bw \cdot \bn_{\mathrm{P}}, \mu_h\right\rangle_{\Sigma}=\left\langle\bw \cdot \bn_{\mathrm{P}}, \mu_h\right\rangle_{\Sigma}  =\left\langle-\bI_{f, h} \bv_f^{\mathrm{S}}\cdot \bn_{\mathrm{S}}-\bI_{s, h} \bw_s^{\mathrm{P}} \cdot \bn_{\mathrm{P}}, \mu_h\right\rangle_{\Sigma},
$$
which implies \eqref{fp1}.

The approximation properties of the components of $I_h$ are the following.
\begin{lemma}
 For all sufficiently smooth $\bv_f^{\mathrm{S}}, \bv_r^{\mathrm{P}}$, $\bw_s^{\mathrm{P}}$, and for $0 \leq r_{k_p} \leq k_p$, $0 \leq r_{k_f} \leq k_f$,  $0 \leq r_{k_s} \leq k_s$, there holds  
 \begin{subequations}
\begin{align}
 \|\bv_f^{\mathrm{S}}-\bI_{f, h} \bv_f^{\mathrm{S}}\|_{1,\Omega_{\mathrm{S}}} &\leq C_1^{\star} h^{r_{k_f}}\|\bv_f^{\mathrm{S}}\|_{r_{k_f}+1,\Omega_{\mathrm{S}}}, \label{ea1}\\
 \|\bw_s^{\mathrm{P}}-\bI_{s,h} \bw_s^{\mathrm{P}}\|_{0,\Omega_{\mathrm{P}}}+h\left|\bw_s^{\mathrm{P}}-\bI_{s,h} \bw_s^{\mathrm{P}}\right|_{1,\Omega_{\mathrm{P}}} &\leq C_1^{\star} h^{r_{k_s}+1}\|\bw_s^{\mathrm{P}}\|_{r_{k_s}+1,\Omega_{\mathrm{P}}}, \label{ea2} \\
 \|\bv_r^{\mathrm{P}}-\bI_{r, h} \bv_r^{\mathrm{P}}\|_{0,\Omega_{\mathrm{P}}}  &\leq C_1^{\star} (h^{r_{k_p}+1}\|\bv_r^{\mathrm{P}}\|_{r_{k_p}+1,\Omega_{\mathrm{P}}}+h^{r_k}\|\bv_f^{\mathrm{S}}\|_{r_{k_f}+1,\Omega_{\mathrm{S}}}+h^{r_{k_s}}\|\bw_s^{\mathrm{P}}\|_{r_{k_s}+1,\Omega_{\mathrm{P}}}). \label{ea3}
\end{align}
\end{subequations}
\end{lemma}
\begin{proof}
The bounds \eqref{ea1} and \eqref{ea2} follow immediately from \eqref{sp3} and \eqref{sz1}. Next, using \eqref{eqop}, we have
$$
\|\bv_r^{\mathrm{P}}-\bI_{r, h} \bv_r^{\mathrm{P}}\|_{0,\Omega_{\mathrm{P}}}  =\|\bv_r^{\mathrm{P}}-\bPi_{r, h} \bv_p-\bPi_{r, h} \bzeta\|_{0,\Omega_{\mathrm{P}}} 
 \leq\|\bv_p-\bPi_{r, h} \bv_p\|_{0,\Omega_{\mathrm{P}}}+\|\bPi_{r, h} \bzeta\|_{0,\Omega_{\mathrm{P}}} .
$$
Elliptic regularity for \eqref{auxi1} implies,
\begin{align}\label{ea4}
\|\bzeta\|_{1,\Omega_{\mathrm{P}}} \leq & C^{\star}(\|(\bv_f^{\mathrm{S}}-\bI_{f, h} \bv_f^{\mathrm{S}}) \cdot \bn_{\mathrm{S}}\|_{1 / 2,\Sigma}+\|(\xi_p-\bI_{s, h} \bw_s^{\mathrm{P}}) \cdot \bn_{\mathrm{P}}\|_{1 / 2,\Sigma}) .
\end{align}
Using \eqref{mfe333}, \eqref{ea4} and trace inequality, we get
\begin{align}
\|\bPi_{r, h} \bzeta\|_{0,\Omega_{\mathrm{P}}} & \leq  C_1^{\star}\|\bzeta\|_{1,\Omega_{\mathrm{P}}} \nonumber \\
& \leq  C_1^{\star}(\|(\bv_f^{\mathrm{S}}-\bI_{f, h} \bv_f^{\mathrm{S}}) \cdot \bn_{\mathrm{S}}\|_{1 / 2,\Sigma}+\|(\bw_s^{\mathrm{P}}-\bI_{s, h} \bw_s^{\mathrm{P}}) \cdot \bn_{\mathrm{P}}\|_{1 / 2,\Sigma}) \nonumber  \\
& \leq C_1^{\star} \left(\|\bv_f^{\mathrm{S}}-\bI_{f, h} \bv_f^{\mathrm{S}}\|_{1,\Omega_{\mathrm{S}}}+\|\bw_s^{\mathrm{P}}-\bI_{s, h} \bw_s^{\mathrm{P}}\|_{1,\Omega_{\mathrm{P}}}\right) \label{ea5}.
\end{align}
A combination of \eqref{ea4}, \eqref{ea5}, \eqref{sp3}, \eqref{mfe3}, and \eqref{sz1} implies \eqref{ea3}.
\end{proof}
\subsection{Error estimates}
We recall that, due to \eqref{mixed-primal}, $(\bu_f^{\mathrm{S}}, \bu_r^{\mathrm{P}}, \partial_t \by_s^{\mathrm{P}}) \in \mathbf{U}$ and {so} we can apply the interpolant to get  $(\bI_{f, h} \bu_f^{\mathrm{S}}, \bI_{r, h} \bu_r^{\mathrm{P}}, \bI_{s, h} \partial_t \by_s^{\mathrm{P}}) \in \mathbf{U}_h$ for any $t \in(0, T]$. We split individual errors into approximation and discretization contributions
\begin{align*}
\mathbf{e}_f & \coloneqq \bu_f^{\mathrm{S}}-\bu_{f, h}^{\mathrm{S}}=\left(\bu_f^{\mathrm{S}}-\bI_{f, h} \bu_f^{\mathrm{S}}\right)+\left(\bI_{f, h} \bu_f^{\mathrm{S}}-\bu^{\mathrm{S}}_{f,h}\right) \coloneqq\bchi_f+\bphi_{f,h}, \\
\mathbf{e}_r & :=\bu_r-\bu_{r,h}^{\mathrm{P}}=\left(\bu_r-\bI_{r, h} \bu_r^{\mathrm{P}}\right)+\left(\bI_{r, h} \bu_r-\bu_{r,h}^{\mathrm{P}}\right):=\bchi_r+\bphi_{r,h}, \\
\mathbf{e}_s & :=\by_s^{\mathrm{P}}-\by_{s, h}^{\mathrm{P}}=\left(\by_s^{\mathrm{P}}-\bI_{s, h} \by_s^{\mathrm{P}}\right)+\left(\bI_{s, h} \by_s^{\mathrm{P}}-\by_{s, h}^{\mathrm{P}}\right):=\bchi_s+\bphi_{s,h}, \\
\mathbf{e}_{ss} & :=\bu_{s}^{\mathrm{P}}-\bu_{s,h}^{\mathrm{P}}=\left(\bu_{s}^{\mathrm{P}}-\bQ_{s, h} \bu_s^{\mathrm{P}}\right)+\left(\bQ_{s, h} \bu_{s}^{\mathrm{P}}-\bu_{s,h}^{\mathrm{P}}\right):=\bchi_{ss}+\bphi_{ss,h}, \\
e_{f p} & :=p^{\mathrm{S}}-p^{\mathrm{S}}_h=\left(p^{\mathrm{S}}-Q_{f, h} p^{\mathrm{S}}\right)+\left(Q_{f, h} p^{\mathrm{S}}-p^{\mathrm{S}}_h\right):=\chi_{f p}+\phi_{f p, h}, \\
e_{p p} & :=p^{\mathrm{P}}-p^{\mathrm{P}}_h=\left(p^{\mathrm{P}}-Q_{p, h} p^{\mathrm{P}}\right)+\left(Q_{p, h} p^{\mathrm{P}}-p^{\mathrm{P}}_h\right):=\chi_{p p}+\phi_{p p, h}, \\
e_\lambda & :=\lambda-\lambda_h=\left(\lambda-Q_{\lambda, h} \lambda\right)+\left(Q_{\lambda, h} \lambda-\lambda_h\right):=\chi_\lambda+\phi_{\lambda, h} .
\end{align*}
Subtracting \eqref{semi1}-\eqref{semi2} from \eqref{mixed-primal} and adding the resulting equation, we obtain the following error equation
\begin{align*}
& a_f^{\mathrm{S}}(\mathbf{e}_{f}, \bv_{f,h}^{\mathrm{S}})+a_{f}^{\mathrm{P}}(\mathbf{e}_{r}, \bw_{s,h}^{\mathrm{P}})+a_s^{\mathrm{P}}(\mathbf{e}_{s}, \bw_{s,h}^{\mathrm{P}})+ a_{f}^{\mathrm{P}}(\mathbf{e}_{r}, \bv_{r,h}^{\mathrm{P}})+ a_{f}^{\mathrm{P}}(\partial_t \mathbf{e}_{s}, \bv_{r,h}^{\mathrm{P}})  + a_{f}^{\mathrm{P}}(\partial_t \mathbf{e}_{s}, \bw_{s,h}^{\mathrm{P}}) +a_{\mathrm{BJS}}(\mathbf{e}_{f}, \partial_t \mathbf{e}_{s} ; \bv_{f,h}^{\mathrm{S}}, \bw_{s,h}^{\mathrm{P}}) \nonumber \\&\quad 
 +b^{\mathrm{S}}\left(\bv_{f,h}^{\mathrm{S}}, e_{fp}\right)+b_s^{\mathrm{P}}\left(\bw_{s,h}^{\mathrm{P}}, e_{pp}\right)  + b_f^{\mathrm{P}}\left(\bv_{r,h}^{\mathrm{P}}, e_{pp}\right) +b_{\Gamma}\left(\bv_{f,h}^{\mathrm{S}}, \bv_{r,h}^{\mathrm{P}}, \bw_{s,h}^{\mathrm{P}} ; e_{\lambda} \right) -b_s^{\mathrm{P}}\left(\partial_t \mathbf{e}_{s}, q^{\mathrm{P}}_h\right)-b_f^{\mathrm{P}}\left(\mathbf{e}_{r}, q^{\mathrm{P}}_h\right) -b^{\mathrm{S}}\left(\mathbf{e}_{f}, q^{\mathrm{S}}_h\right) \\ & \quad - m_{\theta}(\mathbf{e}_{r}, \bw_{s,h}^{\mathrm{P}}) - m_{\theta}(\partial_t \mathbf{e}_{s}, \bw_{s,h}^{\mathrm{P}})   - m_{\theta}(\mathbf{e}_{r}, \bv_{r,h}^{\mathrm{P}})   - m_{\theta}(\partial_t \mathbf{e}_{s}, \bv_{r,h}^{\mathrm{P}}) + m_{\phi^2/\kappa}(\mathbf{e}_{r}, \bv_{r,h}^{\mathrm{P}}) + m_{\rho_f \phi}(\partial_t\mathbf{e}_{r}, \bw_{s,h}^{\mathrm{P}})  \\ & \quad  + m_{\rho_p}(\partial_t\mathbf{e}_{ss}, \bw_{s,h}^{\mathrm{P}}) + m_{\rho_f \phi }(\partial_t\mathbf{e}_{r}, \bv_{r,h}^{\mathrm{P}})  + m_{\rho_f \phi}(\partial_t\mathbf{e}_{ss}, \bv_{r,h}^{\mathrm{P}})+ \left((1-\phi)^2 K^{-1} \partial_t e_{pp}, q^{\mathrm{P}}_h\right)_{\Omega_{\mathrm{P}}} =0 .
\end{align*}
Setting $\bv_{f,h}^{\mathrm{S}}= \bphi_{f,h}, \bv_{r,h}^{\mathrm{P}} = \bphi_{r,h}, \bw_{s,h}^{\mathrm{P}} = \partial_t \bphi_{s,h}, \bv_{s,h}^{\mathrm{P}}= \bphi_{ss,h}, q^{\mathrm{S}}_h  = \phi_{fp,h}, \text{ and } q^{\mathrm{P}}_h = \phi_{pp,h}$, 
we conclude that the following terms simplify due to the properties of the projection operators \eqref{ees2}, \eqref{fp2}, and \eqref{fp3}:
$$
b^{\mathrm{S}}\left(\bchi_f, \phi_{f p, h}\right) =b_f^{\mathrm{P}}\left(\bchi_r, \phi_{p p, h}\right)=b_f^{\mathrm{P}}\left(\bphi_{r,h}, \chi_{p p}\right)=0, \qquad 
\left((1-\phi)^2 K^{-1} \partial_t \chi_{p p}, \phi_{p p, h}\right) = 0.
$$
Moreover, from \eqref{fp1} and \eqref{semi3} we have 
\[
b_{\Gamma}\left(\bphi_{f,h}, \bphi_{r,h}, \partial_t \bphi_{s,h} ; \phi_{\lambda, h}\right)  =0.
\]
On  the other hand, by definition of $H^{-1 / 2} (\Sigma)$ \cite{MR3117433}, there exists $\bw \in \mathbf{H}^1\left(\Omega_{\mathrm{P}}\right)$ such that
\begin{align}\label{another}
\langle \lambda, \bw \cdot \bn_{\mathrm{P}} \rangle_{\Sigma}=\|\lambda\|_{\Lambda_h}^2, \quad \text { and } \quad\|\bw\|_{1,\Omega_{\mathrm{P}}}=\|\lambda\|_{\Lambda_h} \text {. }
\end{align}
This implies that 
$$
\langle \chi_{\lambda}, \bphi_{f,h} \cdot \bn_{\mathrm{P}} \rangle_{\Sigma} =
\langle \chi_{\lambda}, \bphi_{r,h} \cdot \bn_{\mathrm{P}} \rangle_{\Sigma} =\|\chi_{\lambda}\|_{\Lambda_h}^2.
$$
Rearranging terms and using the results above, the error equation 
becomes 
\begin{align}\label{ne1}
&a_f^{\mathrm{S}}(\bphi_{f,h},\bphi_{f,h})+ a_{f}^{\mathrm{P}}(\bphi_{r,h}, \partial_t \bphi_{s,h}) +a_{f}^{\mathrm{P}}(\bphi_{r,h}, \bphi_{r,h}) + a_{f}^{\mathrm{P}}(\partial_t \bphi_{s,h}, \bphi_{r,h}) + a_{f}^{\mathrm{P}}(\partial_t \bphi_{s,h}, \partial_t \bphi_{s,h}) +a_s^{\mathrm{P}}(\bphi_{s,h},\partial_t \bphi_{s,h}) \nonumber \\ 
&\quad  +a_{\mathrm{BJS}}(\bphi_{f,h}, \partial_t \bphi_{s,h} ; \bphi_{f,h}, \partial_t \bphi_{s,h}) + m_{\rho_f \phi}(\partial_t\bphi_{r,h}, \partial_t \bphi_{s,h})   + m_{\rho_p}(\partial_t\bphi_{ss,h},  \partial_t \bphi_{s,h}) + m_{\rho_f \phi }(\partial_t\bphi_{r,h}, \bphi_{r,h})\nonumber \\
&\quad   + m_{\rho_f \phi}(\partial_t\bphi_{ss,h}, \bphi_{r,h})  + ((1-\phi)^2 K^{-1} \partial_t \phi_{pp,h}, \phi_{pp,h})_{\Omega_{\mathrm{P}}}  - m_{\theta}(\bphi_{r,h}, \partial_t \bphi_{s,h})  - m_{\theta}(\partial_t \bphi_{s,h}, \partial_t \bphi_{s,h})  - m_{\theta}(\bphi_{r,h}, \bphi_{r,h})  \nonumber \\ 
& \quad  - m_{\theta}(\partial_t \bphi_{s,h}, \bphi_{r,h}) + m_{\phi^2/\kappa}(\bphi_{r,h}, \bphi_{r,h}) = 
\mathcal{J}_1 + \mathcal{J}_2 + \mathcal{J}_3 + \mathcal{J}_4 + \mathcal{J}_5 +\mathcal{J}_6,
\end{align}
where the right-hand side terms are defined as follows
\begin{align*}
     \mathcal{J}_1  := & -a_f^{\mathrm{S}}\left(\bchi_f, \bphi_{f,h}\right)  + m_{\theta}(\bchi_r, \bphi_{ss,h})   + m_{\theta}\left(\bchi_r, \bphi_{r,h}\right) + m_{\theta}(\partial_t \bchi_s, \bphi_{r,h}) -m_{\phi^2/\kappa}(\bchi_r, \bphi_{r,h})  - m_{\rho_p}(\partial_t\bchi_{ss}, \bphi_{ss,h}) \nonumber  \\& \quad -m_{\rho_f \phi }(\partial_t\bchi_r, \bphi_{r,h})  - m_{\rho_f \phi}\left(\partial_t\bchi_{ss}, \bphi_{r,h}\right)  + m_{\theta}(\partial_t \bchi_s, \bphi_{ss,h}) -  m_{\rho_f \phi}(\partial_t\bchi_r, \bphi_{ss,h}), \\
    \mathcal{J}_2: = & - a_{f}^{\mathrm{P}}\left(\bchi_r, \bphi_{r,h}\right)- a_{f}^{\mathrm{P}}\left(\partial_t \bchi_s, \bphi_{r,h}\right), \\ 
    \mathcal{J}_3 := & -\sum_{j=1}^{d-1}\left\langle\mu \alpha_{\mathrm{BJS}} \sqrt{Z_j^{-1}}\left(\bchi_f-\partial_t \bchi_s\right) \cdot \btau_{f, j},\left(\bphi_{f,h}-\partial_t \bphi_{s,h}\right) \cdot \btau_{f, j}\right\rangle_{\Sigma}, \\
    \mathcal{J}_4 := & -b^{\mathrm{S}}\left(\bphi_{f,h}, \chi_{fp}\right) +b_s^{\mathrm{P}}\left(\partial_t \bchi_s, \phi_{pp,h}\right) - b_{\Gamma}\left(\bphi_{f,h}, \bphi_{r,h}, 0 ; \chi_{\lambda} \right), \\  
    \mathcal{J}_5 :=& -  a_{f}^{\mathrm{P}}\left(\bchi_r,   \partial_t \bphi_{s,h}\right) -a_s^{\mathrm{P}}\left(\bchi_s,\partial_t \bphi_{s,h}\right) - a_{f}^{\mathrm{P}}\left(\partial_t \bchi_s,   \partial_t \bphi_{s,h}\right), \\ 
    \mathcal{J}_6 :=& -\left\langle\partial_t \bphi_{s,h} \cdot \bn_{\mathrm{P}}, \chi_\lambda\right\rangle_{\Sigma} -  b_s^{\mathrm{P}}\left(\partial_t \bphi_{s,h}, \chi_{p p}\right).
\end{align*}

It is important to remark that we have the equation
    \begin{align*}
   & m_{\rho_f \phi}(\partial_t\bphi_{r,h}, \partial_t \bphi_{s,h})   + m_{\rho_s (1-\phi)}\left(\partial_t\bphi_{ss,h}, \partial_t \bphi_{s,h}\right ) + m_{\rho_f \phi}\left(\partial_t\bphi_{ss,h}, \partial_t \bphi_{s,h}\right )  + m_{\rho_f \phi }(\partial_t\bphi_{r,h}, \bphi_{r,h}) + m_{\rho_f \phi}(\partial_t\bphi_{ss,h}, \bphi_{r,h}) \\&  = \frac{1}{2} \partial_t \left( \|\sqrt{\rho_f \phi} (\bphi_{r,h} + \bphi_{ss,h}) \|^2_{0,\Omega_{\mathrm{P}}} + \| \sqrt{\rho_s(1-\phi)} \bphi_{ss,h}\|^2_{0,\Omega_{\mathrm{P}}}\right).
    \end{align*}
Then, combining this equation with the inequalities \eqref{ineq:aux} and 
$\|\bphi_{r,h}\|_{0,\Omega_{\mathrm{P}}}^2 \leq \|\bphi_{r,h} + \bphi_{ss,h}\|^2_{0,\Omega_{\mathrm{P}}}$, and using the coercivity of the bilinear forms, we can arrive at the following estimate 
\[
\texttt{LHS}_{\eqref{ne1}}\gtrsim  \frac{1}{2} \partial_t \left( \|\bphi_{r,h}\|_{0,\Omega_{\mathrm{P}}}^2 + \|\phi_{pp,h}\|_{0,\Omega_{\mathrm{P}}}^2 + \|\bphi_{s,h} \|_{1,\Omega_{\mathrm{P}}} ^2 + \|\bphi_{ss,h} \|_{0,\Omega_{\mathrm{P}}} ^2 \right)+ \left| \bphi_{f,h} - \partial_t \bphi_{s,h}\right|_{\mathrm{BJS}}^2  + \|\bphi_{r,h}\|_{0,\Omega_{\mathrm{P}}}^2  + \|  \bphi_{f,h} \|_{1,\Omega_{\mathrm{S}}}^2.
\]
We proceed to bound the terms on the right-hand side in \eqref{ne1}. Using Lemma \ref{coercivity-continuity},  Cauchy--Schwarz and Young's inequalities, as well as  \eqref{another}, we have 
\begin{align*}
\mathcal{J}_1 & \leq C \epsilon_1^{-1}\left(\|\bchi_f\|_{1,\Omega_{\mathrm{S}}}^2  +\|\bchi_r\|_{1,\Omega_{\mathrm{P}}}^2   + \|\bchi_r\|_{0,\Omega_{\mathrm{P}}}^2  + \| \partial_t \bchi_s \|_{0,\Omega_{\mathrm{P}}}^2 + \| \partial_t \bchi_{ss} \|_{0,\Omega_{\mathrm{P}}}^2 
 +  \| \partial_t \bchi_r \|_{0,\Omega_{\mathrm{P}}}^2\right) \nonumber  \\
 &\quad +\epsilon_1\left(   \|\bphi_{f,h}\|_{1,\Omega_{\mathrm{S}}}^2 
+\|\bphi_{ss,h}\|_{0,\Omega_{\mathrm{P}}}^2+ \|\bphi_{r,h}\|_{0,\Omega_{\mathrm{P}}}^2  \right) .
\end{align*}
Next we estimate the terms involving $\bphi_{r,h}$ in $\mathbf{H}^1(\Omega_{\mathrm{P}})$, since we do not have bounds on $\bphi_{r,h}$ in the energy norm on the left-hand side. We handle this by using inverse, Cauchy--Schwarz, and Young's inequalities. Taking $\epsilon = \epsilon_1 h^{2}$, we get 
\begin{align}\label{h_11}
     \mathcal{J}_2 & \leq \epsilon h ^{-2}\|\bphi_{r,h}\|_{0,\Omega_{\mathrm{P}}}^2 
 + C\epsilon^{-1}(\| \partial_t \bchi_s \|_{1,\Omega_{\mathrm{P}}}^2 + \| \bchi_r \|_{1,\Omega_{\mathrm{P}}}^2 )  \leq
 \epsilon_1 \|\bphi_{r,h}\|_{0,\Omega_{\mathrm{P}}}^2 + C\epsilon_1^{-1} h^{-2}(\| \partial_t \bchi_s \|_{1,\Omega_{\mathrm{P}}}^2 + \| \bchi_r \|_{1,\Omega_{\mathrm{P}}}^2 ).
\end{align}
Similarly, using Cauchy--Schwarz, trace and Young's inequalities,  we obtain
\[ \mathcal{J}_3 \leq \epsilon_1\left|\bphi_{f,h}-\partial_t \bphi_{s,h}\right|_{\mathrm{BJS}}^2+C \epsilon_1^{-1}\left(\|\bchi_f\|_{1,\Omega_{\mathrm{S}}}^2+\|\partial_t \bchi_s\|_{1,\Omega_{\mathrm{P}}}^2\right) .
\]
Finally, using  Cauchy--Schwarz and  Young's inequality as well as \eqref{another}, we can derive the bound 
\begin{align}\label{me3}
  \mathcal{J}_4  & \leq C \epsilon_1^{-1}\left(\|\chi_{f p}\|_{0,\Omega_{\mathrm{S}}}^2+  \|\nabla \cdot \partial_t \bchi_s\|_{0,\Omega_{\mathrm{P}}}^2 \right) +\epsilon_1\left(\|\nabla \cdot \bphi_{f,h}\|_{0,\Omega_{\mathrm{S}}}^2+\|\phi_{p p, h}\|_{0,\Omega_{\mathrm{P}}}^2\right) + C \|\chi_\lambda\|_{-1/2,\Sigma}^2 \nonumber \\&  \leq C \epsilon_1^{-1} 
 \left(\|\chi_{f p}\|_{0,\Omega_{\mathrm{S}}}^2+\|\partial_t \bchi_s\|_{1,\Omega_{\mathrm{P}}}^2+ \| \chi_{\lambda}\|^2_{0,\Sigma}\right) +\epsilon_1\left(\|\bphi_{f,h}\|_{1,\Omega_{\mathrm{S}}}^2 
 +\|\phi_{p p, h}\|_{0,\Omega_{\mathrm{P}}}^2\right) + C\|\chi_\lambda\|_{-1/2,\Sigma}^2 .
\end{align}
Combining \eqref{ne1}-\eqref{me3}, integrating over $[0, t]$, where $0<t \leq T$, and taking $\epsilon_1$ small enough, we can assert that 
\begin{align}\label{me4}
&  \|\bphi_{ss,h}(t)\|_{0,\Omega_{\mathrm{P}}}^2+ \|\bphi_{r,h}(t)\|_{0,\Omega_{\mathrm{P}}}^2 + \|\phi_{pp,h}(t)\|_{0,\Omega_{\mathrm{P}}}^2 + \|\bphi_{s,h} (t) \|_{1,\Omega_{\mathrm{P}}} ^2 + \int_0^t\left| \bphi_{f,h} - \partial_t \bphi_{s,h}\right|^2_{\mathrm{BJS}} \nonumber \\
&\qquad + \|\bphi_{r,h}\|_{L^2(0,T;{\mathbf{L}^2(\Omega_{\mathrm{P}})})}^2   + \|  \bphi_{f,h} \|_{L^2(0,T;\mathbf{H}^1(\Omega_{\mathrm{S}}))}^2  \nonumber \\
& \leq C\epsilon_1^{-1} \left(\|\bchi_f\|_{L^2(0,T;\mathbf{H}^1\left(\Omega_{\mathrm{S}}\right))}^2  +\|\bchi_r\|_{L^2(0,T;\mathbf{H}^1\left(\Omega_{\mathrm{P}}\right))}^2  + \| \partial_t \bchi_s \|_{L^2(0,T;\mathbf{H}^1\left(\Omega_{\mathrm{P}}\right))}^2 +\|\bchi_r\|_{L^2(0,T;\mathbf{L}^2\left(\Omega_{\mathrm{P}}\right))}^2  \nonumber \right. \\ & \qquad  \left.  + \| \partial_t \bchi_s \|_{L^2(0,T;\mathbf{L}^2\left(\Omega_{\mathrm{P}}\right))}^2  + \| \partial_t \bchi_{ss} \|_{L^2(0,T;\mathbf{L}^2\left(\Omega_{\mathrm{P}}\right))}^2 +  \| \partial_t \bchi_r \|_{L^2(0,T;\mathbf{L}^2\left(\Omega_{\mathrm{P}}\right))}^2 + \|\chi_{f p}\|_{L^2(0,T;L^2\left(\Omega_{\mathrm{S}}\right))}^2  \right. \nonumber  \\
&   \qquad \left. + h^{-2} \| \partial_t \bchi_s \|_{L^2(0,T;\mathbf{H}^1\left(\Omega_{\mathrm{P}}\right))}^2 + h^{-2}\|\bchi_r\|_{L^2(0,T;\mathbf{L}^2\left(\Omega_{\mathrm{P}}\right))}^2  \right) + C\|\chi_{\lambda}\|_{L^2(0,T;H^{-1/2}(\Sigma))}^2  + \|\bphi_{ss,h}(0)\|_{0,\Omega_{\mathrm{P}}}^2   \nonumber \\
&\qquad+\epsilon_1\left(\|\bphi_{f,h}\|_{L^2(0,T;\mathbf{H}^1\left(\Omega_{\mathrm{S}}\right))}^2 +\|\bphi_{ss,h}\|_{L^2(0,T;\mathbf{L}^2\left(\Omega_{\mathrm{P}}\right))}^2   + \|\bphi_{r,h}\|_{L^2(0,T;\mathbf{L}^2\left(\Omega_{\mathrm{P}}\right))}^2 +\|\phi_{p p, h}\|_{L^2(0,T;L^2\left(\Omega_{\mathrm{P}}\right))}^2 \nonumber \right.  \\ 
& \qquad \left. + \|\bphi_{r,h}\|_{L^2(0,T;\mathbf{L}^2\left(\Omega_{\mathrm{P}}\right))}^2    \right)  + \|\bphi_{s,h} (0) \|_{1,\Omega_{\mathrm{P}}} ^2  + \|\bphi_{r,h} (0) \|_{0,\Omega_{\mathrm{P}}} ^2 + \|\phi_{pp,h}(0)\|_{L^2(\Omega)}^2  + C  \int_{0}^t\left(\mathcal{J}_5 +\mathcal{J}_6 \right) \ds. 
\end{align}
The interpolations $ \bu_{r,h}^{\mathrm{P}}(0)=\bI_{r, h} \bu_{r, 0}, \quad \bu_{s, h}^{\mathrm{P}}(0)=\bQ_{s, h} \bu_{s, 0}, \quad   p_{p, h}^{\mathrm{P}}(0)=Q_{p, h} p_{p, 0}$ and $\by_{s, h}^{\mathrm{P}}(0)=\bI_{s, h} \by_{s, 0}$, give
\begin{align}\label{me5}
\bphi_{r,h}(0)=0, \quad \bphi_{ss,h}(0)=0, \quad \bphi_{s,h}(0)=0, \quad \phi_{p p, h}(0)=0.
\end{align}
Next, we bound the terms on the right hand side involving $\partial_t \bphi_{s,h}$. By applying integration by parts in time, along with Cauchy--Schwarz and Young's inequalities, Lemma \ref{coercivity-continuity}, and equation \eqref{another}, we obtain 
\begin{align*}
\int_0^t\mathcal{J}_5  \ds &=\left.a_f^{\mathrm{P}}\left(\bchi_r, \bphi_{s,h}\right)\right|_0 ^t-\int_0^ta_f^{\mathrm{P}}\left(\partial_t \bchi_r, \bphi_{s,h}\right) \ds  + \left.a_s^{\mathrm{P}}\left(\bchi_s, \bphi_{s,h}\right)\right|_0 ^t-\int_0^ta_s^{\mathrm{P}}\left(\partial_t \bchi_s, \bphi_{s,h}\right) \ds \nonumber \\ 
& \quad + a_f^{\mathrm{P}}\left(\partial_t \bchi_s, \bphi_{s,h}\right)|_0 ^t-\int_0^ta_f^{\mathrm{P}}\left(\partial_{tt} \bchi_s, \bphi_{s,h}\right)\ds 
\nonumber  \\
& \leq C   \left(\|\partial_t \bchi_r \|_{\mathbf{L}^2\left(0, t ; H^1\left(\Omega_{\mathrm{P}}\right)\right)}^2 \nonumber  +\|\partial_t \bchi_s\|_{\mathbf{L}^2\left(0, t ; H^1\left(\Omega_{\mathrm{P}}\right)\right)}^2 + \|\partial_{tt} \bchi_s\|_{\mathbf{L}^2\left(0, t ; H^1\left(\Omega_{\mathrm{P}}\right)\right)}^2 +\|\bphi_{s,h}\|_{\mathbf{L}^2\left(0, t ; H^1\left(\Omega_{\mathrm{P}}\right)\right)}^2 \right) \\
& \quad + C \epsilon_1^{-1}\left(\|\bchi_r(t)\|_{1,\Omega_{\mathrm{P}}}^2+\|\bchi_s(t)\|_{1,\Omega_{\mathrm{P}}}^2 + \|\partial_t \bchi_s(t)\|_{1,\Omega_{\mathrm{P}}}^2 \right) +\epsilon_1\|\bphi_{s,h}(t)\|_{1,\Omega_{\mathrm{P}}}^2.
\end{align*}
Employing again Cauchy--Schwarz and Young's inequalities, we have
\begin{align}\label{me7}
\int_0^t\mathcal{J}_6 \ds&=\left.\left\langle\bphi_{s,h} \cdot \bn_{\mathrm{P}}, \chi_\lambda\right\rangle_{\Sigma}\right|_0 ^t- \int_0^t\left\langle\bphi_{s,h} \cdot \bn_{\mathrm{P}}, \partial_t \chi_\lambda\right\rangle_{\Sigma}\ds 
+\left. b_s^{\mathrm{P}}\left(\bphi_{s,h}, \chi_{p p}\right)\right|_0 ^t-\int_0^tb_s^{\mathrm{P}} \left(\bphi_{s,h}, \partial_t \chi_{p p}\right)\ds \nonumber  \\
 & \leq  C\left(\epsilon_1^{-1}\|\chi_{p p}(t)\|_{0,\Omega_{\mathrm{P}}}^2 +\|\partial_t \chi_{p p}\|_{L^2\left(0, t ; L^2\left(\Omega_{\mathrm{P}}\right)\right)}^2  \nonumber +  \|\partial_t \chi_{\lambda} \|_{L^2(0,T;H^{-1/2}(\Sigma))}^2\right)\\
 &\quad + \| \chi_{\lambda} (t) \|_{-1/2,\Sigma}^2 +\epsilon_1\|\bphi_{s,h}(t)\|_{1,\Omega_{\mathrm{P}}}^2+\| \bphi_{s,h}\|_{L^2\left(0, T ; \mathbf{H}^1\left(\Omega_{\mathrm{P}}\right)\right)}^2.
\end{align}
Choosing a sufficiently small $\epsilon_1$ in \eqref{me5}-\eqref{me7}, we derive the following bound from equation \eqref{me4}
\begin{align}\label{me8}
&  \|\bphi_{ss,h}(t)\|_{0,\Omega_{\mathrm{P}}}^2+\|\bphi_{r,h}(t)\|_{0,\Omega_{\mathrm{P}}}^2 + \|\phi_{pp,h}(t)\|_{0,\Omega_{\mathrm{P}}}^2 + \|\bphi_{s,h} (t) \|_{1,\Omega_{\mathrm{P}}} ^2 + \left| \bphi_{f,h} - \partial_t \bphi_{s,h}\right|_{L^2(0,T;\mathrm{BJS})}^2  \nonumber \\
&\quad  + \|\bphi_{r,h}\|_{L^2(0,T;{\mathbf{L}^2(\Omega_{\mathrm{P}})})}^2+ \|  \bphi_{f,h} \|_{L^2(0,T;\mathbf{H}^1(\Omega_{\mathrm{S}}))}^2 \nonumber \\
& \leq \frac{C}{\epsilon_1} \bigl(\|\bchi_f\|_{L^2(0,T;\mathbf{H}^1\left(\Omega_{\mathrm{S}}\right))}^2  + h^{-2} \|\bchi_r\|_{L^2(0,T;\mathbf{H}^1\left(\Omega_{\mathrm{P}}\right))}^2   + h^{-2} \| \partial_t \bchi_s \|_{L^2(0,T;\mathbf{H}^1\left(\Omega_{\mathrm{P}}\right))}^2 +\|\bchi_r\|_{L^2(0,T;\mathbf{L}^2\left(\Omega_{\mathrm{P}}\right))}^2   + \| \partial_t \bchi_s \|_{L^2(0,T;\mathbf{L}^2\left(\Omega_{\mathrm{P}}\right))}^2   \nonumber \\
&\quad  + \| \partial_t \bchi_{ss} \|_{L^2(0,T;\mathbf{L}^2\left(\Omega_{\mathrm{P}}\right))}^2  +  \| \partial_t \bchi_r \|_{L^2(0,T;\mathbf{L}^2\left(\Omega_{\mathrm{P}}\right))}^2  + \|\chi_{f p}\|_{L^2(0,T;L^2\left(\Omega_{\mathrm{S}}\right))}^2+\|\bchi_r(t)\|_{1,\Omega_{\mathrm{P}}}^2+\|\bchi_s(t)\|_{1,\Omega_{\mathrm{P}}}^2+\|\partial_t \bchi_s(t)\|_{1,\Omega_{\mathrm{P}}}^2 \nonumber  \\
& \quad +  \| \chi_{\lambda} (t) \|_{-1/2,\Sigma}^2 +\|\chi_{p p}(t)\|_{0,\Omega_{\mathrm{P}}}^2\bigr)  +C  \bigl(  \|\partial_t \bchi_r \|_{\mathbf{L}^2\left(0, t ; H^1\left(\Omega_{\mathrm{P}}\right)\right)}^2  + \|\partial_t \chi_{\lambda} \|_{L^2(0,T;H^{-1/2}(\Sigma))} ^2 \nonumber  \\
&  \quad +\|\partial_t \chi_{p p}\|_{L^2\left(0, t ; L^2\left(\Omega_{\mathrm{P}}\right)\right)}^2  + \|\partial_{tt}\bchi_s\|_{L^2(0,t;\mathbf{H^1}(\Omega_{\mathrm{P}}))}^2
+\|\chi_\lambda\|_{L^2(0,T;H^{-1/2}(\Sigma))}^2 +\|\bphi_{s,h}\|_{\mathbf{L}^2\left(0, t ; \mathbf{H}^1\left(\Omega_{\mathrm{P}}\right)\right)}^2 \bigr) .
\end{align}
Next, we employ the inf–sup condition \eqref{Dinf-sup2} with the choice $(q^{\mathrm{S}},q^{\mathrm{P}}, \mu_h) = \left( \phi_{fp,h}, \phi_{pp,h}, \phi_{\lambda,h}\right)$ and utilize the error equation derived by subtracting \eqref{semi1} from \eqref{mixed-primal}
\begin{align*}
&\|(\phi_{fp,h}, \phi_{pp,h}, \phi_{\lambda,h})\|_{W_{f,h} \times W_{p,h} \times \Lambda_h} \lesssim\sup _{\left(\bv_{f,h}^{\mathrm{S}}, \bv_{r,h}^{\mathrm{P}}\right) \in \mathbf{V}_{f,h} \times \mathbf{V}_{r,h}}  \frac{b^{\mathrm{S}}\left(\bv_{f,h}^{\mathrm{S}}, \phi_{fp,h}\right)+b_f^{\mathrm{P}}\left(\bv_{r,h}^{\mathrm{P}}, \phi_{pp,h}\right)+b_{\Gamma}\left(\bv_{f,h}^{\mathrm{S}}, \bv_{r,h}^{\mathrm{P}}, \mathbf{0} ; \phi_{\lambda_h}\right)}{\|(\bv_{f,h}^{\mathrm{S}}, \bv_{r,h}^{\mathrm{P}})\|_{\mathbf{V}_{f,h} \times \mathbf{V}_{r,h}}} \\
&\quad \lesssim \sup _{\left(\bv_{f,h}^{\mathrm{S}}, \bv_{r,h}^{\mathrm{P}}\right) \in \mathbf{V}_{f,h} \times \mathbf{V}_{r,h}}\left(\frac{-a_f^{\mathrm{S}}\left(\mathbf{e}_f, \bv_{f,h}^{\mathrm{S}}\right)-a_f^{\mathrm{P}}\left({\mathbf{e}}_r, \bv_{r,h}^{\mathrm{P}}\right)-a_f^{\mathrm{P}}\left(\partial_t {\mathbf{e}}_s, \bv_{r,h}^{\mathrm{P}}\right)-a_{\mathrm{BJS}}\left({\mathbf{e}}_f, \partial_t {\mathbf{e}}_s ; \bv_{f,h}^{\mathrm{S}}, \mathbf{0}\right)}{\|(\bv_{f,h}^{\mathrm{S}}, \bv_{r,h}^{\mathrm{P}})\|_{\mathbf{V}_{f,h} \times \mathbf{V}_{r,h}}}\right. \\
&\qquad \qquad \qquad \qquad\left.\frac{+m_{\theta}\left(\partial_t {\mathbf{e}}_s,\bv_{r,h}^{\mathrm{P}}  \right)+m_{\theta}\left( {\mathbf{e}}_r,\bv_{r,h}^{\mathrm{P}}  \right)-k\left( {\mathbf{e}}_r,\bv_{r,h}^{\mathrm{P}}  \right)-m_{\rho_f \phi}\left(\partial_t {\mathbf{e}}_r,\bv_{r,h}^{\mathrm{P}}  \right)} {\|(\bv_{f,h}^{\mathrm{S}}, \bv_{r,h}^{\mathrm{P}})\|_ {\mathbf{V}_{f,h} \times \mathbf{V}_{r,h}}} \right. \\
&\qquad \qquad\qquad \qquad \left. \frac{ -m_{\rho_f \phi}\left(\partial_t {\mathbf{e}}_s,\bv_{r,h}^{\mathrm{P}}  \right)-b^{\mathrm{S}}\left(\bv_{f,h}^{\mathrm{S}}, \chi_{fp,h}\right)-b_f^{\mathrm{P}}\left(\bv_{r,h}^{\mathrm{P}}, \chi_{pp,h}\right)-b_{\Gamma}\left(\bv_{f,h}^{\mathrm{S}}, \bv_{r,h}^{\mathrm{P}}, \mathbf{0} ; \chi_{\lambda_h}\right)}{\|(\bv_{f,h}^{\mathrm{S}}, \bv_{r,h}^{\mathrm{P}})\|_ {\mathbf{V}_{f,h} \times \mathbf{V}_{r,h}}}\right). 
\end{align*}
We treat the term $a_f^{\mathrm{P}}\left({\mathbf{e}}_r, \bv_{r,h}^{\mathrm{P}}\right)$ similarly as in \eqref{h_11}. Integrating over the interval $[0,T]$ and applying Lemma \ref{coercivity-continuity} along with the trace inequality, we obtain
\begin{align}\label{me9}
\nonumber &  \| \phi_{fp,h} \|_{L^2(0,T;L^2(\Omega_{\mathrm{S}}))}^2 + \| \phi_{pp,h} \|_{L^2(0,T;L^2(\Omega_{\mathrm{P}}))}^2 + \| \phi_{\lambda,h} \|_{L^2(0,T;H^{-1/2}(\Sigma))}^2 
\lesssim   \|  \bphi_{f,h} \|_{L^2(0,T;\mathbf{H}^1(\Omega_{\mathrm{S}}))}^2 + \| \bphi_{r,h} \|_{L^2(0,T;\mathbf{L}^2(\Omega_{\mathrm{P}}))}^2 \nonumber \\& \quad +  | \bphi_{f,h} - \partial_t \bphi_{s,h} |_{L^2(0,T;\mathrm{BJS})}^2 + \| \bphi_{r,h} \|_{L^2(0,T;\mathbf{L}^2(\Omega_{\mathrm{P}}))}^2   + \| \bchi_f \|_{L^2(0,T;\mathbf{H}^1(\Omega_{\mathrm{S}}))}^2   + \| \bchi_r \|_{L^2(0,T;\mathbf{H}^1(\Omega_{\mathrm{P}}))}^2  + \|\bchi_r \|_{L^2(0,T;\mathbf{L}^2(\Omega_{\mathrm{P}}))}^2  \nonumber  \\
&  \quad  + \| \chi_{fp}\|_{L^2(0,T;\mathbf{L}^2(\Omega_{\mathrm{P}}))}^2 + \| \chi_{pp}\|_{L^2(0,T;\mathbf{L}^2(\Omega_{\mathrm{P}}))}^2  + \| \chi_{\lambda}\|_{L^2(0,T;H^{-1/2}(\Sigma))}^2 
.
\end{align}
Adding \eqref{me8} and \eqref{me9}, and taking 
$\epsilon_2$ small enough, and then $\epsilon_1$ small enough, gives
\begin{align*}
&  \|\bphi_{ss,h}(t)\|_{0,\Omega_{\mathrm{P}}}^2 + \|\bphi_{r,h}(t)\|_{0,\Omega_{\mathrm{P}}}^2 + \|\phi_{pp,h}(t)\|_{L^2(\Omega)}^2 + \|\bphi_{s,h} (t) \|_{1,\Omega_{\mathrm{P}}} ^2 + \left| \bphi_{f,h} - \partial_t \bphi_{s,h}\right|_{L^2(0,T;{\mathrm{BJS}})}^2 + \|\bphi_{r,h}\|_{L^2(0,T;{\mathbf{L}^2(\Omega_{\mathrm{P}})})}^2 \nonumber \\&\quad  + \|  \bphi_{f,h} \|_{L^2(0,T;\mathbf{H}^1(\Omega_{\mathrm{S}}))}^2 + \| \phi_{fp,h} \|_{L^2(0,T;L^2(\Omega_{\mathrm{S}}))}^2 + \| \phi_{pp,h} \|_{L^2(0,T;L^2(\Omega_{\mathrm{P}}))}^2  + \| \phi_{\lambda,h} \|_{L^2(0,T;H^{-1/2}(\Sigma))}^2   \nonumber \\& \quad \leq C  \bigl(   \|\bphi_{s,h}\|_{\mathbf{L}^2\left(0, t ; H^1\left(\Omega_{\mathrm{P}}\right)\right)}^2+\|\bchi_f\|_{L^2(0,T;\mathbf{H}^1\left(\Omega_{\mathrm{S}}\right))}^2   + h^{-2} \|\bchi_r\|_{L^2(0,T;\mathbf{H}^1\left(\Omega_{\mathrm{P}}\right))}^2  + h^{-2} \| \partial_t \bchi_s\|_{L^2(0,T;\mathbf{H}^1\left(\Omega_{\mathrm{P}}\right))}^2\nonumber \\
&\quad    +\|\bchi_r\|_{L^2(0,T;\mathbf{L}^2\left(\Omega_{\mathrm{P}}\right))}^2   + \| \partial_t \bchi_s \|_{L^2(0,T;\mathbf{L}^2\left(\Omega_{\mathrm{P}}\right))}^2  + \| \partial_t \bchi_{ss} \|_{L^2(0,T;\mathbf{L}^2\left(\Omega_{\mathrm{P}}\right))}^2 +  \| \partial_t \bchi_r \|_{L^2(0,T;\mathbf{L}^2\left(\Omega_{\mathrm{P}}\right))}^2 + \|\chi_{f p}\|_{L^2(0,T;L^2\left(\Omega_{\mathrm{S}}\right))}^2  \nonumber  \\& \quad   +\|\bchi_r(t)\|_{1,\Omega_{\mathrm{P}}}^2  +\|\bchi_s(t)\|_{1,\Omega_{\mathrm{P}}}^2 +  \| \chi_{\lambda} (t) \|_{-1/2,\Sigma}   +\|\chi_{p p}(t)\|_{0,\Omega_{\mathrm{P}}}^2 +\|\partial_t \bchi_s(t)\|_{1,\Omega_{\mathrm{P}}}^2+ \|\partial_t \bchi_r \|_{\mathbf{L}^2\left(0, t ; H^1\left(\Omega_{\mathrm{P}}\right)\right)}^2   \nonumber \\& \quad      +  \|\partial_t \chi_{\lambda} \|_{L^2(0,T;H^{-1/2}(\Sigma))}  +\|\partial_t \chi_{p p}\|_{L^2\left(0, t ; L^2\left(\Omega_{\mathrm{P}}\right)\right)}^2 + \| \partial_{tt} \bchi_s \|_{L^2(0,t;\mathbf{H}^1(\Omega_{\mathrm{P}}))} +  \| \chi_{pp}\|_{L^2(0,T;\mathbf{L}^2(\Omega_{\mathrm{P}}))}^2 
 +\|\chi_\lambda\|_{L^2(0,T;H^{-1/2}(\Sigma))}^2 \bigr).
\end{align*}
Gr\"onwall’s and triangle inequalities alongside the approximation properties \eqref{ps}-\eqref{vh}, \eqref{lambdah}, and \eqref{ea1}-\eqref{ea3}, imply the following result.

\begin{theorem}\label{semidiscreteerror}
Assuming \ref{(H1)}-\ref{(H3)} and  sufficient smoothness for the solution of \eqref{mixed-primal},  the solution of \eqref{semi1}-\eqref{semi4} with $\bu_{r,h}^{\mathrm{P}}(0)=\bI_{r, h} \bu_{r,h}^{\mathrm{P}}$,  $\by_{s, h}^{\mathrm{P}}(0)=\bI_{s, h} \by_{s, 0}$,  $p_{h}^{\mathrm{P}}(0)=Q_{r, h} p_{p, 0}$ and $\bu_{s, h}^{\mathrm{P}}(0)=\bQ_{s, h} \bu_{s, 0}$ satisfies
\begin{align*}
&  \|\bu_{s}^{\mathrm{P}} - \bu_{s,h}^{\mathrm{P}}\|_{L^{\infty}(0,T;\mathbf{L}^2(\Omega_{\mathrm{P}}))} + \|p^{\mathrm{P}}-p^{\mathrm{P}}_h\|_{L^{\infty}(0,T;L^2(\Omega))} + \|\by_s^{\mathrm{P}}-\by_{s,h}^{\mathrm{P}}\|_{L^{\infty}(0,T;\mathbf{H}^1(\Omega_{\mathrm{P}}))} +  \|\bu_r^{\mathrm{P}}-\bu_{r,h}^{\mathrm{P}}\|_{L^{\infty}(0,T;\mathbf{L}^2(\Omega_{\mathrm{P}}))} \\ 
& \quad + \|\bu_f^{\mathrm{S}}- \bu^{\mathrm{S}}_{f,h} \|_{L^2(0,T;\mathbf{H}^1(\Omega_{\mathrm{S}}))} + \|\bu_r^{\mathrm{P}}-\bu_{r,h}^{\mathrm{P}}\|_{L^2(0,T;{\mathbf{L}^2(\Omega_{\mathrm{P}})})} + \left| \left(\bu_f^{\mathrm{S}}- \partial_t \by_s^{\mathrm{P}} \right) - \left(\bu^{\mathrm{S}}_{f,h} - \partial_t \by_{s,h}^{\mathrm{P}} \right)\right|_{L^2(0,T;\mathrm{BJS})} \\
& \quad + \| p^{\mathrm{S}} -  p^{\mathrm{S}}_h \|_{L^2(0,T;L^2(\Omega_{\mathrm{S}}))} + \| p^{\mathrm{P}} - p^{\mathrm{P}}_h \|_{L^2(0,T;L^2(\Omega_{\mathrm{P}}))} + \| \lambda - \lambda_h \|_{L^2(0,T;H^{-1/2}(\Sigma))} \\
& \leq C \sqrt{\exp(T)} \bigg[ h^{r_{k_f}}  \|\bu_f^{\mathrm{S}}\|_{L^2(0,T;\mathbf{H}^{r_{k_f} +1}\left(\Omega_{\mathrm{S}}\right))}  + h^{r_{s_f}+1} \|p^{\mathrm{S}}\|_{L^2(0,T;H^{r_{s_f}+1}\left(\Omega_{\mathrm{S}}\right))} \\
&\quad + h^{r_{k_p}-1}\left(\|\bu_r^{\mathrm{P}}\|_{L^2(0,T;\mathbf{H}^{r_{k_p}+1}\left(\Omega_{\mathrm{P}}\right))} + \|\bu_r^{\mathrm{P}}\|_{L^{\infty}(0,T;\mathbf{H}^{r_{k_p}+1}\left(\Omega_{\mathrm{P}}\right))} + \|\partial_t \bu_r^{\mathrm{P}}\|_{L^2(0,T;\mathbf{H}^{r_{k_p}+1}\left(\Omega_{\mathrm{P}}\right))}   \right)\\
&\quad + h^{r_{s_p}+1} \left( \|p^{\mathrm{P}}\|_{L^{\infty}(0,T;H^{r_{s_p}+1}\left(\Omega_{\mathrm{P}}\right))} +\|p^{\mathrm{P}}\|_{L^2(0,T;H^{r_{s_p}+1}\left(\Omega_{\mathrm{P}}\right))} +\|\partial_t p^{\mathrm{P}} \|_{L^2\left(0, T ; H^{r_{s_p}+1 }\left(\Omega_{\mathrm{P}}\right)\right)} \right) \\
&\quad + h^{r_{k_s}-1}  \left( \|\by_s^{\mathrm{P}}\|_{L^{\infty}(0,T;\mathbf{H}^{r_{k_s}+1}\left(\Omega_{\mathrm{P}}\right))}  +\|\by_s^{\mathrm{P}}\|_{L^{2}(0,T;\mathbf{H}^{r_{k_s}+1}\left(\Omega_{\mathrm{P}}\right))}+\|\partial_t\by_s^{\mathrm{P}}\|_{L^2(0,T;\mathbf{H}^{r_{k_s}+1}\left(\Omega_{\mathrm{P}}\right))} +\|\partial_t\by_s^{\mathrm{P}}\|_{L^{\infty}(0,T;\mathbf{H}^{r_{k_s}+1}\left(\Omega_{\mathrm{P}}\right))} \right. \\
& \qquad \qquad \qquad \left. +\|\partial_{tt}\by_s^{\mathrm{P}}\|_{L^{2}(0,T;\mathbf{H}^{r_{k_s}+1}\left(\Omega_{\mathrm{P}}\right))}\right) + h^{r_{s_s}+1} \left( \|\bu_s^{\mathrm{P}}\|_{L^{2}(0,T;\mathbf{H}^{r_{s_s}+1}\left(\Omega_{\mathrm{P}}\right))} +\|\partial_t\bu_s^{\mathrm{P}}\|_{L^2(0,T;\mathbf{H}^{r_{s_s}+1}\left(\Omega_{\mathrm{P}}\right))} \right) \\
& \quad + h^{{r}_{\tilde{k}_p}+\frac{1}{2}} \left(\|\lambda\|_{L^2(0,T;H^{{r}_{\tilde{k}_p}}\left(\Sigma \right))} + 
\|\lambda\|_{L^{\infty}(0,T;{H}^{{r}_{\tilde{k}_p}}\left(\Sigma \right))} + \|\partial_t \lambda\|_{L^2(0,T;{H}^{{r}_{\tilde{k}_p}}\left(\Sigma \right))}\right)\bigg], 
\end{align*}
where $0 \leq r_{k_f} \leq k_f$, $0 \leq r_{s_f} \leq s_f$, $1 \leq r_{k_p}  \leq k_p$, $0 \leq r_{s_p} \leq s_p$,  
$ 1 \leq r_{k_s} \leq k_s$,  $0 \leq r_{s_s} \leq s_s$, $-1/2 \leq {r}_{\tilde{k}_p} \leq \tilde{k}_p -1/2$.
\end{theorem}
\begin{remark}
  We observe that both relative velocity and solid displacement exhibit sub-optimal convergence. The solid and relative velocity blocks in the diagonal of the system matrix makes it difficult to  derive a bound in the energy norm.
\end{remark}

\section{Fully discrete formulation}\label{section7}
\subsection{Definition and unique solvability}
For the time discretization we employ the backward Euler method with constant time-step $\tau$, $T = N \tau$, and let $t_n = n\tau$, $0 \leq n \leq N$. Let $d_{\tau} u^n := \tau^{-1}(u^n - u^{n-1})$ be the
first order (backward) discrete time derivative, where $u^n \approx u(t_n)$. The fully discrete problem reads: given $\bu_{r,h}^0=\bu_{r,h}^{\mathrm{P}}(0), \bv_{s,h}^0=\bv_{s,h}^{\mathrm{P}}(0), \by_{s,h}^0=\by_{s,h}^{\mathrm{P}}(0), \text{ and } p^{\mathrm{P},0}_h={p}_{h}^{\mathrm{P}}(0)$, find $\bu_{f, h}^{S,n} \in \mathbf{V}_{f, h}, p^{\mathrm{S}, n}_h \in W_{f,h},  \bu_{r,h}^{\mathrm{P},n} \in \mathbf{V}_{r, h} , p^{\mathrm{P},n}_h \in W_{r,h}, \by_{s, h}^{\mathrm{P},n} \in \mathbf{V}_{s,h}, \bu_{s,h}^{\mathrm{P},n} \in \mathbf{W}_{s, h}, \lambda_h^n \in\Lambda_h,$   such that for $1 \leq n \leq N $, there holds 
\begin{subequations}
\begin{align}
& a_f^{\mathrm{S}}(\bu^{\mathrm{S},n}_{f,h}, \bv_{f,h}^{\mathrm{S}})+a_{f}^{\mathrm{P}}(\bu_{r,h}^{\mathrm{P},n}, \bw_{s,h}^{\mathrm{P}})+a_s^{\mathrm{P}}(\by_{s,h}^{\mathrm{P},n}, \bw_{s,h}^{\mathrm{P}})+ a_{f}^{\mathrm{P}}(\bu_{r,h}^{\mathrm{P},n}, \bv_{r,h}^{\mathrm{P}})+ a_{f}^{\mathrm{P}}(d_{\tau} \by_{s,h}^{\mathrm{P},n}, \bv_{r,h}^{\mathrm{P}}) + a_{f}^{\mathrm{P}}(d_{\tau} \by_{s,h}^{\mathrm{P},n}, \bw_{s,h}^{\mathrm{P}})\nonumber \\&\quad +a_{\mathrm{BJS}}(\bu^{\mathrm{S},n}_{f,h}, d_{\tau} \by_{s,h}^{\mathrm{P},n} ; \bv_{f,h}^{\mathrm{S}}, \bw_{s,h}^{\mathrm{P}}) 
 +b^{\mathrm{S}}(\bv_{f,h}^{\mathrm{S}}, p^{\mathrm{S},n}_h)+b_s^{\mathrm{P}}(\bw_{s,h}^{\mathrm{P}}, p^{\mathrm{P},n}_h) + b_f^{\mathrm{P}}(\bv_{r,h}^{\mathrm{P}}, (p^{\mathrm{P}}_h)^n) +b_{\Gamma}(\bv_{f,h}^{\mathrm{S}}, \bv_{r,h}^{\mathrm{P}}, \bw_{s,h}^{\mathrm{P}} ; \lambda_h^n ) \nonumber \\
& \quad  - m_{\theta}(\bu_{r,h}^{\mathrm{P},n}, \bw_{s,h}^{\mathrm{P}}) - m_{\theta}(d_{\tau} \by_{s,h}^{\mathrm{P},n}, \bw_{s,h}^{\mathrm{P}}) - m_{\theta}(\bu_{r,h}^{\mathrm{P},n}, \bv_{r,h}^{\mathrm{P}})   - m_{\theta}(d_{\tau} \by_{s,h}^{\mathrm{P},n}, \bv_{r,h}^{\mathrm{P}})  + m_{\phi^2/\kappa}(\bu_{r,h}^{\mathrm{P},n}, \bv_{r,h}^{\mathrm{P}})  \nonumber \\ & \quad+ m_{\rho_f \phi}(d_{\tau}\bu_{r,h}^{\mathrm{P},n}, \bw_{s,h}^{\mathrm{P}}) + m_{\rho_p}(d_{\tau}\bu_{s,h}^{\mathrm{P},n}, \bw_{s,h}^{\mathrm{P}}) + m_{\rho_f \phi }(d_{\tau}\bu_{r,h}^{\mathrm{P},n}, \bv_{r,h}^{\mathrm{P}})  + m_{\rho_f \phi}(d_{\tau}\bu_{s,h}^{\mathrm{P},n}, \bv_{r,h}^{\mathrm{P}})\nonumber \\& =\langle\ff_{\mathrm{S}}^n,\bv_{f,h}^{\mathrm{S}}\rangle_{\Omega_{\mathrm{P}}}  +(\rho_p \ff_{\mathrm{P}}^n, \bw_{s,h}^{\mathrm{P}})_{\Omega_{\mathrm{P}}} +(\rho_f \phi  \ff_{\mathrm{P}}^n, \bv_{r,h}^{\mathrm{P}})_{\Omega_{\mathrm{P}}}, \label{fully1}
\\ & ((1-\phi)^2 K^{-1} d_{\tau} p^{\mathrm{P},n}_h, q^{\mathrm{P}}_h)_{\Omega_{\mathrm{P}}}-b_s^{\mathrm{P}}(d_{\tau} \by_{s,h}^{\mathrm{P},n}, q^{\mathrm{P}}_h)-b_f^{\mathrm{P}}(\bu_{r,h}^{\mathrm{P},n}, q^{\mathrm{P}}_h)-b^{\mathrm{S}}(\bu^{\mathrm{S},n}_{f,h}, q^{\mathrm{S}}_h) =(r_{\mathrm{S}}^n, q^{\mathrm{S}}_h)_{\Omega_{\mathrm{S}}}+(\rho_f^{-1} \theta^n, q^{\mathrm{P}}_h)_{\Omega_{\mathrm{P}}}, \label{fully2}\\
& b_{\Gamma}(\bu^{\mathrm{S},n}_{f,h}, \bu_{r,h}^{\mathrm{P},n}, \partial_t \by_{s,h}^{\mathrm{P},n} ; \mu_h)=0, \label{fully3} \\
& -m_{\rho_p}(\partial_t\by_{s,h}^{\mathrm{P},n}, \bv_{s,h}^{\mathrm{P}}) + m_{\rho_p}(\bu_{s,h}^{\mathrm{P},n}, \bv_{s,h}^{\mathrm{P}}) =0 \label{fully4},
\end{align}
\end{subequations}
for all  $\bv_{f,h}^{\mathrm{S}} \in \mathbf{V}_{f, h}, q^{\mathrm{S}}_h \in W_{f,h},  \bv_{r,h}^{\mathrm{P}} \in \mathbf{V}_{r, h} , q^{\mathrm{P}}_h \in W_{r,h}, \bw_{s,h}^{\mathrm{P}} \in \mathbf{V}_{s,h}, \bv_{s,h}^{\mathrm{P}} \in \mathbf{W}_{s, h}, \mu_h \in\Lambda_h  $.
The method requires solving at each time step the algebraic system
\begin{align}\label{discrete_matrix}
    \mathcal{L} \mathbf{X} = \tilde{\mathcal{F}},
\end{align}
where the matrix $\mathcal{L}$ is the sum of the matrices $\mathbf{E}$ \eqref{E_matrix} and $\mathbf{H}$ \eqref{H_matrix}, with $\mathbf{E}$ scaled by $\frac{1}{\tau }$. The tilde notation on the right-hand side vectors signifies that they include contributions from the backward Euler time discretization.

\begin{theorem}
    The fully discrete method \eqref{discrete_matrix} has a unique solution under the assumptions \ref{(H1)}-\ref{(H3)}.
\end{theorem}
\begin{proof}
    The aim is to show that the matrix $\frac{1}{\tau }\mathbf{E}+{\mathbf{H}}$ is invertible. 
    Now, we proceed directly from the bilinear forms testing with $(\bu^{\mathrm{S},n}_{f,h},p^{\mathrm{S},n}_h, \bu_{r,h}^{\mathrm{P},n}, p^{\mathrm{P},n}_h, \frac{1}{\tau  }\by_{s,h}^{\mathrm{P},n}, \bu_{s,h}^{\mathrm{P},n}, \lambda_{h}^n)$, we get 
    \begin{align*}
    \frac{1}{\tau }\mathbf{E}+{\mathbf{H}}  &= a_f^{\mathrm{S}}(\bu^{\mathrm{S},n}_{f,h}, \bu^{\mathrm{S},n}_{f,h})+ \frac{1}{\tau } a_{f}^{\mathrm{P}}(\bu_{r,h}^{\mathrm{P},n}, \by_{s,h}^{\mathrm{P},n})+ \frac{1}{\tau } a_s^{\mathrm{P}}(\by_{s,h}^{\mathrm{P},n}, \by_{s,h}^{\mathrm{P},n}) + a_{f}^{\mathrm{P}}(\bu_{r,h}^{\mathrm{P},n},\bu_{r,h}^{\mathrm{P},n}) + \frac{1}{\tau } a_{f}^{\mathrm{P}}(\by_{s,h}^{\mathrm{P},n}, \bu_{r,h}^{\mathrm{P},n}) \\& + \frac{1}{\tau^2} a_{f}^{\mathrm{P}}(\by_{s,h}^{\mathrm{P},n}, \by_{s,h}^{\mathrm{P},n}) + \frac{1}{\tau^2} a_{\mathrm{BJS}}(\bu^{\mathrm{S},n}_{f,h},  \by_{s,h}^{\mathrm{P},n} ; \bu^{\mathrm{S},n}_{f,h}, \by_{s,h}^{\mathrm{P},n}) 
  - \frac{1}{\tau } m_{\theta}(\bu_{r,h}^{\mathrm{P},n}, \by_{s,h}^{\mathrm{P},n}) - \frac{1}{\tau^2} m_{\theta}( \by_{s,h}^{\mathrm{P},n}, \by_{s,h}^{\mathrm{P},n})\\&  - m_{\theta}(\bu_{r,h}^{\mathrm{P},n}, \bu_{r,h}^{\mathrm{P},n}) - \frac{1}{\tau } m_{\theta}( \by_{s,h}^{\mathrm{P},n}, \bu_{r,h}^{\mathrm{P},n})  + m_{\phi^2/\kappa}(\bu_{r,h}^{\mathrm{P},n}, \bu_{r,h}^{\mathrm{P},n}) +\frac{1}{\tau^2} m_{\rho_f \phi}(\bu_{r,h}^{\mathrm{P},n}, \by_{s,h}^{\mathrm{P},n}) + \frac{1}{\tau^2} m_{\rho_p}(\bu_{s,h}^{\mathrm{P},n}, \by_{s,h}^{\mathrm{P},n}) \\&+\frac{1}{\tau } m_{\rho_f \phi }(\bu_{r,h}^{\mathrm{P},n}, \bu_{r,h}^{\mathrm{P},n})   +\frac{1}{\tau } m_{\rho_f \phi}(\bu_{s,h}^{\mathrm{P},n}, \bu_{r,h}^{\mathrm{P},n}) + \frac{1}{\tau }
((1-\phi)^2 K^{-1} p^{\mathrm{P},n}_h, p^{\mathrm{P}}_h)_{\Omega_{\mathrm{P}}}  
\qquad \\& =a_f^{\mathrm{S}}(\bu^{\mathrm{S},n}_{f,h}, \bu^{\mathrm{S},n}_{f,h})+ a_{f}^{\mathrm{P}}(\bu_{r,h}^{\mathrm{P},n}, \bu_{r,h}^{\mathrm{P},n})  + \frac{1}{\tau^2} a_{\mathrm{BJS}}(\bu^{\mathrm{S},n}_{f,h},  \by_{s,h}^{\mathrm{P},n} ; \bu^{\mathrm{S},n}_{f,h}, \by_{s,h}^{\mathrm{P},n}) - \frac{1}{\tau^2} m_{\theta}( \by_{s,h}^{\mathrm{P},n}, \by_{s,h}^{\mathrm{P},n}) \\& - m_{\theta}(\bu_{r,h}^{\mathrm{P},n}, \bu_{r,h}^{\mathrm{P},n}) + m_{\phi^2/\kappa}(\bu_{r,h}^{\mathrm{P},n}, \bu_{r,h}^{\mathrm{P},n}) +\frac{1}{\tau } m_{\rho_f \phi }(\bu_{r,h}^{\mathrm{P},n}, \bu_{r,h}^{\mathrm{P},n}) + \frac{1}{\tau }\left((1-\phi)^2 K^{-1} p^{\mathrm{P},n}_h, p^{\mathrm{P}}_h\right)_{\Omega_{\mathrm{P}}} \\& + \frac{1}{\tau } a_s^{\mathrm{P}}(\by_{s,h}^{\mathrm{P},n}, \by_{s,h}^{\mathrm{P},n}) + \frac{1}{\tau^2} m_{\rho_p}(\bu_{s,h}^{\mathrm{P},n}, \by_{s,h}^{\mathrm{P},n})+ \frac{1}{\tau } a_{f}^{\mathrm{P}}(\bu_{r,h}^{\mathrm{P},n}, \by_{s,h}^{\mathrm{P},n})+ \frac{1}{\tau } a_{f}^{\mathrm{P}}(\by_{s,h}^{\mathrm{P},n}, \bu_{r,h}^{\mathrm{P},n}) \\&- \frac{1}{\tau } m_{\theta}(\bu_{r,h}^{\mathrm{P},n}, \by_{s,h}^{\mathrm{P},n}) - \frac{1}{\tau } m_{\theta}( \by_{s,h}^{\mathrm{P},n}, \bu_{r,h}^{\mathrm{P},n}) +\frac{1}{\tau^2} m_{\rho_f \phi}(\bu_{r,h}^{\mathrm{P},n}, \by_{s,h}^{\mathrm{P},n}) +\frac{1}{\tau } m_{\rho_f \phi}(\bu_{s,h}^{\mathrm{P},n}, \bu_{r,h}^{\mathrm{P},n}).
    \end{align*}
We bound some of the terms above using the inequality \eqref{ineq:aux}. This gives 
\begin{align*}
& \frac{1}{\tau } a_{f}^{\mathrm{P}}(\bu_{r,h}^{\mathrm{P},n}, \by_{s,h}^{\mathrm{P},n})+ \frac{1}{\tau } a_{f}^{\mathrm{P}}(\by_{s,h}^{\mathrm{P},n}, \bu_{r,h}^{\mathrm{P},n}) - \frac{1}{\tau } m_{\theta}(\bu_{r,h}^{\mathrm{P},n}, \by_{s,h}^{\mathrm{P},n}) - \frac{1}{\tau } m_{\theta}( \by_{s,h}^{\mathrm{P},n}, \bu_{r,h}^{\mathrm{P},n}) +\frac{1}{\tau^2} m_{\rho_f \phi}(\bu_{r,h}^{\mathrm{P},n}, \by_{s,h}^{\mathrm{P},n})\\& +\frac{1}{\tau } m_{\rho_f \phi}(\bu_{s,h}^{\mathrm{P},n}, \bu_{r,h}^{\mathrm{P},n}) \geq - \frac{1}{\tau } a_{f}^{\mathrm{P}}\left(\bu_{r,h}^{\mathrm{P},n}, \bu_{r,h}^{\mathrm{P},n}\right) - \frac{1}{\tau } a_{f}^{\mathrm{P}}(\by_{s,h}^{\mathrm{P},n}, \by_{s,h}^{\mathrm{P},n}) + \frac{1}{\tau } m_{\theta}(\bu_{r,h}^{\mathrm{P},n}, \bu_{r,h}^{\mathrm{P},n}) + \frac{1}{\tau } m_{\theta}(\by_{s,h}^{\mathrm{P},n}, \by_{s,h}^{\mathrm{P},n}) \\ & - \frac{1}{\tau} m_{\rho_f \phi}(\bu_{r,h}^{\mathrm{P},n}, \bu_{r,h}^{\mathrm{P},n}) - \frac{1}{\tau} m_{\rho_f \phi}(\bu_{s,h}^{\mathrm{P},n}, \bu_{s,h}^{\mathrm{P},n}) + \frac{1}{\tau } m_{\rho_p}(\bu_{s,h}^{\mathrm{P},n}, \bu_{s,h}^{\mathrm{P},n}). 
\end{align*}
And by combining both estimates above, we arrive at 
\begin{align*}
& \frac{1}{\tau }\mathbf{E}+{\mathbf{H}}  \\
&\quad  \gtrsim  \| \bu^{\mathrm{S},n}_{f,h}\|^2_{1,\Omega_{\mathrm{P}}}+  \| \bu_{r,h}^{\mathrm{P},n}\|^2_{1,\Omega_{\mathrm{P}}} +  \| \by_{s,h}^{\mathrm{P},n}\|^2_{1,\Omega_{\mathrm{P}}} +  |\bu^{\mathrm{S},n}_{f,h}-\by_{s,h}^{\mathrm{P},n}|^2_{\mathrm{BJS}} +  \| \by_{s,h}^{\mathrm{P},n}\|^2_{0,\Omega_{\mathrm{P}}} +  \| \bu_{r,h}^{\mathrm{P},n}\|^2_{0,\Omega_{\mathrm{P}}} +  \| \bu_{s,h}^{\mathrm{P},n}\|^2_{0,\Omega_{\mathrm{P}}} + \|p^{\mathrm{P},n}_h\|^2_{0,\Omega_{\mathrm{P}}}.
\end{align*}
It is clear that all coefficients that form the terms on the right-hand side are positive. Hence, the matrix on the left-hand side is coercive and consequently the matrix obtained from the system \eqref{discrete_matrix} is non-singular. The proof of uniqueness  follows from the semi-discrete case since the system \eqref{discrete_matrix} is a scaled matrix.
\end{proof}

\subsection{Stability analysis of the fully discrete scheme}
In this lemma, we discuss the stability analysis of fully discrete problem.
We will make use of the discrete space–time
norms 
$$
\| \phi\|_{l^2(0,T;X)}^2 := \tau \sum_{n=1}^N \| \phi^n \|_X^2, \quad \quad \| \phi\|_{l^{\infty}(0,T;X)}^2 := \max_{0 \leq n \leq N}\| \phi^n \|_X^2, \quad \quad \left| \bphi\right|_{l^2(0,T;{\mathrm{BJS}})} = \tau \sum_{n=1}^N \left| \bphi\right|^2_{\mathrm{BJS}}.
$$
\begin{lemma}
Under assumptions~\ref{(H1)}-\ref{(H3)}, the fully discrete solution to \eqref{fully1}-\eqref{fully4} satisfies 
\begin{align*}
& \|\bu_{s,h}^{\mathrm{P}}\|_{l^{\infty}\left(0, T ; \mathbf{L}^2(\Omega_{\mathrm{P}})\right)}+\|p^{\mathrm{P}}_h\|_{l^{\infty}(0, T ; L^2(\Omega_{\mathrm{P}}))} + \left| \bu^{\mathrm{S}}_{f,h} - d_{\tau} \by_{s,h}^{\mathrm{P}}\right|_{l^2(0,T;{\mathrm{BJS}})}  +\|\by_{s,h}^{\mathrm{P}}\|_{l^{\infty}\left(0, T ; \mathbf{H}^1(\Omega_{\mathrm{P}})\right)} +\|\bu^{\mathrm{S}}_{f,h}\|_{l^2\left(0, T ; \mathbf{H}^1(\Omega_{\mathrm{S}})\right)} \\
&\quad  + \|\bu_{r,h}^{\mathrm{P}}\|_{l^2\left(0, T ; {L}^2(\Omega_{\mathrm{P}})\right)}  +  \|\bu_{s,h}^{\mathrm{P}} \|_{l^2(0,T;\mathbf{L}^2(\Omega_{\mathrm{P}}))}  + \|p^{\mathrm{S}}_h\|_{l^2(0,T; L^2(\Omega_{\mathrm{S}}))} +\|p^{\mathrm{P}}_h\|_{l^2\left(0,T; L^2(\Omega_{\mathrm{P}})\right)}    + \|\lambda_h\|_{l^2\left(0,T; \Lambda_h\right)}    \\& 
\leq \hat{C} \left( \|\ff_{\mathrm{P}}\|_{l^2(0,T;\mathbf{L}^2(\Omega_{\mathrm{P}}))}+ \|\ff_{\mathrm{P}}\|_{l^2(0,T;\mathbf{H}^{-1}(\Omega_{\mathrm{P}}))} + \|\ff_{\mathrm{S}}\|_{l^2(0,T;\mathbf{H}^{-1}(\Omega_{\mathrm{S}}))} 
 + \|\theta\|_{l^2(0,T;L^2(\Omega_{\mathrm{P}}))} + \|r_{\mathrm{S}}\|_{l^2(0,T;L^2(\Omega_{\mathrm{S}}))} \right. \\& \quad \left. +\|\bu_{s,h}^{\mathrm{P}}(0)\|_{0,\Omega_{\mathrm{P}}} +\|p^{\mathrm{P}}_h(0)\|_{0,\Omega_{\mathrm{P}}}^2 +\|\by_{s,h}^{\mathrm{P}}(0)\|_{1,\Omega_{\mathrm{P}}} \right),
\end{align*}
where 
$\hat{C}(K,\kappa,\rho_f,\rho_s,\lambda_p, \mu_f,\mu_p,\phi, \alpha_{\mathrm{BJS}}, C_T, C_I, C_K)$ is a positive constant.  
\end{lemma}
\begin{proof}
    We choose $ (\bv_{f,h}^{\mathrm{S}}, q^{\mathrm{S}}_h, \bv_{r,h}^{\mathrm{P}}, q^{\mathrm{P}}_h, \bw_{s,h}^{\mathrm{P}}, \mu_h ) = (\bu^{\mathrm{S},n}_{f,h}, p^{\mathrm{S},n}_h, \bu_{r,h}^{\mathrm{P},n}, p^{\mathrm{P},n}_h, d_{\tau} \by_{s,h}^{\mathrm{P},n}, \lambda_h^n )$ in \eqref{fully1}-\eqref{fully4} and using the  \eqref{ineq:aux}, we have 
    \begin{align*}
   & a_f^{\mathrm{S}}(\bu^{\mathrm{S},n}_{f,h}, \bu_{f,h}^{\mathrm{S},n} )+a_s^{\mathrm{P}}(\by_{s,h}^{\mathrm{P},n}, d_{\tau}\by_{s,h}^{\mathrm{P},n}) + \left((1-\phi)^2 K^{-1} d_{\tau} p^{\mathrm{P},n}_h, p^{\mathrm{P},n}_h\right)_{\Omega_{\mathrm{P}}}+  m_{\rho_s (1-\phi)}(d_{\tau}\bu_{s,h}^{\mathrm{P},n}, d_{\tau} \by_{s,h}^{\mathrm{P},n}) \\
   &\quad +m_{\rho_f \phi}\left(d_{\tau}(\bu_{r,h}^{\mathrm{P},n} + \bu_{s,h}^{\mathrm{P},n}), (\bu_{r,h}^{\mathrm{P},n} + \bu_{s,h}^{\mathrm{P},n})\right)  +a_{\mathrm{BJS}}(\bu^{\mathrm{S},n}_{f,h}, d_{\tau}  \by_{s,h}^{\mathrm{P},n} ; \bu^{\mathrm{S},n}_{f,h}, d_{\tau} \by_{s,h}^{\mathrm{P},n}) 
   + m_{\phi^2/\kappa}(\bu_{r,h}^{\mathrm{P},n}, \bu_{r,h}^{\mathrm{P},n}) \\ 
   & = \langle\ff_{\mathrm{S}}^n,\bu^{\mathrm{S},n}_{f,h}\rangle_{\Omega_{\mathrm{P}}} +(\rho_p \ff_{\mathrm{P}}^n, d_{\tau} \by_{s,h}^{\mathrm{P},n})_{\Omega_{\mathrm{P}}} +(\rho_f \phi  \ff_{\mathrm{P}}^n, \bu_{r,h}^{\mathrm{P},n})_{\Omega_{\mathrm{P}}}+ ( r_{\mathrm{S}}^n, p^{\mathrm{S},n}_h)_{\Omega_{\mathrm{S}}}+(\rho_f^{-1} \theta^n, p^{\mathrm{P},n}_h)_{\Omega_{\mathrm{P}}}.   
    \end{align*}
    Now, as a consequence of the following identity 
    \begin{equation}\label{discrete_identity}
    \int_{\Omega_{\mathrm{P}}} {\Upsilon}^n d_{\tau} \Upsilon^n = \frac{1}{2} d_{\tau} \| \Upsilon^n \|^2_{0,\Omega_{\mathrm{P}}} + \frac{1}{2} \tau \| d_{\tau} \Upsilon^n \|^2_{0,\Omega_{\mathrm{P}}},
    \end{equation}
we readily  obtain the energy inequality 
   \begin{align*}
&  \frac{d_{\tau}}{2}  \bigl( \rho_s (1-\phi) \| \bu_{s,h}^{\mathrm{P},n}\|_{0,\Omega_{\mathrm{P}}}^2 + {(1-\phi)^2}{K}^{-1} \|p^{\mathrm{P},n}_h\|^2_{0,\Omega_{\mathrm{P}}} +  2 \mu_p \|\beps(\by_{s,h}^{\mathrm{P},n})\|^2_{0,\Omega_{\mathrm{P}}} + \lambda_p \|\nabla \cdot \by_{s,h}^{\mathrm{P},n}\|^2_{0,\Omega_{\mathrm{P}}} +\rho_f \phi \|(\bu_{s,h}^{\mathrm{P},n}+\bu_{r,h}^{\mathrm{P},n})\|^2_{0,\Omega_{\mathrm{P}}} \bigr) \\
& \quad  + \frac{\tau}{2} \left( \rho_s (1-\phi) \| d_{\tau}\bu_{s,h}^{\mathrm{P},n}\|_{0,\Omega_{\mathrm{P}}}  +(1-\phi)^2{K}^{-1} \|d_{\tau} p^{\mathrm{P},n}_h\|_{0,\Omega_{\mathrm{P}}} +   2 \mu_p\| d_{\tau} \beps( \by_{s,h}^{\mathrm{P},n})\|_{0,\Omega_{\mathrm{P}}} + \lambda_p \|d_{\tau} \nabla \cdot \by_{s,h}^{\mathrm{P},n}\|_{0,\Omega_{\mathrm{P}}}  \right. \\
& \quad \left.+ \rho_f \phi \|\bu_{s,h}^{\mathrm{P},n}+\bu_{r,h}^{\mathrm{P},n}\|_{0,\Omega_{\mathrm{P}}}^2 \right) + \left| \bu^{\mathrm{S},n}_{f,h} - d_{\tau} \by_{s,h}^{\mathrm{P},n}\right|_{\mathrm{BJS}} + \left({\phi^2}{\kappa}^{-1} \bu_{r,h}^{\mathrm{P},n}, \bu_{r,h}^{\mathrm{P},n} \right)_{\Omega_{\mathrm{P}}}  +  2 \mu_f  (\beps(\bu^{\mathrm{S},n}_{f,h}), \beps(\bu^{\mathrm{S},n}_{f,h}) )_{\Omega_{\mathrm{S}}} \\
& \leq \langle \ff_{\mathrm{S}}(t_n), \bu^{\mathrm{S},n}_{f,h} \rangle_{\Omega_{\mathrm{S}}} + ( \rho_p \ff_{\mathrm{P}}(t_n), \bu_{s,h}^{\mathrm{P},n} )_{\Omega_{\mathrm{P}}}   + (\rho_f \phi  \ff_{\mathrm{P}}(t_n), \bu_{r,h}^{\mathrm{P},n} )_{\Omega_{\mathrm{P}}}  + ( r_{\mathrm{S}}(t_n), p^{\mathrm{S},n}_h)_{\Omega_{\mathrm{S}}} + (\rho_f^{-1} \theta(t_n), p^{\mathrm{P},n}_h )_{\Omega_{\mathrm{P}}}. 
\end{align*} 
Using Cauchy--Schwarz and Young's inequalities along with Lemma \ref{coercivity-continuity},  noting that   $\tilde{C}$ depends on the parameters, and using  $\|\bu_{r,h}^{\mathrm{P},n}\|_{0,\Omega_{\mathrm{P}}} \leq \|\bu_{s,h}^{\mathrm{P},n}+\bu_{r,h}^{\mathrm{P},n}\|_{0,\Omega_{\mathrm{P}}} $, we can sum over  $n = 1,\ldots,N$ and multiply  by $\tau$, to obtain
\begin{align}\label{dis1}
&\|\bu_{s,h}^{\mathrm{P},N}\|_{0,\Omega_{\mathrm{P}}}^2 + \|p^{\mathrm{P},N}_h\|_{0,\Omega_{\mathrm{P}}}^2 + \|\by_{s,h}^{\mathrm{P},N} \|_{1,\Omega_{\mathrm{P}}} ^2 +\|\bu_{r,h}^{\mathrm{P},N}\|^2_{0,\Omega_{\mathrm{P}}} + \tau \sum_{n =1 }^N \left( \left| \bu^{\mathrm{S},n}_{f,h} - d_{\tau} \by_{s,h}^{\mathrm{P},n}\right|_{\mathrm{BJS}}^2 + \| \bu_{r,h}^{\mathrm{P},n} \|_{0,\Omega_{\mathrm{P}}}  + \|  \bu^{\mathrm{S},n}_{f,h} \|_{1,\Omega_{\mathrm{S}}}^2 \right)\nonumber \\
& \quad + {\tau}^{2} \sum_{n=1}^N\left( \|d_{\tau} \bu_{s,h}^{\mathrm{P},n}\|_{0,\Omega_{\mathrm{P}}}^2 + \| d_{\tau} p^{\mathrm{P},n}_h\|_{0,\Omega_{\mathrm{P}}}^2   + \| d_{\tau} \by_{s,h}^{\mathrm{P},n} \|_{1,\Omega_{\mathrm{P}}} ^2  + \|d_{\tau}\bu_{r,h}^{\mathrm{P},n}\|^2_{0,\Omega_{\mathrm{P}}}\right)  \nonumber  \\
&\leq \tilde{C} \bigg( \|\bu_{s,h}^{\mathrm{P},0}\|_{0,\Omega_{\mathrm{P}}}^2 + \|p^{\mathrm{P},0}_h\|_{0,\Omega_{\mathrm{P}}}^2 + \|\by_{s,h}^{\mathrm{P},0} \|_{1,\Omega_{\mathrm{P}}} ^2 + \|\bu_{r,h}^{\mathrm{P},0} \|_{0,\Omega_{\mathrm{P}}} ^2 + 
{\epsilon_1}^{-1} \tau \sum_{n=1}^N ( \|\ff_{\mathrm{P}}(t_n)\|_{0,\Omega_{\mathrm{P}}}^2 + \|\ff_{\mathrm{S}}(t_n)\|_{-1,\Omega_{\mathrm{S}}}^2+\|\theta(t_n)\|_{0,\Omega_{\mathrm{P}}}^2 \nonumber  \\ &\qquad   + \|r_{\mathrm{S}} (t_n)\|_{0,\Omega_{\mathrm{S}}}^2)  +  \epsilon_1 \tau \sum_{n=1}^N (\|\bu^{\mathrm{S},n}_{f,h}\|_{1,\Omega_{\mathrm{S}}}^2+\|p^{\mathrm{P},n}_h\|_{0,\Omega_{\mathrm{P}}}^2 +\|p^{\mathrm{S},n}_h\|_{L^2(\Omega_{\mathrm{S}})}^2+\|\bu_{s,h}^{\mathrm{P},n}\|_{0,\Omega_{\mathrm{P}}}^2  +\|\bu_{r,h}^{\mathrm{P},n}\|_{0,\Omega_{\mathrm{P}}}^2) \bigg).
\end{align}
Next, employing the inf–sup condition for $(p^{\mathrm{S},n}_h, p^{\mathrm{P},n}_h, \lambda_h^n)$, in a similar way, we readily obtain 
\begin{align}\label{dis2}
&\epsilon_2 \tau \sum_{n=1}^N \bigl( \|p^{\mathrm{S},n}_h\|_{W_{f,h}}^2 + \|p^{\mathrm{P},n}_h\|_{W_{p,h}}^2 + \|\lambda_h^n\|_{\Lambda_h}^2 \bigr) \nonumber \\
&\quad \leq  \tilde{C} \epsilon_2 \tau \sum_{n=1}^N \bigl(\|\bu^{\mathrm{S},n}_{f,h} \|_{1,\Omega_{\mathrm{S}}}^2 + \| \bu_{r,h}^{\mathrm{P},n}\|_{0,\Omega_{\mathrm{P}}}^2   + |\bu^{\mathrm{S},n}_{f,h}- d_{\tau} \by_{s,h}^{\mathrm{P},n}|_{\mathrm{BJS}}^2 + \|\ff_{\mathrm{S}}(t_n)\|_{0,\Omega_{\mathrm{S}}}^2 +\|\ff_{\mathrm{P}}(t_n)\|_{0,\Omega_{\mathrm{P}}}^2 \bigr).
\end{align}
Combining \eqref{dis1} and \eqref{dis2}, and taking $\epsilon_2$ small enough, and then $\epsilon_1$ small enough, we can assert that 
\begin{align*}
&\|\bu_{s,h}^{\mathrm{P},N}\|_{0,\Omega_{\mathrm{P}}}^2 + \|p^{\mathrm{P},N}_h\|_{0,\Omega_{\mathrm{P}}}^2 + \|\by_{s,h}^{\mathrm{P},N} \|_{1,\Omega_{\mathrm{P}}}^2 +\|\bu_{r,h}^{\mathrm{P},N}\|^2_{0,\Omega_{\mathrm{P}}}
+ \tau \sum_{n =1 }^N \left( \left| \bu^{\mathrm{S},n}_{f,h} - d_{\tau}\by_{s,h}^{\mathrm{P},n}\right|_{\mathrm{BJS}}^2 + \| \bu_{r,h}^{\mathrm{P},n} \|_{0,\Omega_{\mathrm{P}}}   + \|  \bu^{\mathrm{S},n}_{f,h} \|_{1,\Omega_{\mathrm{S}}}^2  \right) \\
& \quad + {\tau}^{2} \sum_{n=1}^N\left( \|d_{\tau} \bu_{s,h}^{\mathrm{P},n}\|_{0,\Omega_{\mathrm{P}}}^2  + \| d_{\tau} p^{\mathrm{P},n}_h\|_{0,\Omega_{\mathrm{P}}}^2    + \| d_{\tau} \by_{s,h}^{\mathrm{P},n} \|_{1,\Omega_{\mathrm{P}}} ^2 + \|d_{\tau}\bu_{r,h}^{\mathrm{P},n}\|^2_{0,\Omega_{\mathrm{P}}}\right) + \tau \sum_{n=1}^N \left( \|p^{\mathrm{S},n}_h\|_{W_{f,h}}^2 + \|p^{\mathrm{P},n}_h\|_{W_{p,h}}^2  + \|\lambda_h^n\|_{\Lambda_h}^2 \right) \nonumber  \\
&  \lesssim \tau \sum_{n=1}^N (\|\ff_{\mathrm{P}}(t_n)\|_{0,\Omega_{\mathrm{P}}}^2  + \|\ff_{\mathrm{S}}(t_n)\|_{-1,\Omega_{\mathrm{S}}}^2+\|\theta(t_n)\|_{0,\Omega_{\mathrm{P}}}^2 + \|r_{\mathrm{S}} (t_n)\|_{L^2(\Omega_{\mathrm{S}})}^2) + \|\bu_{s,h}^{\mathrm{P},0}\|_{0,\Omega_{\mathrm{P}}}^2  + \|p^{\mathrm{P},0}_h\|_{0,\Omega_{\mathrm{P}}}^2 \\
& \quad + \|\by_{s,h}^{\mathrm{P},0} \|_{1,\Omega_{\mathrm{P}}} ^2 + \|\bu_{r,h}^{\mathrm{P},0} \|_{0,\Omega_{\mathrm{P}}} ^2,
\end{align*}
which completes the proof.
\end{proof}

\begin{lemma}(Discrete Gr\"ownwall inequality \cite{MR1299729})  
Let $\tau >0, B \geq 0,$ and let $a_n, b_n, c_n, d_n\geq 0$ be non-negative sequences such that $a_0 \leq B$ and 
$$
a_n +\tau \sum_{l=1}^n b_l \leq \tau \sum_{l=1}^{n-1} d_l a_l + \tau \sum_{l=1}^n c_l +B, \qquad n \geq 1. 
$$
Then, there holds 
\begin{align}\label{I-7}
a_n +\tau \sum_{l=1}^n b_l \leq \exp\left(\tau \sum_{l=1}^{n-1} d_l\right)\left( \tau \sum_{l=1}^n c_l +B\right), \qquad n \geq 1. 
\end{align}
\end{lemma}
\subsection{Error estimates for the fully discrete scheme}
\begin{theorem}
 Assuming \ref{(H1)}-\ref{(H3)} and sufficient smoothness for the solution of \eqref{mixed-primal},  the solution of 
 \eqref{semi1}-\eqref{semi4} with $\bu_{r,h}^{\mathrm{P}}(0)=\bI_{r, h} \bu_{r,0}$, $\by_{s, h}^{\mathrm{P}}(0)=\bI_{s, h} \by_{s, 0}$, $p_{P, h}^{\mathrm{P}}(0)=Q_{r, h} p_{p, 0}$, and $\bu_{s, h}^{\mathrm{P}}(0)=\bQ_{s, h} \bu_{s, 0}$,  satisfies
\begin{align*}
&  \|\bu_{s}^{\mathrm{P}} - \bu_{s,h}^{\mathrm{P}}\|_{l^{\infty}(0,T;\mathbf{L}^2(\Omega_{\mathrm{P}}))} + \|p^{\mathrm{P}}-p^{\mathrm{P}}_h\|_{l^{\infty}(0,T;L^2(\Omega_{\mathrm{P}}))} + \|\by_s^{\mathrm{P}}-\by_{s,h}^{\mathrm{P}}\|_{l^{\infty}(0,T;\mathbf{H}^1(\Omega_{\mathrm{P}}))} + \|\bu_f^{\mathrm{S}}- \bu^{\mathrm{S}}_{f,h} \|_{l^2(0,T;\mathbf{H}^1(\Omega_{\mathrm{S}}))} \\ 
& \quad +  \|\bu_r^{\mathrm{P}}-\bu_{r,h}^{\mathrm{P}}\|_{l^{\infty}(0,T;\mathbf{L}^2(\Omega_{\mathrm{P}}))}   + \|\bu_r^{\mathrm{P}}-\bu_{r,h}^{\mathrm{P}}\|_{l^2(0,T;{\mathbf{L}^2(\Omega_{\mathrm{P}})})} + | (\bu_f^{\mathrm{S}}- d_{\tau} \by_s^{\mathrm{P}} ) - (\bu^{\mathrm{S}}_{f,h} - d_{\tau} \by_{s,h}^{\mathrm{P}} )|_{l^2(0,T;\mathrm{BJS})} \\
& \quad + \| p^{\mathrm{S}} -  p^{\mathrm{S}}_h \|_{l^2(0,T;L^2(\Omega_{\mathrm{S}}))} + \| p^{\mathrm{P}} - p^{\mathrm{P}}_h \|_{l^2(0,T;L^2(\Omega_{\mathrm{P}}))} + \| \lambda - \lambda_h \|_{l^2(0,T;H^{-1/2}(\Sigma))} \\
& \leq C \sqrt{\exp(T)} \biggl[ h^{r_{k_f}} \|\bu_f^{\mathrm{S}}\|_{l^2(0,T;\mathbf{H}^{r_{k_f} +1}\left(\Omega_{\mathrm{S}}\right))}   + h^{r_{s_s}+1} ( \|\bu_s^{\mathrm{P}}\|_{l^{2}(0,T;\mathbf{H}^{r_{s_s}+1}\left(\Omega_{\mathrm{P}}\right))} +\|\partial_t\bu_s^{\mathrm{P}}\|_{L^2(0,T;\mathbf{H}^{r_{s_s}+1}\left(\Omega_{\mathrm{P}}\right))} ) \\
& \quad +  h^{r_{k_p}-1}(\|\bu_r^{\mathrm{P}}\|_{l^2(0,T;\mathbf{H}^{r_{k_p}+1}\left(\Omega_{\mathrm{P}}\right))}   + \|\bu_r^{\mathrm{P}}\|_{l^{\infty}(0,T;\mathbf{H}^{r_{k_p}+1}\left(\Omega_{\mathrm{P}}\right))}  + \|\partial_t \bu_r^{\mathrm{P}}\|_{L^2(0,T;\mathbf{H}^{r_{k_p}+1}\left(\Omega_{\mathrm{P}}\right))}  ) \\
&\quad + h^{r_{s_p}+1} ( \|p^{\mathrm{P}}\|_{l^{\infty}(0,T;H^{r_{s_p}+1}\left(\Omega_{\mathrm{P}}\right))} +\|p^{\mathrm{P}}\|_{l^2(0,T;H^{r_{s_p}+1}\left(\Omega_{\mathrm{P}}\right))} +\|\partial_t p^{\mathrm{P}} \|_{L^2\left(0, T ; H^{r_{s_p}+1 }\left(\Omega_{\mathrm{P}}\right)\right)} )  \\
& \quad + h^{r_{k_s}-1}  \bigl( \|\by_s^{\mathrm{P}}\|_{l^{\infty}(0,T;\mathbf{H}^{r_{k_s}+1}\left(\Omega_{\mathrm{P}}\right))} +\|\by_s^{\mathrm{P}}\|_{l^{2}(0,T;\mathbf{H}^{r_{k_s}+1}\left(\Omega_{\mathrm{P}}\right))}  +\|\partial_t\by_s^{\mathrm{P}}\|_{L^2(0,T;\mathbf{H}^{r_{k_s}+1}\left(\Omega_{\mathrm{P}}\right))}  \\
& \qquad \qquad \qquad    +\|\partial_t\by_s^{\mathrm{P}}\|_{L^{\infty}(0,T;\mathbf{H}^{r_{k_s}+1}\left(\Omega_{\mathrm{P}}\right))}  +\|\partial_{tt}\by_s^{\mathrm{P}}\|_{L^{2}(0,T;\mathbf{H}^{r_{k_s}+1}\left(\Omega_{\mathrm{P}}\right))}\bigr) \\
&\quad + h^{r_{s_f}+1} \|p^{\mathrm{S}}\|_{l^2(0,T;H^{r_{s_f}+1}\left(\Omega_{\mathrm{S}}\right))} + h^{{r}_{\tilde{k}_p}+\frac{1}{2}} (\|\lambda\|_{l^2(0,T;H^{{r}_{\tilde{k}_p}}\left(\Sigma \right))} + 
\|\lambda\|_{l^{\infty}(0,T;H^{{r}_{\tilde{k}_p}}\left(\Sigma \right))}  + \|\partial_t \lambda\|_{L^2(0,T;H^{{r}_{\tilde{k}_p}}\left(\Sigma \right))})\biggr] \\
&\quad + \tau \biggl[  \|\partial_{tt} \by_s^{\mathrm{P}}\|_{L^2(0,T;\mathbf{H}^1(\Omega_{\mathrm{P}}))} + \|\partial_{tt} \by_s^{\mathrm{P}}\|_{L^2(0,T;\mathbf{L}^2(\Omega_{\mathrm{P}}))} + \|\partial_{tt} \by_s^{\mathrm{P}}\|_{L^{\infty}(0,T;\mathbf{H}^1(\Omega_{\mathrm{P}}))} + \|\partial_{ttt} \by_s^{\mathrm{P}}\|_{L^2(0,T;\mathbf{H}^1(\Omega_{\mathrm{P}}))} \\
& \qquad  \quad  +     \|\partial_{tt} \bu_r^{\mathrm{P}}\|_{L^2(0,T;\mathbf{L}^2(\Omega_{\mathrm{P}}))} +  \|\partial_{tt} \bu_s^{\mathrm{P}}\|_{L^2(0,T;\mathbf{L}^2(\Omega_{\mathrm{P}}))}   +  \|\partial_{tt} p^{\mathrm{P}}\|_{L^2(0,T;L^2(\Omega_{\mathrm{P}}))} \biggr],
\end{align*}
where $0 \leq r_{k_f} \leq k_f$, $0 \leq r_{s_f} \leq s_f$, $1 \leq r_{k_p}  \leq k_p$, $0 \leq r_{s_p} \leq s_p$,  
$ 1 \leq r_{k_s} \leq k_s$,  $0 \leq r_{s_s} \leq s_s$, $-1/2 \leq {r}_{\tilde{k}_p} \leq \tilde{k}_p -1/2$.
\end{theorem}
\begin{proof}
We introduce the errors for all variables and split them into approximation and discretization errors:
\begin{align*}
\mathbf{e}_f^n & :=\bu_f^{{\mathrm{S}},n}-\bu_{f, h}^{{\mathrm{S}},n}=(\bu_f^{{\mathrm{S}},n}-\bI_{f, h} \bu_f^{{\mathrm{S}},n})+(\bI_{f, h} \bu_f^{{\mathrm{S}},n}-\bu^{\mathrm{S},n}_{f,h}):=\bchi_f^n+\bphi_{f,h}^n, \\
\mathbf{e}_r^n & :=\bu_r^{\mathrm{P},n}-\bu_{r,h}^{\mathrm{P},n}=(\bu_r^{\mathrm{P},n}-\bI_{r, h} \bu_r^{\mathrm{P},n})+(\bI_{r, h} \bu_r^{\mathrm{P},n}-\bu_{r,h}^{\mathrm{P},n}):=\bchi_r^n+\bphi_{r,h}^n, \\
\mathbf{e}_s^n & :=\by_s^{{\mathrm{P}},n}-\by_{s, h}^{\mathrm{P},n}=(\by_s^{{\mathrm{P}},n}-\bI_{s, h} \by_s^{{\mathrm{P}},n})+(\bI_{s, h} \by_s^{{\mathrm{P}},n}-\by_{s, h}^{\mathrm{P},n}):=\bchi_s^n+\bphi_{s,h}^n, \\
\mathbf{e}_{ss} & :=\bu_s^{\mathrm{P},n}-\bu_{s,h}^{\mathrm{P},n}=(\bu_s^{\mathrm{P},n}-\bQ_{s, h} \bu_s^{\mathrm{P},n})+(\bQ_{s, h} \bu_s^{\mathrm{P},n}-\bu_{s,h}^{\mathrm{P},n}):=\bchi_{ss}^n+\bphi_{ss,h}^n, \\
e_{f p}^n & :=p^{\mathrm{S},n}-p^{\mathrm{S},n}_h=(p^{\mathrm{S},n}-Q_{f, h} p^{\mathrm{S},n})+(Q_{f, h} p^{\mathrm{S},n}-p^{\mathrm{S}, n}_h):=\chi_{f p}^n+\phi_{f p, h}^n, \\
e_{p p}^n & :=p^{\mathrm{P},n}-p^{\mathrm{P}, n}_h=(p^{\mathrm{P},n}-Q_{p, h} p^{\mathrm{P},n})+(Q_{p, h} p^{\mathrm{P},n}-p^{\mathrm{P},n}_h):=\chi_{p p}^n+\phi_{p p, h}^n, \\
e_\lambda^n & :=\lambda^n-\lambda_h^n=(\lambda^n-Q_{\lambda, h} \lambda^n)+(Q_{\lambda, h} \lambda^n-\lambda_h^n):=\chi_\lambda^n+\phi_{\lambda, h}^n .
\end{align*} 
The error equations are obtained by subtracting the first two equations of  \eqref{fully1}-\eqref{fully2} from the their continuous counterpart \eqref{mixed-primal}.
Let $r_n$ denote the difference between the time derivative and its discrete analogue $r_n(\theta) = \partial_t \theta(t_n) - d_{\tau} \theta^n$.

The proof  follows the structure of the proof of Theorem \ref{semidiscreteerror}:
\begin{itemize}
    \item[(i)] Substitute $(\bv_{f,h}^{\mathrm{S}}, \bv_{r,h}^{\mathrm{P}},  \bw_{s,h}^{\mathrm{P}}, \bv_{s,h}^{\mathrm{P}}, q^{\mathrm{S}}_h, q^{\mathrm{P}}_h)  = (\bphi_{f,h}^n, \bphi_{r,h}^n, d_{\tau} \bphi_{s,h}^n, \bphi_{ss,h}^n, \phi_{fp,h}^n, \phi_{pp,h}^n)$ in the error equation; and
    \item[(ii)] Split the individual errors  
    and apply the properties of the projection operators.  
\end{itemize}
This gives 
\begin{align}\label{fe1}
&a_f^{\mathrm{S}}\left(\bphi_{f,h}^n,\bphi_{f,h}^n\right)+ a_{f}^{\mathrm{P}}\left(\bphi_{r,h}^n, d_{\tau} \bphi_{s,h}^n\right) +a_{f}^{\mathrm{P}}\left(\bphi_{r,h}^n, \bphi_{r,h}^n\right) + a_{f}^{\mathrm{P}}\left(d_{\tau}\bphi_{s,h}^n, \bphi_{r,h}^n\right) + a_{f}^{\mathrm{P}}\left(d_{\tau}\bphi_{s,h}^n, d_{\tau} \bphi_{s,h}^n\right) +a_s^{\mathrm{P}}\left(\bphi_{s,h}^n,d_{\tau}  \bphi_{s,h}^n\right) \nonumber  \\ 
&\quad +a_{\mathrm{BJS}}\left(\bphi_{f,h}^n, d_{\tau} \bphi_{s,h}^n ; \bphi_{f,h}^n, d_{\tau} \bphi_{s,h}^n\right) 
+ m_{\rho_f \phi}(d_{\tau}\bphi_{r,h}^n, d_{\tau} \bphi_{s,h}^n)   + m_{\rho_p}\left(d_{\tau} \bphi_{ss,h}^n, d_{\tau} \bphi_{s,h}^n\right )  + m_{\rho_f \phi }(d_{\tau} \bphi_{r,h}^n, \bphi_{r,h}^n)   \nonumber\\ 
& \quad  + m_{\rho_f \phi}(d_{\tau}\bphi_{ss,h}^n, \bphi_{r,h}^n) + \left((1-\phi)^2 K^{-1} d_{\tau} \phi_{pp,h}^n, \phi_{pp,h}^n\right)_{\Omega_{\mathrm{P}}}   - m_{\theta}(\bphi_{r,h}^n, d_{\tau} \bphi_{s,h}^n)  - m_{\theta}(d_{\tau} \bphi_{s,h}^n, d_{\tau} \bphi_{s,h}^n) \nonumber \\
& \quad    - m_{\theta}(\bphi_{r,h}^n, \bphi_{r,h}^n) - m_{\theta}(d_{\tau} \bphi_{s,h}^n, \bphi_{r,h}^n) + m_{\phi^2/\kappa}(\bphi_{r,h}^n, \bphi_{r,h}^n) \nonumber \\& = 
   \mathcal{E} + \mathcal{H},
\end{align}
where 
\begin{align*}
\mathcal{E} &:= - a_f^{\mathrm{S}}\left(\bchi_f^n, \bphi_{f,h}^n\right)  - a_{f}^{\mathrm{P}}\left(\bchi_r^n, \bphi_{r,h}^n\right)- a_{f}^{\mathrm{P}}\left(d_{\tau} \bchi_s^n, \bphi_{r,h}^n\right) -a_{\mathrm{BJS}}\left(\bchi_f^n, d_{\tau} \bchi_s^n ; \bphi_{f,h}^n, d_{\tau} \bphi_{s,h}^n\right) -b^{\mathrm{S}}\left(\bphi_{f,h}^n, \chi_{fp}^n\right) \nonumber \\ 
& -b_{\Gamma}\left(\bphi_{f,h}^n, \bphi_{r,h}^n,  \mathbf{0}; \chi_{\lambda}^n \right)   +b_s^{\mathrm{P}}\left(d_{\tau} \bchi_s^n, \phi_{pp,h}^n\right) + m_{\theta}(\bchi_r^n, d_{\tau} \bphi_{s,h}^n) + m_{\theta}(d_{\tau} \bchi_s^n, d_{\tau} \bphi_{s,h}^n)  + m_{\theta}(\bchi_r^n, \bphi_{r,h}^n) + m_{\theta}(d_{\tau} \bchi_s^n, \bphi_{r,h}^n) \nonumber \\
& -m_{\phi^2/\kappa}(\bchi_r^n, \bphi_{r,h}^n)  - m_{\rho_f \phi}(d_{\tau} \bchi_r^n, d_{\tau} \bphi_{s,h}^n)  -m_{\rho_p}(d_{\tau} \bchi_{ss}^n, d_{\tau} \bphi_{s,h}^n)  - m_{\rho_f \phi }(d_{\tau}\bchi_r^n, \bphi_{r,h}^n)  - m_{\rho_f \phi}(d_{\tau}\bchi_{ss}^n, \bphi_{r,h}^n)  \nonumber \\
& -a_{f}^{\mathrm{P}}\left(  r_n(\by_{s}^{\mathrm{P}}), \bphi_{r,h}^n\right) 
   + m_{\theta}(  r_n(\by_{s}^{\mathrm{P}}), d_{\tau}\bphi_{s,h}^n)    + m_{\theta}(  r_n(\by_{s}^{\mathrm{P}}), \bphi_{r,h}^n) -m_{\rho_f \phi}(r_n(\bu_r^{\mathrm{P}}), d_{\tau} \bphi_{s,h}^n)  - m_{\rho_p}( r_n(\by_{s}^{\mathrm{P}}), d_{\tau} \bphi_{s,h}^n)\nonumber \\ 
   &  - m_{\rho_f \phi }(r_n(\bu_r^{\mathrm{P}}), \bphi_{r,h}^n)  - m_{\rho_f \phi}(r_n(\bu_{s}^{\mathrm{P}}), \bphi_{r,h}^n)  +b_s^{\mathrm{P}}\left( r_n(\by_{s}^{\mathrm{P}}), \phi_{pp,h}^n\right)  - \left((1-\phi)^2 K^{-1} r_n(p), \phi_{pp,h}^n\right)_{\Omega_{\mathrm{P}}},\\  
\mathcal{H} &:= -a_s^{\mathrm{P}}\left(\bchi_s^n,d_{\tau} \bphi_{s,h}^n\right)-
a_{f}^{\mathrm{P}}\left(\bchi_r^n, d_{\tau} \bphi_{s,h}^n\right)
 -a_{f}^{\mathrm{P}}\left(d_{\tau} \bchi_s^n, d_{\tau} \bphi_{s,h}^n\right) -b_s^{\mathrm{P}}\left(d_{\tau} \bphi_{s,h}^n, \chi_{pp}^n\right) \nonumber \\& -a_{f}^{\mathrm{P}}\left(  r_n(\by_{s}^{\mathrm{P}}),d_{\tau} \bphi_{s,h}^n\right)
 -a_{\mathrm{BJS}}\left(\mathbf{0},   r_n(\by_{s}^{\mathrm{P}}) ; \bphi_{f,h}^n, d_{\tau}\bphi_{s,h}^n\right) - b_{\Gamma}(\mathbf{0},\mathbf{0},d_{\tau} \bphi_{s,h}^n; \chi_{\lambda}^n). 
\end{align*}
By using the inequality  \eqref{ineq:aux} and the identity \eqref{discrete_identity}, the left-hand side of \eqref{fe1} becomes 
\begin{align*}
    \texttt{LHS}_{\eqref{fe1}} & \geq \frac{1}{2} d_{\tau} \left( \|\bphi_{ss,h}^n\|_{0,\Omega_{\mathrm{P}}}^2 + \|\phi_{pp,h}^n\|_{0,\Omega_{\mathrm{P}}}^2 + \|\bphi_{s,h}^n \|_{1,\Omega_{\mathrm{P}}} ^2 +\|\bphi_{r,h}^n\|^2_{0,\Omega_{\mathrm{P}}}\right)+ \frac{\tau}{2}\left( \|d_{\tau} \bphi_{ss,h}^n\|_{0,\Omega_{\mathrm{P}}}^2 \right. \nonumber \\
    & \quad \left.+ \| d_{\tau} \phi_{pp,h}^n\|_{0,\Omega_{\mathrm{P}}}^2  + \| d_{\tau} \bphi_{s,h}^n \|_{1,\Omega_{\mathrm{P}}} ^2 + \|d_{\tau}\bphi_{r,h}^n\|^2_{0,\Omega_{\mathrm{P}}}\right)+ \left| \bphi_{f,h}^n - d_{\tau} \bphi_{s,h}^n\right|_{\mathrm{BJS}}^2 + \| \bphi_{r,h}^n \|_{0,\Omega_{\mathrm{P}}}  + \|  \bphi_{f,h}^n \|_{1,\Omega_{\mathrm{S}}}^2.
\end{align*}
Now, we bound the terms on the right-hand side similarly as in the proof of Theorem \ref{semidiscreteerror}. Consequently we get 
\begin{align}\label{fe3}
\mathcal{I} & \leq \epsilon_1^{-1}\left(\|\bchi_f^n\|_{1,\Omega_{\mathrm{S}}}^2  +\|\bchi_r^n\|_{1,\Omega_{\mathrm{P}}}^2  + \| d_{\tau} \bchi_s^n \|_{1,\Omega_{\mathrm{P}}}^2    + \|\bchi_r^n\|_{0,\Omega_{\mathrm{P}}}^2  + \| d_{\tau} \bchi_s^n \|_{0,\Omega_{\mathrm{P}}}^2 + \| d_{\tau} \bchi_{ss}^n \|_{0,\Omega_{\mathrm{P}}}^2  +  \| d_{\tau} \bchi_r^n \|_{0,\Omega_{\mathrm{P}}}^2  + \|\chi_{f p}^n\|_{0,\Omega_{\mathrm{S}}}^2  \right. \nonumber \\ 
& \quad  \left.
 + \| \chi_{\lambda}^n\|^2_{L^2(\Sigma)} +\|r_n(\by_s^{\mathrm{P}})\|_{1,\Omega_{\mathrm{P}}}^2 + \|r_n(\by_s^{\mathrm{P}})\|_{0,\Omega_{\mathrm{P}}}^2 + \|r_n(\bu_r^{\mathrm{P}})\|_{0,\Omega_{\mathrm{P}}}^2  + \|r_n(\bu_{s}^{\mathrm{P}})\|_{0,\Omega_{\mathrm{P}}}^2 + \|r_n(p^{\mathrm{P}})\|_{0,\Omega_{\mathrm{P}}}^2\right) + C \|\chi_{\lambda}^n\|_{-1/2,\Sigma}^2\nonumber \\
 & \quad   +\epsilon_1\left(    \|\bphi_{f,h} ^n\|_{1,\Omega_{\mathrm{S}}}^2   
 + \|\bphi_{r,h}^n\|_{0,\Omega_{\mathrm{P}}}^2 +\|\bphi_{ss,h}^n\|_{0,\Omega_{\mathrm{P}}}^2   + \|\bphi_{r,h}^n\|_{0,\Omega_{\mathrm{P}}}^2  + \left|\bphi_{f,h}^n-\partial_t \bphi_{s,h}^n\right|_{\mathrm{BJS}}^2  +\|\phi_{p p, h}^n\|_{0,\Omega_{\mathrm{P}}}^2  \right)  . 
 \end{align}
Combining \eqref{fe1} and \eqref{fe3}, summing over $n = 1, \ldots, N$, and multiplying by $\tau$, we use Lemma \ref{coercivity-continuity} to obtain:
\begin{align}\label{fe4}
&\|\bphi_{ss,h}^N\|_{0,\Omega_{\mathrm{P}}}^2 + \|\phi_{pp,h}^N\|_{0,\Omega_{\mathrm{P}}}^2 + \|\bphi_{s,h}^N\|_{1,\Omega_{\mathrm{P}}} ^2 +\|\bphi_{r,h}^N\|^2_{0,\Omega_{\mathrm{P}}} +{\tau}^2 \sum_{n=1}^N\left( \|d_{\tau} \bphi_{ss,h}^n\|_{0,\Omega_{\mathrm{P}}}^2 + \| d_{\tau} \phi_{pp,h}^n\|_{0,\Omega_{\mathrm{P}}}^2 \right. \nonumber \\
    & \quad \left. + \| d_{\tau} \bphi_{s,h}^n \|_{1,\Omega_{\mathrm{P}}} ^2 + \|d_{\tau}\bphi_{r,h}^n\|^2_{0,\Omega_{\mathrm{P}}}\right)+\tau \sum_{n=1}^N\left( \left| \bphi_{f,h}^n - d_{\tau} \bphi_{s,h}^n\right|_{\mathrm{BJS}}^2 + \| \bphi_{r,h}^n \|_{0,\Omega_{\mathrm{P}}}   + \|  \bphi_{f,h}^n \|_{1,\Omega_{\mathrm{S}}}^2 \right) \nonumber \\
    & \leq C \biggl[ \|\bphi_{ss,h}^0\|_{0,\Omega_{\mathrm{P}}}^2 + \|\phi_{pp,h}^0\|_{0,\Omega_{\mathrm{P}}}^2 + \|\bphi_{s,h}^0\|_{1,\Omega_{\mathrm{P}}} ^2 +\|\bphi_{r,h}^0\|^2_{0,\Omega_{\mathrm{P}}} + \epsilon_1^{-1} \tau \sum_{n=1}^N \left(\|\bchi_f^n\|_{1,\Omega_{\mathrm{S}}}^2  +\|\bchi_r^n\|_{1,\Omega_{\mathrm{P}}}^2 
 + \| d_{\tau} \bchi_s^n \|_{1,\Omega_{\mathrm{P}}}^2 \right.  \nonumber  \\ 
 &  \quad   + \|\bchi_r^n\|_{0,\Omega_{\mathrm{P}}}^2  + \| d_{\tau} \bchi_s^n \|_{0,\Omega_{\mathrm{P}}}^2 + \| d_{\tau} \bchi_{ss}^n \|_{0,\Omega_{\mathrm{P}}}^2 
 +  \| d_{\tau} \bchi_r^n \|_{0,\Omega_{\mathrm{P}}}^2  + \|\chi_{f p}^n\|_{0,\Omega_{\mathrm{S}}}^2  +\|r_n(\by_s^{\mathrm{P}})\|_{1,\Omega_{\mathrm{P}}}^2 + \|r_n(\by_s^{\mathrm{P}})\|_{0,\Omega_{\mathrm{P}}}^2   \nonumber  \\
 & \quad \left.
+ \|r_n(\bu_r^{\mathrm{P}})\|_{0,\Omega_{\mathrm{P}}}^2 + \|r_n(\bu_{s}^{\mathrm{P}})\|_{0,\Omega_{\mathrm{P}}}^2 +\|r_n(p^{\mathrm{P}})\|_{0,\Omega_{\mathrm{P}}}^2\right)  + \epsilon_1 \tau \sum_{n=1}^N(\|\bphi_{f,h} ^n\|_{1,\Omega_{\mathrm{S}}}^2   
 + \|\bphi_{r,h}^n\|_{0,\Omega_{\mathrm{P}}}^2 +\|\bphi_{ss,h}^n\|_{0,\Omega_{\mathrm{P}}}^2  \nonumber  \\ 
 & \quad + \|\bphi_{r,h}^n\|_{0,\Omega_{\mathrm{P}}}^2  + \left|\bphi_{f,h}^n-\partial_t \bphi_{s,h}^n\right|_{\mathrm{BJS}}^2  +\|\phi_{p p, h}^n\|_{0,\Omega_{\mathrm{P}}}^2  )  + C \tau \sum_{n=1}^N ( \|\chi_{\lambda}^n\|_{-1/2,\Sigma}^2 + \mathcal{H})\biggr].
\end{align}
Next, for each term in $\mathcal{H}$, we use the following summation by parts
$$
\tau \sum_{n=1}^N\left(\upsilon\left(t_n\right), d_\tau \bphi_{s,h}^n\right)=\left(\upsilon\left(t_N\right), \bphi_{s,h}^N\right)-\left(\upsilon(0), \bphi_{s,h}^0\right)-\tau \sum_{n=1}^{N-1}\left(d_\tau \upsilon ^n, \bphi_{s,h}^n\right),
$$
where $\upsilon$ stands for any of the functions  in $\mathcal{H}$.
Then, we apply Cauchy--Schwarz and Young's inequalities to obtain
\begin{align*}
& \tau \sum_{n=1}^N\left(\upsilon(t_n), d_\tau \bphi_{s,h}^n\right) \\
&\quad \leq \frac{\epsilon_1}{2}\|\bphi_{s,h}^N\|_{0,\Omega_{\mathrm{P}}}^2+\frac{1}{2 \epsilon_1}\|\upsilon(t_N)\|_{0,\Omega_{\mathrm{P}}}^2+\frac{\tau}{2} \sum_{n=1}^{N-1}\|\bphi_{s,h}^n\|_{0,\Omega_{\mathrm{P}}}^2 
+\frac{1}{2}\bigl(\|\bphi_{s,h}^0\|_{0,\Omega_{\mathrm{P}}}^2+\|\upsilon(0)\|_{0,\Omega_{\mathrm{P}}}^2+\tau \sum_{n=1}^{N-1}\|d_\tau \upsilon^n\|_{0,\Omega_{\mathrm{P}}}^2\bigr).
\end{align*}
Next we proceed to bound each term of $\mathcal{H}$. It follows that 
\begin{align*}
\tau \sum_{n=1}^N \mathcal{H} & \leq {\epsilon_1}\|\bphi_{s,h}^N\|_{1,\Omega_{\mathrm{P}}}^2 +{\tau}\sum_{n=1}^{N-1}\|\bphi_{s,h}^n\|_{1,\Omega_{\mathrm{P}}}^2 + \|\bphi_{s,h}^0\|_{1,\Omega_{\mathrm{P}}}^2 + \epsilon_1^{-1} \left( \|\bchi_s^N\|_{1,\Omega_{\mathrm{P}}}^2  \nonumber \right. \\& \quad \left. +\|\bchi_r^N\|_{1,\Omega_{\mathrm{P}}}^2 + \|d_{\tau}\bchi_s^N\|_{1,\Omega_{\mathrm{P}}}^2 + \| \chi_{pp}^N \|^2_{0,\Omega_{\mathrm{P}}} + \| r_N(\by_s^{\mathrm{P}}) \|_{1,\Omega_{\mathrm{P}}} + \|\chi_{\lambda}^N\|_{H^{-1/2}(\Omega_{\mathrm{P}})}\right) \nonumber \\&  \quad + C\left( \|\bchi_s^0\|_{1,\Omega_{\mathrm{P}}}^2 +\|\bchi_r^0\|_{1,\Omega_{\mathrm{P}}}^2 +\|d_{\tau}\bchi_s^0\|_{1,\Omega_{\mathrm{P}}}^2 + \| \chi_{pp}^0 \|^2_{0,\Omega_{\mathrm{P}}} + \| r_0(\by_s^{\mathrm{P}}) \|_{1,\Omega_{\mathrm{P}}} + \|\chi_{\lambda}^0\|_{H^{-1/2}(\Omega_{\mathrm{P}})}\right) \nonumber \\& \quad + \tau \sum_{n=1}^{N-1} \left(\|d_{\tau}\bchi_s^n\|_{1,\Omega_{\mathrm{P}}}^2 +  \|d_{\tau}\bchi_r^n\|_{1,\Omega_{\mathrm{P}}}^2 + \|d_{\tau} d_{\tau}\bchi_s^n\|_{1,\Omega_{\mathrm{P}}}^2 + \| d_{\tau} \chi_{pp}^n \|^2_{0,\Omega_{\mathrm{P}}} + \| d_{\tau} r_n(\by_s^{\mathrm{P}}) \|_{1,\Omega_{\mathrm{P}}}   + \| d_{\tau}\chi_{\lambda}^n\|_{-1/2,\Omega_{\mathrm{P}}} \right).
\end{align*}
For the initial conditions, we set $ \bu_{r,h}^{\mathrm{P}}=\bI_{r, h} \bu_{r, 0}$,   $\bu_{s, h}=\bQ_{s, h} \bu_{s, 0}$,  $p_{h}^{\mathrm{P}}(0)=Q_{p, h} p_{p, 0}$, and $\by_{s, h}^{\mathrm{P}}(0)=\bI_{s, h} \by_{s, 0}$, implying
\[
\bphi_{r,h}^0=\cero, \quad \bphi_{ss,h}^0=\cero, \quad \bphi_{s,h}^0=\cero, \quad \phi_{p p, h}^0=0.
\]
Analogously to equation \eqref{me9}, we have
\begin{align*}
&\epsilon_2 \tau \sum_{n=1}^N\bigl( \| \phi_{fp,h}^n \|_{L^2(\Omega_{\mathrm{S}})}^2 + \| \phi_{pp,h}^n \|_{0,\Omega_{\mathrm{P}}}^2 + \| \phi_{\lambda,h}^n \|_{-1/2, \Sigma}^2\bigr) \lesssim \epsilon_2 \tau \sum_{n=1}^N \bigl( \|  \bphi_{f,h}^n \|_{1,\Omega_{\mathrm{S}}}^2 + h^{-2} \| \bphi_{r,h}^n \|_{0,\Omega_{\mathrm{P}}}^2   \nonumber \\&  +  | \bphi_{f,h}^n - d_\tau \bphi_{s,h}^n |_{\mathrm{BJS}}^2 + \| \bphi_{r,h}^n \|_{0,\Omega_{\mathrm{P}}}^2   + \| \bchi_f^n \|_{1,\Omega_{\mathrm{S}}}^2  + \| \bchi_r^n \|_{1,\Omega_{\mathrm{P}}}^2  + \|\bchi_r^n \|_{0,\Omega_{\mathrm{P}}}^2   + \| \chi_{fp}^n\|_{0,\Omega_{\mathrm{P}}}^2  + \| \chi_{pp}^n\|_{0,\Omega_{\mathrm{P}}}^2  + \| \chi_{\lambda}^n\|_{-1/2, \Sigma}^2 
\bigr).
\end{align*}
We note that the terms involving $d_{\tau}$ on the right-hand side require special treatment similarly as in \cite{MR4353225}. Consequently 
\begin{align*}
    &\tau \sum_{n=1}^{N-1}  \left(\|d_{\tau}\bchi_s^n\|_{1,\Omega_{\mathrm{P}}}^2 +  \|d_{\tau}\bchi_r^n\|_{1,\Omega_{\mathrm{P}}}^2  + \| d_{\tau} \chi_{pp}^n \|^2_{0,\Omega_{\mathrm{P}}}  + \| d_{\tau}\chi_{\lambda}^n\|_{-1/2,\Omega_{\mathrm{P}}} \right) \nonumber \\& \leq \int_0^{t_N} \left(\|\partial_{\tau}\bchi_s\|_{1,\Omega_{\mathrm{P}}}^2 +  \|\partial_{\tau}\bchi_r\|_{1,\Omega_{\mathrm{P}}}^2 + \| \partial_{\tau} \chi_{pp} \|^2_{0,\Omega_{\mathrm{P}}}  + \| \partial_{\tau}\chi_{\lambda}\|_{-1/2,\Omega_{\mathrm{P}}} \right).
\end{align*}
To bound  $\| d_\tau d_\tau \bchi_s^n \|^2_{1,\Omega_{\mathrm{P}}}$, we use the Integral Mean Value and Mean Value Theorems. Therefore 
\[
\tau \sum_{n=1}^{N-1}\|d_\tau d_\tau \bchi_s^n \|_{1,\Omega_{\mathrm{P}}}^2 \leq C \underset{t \in\left(0, t_N\right)}{\operatorname{esssup}}\|\partial_{t t} \bchi_s\|_{1,\Omega_{\mathrm{P}}}^2 .
\]
On the other hand, regarding the time discretization error,   Taylor's expansion gives
\[
\tau \sum_{n=1}^N\|r_n(\phi)\|_{H^k(S)}^2 \leq C \tau^2\|\partial_{t t} \phi \|_{L^2\left(0, T ; H^k(S)\right)}^2,
\]
and similarly for the higher-order derivatives, we obtain 
\begin{align}\label{fe10}
\tau \sum_{n=1}^N\|d_{\tau} r_n(\phi)\|_{H^k(S)}^2 \leq C \tau^2 
\|\partial_{t t t} \phi \|_{L^2\left(0, T ; H^k(S)\right)}^2.
\end{align}
The assertion of the theorem follows from combining  \eqref{fe4}-\eqref{fe10},  the discrete Gr\"onwall inequality \eqref{I-7} for $a_n = |\bphi_{s,h}^N|^2_{1,\Omega_{\mathrm{P}}}$, triangle inequality, and the approximation properties  \eqref{ps} (resp. \eqref{ea1}) applied to \eqref{vh}, \eqref{lambdah}, (resp. \eqref{ea3}).
\end{proof}

\section{Numerical experiments}\label{section8}
All the routines have been implemented using the open-source finite element library FEniCS \cite{alnaes2015fenics}, along with the specialized module multiphenics \cite{Ballarin}. This module was used to handle specific terms related to subdomains and boundaries. The solvers used in this work are monolithic. We utilized the MUMPS \cite{amestoy2000mumps} distributed direct solver for the linear systems in the first three examples and employed UMFPACK \cite{davis2007umfpack} for the fourth example. We showcase four examples: convergence tests (example 1), simulation of subsurface fracture flow (example 2), extensible channel flow  (example 3), and flow and deformation patterns in 2D slices of the brain (example 4). 

\subsection{Convergence tests against manufactured solutions}
The accuracy of the spatio-temporal discretization is verified using the following closed-form solutions defined on the domains $\Omega_{\mathrm{S}}=(0,1) \times(0,1), \Omega_{\mathrm{P}}=(0,1) \times(1,2)$, separated by the interface $\Sigma=(0,1) \times\{1\}$
\begin{gather}\label{num}
\nonumber
\bu_f^{\mathrm{S}}= \sin (t)\left(\begin{array}{c}
-\cos (\pi x) \sin (\pi y) \\
\sin (\pi x) \cos (\pi y)
\end{array}\right), \qquad 
p^{\mathrm{S}} = \sin (t) \cos (\pi x) \cos (\pi y),  \qquad 
\bu_r^{\mathrm{P}} = \left(\begin{array}{c}
t^2 \sin^2(4\pi y) - t x^3 \cos(4 \pi y)\\
t^2 \sin^2(4\pi y) + 2 t x^3 \sin(4 \pi y)
\end{array}\right), \\
\bu_s^{\mathrm{P}} = \left(\begin{array}{c}
 t x^3 \cos(4 \pi y)\\
-2 t x^3 \sin(4 \pi y)
\end{array}\right), 
\qquad 
\by_s^{\mathrm{P}} = \left(\begin{array}{c}
0.5 t^2 x^3 \cos(4 \pi y)\\
- t^2 x^3 \sin(4 \pi y)
\end{array}\right), 
\qquad 
p^{\mathrm{P}} = \cos (t) \sin (\pi x) \sin (\pi y).
\end{gather}
The synthetic model parameters are taken as
$$
\lambda_p=10, \quad \mu_p=10, \quad \mu_f=10, \quad \alpha_{\mathrm{BJS}}=1, \quad \phi=0.1, \quad \kappa=1, \quad \rho_p=1, \quad \rho_f=1, \quad K =1, \quad \theta =-0.01,
$$
all regarded non-dimensional and do not have physical relevance in this case, as we will be simply testing the convergence of the FE approximations. 
The model problem is then complemented with the appropriate Dirichlet boundary conditions and initial data. These functions do not necessarily fulfill the interface conditions, so additional terms are required giving modified relations on $\Sigma$:
$$
\bu_f^{\mathrm{S}}\cdot \bn_{\mathrm{S}}+\left(\partial_t \by_s^{\mathrm{P}}+\bu_r^{\mathrm{P}}\right) \cdot \bn_{\mathrm{P}}=m_{\Sigma,\mathrm{ex}}^1 , \quad 
-\left(\bsigma^{\mathrm{S}} \bn_{\mathrm{S}}\right) \cdot \bn_{\mathrm{S}}= -\left(\bsigma_f^{\mathrm{P}} \bn_{\mathrm{P}}\right) \cdot \bn_{\mathrm{P}} + m_{\Sigma,\mathrm{ex}}^2 = \lambda  , 
$$
$$
\bsigma_f^{\mathrm{S}} \bn_{\mathrm{S}}+\bsigma_f^{\mathrm{P}} \bn_{\mathrm{P}} +\bsigma_s^{\mathrm{P}} \bn_{\mathrm{P}}= m_{\Sigma,\mathrm{ex}}^3  , \quad 
-\left(\bsigma_f^{\mathrm{S}} \bn_{\mathrm{S}}\right) \cdot \btau_{f, j}=\mu_{f}
  \alpha_{\mathrm{BJS}} \sqrt{Z_j^{-1}}\left(\bu_f^{\mathrm{S}}-{\partial_t \by_s^{\mathrm{P}}}\right) \cdot \btau_{f, j} + m_{\Sigma,\mathrm{ex}}^4,
\quad 
(\bsigma_f^{\mathrm{P}} \bn_{\mathrm{P}}) \cdot \tau_{f,j} = m_{\Sigma,\mathrm{ex}}^5,
$$
and the additional scalar and vector terms $m_{\Sigma,\mathrm{ex}}^i$ (computed with the exact solutions \eqref{num} entail the following changes in the linear functionals
\begin{gather*}
 F(\bv_f^{\mathrm{S}})= \int_{\Omega_{\mathrm{S}}}  \ff_{\mathrm{S}} \bv_f^{\mathrm{S}} - \langle  m_{\Sigma,\mathrm{ex}}^4 , \bv_f^{\mathrm{S}} \cdot \tau_{f,j} \rangle_{\Sigma}, \qquad  F(\bv_r^{\mathrm{P}}):= \int_{\Omega_{\mathrm{P}}} \rho_f \phi \ff_{\mathrm{P}} \bv_f^{\mathrm{S}} + \langle  m_{\Sigma,\mathrm{ex}}^2 , \bv_r^{\mathrm{P}} \cdot \bn_{\mathrm{P}} \rangle_{\Sigma} + \langle  m_{\Sigma,\mathrm{ex}}^5 , \bv_r^{\mathrm{P}} \cdot \tau_{f,j} \rangle_{\Sigma}, \\
 F(\bw_s^{\mathrm{P}}):=  \int_{\Omega_{\mathrm{P}}} \rho_p \ff_{\mathrm{P}} \bw_s^{\mathrm{P}}  +  \langle  m_{\Sigma,\mathrm{ex}}^3 , \bw_s^{\mathrm{P}} + \langle  m_{\Sigma,\mathrm{ex}}^4 , \bw_s^{\mathrm{P}} \cdot \tau_{f,j} \rangle_{\Sigma}\rangle_{\Sigma}, \qquad 
 F(\mu) := - \langle  m_{\Sigma,\mathrm{ex}}^1 , \mu \rangle_{\Sigma}.
\end{gather*}
We generate successively refined simplicial grids and use a sufficiently small (non dimensional) time step $\tau = h^2$ and final time $T=1$, to guarantee that the error produced by the time discretization does not dominate. Errors between the approximate and exact solutions are shown in Table \ref{conv_test_space}. This error history confirms the optimal convergence of the FE scheme for all variables in their respective norms. 

\begin{table}[t!]
\centering
\begin{tabular}{r|c|c|c|c|c|c|c|c|c}
\toprule
DoFs & $h$ & $\|e_{\bu_f^{\mathrm{S}}}\|_{l^2(\mathbf{H}^1)}$ & rate & $\|e_{p^{\mathrm{S}}}\|_{l^2(L^2)}$ & rate & $\|e_{\bu_r}\|_{l^2(\mathbf{L}^2)}$ & rate & $\|e_{p^{\mathrm{P}}}\|_{l^2(L^2)}$ & rate \\
\midrule 
$1107$ & $0.2795$ & $ 0.214600 $ & $ -  $& $28.57000 $ & $-$ & $4.998000 $ & $ -$  & $0.375800$ & $-$\\
$3995$ & $0.1398$  & $ 0.035700 $ & $  2.586 $& $0.900800 $ &$4.987$ & $0.136400 $ & $5.196$ & $0.060180$ & $2.643$ \\
$15147$& $ 0.0699$ &$ 0.009108$ & $ 1.972 $&	$0.335000 $ & $1.427$ & $0.045020 $ & $1.599$ & $0.016930$ & $1.830$ \\
$58955$& $0.0349$ &  $ 0.002272$ & $  2.003 $& $0.047570$ & $2.816$ & $0.007665$ & $2.554$ & $0.004460$ & $1.924$ \\
$232587$ &$0.0175$ & $0.000567$ & $ 2.003 $& $0.006113$ & $2.960$ & $0.001311 $ & $2.547$ & $0.001058$ & $2.075$ \\
\bottomrule 
\end{tabular}
\bigskip 

\begin{tabular}{r|c|c|c|c|c|c|c}
\toprule 
DoFs & $h$ & $\|e_{\by_s^{\mathrm{P}}}\|_{l^2(\mathbf{H}^1)}$ & rate & $\|e_{\bu_s^{\mathrm{P}}}\|_{l^2(\mathbf{L}^2)}$ & rate & $\|e_{\lambda}\|_{l^2(H^{-1/2})}$ & rate  \\
\midrule 
$1107$ & $0.2795$ & $ 0.454600 $ & $ -  $& $0.241100 $ & $-$ & $13.81000 $ & $ -$  \\
$3995$ & $0.1398$  & $ 0.181500$ & $  1.325 $& $0.048640$ &$2.309$ & $0.297700 $ & $5.536$ \\
$15147$&$ 0.0699$ &$ 0.048740$ & $  1.897 $&	$0.012250 $ & $1.989$ & $0.080220 $ & $1.892$ \\
$58955$&$0.0349$ &  $ 0.012440$ & $  1.970 $& $0.003070$ & $1.996$ & $0.007744$ & $3.373$ \\
$232587$ &$0.0175$ & $ 0.003128$ & $ 1.991 $& $0.000763$ & $2.009$ & $0.000654 $ & $3.566$ \\
\bottomrule 
\end{tabular}
\caption{Experimental errors related to spatial discretization and convergence rates are computed for the approximate solutions $\bu_f^{\mathrm{S}},p^{\mathrm{S}},\bu_r^{\mathrm{P}},p_h,\by_s^{\mathrm{P}},\bu_{s}^{\mathrm{P}}$ and $\lambda_h$, using $\mathbb{P}_2^2-\mathbb{P}_1-\mathbb{P}^2_2-\mathbb{P}^1-\mathbb{P}^2_2-\mathbb{P}_1^2-\mathbb{P}_1$. The computations are performed at the last time step.} 
\label{conv_test_space}
\end{table}

The backward Euler method is assessed for time convergence and verified by partitioning the time interval $(0, 1)$ into successively refined uniform discretizations and computing cumulative errors
$$
\hat{e}_s = \bigg( \sum_{n=1}^N \tau  \| s(t_{n+1}) - s_h^{n+1}\|^2_{\star}\bigg)^{1/2},
$$
where $\| \cdot \|_{\star}$ denotes the appropriate space norm for the generic vector or scalar field $s$. For this test we use a fixed mesh involving $923915$ DoFs. The results are shown in Table \ref{conv_test_time_M}, confirming the expected first-order convergence.

    \begin{table}[t!]
    \centering
  \begin{tabular}{r|c|c|c|c|c|c|c|c|c|c}
    \toprule 
   	$\tau $   & $\hat{e}_{\bu_f^{\mathrm{S}}}$ & rate & $\hat{e}_{p_f}$ & rate & $ \hat{e}_{\bu_r} $ & rate & $\hat{e}_{p_h}$ & rate & $\hat{e}_{\by_s^{\mathrm{P}}}$ & rate \\
    \midrule 
$0.5$  & $ 0.0293 $ & $  - $& $ 0.5374$ & $-$ & $ 0.0514$ & $ -$ & $ 0.1013$ & $ -$ & $1.3238$ & $-$\\
$0.225$  & $ 0.0146 $ & $  1.005 $ & $ 0.2688 $ & $1.000$ & $ 0.0231$ & $1.150$ & $ 0.0504$ & $ 1.008$ & $0.5736$ & $1.207$\\
$ 0.125$ & $0.0073$ & $  1.000$& $ 0.1345 $ & $0.990$ & $ 0.0109 $ & $ 1.078 $ & $ 0.0251 $ & $ 1.003$ & $0.2645$ & $1.117$\\
 $0.0625$ & $0.0036 $ & $  0.999$& $ 0.0673 $ & $0.998$ & $ 0.0053$ & $ 1.033$ & $ 0.0126$ & $ 1.001$ & $0.1266$ & $1.063$\\
 $ 0.03125$ & $0.0018$ & $  0.998$& $ 0.0337 $ & $0.998$ & $ 0.0027$ & $ 1.001 $ & $ 0.0063 $ & $ 0.999$ & $0.0619$ & $1.033$\\
 $0.015625$ & $0.0009 $ & $  0.997$ & $ 0.0169 $ & $0.997$ & $ 0.0014$ & $ 0.967$ & $ 0.0031$ & $ 0.996$ & $0.0306$ & $1.018$\\
  	 \bottomrule 
  			\end{tabular}
\bigskip 

\begin{tabular}{r|c|c|c|c}
    \toprule 
$ \tau $ & $ \hat{e}_{\bu_s^{\mathrm{P}}}$ & rate & $\|\hat{e}_{\lambda}\|_{l^2(H^{-1/2})}$ & rate  \\
    \midrule 
$0.5$ & $ 0.1458  $ & $-$ & $0.1505$ & $-$ \\
$0.225$  & $ 0.0729$ & $1.000$ & $0.0749$ & $1.007$ \\
$ 0.125$ & $ 0.0365 $ & $1.000$ & $0.0374$ & $1.003$ \\
 $0.0625$ & $ 0.0182 $ & $1.000$ & $0.0187$ & $1.002$ \\
 $ 0.03125$ & $ 0.0091 $ & $1.000$ & $0.0093$ & $1.000$\\
 $0.015625$ & $ 0.0046 $ & $0.999$ & $0.0047$ & $1.000$ \\
    \bottomrule 
\end{tabular}
\caption{Experimental cumulative errors associated with the temporal discretization and convergence rates for the approximate solutions $\bu_f^{\mathrm{S}},p^{\mathrm{S}},\bu_r^{\mathrm{P}},p_h,\by_s^{\mathrm{P}}, \text{ and } \bu_{s}^{\mathrm{P}}$, using a backward Euler scheme.}
    \label{conv_test_time_M}
  		\end{table}

\subsection{Simulation of subsurface fracture flow}
This test illustrates the applicability of the formulation in hydraulic fracturing and problem setup is similar to \cite{MR3904522,MR4353225}. The physical units are meters for length, seconds for time, and KPa for pressure.  Consider a rectangular domain $\Omega$ with dimensions $(0, 3.048) \times (0, 6.096)$, which contains a relatively large fracture known as a macro void or open channel $\Omega_{\mathrm{S}}$ filled with an incompressible fluid, and the poro-hyperelastic domain is defined as $\Omega_{\mathrm{P}} = \Omega \setminus \Omega_{\mathrm{S}}$. The permeability $\kappa$ and porosity $\phi$ are heterogeneous but isotropic in the $xy$-plane, and derived from the non-smooth pattern found in the Society of Petroleum Engineers $\mathrm{SPE-10}$ benchmark data/model 2\footnote{www.spe.org/web/csp}. We rescale this pattern as in \cite{MR3851065} and map it onto a piecewise constant field using an unstructured triangular mesh for the poro-hyperelastic region. Note that the present formulation requires smooth porosity. Therefore, we project both the porosity and permeability data onto a $\mathbb{P}_1$ field, as visualized in Figure \ref{figD2}.  There are 85 distinct layers within two general categories. We choose layer 80 from the dataset, which corresponds to the Upper Ness region exhibiting a fluvial fan pattern (flux channels of higher permeability and porosity). No gravity and no external loads are considered and the unstructured triangular mesh has $1629$ elements for the Stokes region and $18897$ elements for the poro-hyperelastic domain. We introduce a minor modification to the PDEs of the poroelastic region. Specifically, we incorporate $\mu_f \phi^2 \kappa_{f}^{-1} \bu_r$ instead of $\phi^2 \kappa_{f}^{-1} \bu_r$.

The boundary conditions are chosen to be 
$$
\begin{array}{lll}
\text { Injection: } & \bu_f^{\mathrm{S}}\cdot \bn_{\mathrm{S}}=10, \quad \bu_f^{\mathrm{S}}\cdot \btau_f=0 & \text { on } \Sigma_{\mathrm{inflow}}, \\
\text { stress free: } & \left(\bsigma_s^{\mathrm{P}} \bn_{\mathrm{P}}\right) = \cero, \left(\bsigma_f^{\mathrm{P}} \bn_{\mathrm{P}}\right) =\cero & \text { on }\Sigma_{\mathrm{left}}, \\
\text { Normal relative velocity: } & \bu_r^{\mathrm{P}} \cdot \bn_{\mathrm{P}}=0  & \text { on } \Sigma_{\mathrm{bottom}} \cup \Sigma_{\mathrm{right}} \cup \Sigma_{\mathrm{top}} , \\
\text { Normal displacement: } & \by_s^{\mathrm{P}} \cdot \bn_{\mathrm{P}}=0 & \text { on } \Sigma_{\mathrm{bottom}} \cup \Sigma_{\mathrm{right}} \cup \Sigma_{\mathrm{top}}, \\
\text { Shear traction: } & \left(\bsigma_s^{\mathrm{P}} \bn_{\mathrm{P}}\right) \cdot \btau_{f,j}=0, \left(\bsigma_f^{\mathrm{P}} \bn_{\mathrm{P}}\right) \cdot \btau_{f,j}=0 & \text { on } \Sigma_{\mathrm{bottom}} \cup \Sigma_{\mathrm{right}} \cup \Sigma_{\mathrm{top}}.
\end{array}
$$
The initial conditions are set accordingly to $\bu_r^{\mathrm{P}}(\cero)=\cero,\quad \by_s^{\mathrm{P}}(\cero)=\cero,  \quad \bu_s^{\mathrm{P}}(\cero)=\cero \text { and } p^{\mathrm{P}}(0)= 0$. The total simulation time is $T=10 \text { hours}$ and the time step is $\tau =30 \mathrm{~s}$. The model parameters are taken as
$$
 \nu = 0.2, \quad \mu_f = 10^{-6}, \quad \alpha_{\mathrm{BJS}} = 1.0, \quad \rho_f = 1000, \quad \rho_p = 1016, \quad \theta = 0, \quad c_0 = 6.89 \times 10^{-2}, \quad K= (1-\phi)^2/c_0 
$$
These parameters are realistic for hydraulic fracturing and are similar to the ones used in \cite{MR3347244, MR3302293}.

 The Lam\'e coefficients and bulk modulus are determined from the Young's modulus $E$ and the Poisson's ratio $\nu$ via the relationships $\lambda_p=E \nu /[(1+\nu)(1-2 \nu)] \text { and } \mu_p=E /[2(1+\nu)]$ . Given the porosity $\phi$, the Young’s modulus is determined from the law $E(\bx) = 10^7 (1 - 2\phi(\bx))^{2.1}$. 
 
For this test, we use the Taylor--Hood $\mathbb{P}_2^2-\mathbb{P}_1$ elements for the fluid velocity and pressure in the fracture region, and $\mathbb{P}_2^2-\mathbb{P}_1-\mathbb{P}_2^2-\mathbb{P}_1^2$ elements for the relative velocity, pressure, solid displacement and solid velocity in the porous medium, and  continuous $\mathbb{P}_1$ elements for the Lagrange multiplier. Snapshots of the approximate solutions (relative velocity, solid velocity, solid displacement, pressure in the poro-hyperelastic region, and fluid pressure and velocity magnitude in the Stokes region) for fluid injection into a fracture porous medium using the SPE10-based benchmark test are shown in Figure \ref{figD3}. We observe that the most of the leak-off occurs through the fracture tip, where the relative velocity in the porous medium is greatest in a channel-like high-permeability region near the tip. The injected fluid causes an increase in pressure at the interface, and the anticipated channel-like progressive filtration from the Stokes domain to the poro-hyperelastic domain can be appreciated in the porous pressure plot on the left panel of Figure \ref{figD3}, demonstrating higher values near the fracture tip. \cmag{We note that the model we are using allows for using a spatially varying porosity within the equation, which leads to a leakage near the base of the injection pointing downwards. A similar effect is seen at nearly two-thirds of the tip, with an upwards flow.} 

\begin{figure}[t!]
\centering
\subfloat[Porosity]{\includegraphics[height=6.12cm, width=4.20cm]{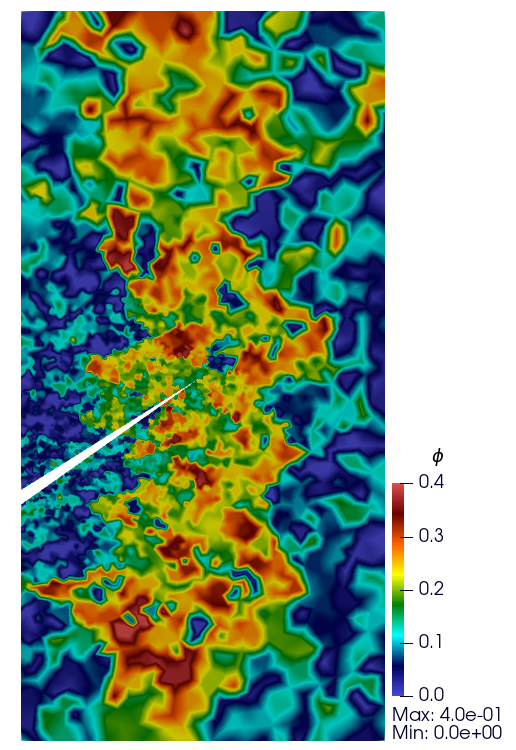}}
\subfloat[Permeability]{\includegraphics[height=6.12cm, width=4.20cm]{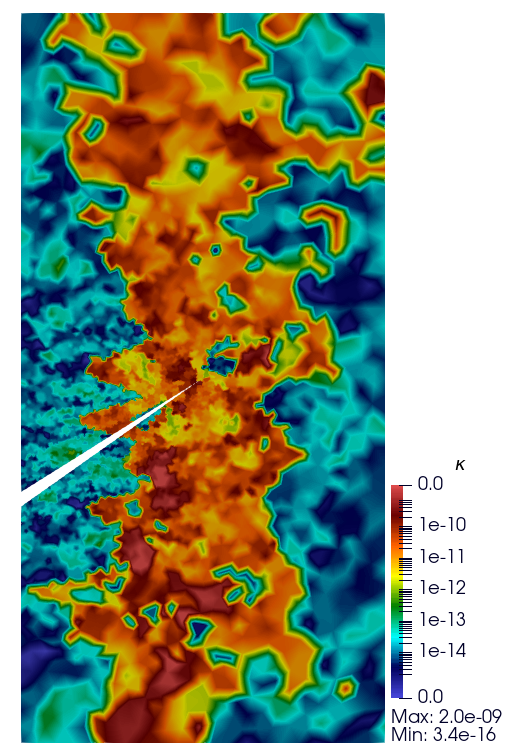}}
\subfloat[Young's modulus]{\includegraphics[height=6.12cm, width=4.20cm]{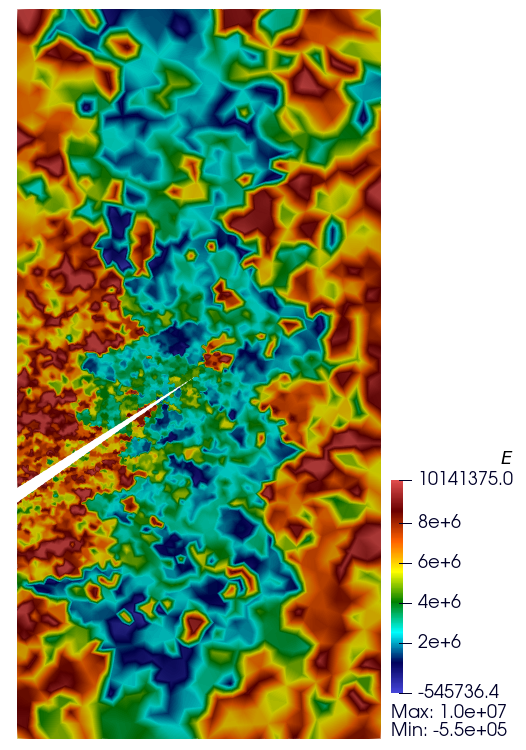}}
\caption{Material properties (porosity $\phi(\bx)$, permeability $\kappa(\bx)$, and Young modulus $E(\bx)$) from layer 80 of the $\mathrm{SPE10}$ benchmark
dataset for reservoir simulations, herein projected onto a $\mathbb{P}_1$ field for the poro-hyperelastic sub-domain.}
\label{figD2}
\end{figure}

\begin{figure}[t!]
\centering
\subfloat[porous pressure]{\includegraphics[height=6.12cm, width=4.20cm]{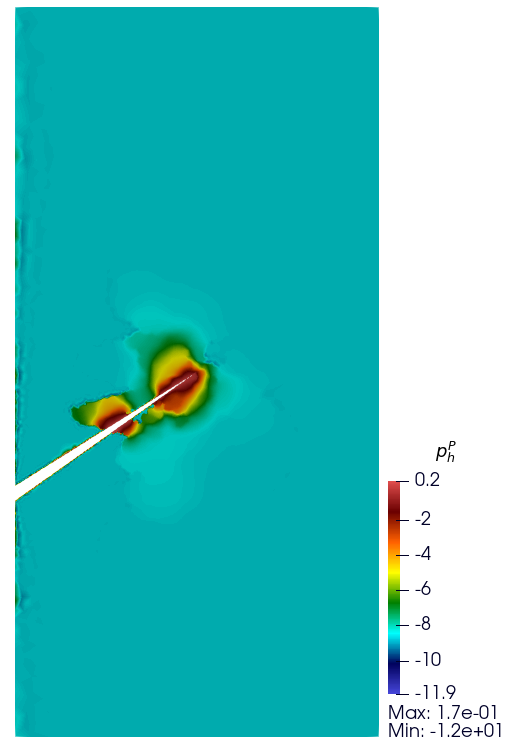}}
\subfloat[relative velocity]{\includegraphics[height=6.12cm, width=4.20cm]{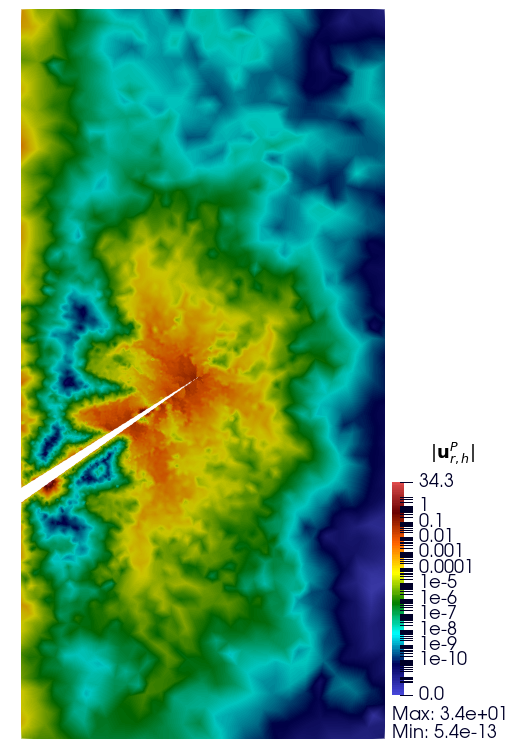}}
\subfloat[solid displacement]{\includegraphics[height=6.12cm, width=4.20cm]{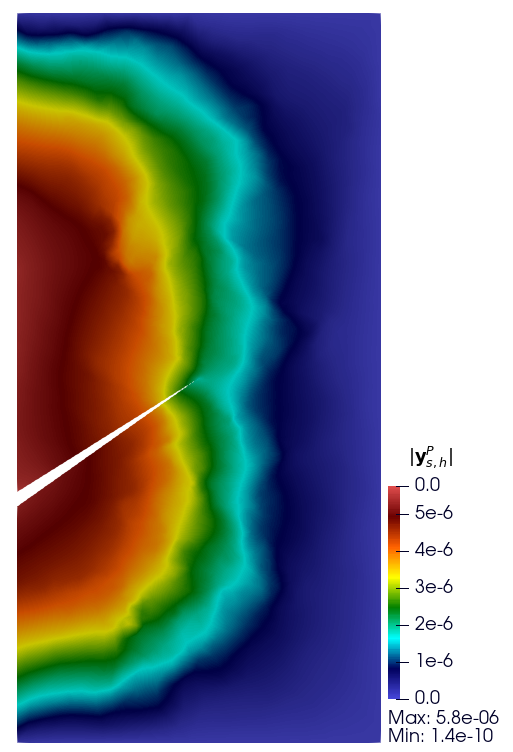}}
\subfloat[solid velocity]{\includegraphics[height=6.12cm, width=4.20cm]{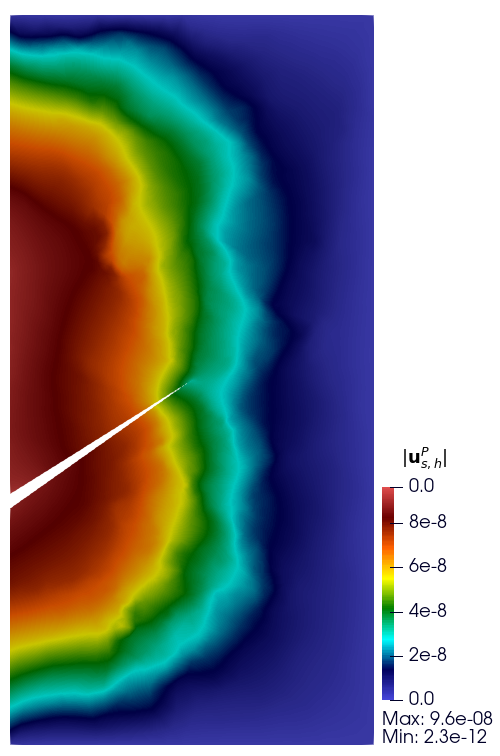}}\\
\subfloat[Stokes pressure]{\includegraphics[height=4.5cm, width=5cm]{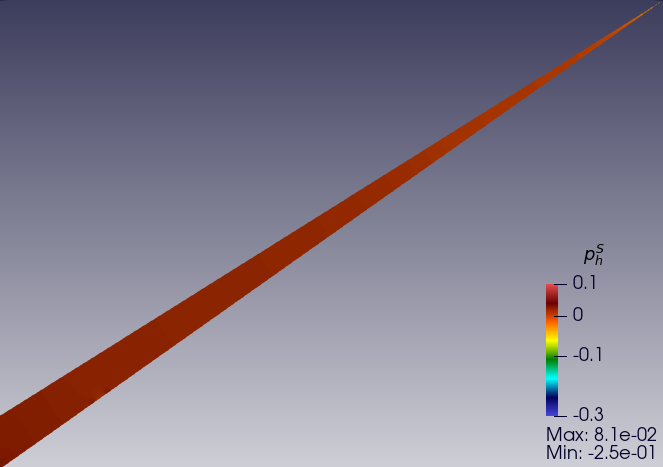}}
\hspace{1cm}
\subfloat[Stokes velocity]{\includegraphics[height=4.5cm, width=5cm]{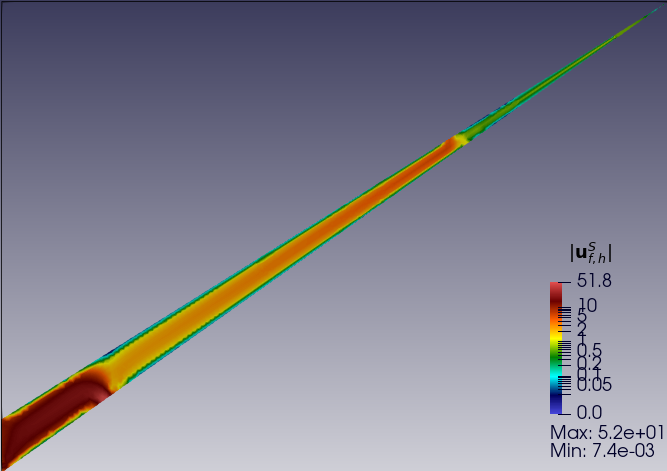}}
\caption{Snapshots of the approximate solutions for fluid injection into a fracture porous medium using the SPE10-based benchmark test.}
\label{figD3}
\end{figure}

\subsection{Channel filtration and stress build-up on interface deformation}
We continue our set of numerical simulations addressing  an important practical problem. We adopt a similar setup as in \cite{MR4353225}. Consider the domains $\Omega_{\mathrm{S}}=(-1,1) \times(0,2), \Omega_{\mathrm{P}}=(-1,1) \times(-2,0)$, separated by the interface $\Sigma=(-1,1) \times\{0\}$. Even if the present model is stated in the limit of small strains, it is possible to have large displacements, likely located near the interface (and without violating the model assumptions). In this scenario, 
 
we smoothly move the fluid domain and the fluid mesh to avoid distortions generated near the interface. For this we use a standard harmonic extension that is solved at each time step: Find $\mathbf{d}_h^*=\mathbf{d}_h+\widehat{\mathbf{d}}_h$ such that
$$
- \Delta \widehat{\mathbf{d}}_h=\mathbf{0} \quad \text { in } \Omega_S, \quad \widehat{\mathbf{d}}_h=\mathbf{d}_h \text { on } \Sigma, \quad \nabla \widehat{\mathbf{d}}_h \cdot \bn_S = \cero  \text { on  walls of } \Omega_S, \quad \text { and } \quad \widehat{\mathbf{d}}_h=\mathbf{0} \text { on \cmag{the} inlet}.
$$

\cmag{Note that, in contrast to \cite{MR4353225}, we are using a homogeneous Neumann boundary condition at the sides. Then}, we perform an $\mathbf{L}^2$-projection of both $\mathbf{d}_h$ and $\widehat{\mathbf{d}}_h$ into $\mathbf{V}_{s,h}+\mathbf{V}_{f,h}$ and add them to obtain the global displacement $\mathbf{d}_h^*$. We enforce the effect of deformation on the Stokes domain by considering that $\Omega_{\mathrm{S}}$ is a deformed configuration, and solve in that domain the pulled-back equations according to the deformation $\varphi(\mathbf x) = \mathbf x + \widehat{\mathbf d}$. 

We illustrate the effect of using harmonic extension by looking at the behavior of normal filtration into a 2D deformable porous medium. The boundary conditions is as follows, assuming that the flow is driven by pressure differences only. On the top segment we impose the fluid pressure $p^{\mathrm{S}}_{\text {in }}=2 \sin ^2(\pi t)$, and on the outlet (the bottom segment) the fluid pressure $p^{\mathrm{P}}_{\text {out }}=0$.  On the vertical walls of $\Omega_{\mathrm{S}}$ we set $\bu_f^{\mathrm{S}}= \mathbf{0}$ while on the vertical walls of $\Omega_{\mathrm{P}}$ we set the slip conditions $\by_s^{\mathrm{P}} \cdot \bn_{\mathrm{P}}=0$ and $\bu_r^{\mathrm{P}} \cdot \bn_{\mathrm{P}}=0$. The model parameters (all adimensional) are taken as 
$$
\kappa=0.005, \quad \lambda_p=10, \quad \mu_p=5, \quad \rho_p=1.07, \quad \rho_f=1, \quad \alpha_{\mathrm{BJS}}=0.1, \quad \mu_f=0.8, \quad \theta = 0, \quad   
c_0 = 0.02, 
$$
$$
\quad \phi = 0.3, \quad K = (1-\phi)^2/c_0.
$$
We assume that there are no body forces or gravity acting on the system. This example uses the same FEs as before and the numerical results are presented in Figure \ref{figD4}. The effect of the interface is clearly seen in the poroelastic domain. 
Close to the interface, the relative velocity, the solid displacement, and the fluid pressure are heterogeneous in the horizontal direction
before recovering the expected constant value (constant in the horizontal direction) expected in the far field. Also, we plot the overall displacement in the domain.  From Figure \ref{figD7} (see panels (b) and (c)) one can see that for large enough interfacial displacements, the elements close to it exhibit a large distortion.

\begin{figure}[t!]
\centering
\subfloat[Stokes velocity]{\includegraphics[height=6.5cm, width=4.5cm]{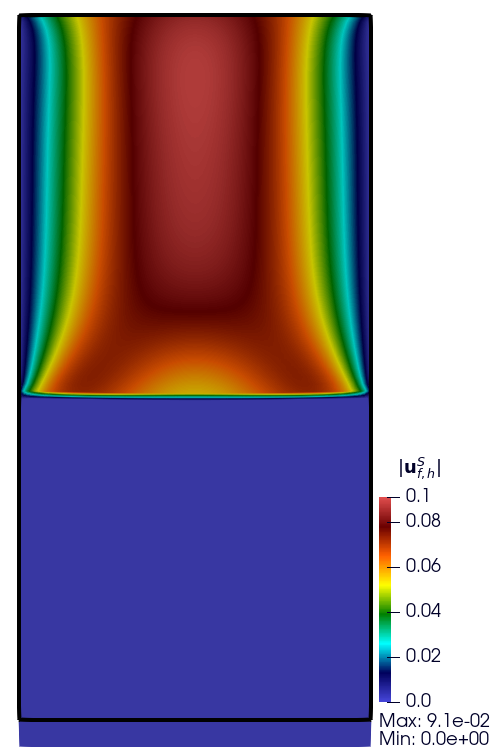}}
\subfloat[Stokes pressure]{\includegraphics[height=6.5cm, width=4.5cm]{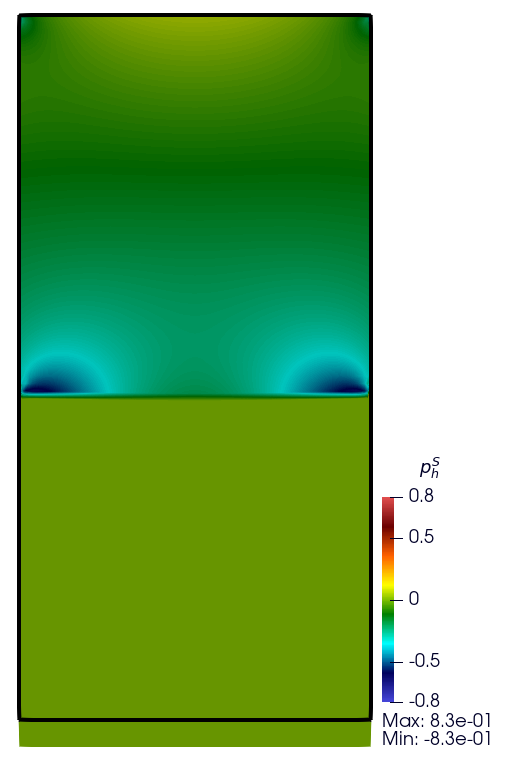}}
\subfloat[relative velocity]{\includegraphics[height=6.5cm, width=4.5cm]{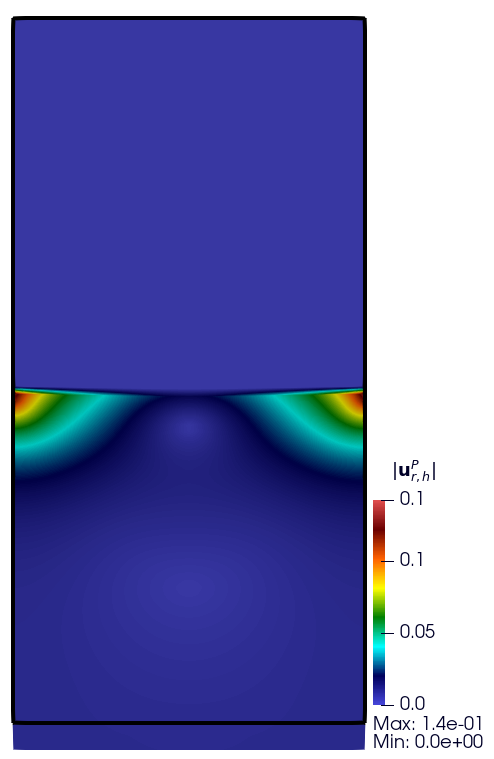}}\\
\subfloat[porous pressure]{\includegraphics[height=6.5cm, width=4.5cm]{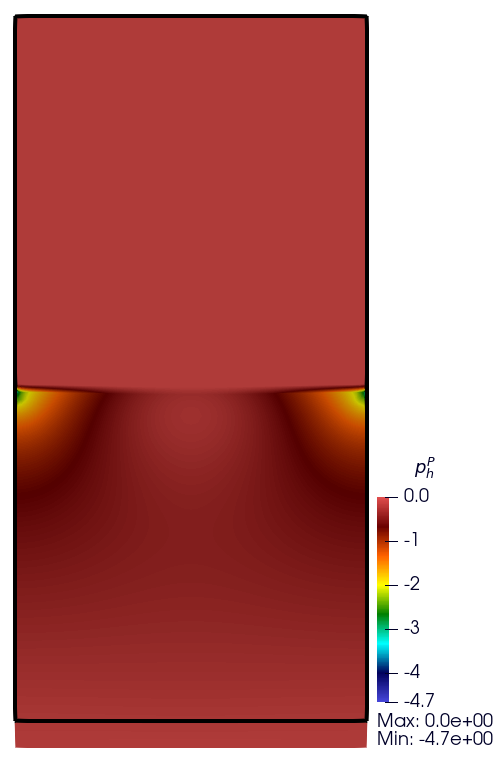}}
\subfloat[solid displacement]
{\includegraphics[height=6.5cm, width=4.5cm]{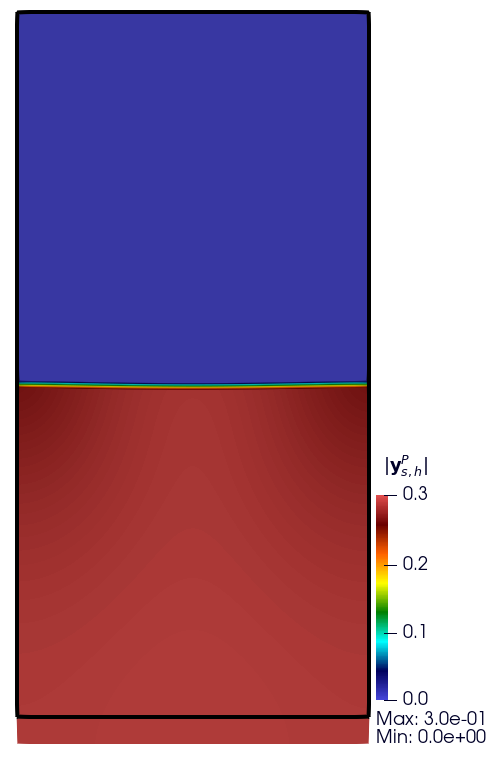}}
\subfloat[global displacement]
{\includegraphics[height=6.5cm, width=4.5cm]{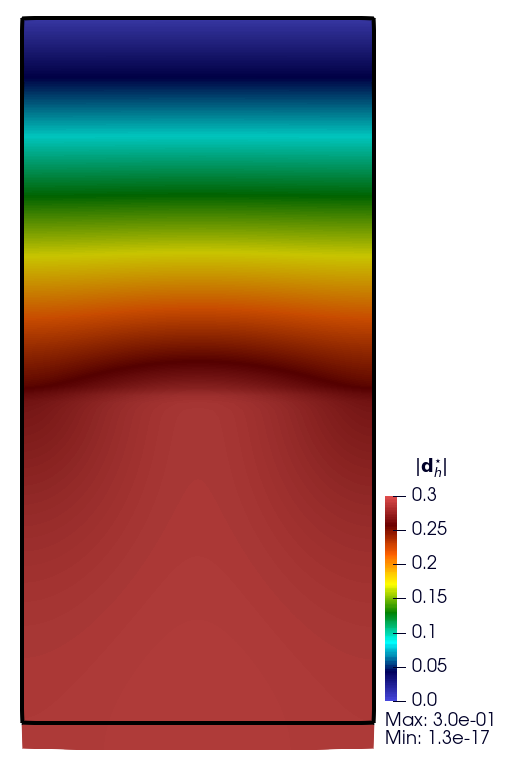}}
\caption{Filtration into a deformable porous medium. All snapshots are taken at time
$t = 2$ with $\tau = 0.1$, and the black outer line indicates the location of the undeformed domain.}
\label{figD4}
\end{figure}

\begin{figure}[t!]
\centering
\subfloat[]
{\includegraphics[height=2.65cm, width=5cm]{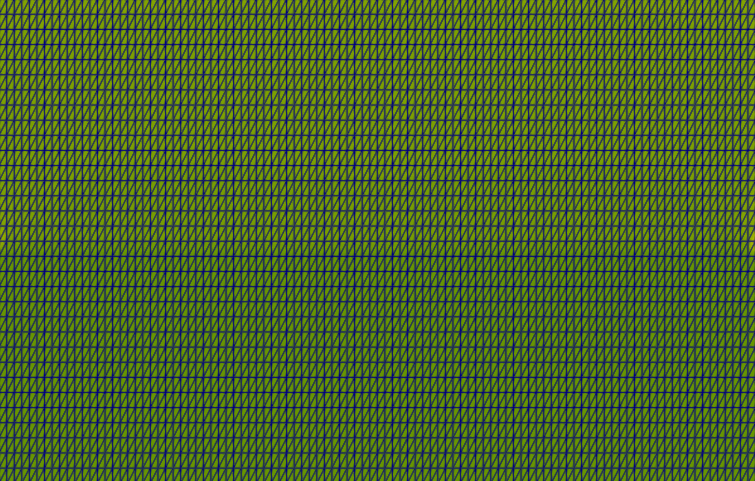}}
\hspace{0.3cm}
\subfloat[]
{\includegraphics[height=2.65cm, width=5cm]{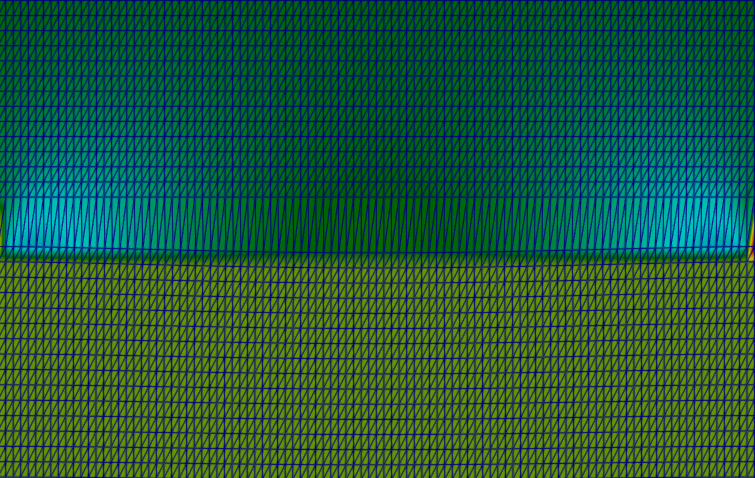}}
\hspace{0.3cm}
\subfloat[]
{\includegraphics[height=2.65cm, width=5cm]{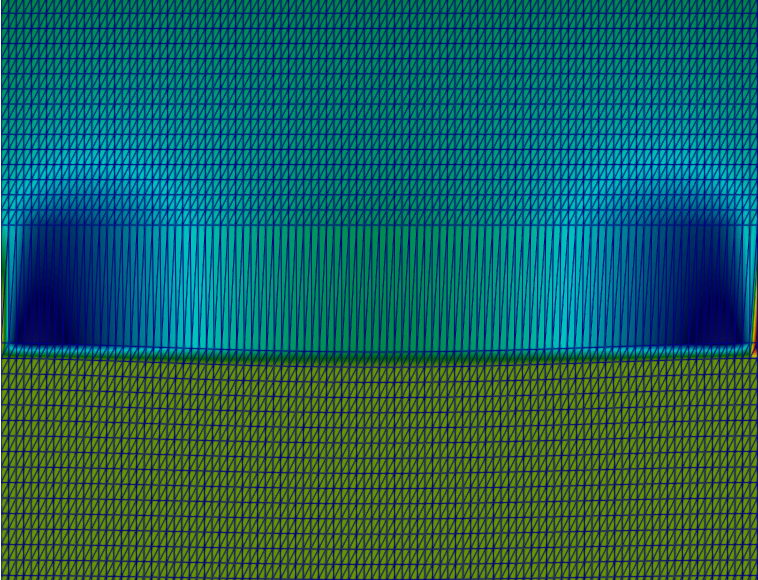}}
 \\
 \subfloat[]
{\includegraphics[height=2.65cm, width=5cm]{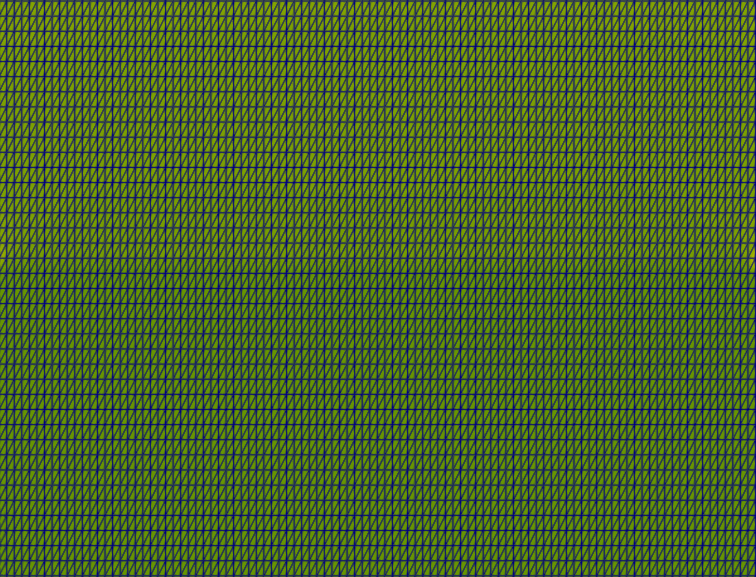}}
\hspace{0.3cm}
\subfloat[]
{\includegraphics[height=2.65cm, width=5cm]{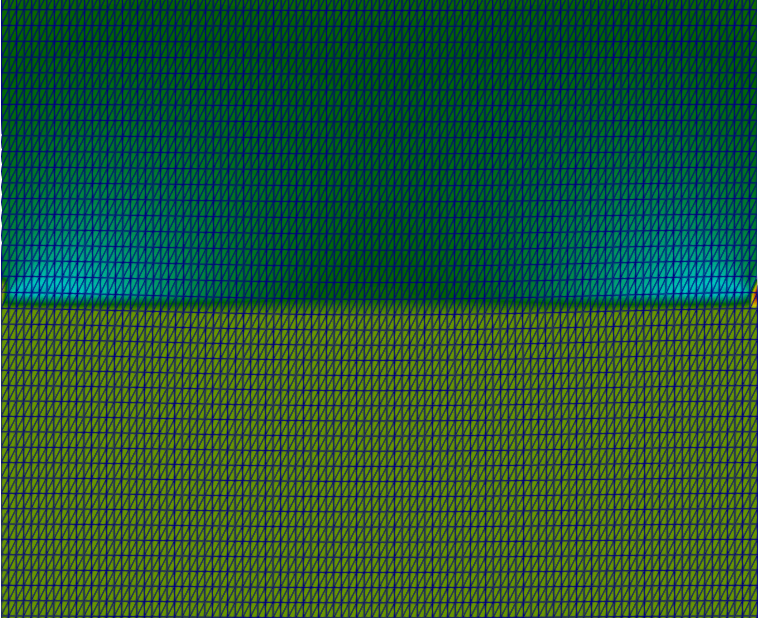}}
\hspace{0.3cm}
\subfloat[]
{\includegraphics[height=2.65cm, width=5cm]{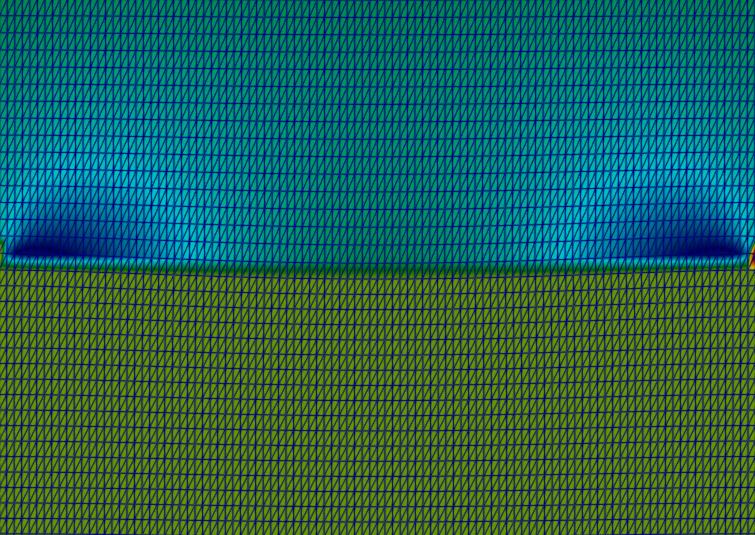}}
\caption{Zoom of the meshes on the interface at times $t = 0$, $t = 1$, and $t = 2$. Effect of using or not  the harmonic extension to move the fluid domain (bottom and top, respectively). }
\label{figD7}
\end{figure}

\subsection{Interfacial ﬂow in the brain}
To conclude this section, we present a \cmag{2D} simulation related to brain biomechanics. More specifically, \cmag{we investigate how  the incoming cerebrospinal fluid (CSF) flow from the spinal canal effects the brain tissues. In other words, we can say that the chosen problem is motivated by real-world applications like modeling the glymphatic system, where the brain porous tissue interacts with the surrounding CSF. For example, a heartbeat creates pressure waves in the CSF around the brain, which then spread through the brain.} A number of mechanical processes constantly affect the brain function: blood entering and leaving, fluid movements such as CSF and interstitial fluid in and around the brain and spine, pressures inside the skull, brain tissue shifts, and fluid flow between cells, for example. 

Following \cite{boon22}, we represent the brain parenchyma  as a 2D poro-hyperelastic sub-domain  $\Omega_P$, and the surrounding CSF-filled spaces as a free fluid (Stokes) sub-domain denoted by $\Omega_S$. These sub-domains share a common boundary $\Sigma = \Omega_S \cap \Omega_P$ with normal vector $\bn_S$, pointing from $\Omega_S$ to $\Omega_P$ on $\Sigma$ and outwards on the boundary $\partial \Omega_S$. The following values for the material parameters are adopted from \cite{boon22,m2an23,Stoverud2016} 
$$
\mu_f = 7 \times 10^{-7}, \quad \alpha_{\mathrm{BJS}} = 1, \quad c_0 = 2 \times 10^{-5}, \quad \phi = 0.2, \quad K = (1 - \phi)^2/ c_0, \quad
\kappa_f = 1 \times 10^{-8}, \quad \mu_p = 267 \times 10^{-3},
$$
$$
\lambda_p = 26488 \times 10^{-3},\quad \theta = 0, \quad
\rho_f = 1.005, \quad \rho_p = 1.03.
$$
The traction boundary conditions are applied as follows: at the top right $(\bsigma_f^S \cdot \bn_S = (10,10))$ and bottom right $(\bsigma_f^S \cdot \bn_S = (1,1))$ regions in the axial slices of the brain, and at the top $(\bsigma_f^S \cdot \bn_s = (10,10))$ region in the coronal slices of the brain. A homogeneous Dirichlet boundary condition is imposed on the remaining parts of the boundary. The snapshots of the approximate solutions (interstitial fluid pressure, interstitial fluid velocity, brain tissue displacement, brain tissue velocity, CSF pressure, CSF velocity) for the axial and coronal slices, which illustrate the interfacial flow in the brain, are shown in Figures \ref{figD10} and \ref{figD11}, respectively. 

\cmag{For the axial slices, the excess pore pressure that develops in the parenchyma drains through the entire interface. However, under the flow and loading rate regime being considered, the localization of Stokes fluid pressure and Stokes velocity (determined by boundary conditions on the bottom left) results in interfacial flow patterns where the pore pressure gradient is much higher near the bottom-left region with larger velocity. The Biot fluid pressure then dissipates throughout the remainder of the deformable porous domain, and the permeating patterns align with the brain displacement. As the flow eventually reaches the interface near the boundary (on the top right), the displacement is significantly smaller compared to the displacement near the center of the parenchyma.}

\cmag{In the coronal slices, we observe that both fluid pressure distributions in the subarachnoid space (region filled with CSF) and brain parenchyma exhibit a higher gradient near the portion of the interface closer to the boundary at the top. This results in a very smooth displacement field with a relatively higher (but still mild) magnitude near the top-center part of the interface. On average, the fluid pressure in the subarachnoid space remains higher than in the parenchyma. In both types of slices, the interstitial fluid (ISF) velocity within the brain parenchyma remains generally low relative to CSF flow. We note that, since the application here is in 2D, we refrain from drawing any conclusions regarding the phenomena in 3D.}

\begin{figure}[t!]
\centering
\subfloat[porous pressure]{\includegraphics[height=6cm, width=6cm]{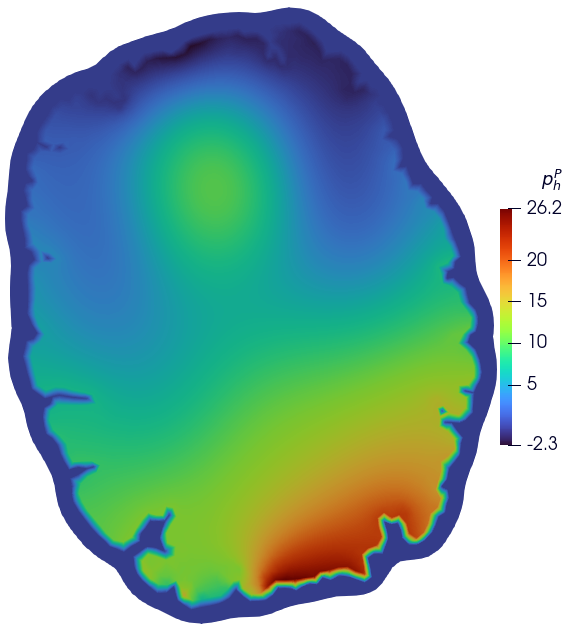}}
\subfloat[relative velocity]{\includegraphics[height=6cm, width=6cm]{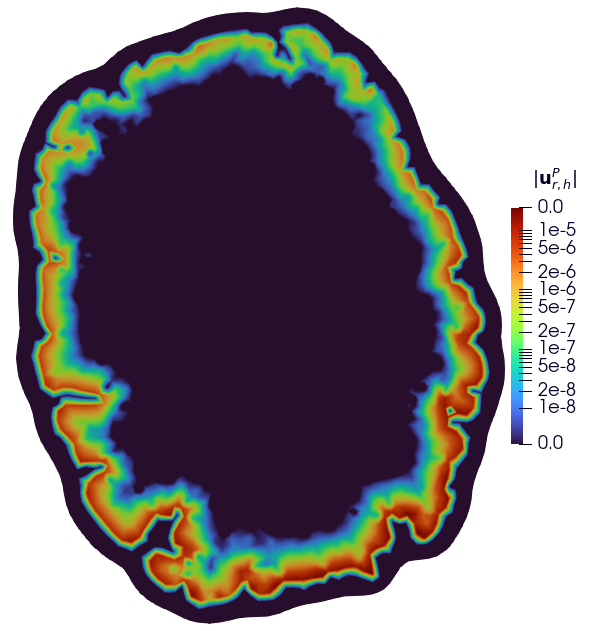}}
\subfloat[solid displacement]{\includegraphics[height=6cm, width=6cm]{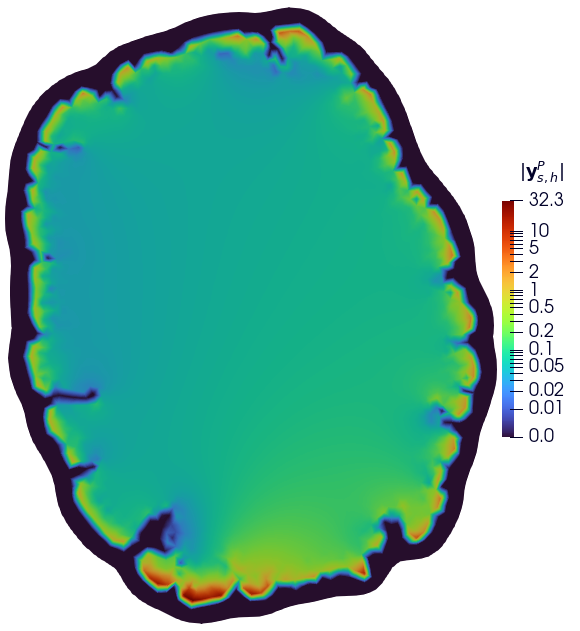}}\\
\subfloat[solid velocity]{\includegraphics[height=6cm, width=6cm]{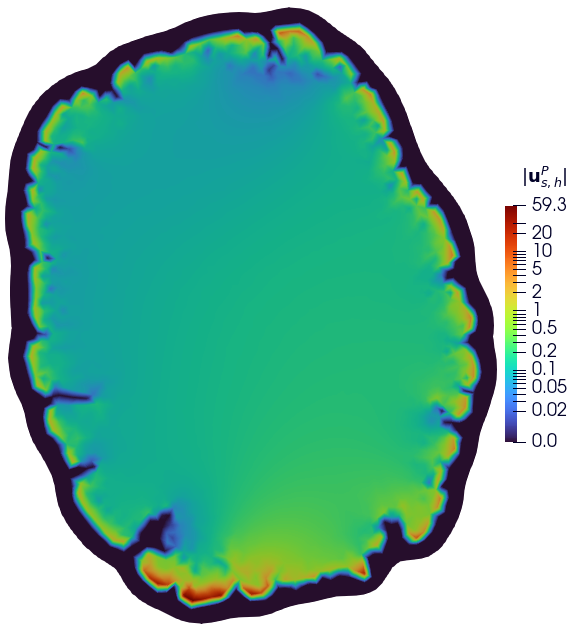}}
\subfloat[Stokes pressure]{\includegraphics[height=6cm, width=6cm]{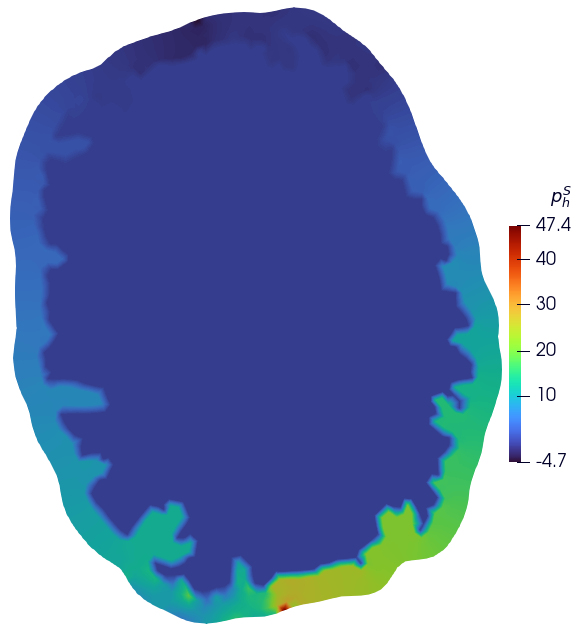}}
\subfloat[Stokes velocity]{\includegraphics[height=6cm, width=6cm]{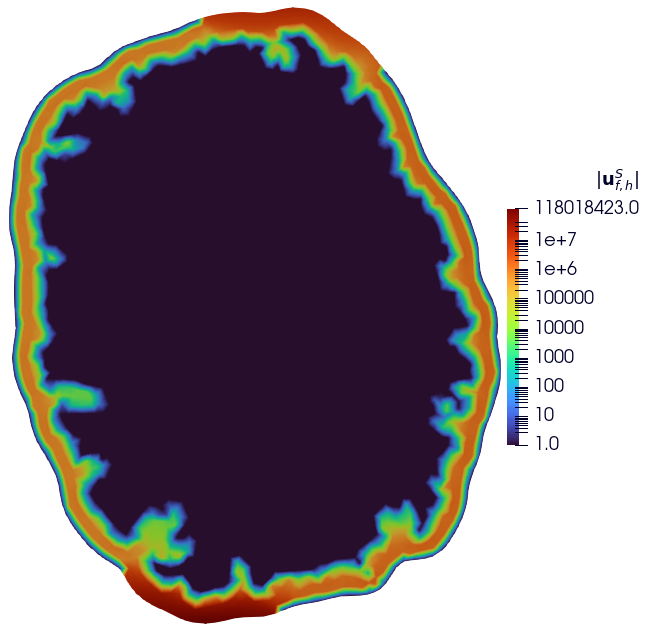}}
\caption{Snapshots of the approximate solutions for the interfacial flow in an idealized geometry at $T=1$ with $dt = 0.005$. The traction boundary conditions in the top right and bottom left corners (axial slices), respectively.} 
\label{figD10}
\end{figure}

\begin{figure}[t!]
\centering
\subfloat[porous pressure]{\includegraphics[width=6cm, height=4.65cm]{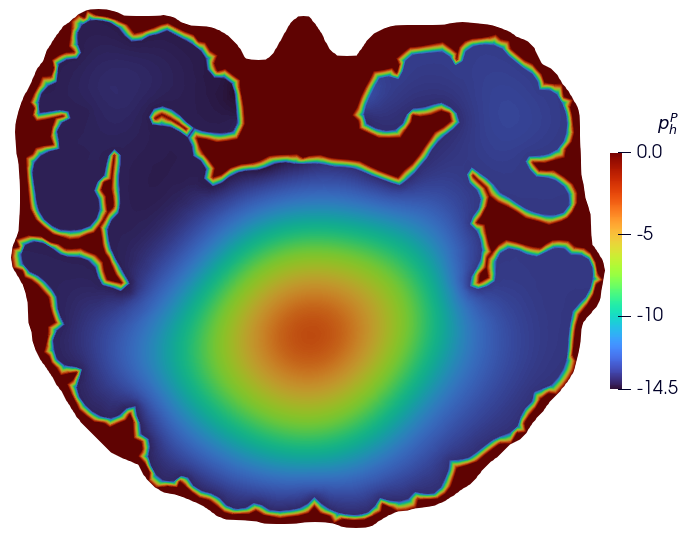}}
\subfloat[relative velocity]{\includegraphics[width=6cm, height=4.65cm]{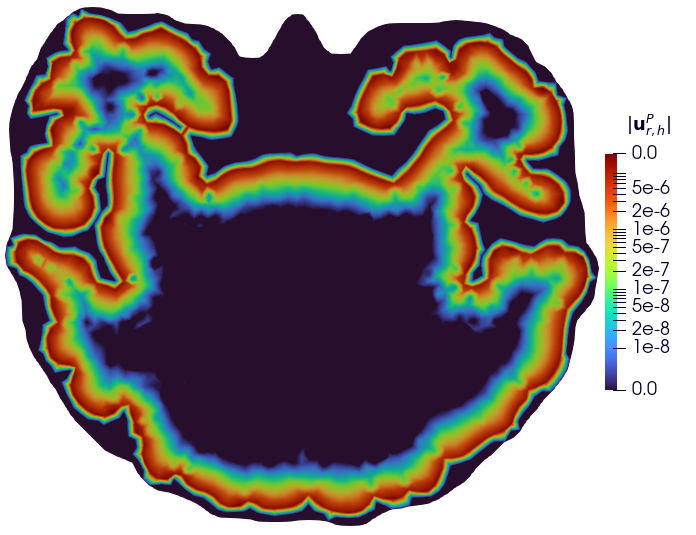}}
\subfloat[solid displacement]{\includegraphics[width=6cm, height=4.65cm]{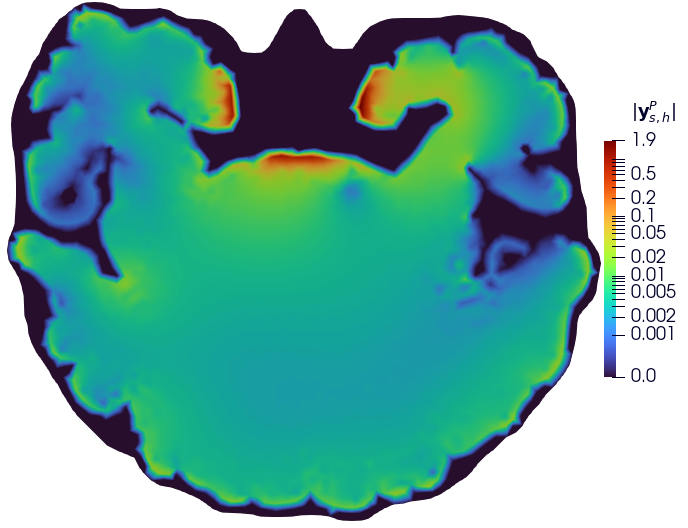}}\\
\subfloat[solid velocity]{\includegraphics[width=6cm, height=4.65cm]{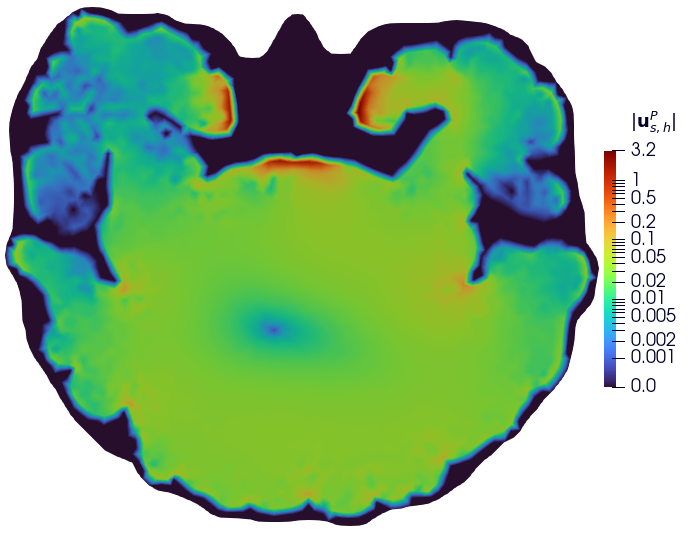}}
\subfloat[Stokes pressure]{\includegraphics[width=6cm, height=4.65cm]{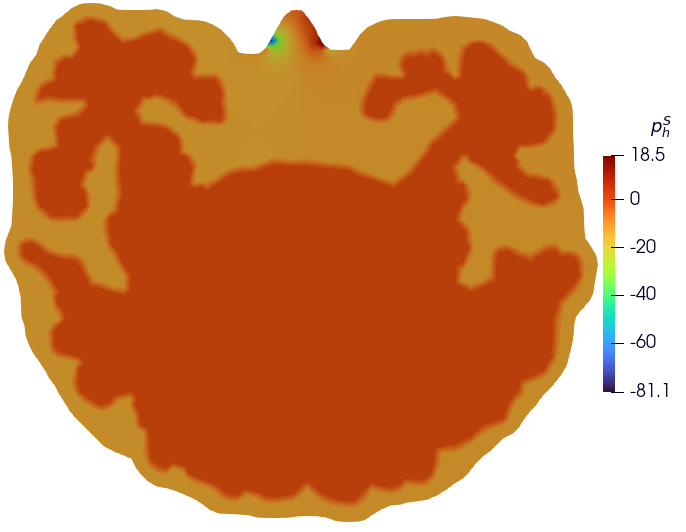}}
\subfloat[Stokes velocity]{\includegraphics[width=6cm, height=4.65cm]{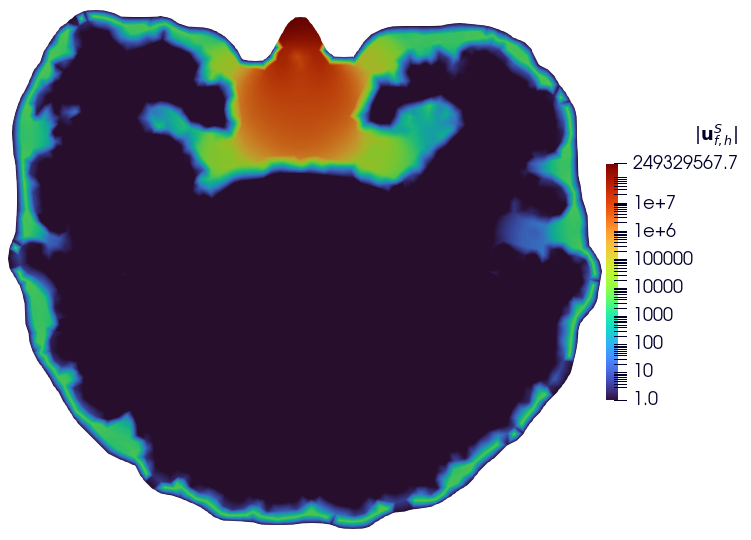}}
\caption{Snapshots of the approximate solutions for the interfacial flow in an idealized geometry at $T=1$ with $dt = 0.005$. The traction boundary conditions at the top (coronal slices), respectively.}
\label{figD11}
\end{figure}

\section{Conclusion}\label{sec:concl}
In this paper, we propose a model for the coupling between free fluid and a poro-hyperelastic body. This model is a novel contribution to the field of theoretical and numerical partial differential equations in interface coupled problems. We employed the Brinkman equation for fluid flow in the porous medium, incorporating inertial effects into the fluid dynamics. A Lagrange multiplier-based formulation is proposed and prove the existence and uniqueness of the continuous and discrete formulations. The proof of the existence relies on an auxiliary multi-valued parabolic problem. A priori error estimates for both the semi- and fully-discrete schemes are derived. 
Theoretically, we observe that both relative velocity and solid displacement exhibit sub-optimal convergence. The solid and relative velocity blocks in the diagonal of the system matrix makes it difficult to derive a bound in the energy norm.
We conducted numerical validation of spatio-temporal accuracy. The numerical convergence tests show a suboptimal convergence for relative velocity and optimal convergence for solid displacement in their respective norms. Additionally, we observe superconvergence for Stokes pressure. We also performed tests of applicative relevance, studying the behavior of poromechanical filtration in subsurface hydraulic fracture with challenging heterogeneous material parameters and channel filtration when stress builds up on interface deformation. The set of tests also includes a typical application in biomechanical modeling of the brain to study how  the incoming cerebrospinal fluid (CSF) flow from the spinal canal effects the brain tissues. These results can contribute in the prediction of important mechanisms in the overall brain function. Further perspectives of this work include the extension to the fully nonlinear regime, as well as other type of transmission conditions that would allow more generality in the type of poromechanical problems we can tackle. While the model and the continuous and discrete analyses are valid also in the 3D case, our numerical examples were only in 2D, due to the substantial increase in computational and memory requirements. Also, the interface in 3D brain geometries is significantly larger and more complex than in 2D cases, increasing the number of degrees of freedom at the interface. We therefore foresee the development of suitable parameter-robust block preconditioners.

\section*{Acknowledgments}
We kindly thank Dr Miroslav Kuchta for providing the brain slice meshes and the model data used in our last example. AB was supported by the Ministry of Education, Government of India - MHRD. NB received support from the ANID Grant \textsc{FONDECYT de Postdoctorado N° 3230326} and from Centro de Modelamiento Matematico (CMM), Proyecto Basal FB210005. RRB received  partial support from by  the Australian Research Council through the \textsc{Future Fellowship} grant FT220100496 and \textsc{Discovery Project} grant DP22010316. 



\bibliographystyle{siam}
\bibliography{references}

\end{document}